\documentclass[10pt]{article}

\usepackage{amstext}
\usepackage{amssymb}
\usepackage{amsmath}
\usepackage{amsthm}
\usepackage[dvips]{graphicx}        % макропакет для включения рисунков eps командой \includegraphicx{XXX.eps}
\graphicspath{{./pict/}}            % Путь к папке с рисунками

\newtheorem{theorem}{Theorem}[section]
\newtheorem{corol}[theorem]{Corollary}%[section]
\newtheorem{conjecture}[theorem]{Conjecture}
\newtheorem{problem}[theorem]{Problem}
\newtheorem{rrule}[theorem]{Rule}

\newcounter{notezz}[section]
\newtheorem{example}[theorem]{Example}
\newtheorem{lemma}[theorem]{Lemma}%[section]

\theoremstyle{definition}
\newtheorem{definition}[theorem]{Definition}%[section]

\theoremstyle{remark}\newtheorem{remark}[theorem]{Remark}%[section]

\DeclareMathOperator{\PR}{P\mathbb R}
\DeclareMathOperator{\Ind}{Ind}
\renewcommand{\Im}{\operatorname{Im}}
\renewcommand{\Re}{\operatorname{Re}}

\def\sgn{\mathop{\rm sign}\nolimits}
\newcommand{\IndI}{\Ind_\infty}
\newcommand{\IndC}{\Ind_{-\infty}^{+\infty}}
\newcommand{\IndPR}{\Ind_{\PR}}

\newcommand{\SC}{\operatorname{{\rm V}}}
\newcommand{\SR}{\operatorname{{\rm P}}}
\newcommand{\SCF}{\operatorname{{\rm V}}}
\newcommand{\SRF}{\operatorname{{\rm P}}}

\def\sgn{\mathop{\rm sign}\nolimits}
\def\eqbd{\mathop{{:}{=}}}
\def\bdeq{\mathop{{=}{:}}}
\def\eop{\hfill
        {\ \vbox{\hrule\hbox{\vrule height1.3ex\hskip0.8ex\vrule}\hrule}}
        \vskip 0.3cm \par}
\begin{document}
\vspace{\baselineskip}

\title{Structured matrices, continued fractions,
 and root localization of polynomials}

\author{Olga Holtz\thanks{The work of  O.H.
was supported by the Sofja Kovalevskaja Research Prize of Alexander von
Humboldt Foundation and by the National Science Foundation under agreement DMS-0635607.
Email: {\tt holtz@math.ias.edu} }  \\
 \small Department of Mathematics,  University of California-Berkeley \\
\small Institut f\"ur Mathematik, Technische Universit\"at Berlin  \\
\small School of Mathematics,   Institute for Advanced Study
 \and  Mikhail Tyaglov\thanks{The work of M.T. was supported by the
Sofja Kovalevskaja Research Prize of Alexander von Humboldt Foundation.
Email: {\tt tyaglov@math.tu-berlin.de}}  \\
\small Institut f\"ur Mathematik,  Technische Universit\"at  Berlin }

\date{\small December 23, 2009; last revision February 9, 2011}

\maketitle

%\begin{titlepage}
%\begin{center}
%\null \vskip 25mm {\sc \large Structured matrices, continued fractions,
% and root localization of polynomials} \\
%\vskip 6mm {\sc  Olga Holtz $\mathrm{and}$ %Yu.S. Barkovskiy,
%Mikhail Tyaglov}
%\date\today
%\end{center}
%
%\end{titlepage}

\begin{abstract}
We give a detailed account of various connections between  several classes of objects:
Hankel, Hurwitz, Toeplitz, Vandermonde and other structured matrices,  Stietjes and
Jacobi-type continued fractions,  Cauchy indices, moment problems, total positivity,
and root localization  of univariate polynomials. Along with a survey of many classical
facts, we provide a number of new results.
\end{abstract}

%\null \thispagestyle{empty}
\tableofcontents
%\vfill
%
%\eject

\setcounter{page}{1}

%%%%%%%%%%%%%%%%%%%%%%%%%%%%%%%%%%%%%%%%%%%%%%%%%%%%%%%%
\section*{Introduction}
%%%%%%%%%%%%%%%%%%%%%%%%%%%%%%%%%%%%%%%%%%%%%%%%%%%%%%%%

This survey and research paper offers a glimpse at several classical
topics going back to Descartes, Gauss, Stieltjes, Hermite, Hurwitz
and Sylvester (see, e.g.,~\cite{CM, Post, Fisk, Prasolov, Rah}), all connected
by the idea that behavior of polynomials can be analyzed  via algebraic
constructs involving their coefficients. Thus a number of inter-related
algebraic  constructs was built, including Hurwitz, Toeplitz and Hankel
matrices, the  corresponding quadratic forms, and the
corresponding continued fractions. The linear-algebraic properties
of these objects were shown to be intimately related to root localization
of polynomials such as stability, whereby the zeros of a polynomial
avoid a specific half-plane, or hyperbolicity, whereby the zeros lie
on a specific line. These methods gave rise to several well-known tests
of root localization, such as the Routh-Hurwitz algorithm or the
Lienard-Chipart test.

In the 20th century, this line of research was developed further by
Schur \cite{DymKats}, P\'olya \cite{PolyaII},  Krein \cite{Krein1, Krein2,
GantKrein} and others (see, e.g.,~\cite{Henrici}), leading to important  notions of total
positivity, P\'olya frequency sequences, stability preservers etc.
A very important part of that research effort was devoted to entire
functions, in particular, the Laguerre-P\'olya and related classes.
However, this area gradually went out of fashion and was essentially
abandoned around 1970-1980s, with a few exceptions, such as Pad\'e approximation.
We should note in this connection that the closely related theory of continued
fractions, initiated by Chebyshev, Stieltjes and Markov and further developed
by Akhiezer, Krein and their collaborators in connection with problems of
mechanics is only now returning to the forefront due to its  connections with
orthogonal polynomials \cite{Khruschev,Lorentzen_Waadeland,
Nik_Sor,Jones_Thron,Jones_Thron2}.

On the other hand, various fast algorithms for structured matrices (Cauchy,
Vandermonde, Toeplitz, Hankel matrices, and more generally quasiseparable 
matrices and matrices with displacement structure) have been developing rapidly
in the last few decades due to the efforts of Gohberg, Olshevsky, Eidelman,
Kailath, Heinig, Sakhnovich,  Lerer, Rodman, Fuhrmann and others,
along with applications to control theory and engineeting:
see the collections \cite{StructuredI,StructuredII,FastAlg,FastAlgStructure,HeinigVolume,
NumLinAlgControl} and  the references therein.
Surveying these developments is, however, outside the scope of this paper.   

Furthermore, we must stress the fact that the results and algorithms in this 
paper are not approximate but exact (i.e., give exact answers in exact 
arithmetic) as is the term \emph{localization}, which is often
used in the sequel. For instance, we perform root localization
with respect to a line or a half plane and are able to determine
\emph{exactly}, e.g., how many zeros of a given polynomial lie
on a given line. For sure, this does not obviate the need for error 
analysis of the corresponding algorithms that use finite precision or
floating-point arithmetic.

We must note that the questions of root localization
are returning into the spotlight also  due to the development of
basic algebraic techniques for   multivariate polynomials~\cite{BGL,Gul,B,BS,BBS,
Ren} and entire functions~\cite{BCC, CC4}, as well as due to newly found
applications of results on multivariate stable and hyperbolic polynomials
to other branches of mathematics \cite{Gur2, Gur3, Wag2, BBS}.  Very promising
applications of current interest include problems of discrete probability,
combinatorics, and statistical physics, such as the analysis of partition
functions arising in classical Ising and Potts models of statistical mechanics
\cite{Sok,ScottSokal}.

The goal of this paper is to provide a comprehensive and coherent
treatment of classical connections between three kinds of objects:
various structured matrices, representations of rational functions
via continued fractions, and root localization of univariate
polynomials. Our additional goal is to demonstrate that, despite
the rich history of this subject, even the univariate case is far
from being exhausted, and that classical algebraic techniques are
useful in answering questions about the behavior of polynomials
and rational functions.

Our future goals include using this work
for generalizations of these classical results to univariate
entire and meromorphic functions, including questions on P\'olya
frequency sequences, Hurwitz rational functions, and entire
functions with all real zeros~\cite{Tya,{Tya_Bark}}. One particular
area of interest is the class of so-called generalized
Hurwitz polynomials, which is a useful large class containing
all Hurwitz stable polynomials. This line of our ongoing research
is closely related to the theory of indefinite metric spaces,
indefinite moment problems, and the corresponding eigenvalue problems
\cite{Derkach,Piv,Piv_Wor,H4}.

We now illustrate our main point above by discussing several new results in this paper.

Our first illustration is provided by the body of work in Section~\ref{s:real.roots.poly.coeffs}.
These results provide explicit criteria for a polynomial to have only real roots in terms of
Hankel and Hurwitz determinants made of its coefficients. To give the reader an  idea of these
statements, we quote two sample theorems from Section~\ref{s:real.roots.poly.coeffs}:

Consider a polynomial
\begin{equation*}
p(z)=a_0z^n+a_1z^{n-1}+\cdots+a_n,\qquad a_1,\dots,a_n\in\mathbb
R,\ a_0>0.
\end{equation*}
and let
$$
L(z)\eqbd \dfrac{p'(z)}{p(z)}=\dfrac{s_0}z+\dfrac{s_1}{z^2}+\dfrac{s_2}{z^3}+\cdots
$$
be its logarithmic derivative. Define the associated infinite Hankel matrix
$$ S\eqbd S(L)\eqbd  \| s_{i+j} \|_{i,j=0}^\infty $$ and its leading principal minors
\begin{equation*}
D_j(S)\eqbd D_j(S(L))\eqbd
\begin{vmatrix}
    s_0 &s_1 &s_2 &\dots &s_{j-1}\\
    s_1 &s_2 &s_3 &\dots &s_j\\
    \vdots&\vdots&\vdots&\ddots&\vdots\\
    s_{j-1} &s_j &s_{j+1} &\dots &s_{2j-2}
\end{vmatrix},\quad j=1,2,3,\dots.
\end{equation*}
%
%Here is our first criterion from Section~\ref{s:real.roots.poly.coeffs} that $p$ has only real roots
% in terms of its  Hankel determinants: \vskip 2mm
%
%\noindent {\bf Theorem (Hankel criterion for real zeros).\/}
%\textit{ The polynomial $p$ has exactly $m\,(\leq n)$
%distinct zeros all of which are real if and only if
%
%\begin{eqnarray*}
%& D_j(S(L))>0, & \quad j=1,2,\ldots,m, \\
%
%& D_j(S(L))=0, & \quad j>m.
%\end{eqnarray*}
%}

Here is a sample result from Section~\ref{s:real.roots.poly.coeffs} that provides an interesting
connection between total positivity of a special Hurwitz matrix and polynomials all whose
roots are real and negative: \vskip 2mm

\noindent
{\bf Theorem (Total positivity criterion for negative zeros).\/}
\textit{The polynomial
$$
g(z)=a_0+a_1z+\cdots+a_nz^n,\quad a_0\neq0,\quad a_n>0
$$
has all negative zeros if and only if the infinite matrix
$$
\mathcal{D}_{\infty}(g)\eqbd
\begin{pmatrix}
a_0&a_1&a_2 &a_3 &a_4 &a_5 &a_6 &\dots\\
0  &a_1&2a_2&3a_3&4a_4&5a_5&6a_6&\dots\\
0  &a_0&a_1 &a_2 &a_3 &a_4 &a_5 &\dots\\
0  &0  &a_1 &2a_2&3a_3&4a_4&5a_5&\dots\\
0  &0  &a_0 &a_1 &a_2 &a_3 &a_4 &\dots\\
0  &0  &0   &a_1 &2a_2&3a_3&4a_4&\dots\\
\vdots &\vdots  &\vdots  &\vdots &\vdots &\vdots  &\vdots &\ddots
\end{pmatrix}
$$
is totally nonnegative.} \vskip 2mm

This result is based on a new connection between Stieltjes continued fractions
representing a rational function $P=q/p$  and a special factorization of the
infinite Hurwitz matrix $H(p,q)$ associated to the pair $(p,q)$ (see
Section~\ref{s:R-functions.Stieltjes}).  This connection implies a criterion
of total nonnegativity of the latter infinite Hurwitz matrix
$H(p,q)$ (Theorem~\ref{Th.Hurwitz.Matrix.Total.Nonnegativity}), which in turn
implies our criterion for negative zeros of a polynomial $g$
via the total nonnegativity of the infinite discriminant matrix $\mathcal{D}_\infty(g)$.
One direction of Theorem~\ref{Th.Hurwitz.Matrix.Total.Nonnegativity} was essentially
developed in the earlier works of Asner and Kemperman\cite{Asner,Kemperman}, while the
other direction is given here for the first time. This result simultaneously provides a
criterion of stability \cite{Asner,Kemperman, Holtz1,Pinkus,Tya}.

Section~\ref{s:real.roots.poly.coeffs} provides many other alternative criteria for
real zeros, including the cases of only positive, only negative, or mixed zeros,
in terms of Hankel and Hurwitz determinants as well as the coefficients of Stieltjes
continued fraction expansion of the logarithmic derivative of a given polynomial.
The special role played by the logarithmic derivative for the analysis of real
zeros is clarified earlier in Section~\ref{s:counting.zeros}, together with the
main ideas behind counting real roots to the left and to the right of the origin.
The underlying theory of Cauchy indices is presented in Section~\ref{s:R-functions.ratios}.
%Several results there are also new, for example, the formulas connecting the Cauchy indices
%of a rational function on the intervals $(-\infty, 0)$ and $(0,\infty)$ with the
%number of Frobenius sign changes in the sequence of its Hankel minors
%(Theorem~\ref{Th.number.of.negative.positive.indices.of.real.rational.function.1}).

Another important class of objects behind the results of this
paper  on polynomial roots is the class of so-called
$R$-functions. These are rational functions mapping the upper-half
plane of the complex plane either to the lower half-plane or to
itself (see Definition~\ref{def.S-function}). The 
theory of $R$-functions connects several key
objects of this work:  continued  fractions, Cauchy indices,
polynomials with real roots, the moment problem, and determinantal
inequalities. As an example, we quote the following new result
from Section~\ref{s:R-functions.ratios}:

\vskip 2mm \noindent
{\bf Theorem (Generalized Lienard-Chipart criterion).\/}
\textit{If a real rational function $R=q/p$, with $p$ and $q$ relatively prime,
 maps the upper half-plane of the complex plane to the  lower half-plane, then
the number of its positive poles is equal to the number  $\SC^-(a_0,\ldots, a_n)$
of strong sign changes in the sequence $(a_0,\ldots, a_n)$ of coefficients of
its denominator  $p$. In particular, $R$ has only negative poles if and only
if $a_j>0$ for $j=1,2,\ldots,n$, and only positive poles if and only if
$a_{j-1}a_j<0$ for $j=1,2,\ldots,n$. }
\vskip 2mm

Our final illustration is the generalized Orlando's formula from Section \ref{s:Resultant.formula}.
Orlando's formula
\textit{per se} expresses the Hurwitz determinant of order $n-1$ associated to
a polynomial $p$ of degree $n$ as the product, up to a certain normalization,
 of all possible pairs  $z_i+z_j$  of the zeros of $p$. Orlando's result goes
back to 1911~\cite{Orlando}. The generalized Orlando formula
obtained in this paper says the following:

\vskip 2mm \noindent {\bf Theorem (Generalized Orlando formula).}
\textit{Let $p$ be a polynomial of degree $n$ and $q$ a polynomial
of degree  $m \leq n$. Then the resultant of these
polynomials can be computed as follows:
$$ \mathbf{R}(p,q) = (-1)^{n(n-1)\over 2} a_0^{m+n} \prod_{1\leq i<j \leq 2n} (z_i+z_j) $$
where $a_0$ is the leading coefficient of $p$ and where $z_i$, $i=1,\ldots, 2n$, are the
zeros of the polynomial}  $$ h(z) \eqbd p(z^2)+zq(z^2).$$
The classical Orlando's formula then follows as a special case by splitting an arbitrary
polynomial into its even and odd part and applying the generalized Orlando's formula.

As these examples show, new connections can be found between several classical matrix
classes (Hurwitz, Hankel, Vandermonde etc.) made of the coefficients of polynomials,
different representations of rational functions (Stieltjes and Jacobi continued
fractions, Laurent series, elementary fractions), and various counting notions (sign
changes, the number of roots in a specific domain of a complex plain, Cauchy indices).
Underlying all these is a coherent theory demonstrating how these connections arise.

The presented theory is completely general in that it does not single out the
special case of stable polynomials. The importance of stability, historically
the first question on root localization~\cite{KreinNaimark,Jury,FullerAper,Rogers,Genin}, 
is by now quite 
well understood. Various stability criteria have been studied in great detail,
especially in the engineering literature. In this survey 
(Section~\ref{s:counting.zeros}), we choose instead to illustrate
applications of the general theory to polynomials with real 
roots, a less studied but arguably equally important special case.

To summarize, the point of this work is to provide a uniform, streamlined, algebraic
treatment for a body of questions centered around the root localization of polynomials.
At present, results of this type are scattered in the literature, whereas even textbooks
and monographs about polynomials sometimes fail to provide answers to some basic questions
(i.e., tests for real roots using only the coefficients of the polynomial). We hope that
this work will serve as a useful reference and source of further research for mathematicians
in various fields interested in polynomials and their zeros.

%\vspace{15cm}

\setcounter{equation}{0}

\setcounter{notezz}{0}

%%%%%%%%%%%%%%%%%%%%%%%%%%%%%%%%%%%%%%%%%%%%%%%%%%%%%%%%
\section{Complex rational functions and related topics\label{s:Complex.rational.functions}}
%%%%%%%%%%%%%%%%%%%%%%%%%%%%%%%%%%%%%%%%%%%%%%%%%%%%%%%%

Consider a rational function
\begin{equation}\label{basic.rational.function}
z \mapsto  R(z)\eqbd \dfrac{q(z)}{p(z)}
\end{equation}
where $p$ and $q$ are polynomials with complex coefficients
\begin{eqnarray}\label{polynomial.1}
& \, p(z) \; \eqbd \; a_0 z^n + a_1z^{n-1}+\cdots+a_n,\qquad &
a_0, \, a_1, \, \dots, \, a_n\in\mathbb{C},\ a_0\neq0, \\
\label{polynomial.2} & q(z) \; \eqbd \;
b_0z^{n}+b_1z^{n-1}+\cdots+b_{n},\qquad & b_0, \, b_1, \, \dots, \,
b_{n}\in\mathbb{C},
\end{eqnarray}
If the greatest common divisor of $p$ and $q$ has degree $l$, $0
\leq l\leq n$, then the rational function $R$ has
exactly $r=n-l$ poles. Note that zeros or poles of a rational
function are always counted with multiplicities unless explicitly
stipulated otherwise. Thus, in the rest of the paper, the phrase
``counting multiplicities'' will be implicit in every statement
about zeros or poles of functions under consideration.

%%%%%%%%%%%%%%%%%%%%%%%%%%%%%%%%%%%%%%%%%%%%%%%%%%%%%%%%%%%%%%%%%%%
\subsection{\label{s:Hurwitz.matrices.and.Hurwitz.formula}Hankel
and Hurwitz matrices. Hurwitz formul\ae}
%%%%%%%%%%%%%%%%%%%%%%%%%%%%%%%%%%%%%%%%%%%%%%%%%%%%%%%%%%%%%%%%%%

In this section, we introduce Hankel and Hurwitz matrices associated
to a given rational functions and discuss the Hurwitz formul\ae\ that
connect these two classes.

Expand the function~\eqref{basic.rational.function} into its
Laurent series at $\infty$:
\begin{equation}\label{basic.rat.func.expansion}
R(z)=s_{-1}+\dfrac{s_{0}}{z}+\dfrac{s_{1}}{z^{2}}+\dfrac{s_{2}}{z^{3}}+\cdots.
\end{equation}
Here $s_j = 0$ for $j<n-1-m$ and $s_{n-1-m}=\dfrac{b_0}{a_0}$,
where $m=\deg q$.

The sequence of coefficients of negative powers of $z$
\begin{equation*}\label{main.sequence}
s_0,s_1,s_2,\ldots
\end{equation*}
defines the infinite \textit{Hankel} matrix $S \eqbd S(R)
\eqbd\|s_{i+j}\|_{i,j=0}^{\infty}$. This gives
 the correspondence
\begin{equation}\label{Hakel.matrix.Rat.func.correspondence}
R\mapsto S(R).
\end{equation}

\begin{definition}\label{def.Hankel.determinants}
For a given infinite sequence $(s_j)_{j=0}^\infty$, consider the
determinants
\begin{equation}\label{Hankel.determinants.1}
D_j(S)\eqbd
\begin{vmatrix}
    s_0 &s_1 &s_2 &\dots &s_{j-1}\\
    s_1 &s_2 &s_3 &\dots &s_j\\
    \vdots&\vdots&\vdots&\ddots&\vdots\\
    s_{j-1} &s_j &s_{j+1} &\dots &s_{2j-2}
\end{vmatrix},\quad j=1,2,3,\dots,
\end{equation}
i.e., the \textit{leading principal minors} of the infinite Hankel
matrix~$S$. These determinants are referred to as the
\textit{Hankel minors} or  \textit{Hankel determinants}.
\end{definition}

An infinite matrix is said to have finite rank $r$ if all its
minors of order greater than $r$ are zero whereas there exists at
least one nonzero minor of order $r$.
Knonecker~\cite{Kronecker} proved that, for any infinite Hankel
matrix, any minor of order $r$ where $r$ is the rank  of the
matrix, is a multiple of its leading principal minor of order $r$.
This implies the following result.
\begin{theorem}[Kronecker \cite{Kronecker}]\label{Th.Hankel.matrix.rank.2}
An infinite Hankel matrix $S=\|s_{i+j}\|_{i,j=1}^{\infty}$ has
finite rank $r$ if and~only~if
\begin{eqnarray}
D_r(S) & \neq & 0, \label{finite.rank.condition.1}  \\
D_j(S)& = & 0, \quad \hbox{\rm for all } j>r.
\label{finite.rank.condition.2}
\end{eqnarray}
\end{theorem}
%
%If the entries of an infinite Hankel matrix $S$ are coefficients
%of the Laurent series~\eqref{basic.rat.func.expansion} of some
%rational function $R$, then we denote the
%minors~\eqref{Hankel.determinants.1} by $D_j(R)$.

Let $\widehat{D}_j(S)$ denote the following determinants
\begin{equation}\label{Hankel.determinants.2}
\widehat{D}_j(S)\eqbd
\begin{vmatrix}
    s_1 &s_2 &s_3 &\dots &s_{j}\\
    s_2 &s_3 &s_4 &\dots &s_{j+1}\\
    \vdots&\vdots&\vdots&\ddots&\vdots\\
    s_{j} & s_{j+1} & s_{j+2} & \dots &s_{2j-1}
\end{vmatrix},\quad j=1,2,3,\dots,
\end{equation}
With a slight abuse of notation, we will also write $D_j(R)
\eqbd D_j(S(R))$ and $\widehat{D}_j(R) \eqbd\widehat{D}_j(S(R))$
 if the matrix $S=S(R)$ is made of the
coefficients~\eqref{basic.rat.func.expansion}
 of the function $R$.

The following theorem was also established by Gantmacher
in~\cite{Gantmakher}.
\begin{theorem}\label{Th.Hankel.matrix.rank.1}
An infinite matrix $S=\|s_{i+j}\|_{i,j=1}^{\infty}$ has finite
rank if and only if the sum of the series
\begin{equation*}
R(z)=\dfrac{s_0}{z}+\dfrac{s_1}{z^2}+\dfrac{s_2}{z^3}+\cdots
\end{equation*}
is a rational function of $z$. In this case the rank of the matrix
$S$ is equal to the number of poles of the~function $R$.
\end{theorem}

Theorems~\ref{Th.Hankel.matrix.rank.2}
and~\ref{Th.Hankel.matrix.rank.1} have a simple corollary, which
will be useful later.

\begin{corol}\label{corol.zero.pole}
A rational function $R$ with exactly $r$ poles represented by the
series~\eqref{basic.rat.func.expansion}  has a pole at the point
$0$ if and only if
\begin{equation}\label{zero.pole.condition}
\widehat{D}_{r-1}(R)\neq0\quad\text{and}\quad
\widehat{D}_{j}(R)=0\quad {\rm for } \;\; j=r, r+1, \ldots .
\end{equation}
\end{corol}%

\proof
Since the function $R$ represented as a series
\eqref{basic.rat.func.expansion} has exactly $r$ poles, the
function
\begin{equation*}\label{corol.zero.pole.proof.1}
G(z) \eqbd
zR(z)-zs_{-1}=s_0+\dfrac{s_{1}}{z}+\dfrac{s_{2}}{z^{2}}+\dfrac{s_{3}}{z^{3}}+\cdots
\end{equation*}
has exactly $r-1$ poles if and only if $R$ has a pole at the
point~$0$. If $R$ does not have a pole at $0$, then $G$ has $r$
poles. Thus, by Theorems~\ref{Th.Hankel.matrix.rank.2}
and~\ref{Th.Hankel.matrix.rank.1}, the function $R$ has a pole
at~$0$  if and only if
\begin{equation}\label{corol.zero.pole.proof.2}
D_{r-1}(G)\neq0\quad\text{and}\quad D_j(G)=0,\quad j=r,r+1,\ldots
.
\end{equation}
Since
\begin{equation*}\label{corol.zero.pole.proof.3}
\widehat{D}_j(R)=D_j(G),\quad j=1,2,\ldots,
\end{equation*}
the formula~\eqref{corol.zero.pole.proof.2} yields the assertion
of the corollary.  \eop

Theorems~\ref{Th.Hankel.matrix.rank.2}
and~\ref{Th.Hankel.matrix.rank.1} imply the following:  if the
greatest common  divisor of the polynomials $p$ and  $q$ defined
in~\eqref{polynomial.1}--\eqref{polynomial.2} has degree $l$,
$0\leq l\leq\deg q$, then the
formul\ae~\eqref{finite.rank.condition.1}--\eqref{finite.rank.condition.2}
hold for $r=n-l$ for the rational
function~\eqref{basic.rational.function}, where $n$ is the degree
of the polynomial~\eqref{polynomial.1}, i.e., the denominator
of~$R$.

Denote by $\nabla_{2j}(p,q)$ the following determinants of
order~$2j$:
\begin{equation}\label{Hurwitz.General.Determinants}
\nabla_{2j}(p,q)\eqbd
\begin{vmatrix}
    a_0 &a_1 &a_2 &\dots &a_{j-1} & a_{j}  &\dots &a_{2j-1}\\
    b_0 &b_1 &b_2 &\dots &b_{j-1} & b_{j}  &\dots &b_{2j-1}\\
     0  &a_0 &a_1 &\dots &a_{j-2} & a_{j-1}&\dots &a_{2j-2}\\
     0  &b_0 &b_1 &\dots &b_{j-2} & b_{j-1}&\dots &b_{2j-2}\\
    \vdots&\vdots&\vdots&\ddots&\vdots&\vdots&\ddots&\vdots\\
     0  &  0 &  0 &\dots &a_{0} & a_{1}&\dots &a_{j}\\
     0  &  0 &  0 &\dots &b_{0} & b_{1}&\dots &b_{j}\\
\end{vmatrix},\quad j=1,2,\dots,
\end{equation}
constructed using the coefficients of the
polynomials~\eqref{polynomial.1}--\eqref{polynomial.2}.  Here we
set $a_i=0$ for all $i>n$ and $b_l=0$ for $l>m$. The determinants
$\nabla_{2j}(p,q)$ are called \textit{determinants of Hurwitz
type} or just \textit{Hurwitz determinants} or \textit{Hurwitz
minors}.

In his celebrated work \cite{Hurwitz}, A.~Hurwitz found
relationships between the minors $D_j(R)$ defined
in~\eqref{Hankel.determinants.1} and the  determinants
$\nabla_{2j}(p,q)$ defined by
\eqref{Hurwitz.General.Determinants}.
\begin{theorem}[\cite{Hurwitz,{KreinNaimark},{Gantmakher},{Barkovsky.2}}]\label{Th.Hurwitz.relations}
Let $R(z)=\dfrac{q(z)}{p(z)}$, where the polynomials $p$ and $q$
are defined by~\eqref{polynomial.1}--\eqref{polynomial.2}. Then
\begin{equation}\label{Hurwitz.General.Relations}
\nabla_{2j}(p,q)=a_0^{2j}D_j(R),\quad j=1,2,\ldots.
\end{equation}
\end{theorem}
\proof
From~\eqref{basic.rat.func.expansion} it follows that coefficients
$s_k,a_k,b_k$ satisfy recurrence relations
\begin{equation}\label{recurrence.relations}
b_j=a_0s_{j-1}+a_1s_{j-2}+\cdots+a_js_{-1},\qquad j=0,1,2,\ldots.
\end{equation}
These recurrence relations imply the
formul\ae~\eqref{Hurwitz.General.Relations} by direct matrix
multiplication, once we take into
account~\eqref{recurrence.relations} (cf.~\cite{KreinNaimark}):
\begin{eqnarray*}\label{equalities.for.proof.of.Hurwitz.General.Relations.1}
a_0^{2j+1}D_j(R) &= &(-1)^{\tfrac{j(j-1)}2}a_0^{2j+1}
\begin{vmatrix}
    1 & 0 & 0 &\dots &0 &0  &\dots &0\\
    0 & 1 & 0 &\dots &0 &0  &\dots &0\\
    %\hdotsfor{7}\\
    \vdots &\vdots &\vdots &\ddots &\vdots &\vdots  &\ddots &\vdots\\
    0 & 0 & 0 &\dots &1 &0  &\dots &0\\
    0  &s_{-1}&s_0&\dots &s_{j-2} &s_{j-1} &\dots &s_{2j-2}\\
    0  & 0 &s_{-1}&\dots &s_{j-3} &s_{j-2} &\dots &s_{2j-3}\\
    \vdots&\vdots&\vdots&\ddots&\vdots&\vdots&\ddots&\vdots\\
    0  &  0 &  0 &\dots&  s_{0} &s_{1} & \dots &s_{j-2}\\
    0  &  0 &  0 &\dots&  s_{-1} &s_{0} & \dots &s_{j-1}\\
\end{vmatrix}=\\
&=&(-1)^j\begin{vmatrix}
    a_0 &0 &\dots &0\\
    a_1 &a_0 &\dots &0\\
    \vdots &\vdots &\ddots &\vdots\\
    a_{2j} &a_{2j-1} &\dots &a_0\\
\end{vmatrix}\cdot
\begin{vmatrix}
    1 & 0 & 0 &\dots &0 &0  &\dots &0\\
    0  &s_{-1}&s_0&\dots &s_{j-2} &s_{j-1} &\dots &s_{2j-2}\\
    0 & 1 & 0 &\dots &0 &0  &\dots &0\\
    0  & 0 &s_{-1}&\dots &s_{j-3} &s_{j-2} &\dots &s_{2j-3}\\
    %\hdotsfor{7}\\
    %\dots &\dots &\dots &\dots &\dots &\dots  &\dots &\dots\\
    \vdots&\vdots&\vdots&\ddots&\vdots&\vdots&\ddots&\vdots\\
    %0  &  0 &  0 &\dots&  s_{0} &s_{1} & \dots &s_{j-2}\\
    %0 & 0 & 0 &\dots &0 &0  &\dots &0\\
    0  &  0 &  0 &\dots&  s_{-1} &s_{0} & \dots &s_{j-1}\\
    0 & 0 & 0 &\dots &1 &0  &\dots &0\\
\end{vmatrix}=\\
&=&(-1)^j\begin{vmatrix}
    a_0 &a_1 &a_2 &\dots &a_{j-1} & a_{j}  &\dots &a_{2j}\\
    0 &b_0 &b_1 &\dots &b_{j-2} & b_{j-1}  &\dots &b_{2j-1}\\
     0  &a_0 &a_1 &\dots &a_{j-2} & a_{j-1}&\dots &a_{2j-1}\\
     0  &0 &b_0 &\dots &b_{j-3} & b_{j-2}&\dots &b_{2j-2}\\
    \vdots&\vdots&\vdots&\ddots&\vdots&\vdots&\ddots&\vdots\\
     0  &  0 &  0 &\dots &b_{0} & b_{1}&\dots &b_{j}\\
     0  &  0 &  0 &\dots &a_{0} & a_{1}&\dots &a_{j}\\
\end{vmatrix}=a_0\nabla_{2j}(p,q). \qquad \qquad \qquad
\end{eqnarray*}

\eop

\begin{remark}\label{remark.1.1}
Since the formul\ae~\eqref{Hurwitz.General.Relations} are
algebraic identities also valid for polynomials
$\widetilde{p}=\sum_{j=0}^na_jz^j$, $a_0\neq0$, and
$\widetilde{q}=\sum_{j=0}^nb_jz^j$, they also hold when
$p$ and $q$ are entire functions and $R$ is a meromorphic
function.
\end{remark}

From~\eqref{Hurwitz.General.Determinants} one can see that, whenever
 $b_0\neq 0$, we have
\begin{equation}\label{Hurwitz.determinants.equality}
\nabla_{2j}(p,q)=\nabla_{2j}(-q,p),\quad j=1,2,\ldots .
\end{equation}
This observation, coupled with Theorem~\ref{Th.Hurwitz.relations},
yields the following well-known fact (cf.~\cite{Edrei_1}\footnote{It appears that
this basic fact was known much earlier than the work of Edrei
but we do not know of the original reference.}).

\begin{corol}
For two infinite sequences
%\begin{equation*}\label{main.sequences}
$S\eqbd (s_j)_{j=-1}^\infty,$  $T\eqbd (t_j)_{j=-1}^\infty$ with
 $s_{-1}\neq 0$ and $t_{-1} \neq 0$,
%\end{equation*}
the following conditions are equivalent:
\begin{eqnarray*}   1)
  \qquad\qquad \qquad \qquad\qquad \qquad \qquad  \qquad \qquad
s_{-1} t_{-1} +1 & = & 0, \\
s_it_{-1}+s_{i-1}t_0+\cdots+s_0t_{i-1}+s_{-1}t_i & = & 0, \qquad\qquad
i=0,1,2,\ldots. \qquad \qquad\qquad \qquad \qquad
\end{eqnarray*}
\begin{equation*}\label{Hankel.determinants.equality}
2) \qquad \qquad \qquad \qquad \qquad\qquad \qquad \qquad \qquad
D_k(S)=s_{-1}^{2k}D_k(T),\qquad\quad \; \; k=1,2,\ldots,  \qquad
\qquad \qquad \qquad\qquad
\end{equation*}
where $D_k(S)$ and $D_k(T)$ are the Hankel
minors~\eqref{Hankel.determinants.1} for the sequences $S$ and
$T$, respectively.
\end{corol}
\proof
Let
\begin{equation*}
R(z)=\dfrac{q(z)}{p(z)}=s_{-1}+\dfrac{s_0}{z}+\dfrac{s_1}{z^2}+\dfrac{s_2}{z^3}+\cdots,
\end{equation*}
then
\begin{equation*}
-\dfrac1{R(z)}=-\dfrac{p(z)}{q(z)}=t_{-1}+\dfrac{t_0}{z}+\dfrac{t_1}{z^2}+\dfrac{t_2}{z^3}+\cdots.
\end{equation*}
The assertion of the corollary follows from
Theorem~\ref{Th.Hurwitz.relations},
formula~\eqref{Hurwitz.determinants.equality} and the identity
\begin{equation*}
\left(s_{-1}+\dfrac{s_0}{z}+\dfrac{s_1}{z^2}+\dfrac{s_2}{z^3}+\cdots\right)\left(t_{-1}+\dfrac{t_0}{z}+\dfrac{t_1}{z^2}+\dfrac{t_2}{z^3}+\cdots\right)
\equiv -1.
\end{equation*}
\eop
%

%%%%%%%%%%%%%%%%%%%%%%%%%%%%%%%%%%%%%%%%%%%%%%%%%%%%%%%%%%%%%%%%%
\subsection{\label{s:Resultant.formula}Resultant formul\ae\
and applications: discriminant and Orlando formul\ae}
%%%%%%%%%%%%%%%%%%%%%%%%%%%%%%%%%%%%%%%%%%%%%%%%%%%%%%%%%%%%%%%%%

This section is devoted to the resultant,  an important object that
is closely related to the greatest common divisor of two polynomials.
The discriminant of a single polynomial occurs as a special case of
the resultant of that polynomial and its derivative. Moreover, the
generalized Orlando formula connects the resultant of two
polynomials, $p$ and $q$, with the roots of the aggregate polynomial
$h(z) \eqbd  p(z^2) + z q(z^2)$. Although the resultant is classically
known, this last connection is new.

\begin{definition}\label{definition.resultant}
Let two polynomials $p$ and $q$ be given
in~\eqref{polynomial.1}--\eqref{polynomial.2} with $n\eqbd \deg p$ and
$m\eqbd\deg q$. The \textit{resultant} of $p$~and $q$ is the following determinant of
order $n+m$
%\footnote{In case $b_0=0$ and $\deg q=m<n$, the
%determinant~\eqref{Resultant} is equal to
%$a_0^{n-m}\mathbf{R}(p,q)$. In that case, $\mathbf{R}(p,q)$ is a
%determinant of order $(m+n)$ (see, for example,~\cite{Kurosh}).}
%
\begin{equation}\label{Resultant}
\mathbf{R}(p,q)\eqbd
\begin{vmatrix}
    a_0 &a_1 &\dots &a_{m-1}   &a_{m}   &\dots &a_{n-1} & a_{n}  &\dots &a_{n+m-1}\\
     0  &a_0 &\dots &a_{m-2}   &a_{m-1} &\dots &a_{n-2} & a_{n-1}&\dots &a_{n+m-2}\\
    \vdots&\vdots&\ddots&\vdots&\vdots&\ddots&\vdots&\vdots&\ddots&\vdots\\
     0  &  0 &\dots &a_0  &a_1  &\dots &a_{n-m}   & a_{n-m+1}&\dots &a_{n}\\
    b_{n-m} &b_{n-m+1}  &\dots &b_{n-1}&b_{n}   &\dots &b_{2n-m-1} & b_{2n-m}  &\dots &b_{2n-1}\\
     0  &b_{n-m} &\dots &b_{n-2}&b_{n-1} &\dots &b_{2n-m-2} & b_{2n-m-1}&\dots &b_{2n-2}\\
    \vdots&\vdots&\ddots&\vdots&\vdots&\ddots&\vdots&\vdots&\ddots&\vdots\\
     0  &  0 &\dots &0   &0   &\dots &b_{n-m} & b_{n-m+1}&\dots &b_{n}\\
\end{vmatrix},
\end{equation}
%
%\begin{equation}\label{Resultant}
%\mathbf{R}(p,q)\eqbd
%\begin{vmatrix}
%    a_0 &a_1 &\dots &a_{n-1} & a_{n}  &\dots &a_{2n-1}\\
%     0  &a_0 &\dots &a_{n-2} & a_{n-1}&\dots &a_{2n-2}\\
%    \vdots&\vdots&\ddots&\vdots&\vdots&\ddots&\vdots\\
%     0  &  0 &\dots &a_{0}   & a_{1}&\dots &a_{n}\\
%    b_0 &b_1 &\dots &b_{n-1} & a_{n}  &\dots &b_{2n-1}\\
%     0  &b_0 &\dots &b_{n-2} & b_{n-1}&\dots &b_{2n-2}\\
%    \vdots&\vdots&\ddots&\vdots&\vdots&\ddots&\vdots\\
%     0  &  0 &\dots &b_{0}   & b_{1}&\dots &b_{n}\\
%\end{vmatrix},
%\end{equation}
%
where we set $a_i\eqbd 0$ and $b_i\eqbd 0$ for all $i>n$.
\end{definition}

From~\eqref{Resultant} it can be immediately seen that
\begin{equation*}\label{Resultant.property.1}
\mathbf{R}(p,q)=(-1)^{nm}\mathbf{R}(q,p).
\end{equation*}

The resultant of two polynomials is a multi-affine function of the
roots of these polynomials, as the following formula shows.
\begin{theorem}\label{Th.Resultant.formula}
Given polynomials $p$ and $q$ as in~\eqref{polynomial.1}--\eqref{polynomial.2}
with $b_0\neq 0$, let
$\lambda_i$ $(i=1,\ldots,n)$ denote the zeros of  $p$, and let $\mu_j$
$(j=1,\ldots,n)$ denote the zeros of  $q$. Then
\begin{equation}\label{Resultant.formula}
\mathbf{R}(p,q)=(-1)^{\tfrac{n(n-1)}2}\nabla_{2n}(p,q)=a_0^n\prod_{i=1}^nq(\lambda_i)=a_0^nb_0^n\prod_{i,j=1}^n(\lambda_i-\mu_j)=(-1)^nb_0^n\prod_{j=1}^np(\mu_j).
\end{equation}
\end{theorem}
\proof
At first, assume that all roots of the polynomial~$p$ are
distinct. Then the function~$R$ admits the representation
\begin{equation}\label{rat.func.elementary.fractions}
R(z)=\displaystyle\beta+\sum^{n}_{j=1}\frac
{\alpha_j}{z-\lambda_j},
\end{equation}
where
\begin{equation}\label{residues}
\alpha_j=\dfrac{p(\lambda_j)}{q'(\lambda_j)},\quad j=1,\ldots,n.
\end{equation}
Here $\lambda_j\neq\lambda_i$ whenever $i\neq j$, according to the
assumption. Comparing the
representation~\eqref{rat.func.elementary.fractions} with the
expansion~\eqref{basic.rat.func.expansion}, we obtain
\begin{equation}\label{Markov.parameters}
s_k=\sum_{i=1}^n\alpha_i\lambda_i^k,\qquad k=0,1,2,\ldots
\end{equation}
and $s_{-1}=\dfrac{b_0}{a_0}\neq0$.
From~\eqref{Hurwitz.General.Determinants},
\eqref{Hurwitz.General.Relations} and~\eqref{Resultant} it follows
that
\begin{equation}\label{Resultant.formula.working.1}
\mathbf{R}(p,q)=(-1)^{\tfrac{n(n-1)}2}\nabla_{2n}(p,q)=(-1)^{\tfrac{n(n-1)}2}a_0^{2n}D_n(R).
\end{equation}
The formul\ae~\eqref{Markov.parameters} yield
\begin{equation*}\label{Resultant.formula.working.2}
\begin{array}{c}
\begin{pmatrix}
    s_0 &s_1 &s_2 &\cdots &s_{n-1}\\
    s_1 &s_2 &s_3 &\cdots &s_n\\
    \vdots&\vdots&\vdots&\ddots&\vdots\\
    s_{n-1} &s_n &s_{n+1} &\cdots &s_{2n-2}
\end{pmatrix}=
 \qquad \qquad \qquad \qquad  \qquad \qquad \qquad \qquad \qquad \quad  \\
=\begin{pmatrix}
    \alpha_1 &\alpha_2 &\alpha_3 &\dots &\alpha_n\\
    \alpha_1\lambda_1 &\alpha_2\lambda_2 &\alpha_3\lambda_3 &\dots &\alpha_n\lambda_n\\
    \alpha_1\lambda_1^2 &\alpha_2\lambda_2^2 &\alpha_3\lambda_3^2 &\dots &\alpha_n\lambda_n^2\\
    \vdots&\vdots&\vdots&\ddots&\vdots\\
    \alpha_1\lambda_1^{n-1} &\alpha_2\lambda_2^{n-1} &\alpha_3\lambda_3^{n-1} &\dots &\alpha_n\lambda_n^{n-1}\\
\end{pmatrix}\cdot
\begin{pmatrix}
     1 &\lambda_1 &\lambda_1^2 &\dots &\lambda_1^{n-1}\\
     1 &\lambda_2 &\lambda_2^2 &\dots &\lambda_2^{n-1}\\
     1 &\lambda_3 &\lambda_3^2 &\dots &\lambda_3^{n-1}\\
    \vdots&\vdots&\vdots&\ddots&\vdots\\
     1 &\lambda_n &\lambda_n^2 &\dots &\lambda_n^{n-1}\\
\end{pmatrix}.
\end{array}
\end{equation*}
This equality implies
\begin{equation}\label{Resultant.formula.working.3}
D_n(R)=\begin{vmatrix}
    s_0 &s_1 &s_2 &\dots &s_{n-1}\\
    s_1 &s_2 &s_3 &\dots &s_n\\
    \vdots&\vdots&\vdots&\ddots&\vdots\\
    s_{n-1} &s_n &s_{n+1} &\dots &s_{2n-2}
\end{vmatrix}=\prod_{i=1}^n\alpha_i\cdot
\begin{vmatrix}
    1 &1 &1 &\dots &1\\
    \lambda_1 &\lambda_2 &\lambda_3 &\dots &\lambda_n\\
    \lambda_1^2 &\lambda_2^2 &\lambda_3^2 &\dots &\lambda_n^2\\
    \vdots&\vdots&\vdots&\ddots&\vdots\\
    \lambda_1^{n-1} &\lambda_2^{n-1} &\lambda_3^{n-1} &\dots &\lambda_n^{n-1}\\
\end{vmatrix}^2
\end{equation}
This formula combined with the residue formula~\eqref{residues}
implies
\begin{equation}\label{Resultant.formula.working.4}
D_n(R)=\prod_{i=1}^n\dfrac{q(\lambda_i)}{p'(\lambda_i)}\cdot\prod_{j<i}(\lambda_i-\lambda_j)^2.
\end{equation}
Since
\begin{equation*}\label{Resultant.formula.working.5}
p'(z)=a_0\sum_{i=1}^n\prod_{\substack{k=1\\
k\neq i}}^n(z-\lambda_k),
\end{equation*}
we have
\begin{equation*}\label{Resultant.formula.working.6}
\prod_{i=1}^np'(\lambda_i)=a_0^n\prod_{i=1}^n\prod_{\substack{k=1\\
k\neq
i}}^n(\lambda_i-\lambda_k)=a_0^n\prod_{j<i}(\lambda_i-\lambda_j)
\prod_{i<j}(\lambda_i-\lambda_j).
\end{equation*}
This product has exactly $n(n-1)$ factors of the form
$\lambda_i-\lambda_j$, therefore
\begin{equation}\label{Resultant.formula.working.7}
\prod_{i=1}^np'(\lambda_i)=a_0^n(-1)^{\tfrac{n(n-1)}2}\prod_{j<i}(\lambda_i-\lambda_j)^2.
\end{equation}
Now
from~\eqref{Resultant.formula.working.4}--\eqref{Resultant.formula.working.7}
we obtain
\begin{equation}\label{Resultant.formula.working.8}
D_n(R)=\dfrac{(-1)^{\tfrac{n(n-1)}2}}{a_0^n}\prod_{i=1}^nq(\lambda_i)=(-1)^{\tfrac{n(n-1)}2}\cdot\dfrac{b_0^n}{a_0^n}\prod_{i,j=1}^n(\lambda_i-\mu_j).
\end{equation}
The formula~\eqref{Resultant.formula} follows
from~\eqref{Resultant.formula.working.1}
and~\eqref{Resultant.formula.working.8}.

If the polynomial $p$ has multiple zeros, we can consider an
approximating polynomial  $p_\varepsilon$  with simple zeros such
that
\begin{equation*}\label{Resultant.formula.working.9}
\displaystyle\lim_{\varepsilon\to 0} p_\varepsilon(z) =p(z) \quad
\hbox{\rm for all } \;  z.
\end{equation*}
Then
\begin{equation*}\label{Resultant.formula.working.10}
\displaystyle\nabla_{2n}(R_\varepsilon)\xrightarrow[\varepsilon\to0]{}\nabla_{2n}(R),
\end{equation*}
where $R_\varepsilon(z) \eqbd \dfrac{q(z)}{p_\varepsilon(z)}$. The
formula~\eqref{Resultant.formula} is valid for the polynomials $q$
and $p_\varepsilon$, so it is also valid at the limit, i.e.,
for the polynomials $q$ and $p$.  \eop
\begin{corol}
Under the conditions of Theorem~\ref{Th.Resultant.formula}, let
$\deg q=m\leq n$. Then
\begin{equation}\label{Resultant.formula.2}
\begin{array}{rclcl}
\mathbf{R}(p,q) & = & (-1)^{\tfrac{n(n-1)}2}a_0^{m-n}\nabla_{2n}(p,q) & = & \displaystyle a_0^m\prod_{i=1}^nq(\lambda_i)\ \ = \\
=\ \ \displaystyle a_0^mb_{n-m}^n\prod_{i=1}^n \prod_{j=1}^m
(\lambda_i-\mu_j) & = &
\displaystyle(-1)^{nm}b_{n-m}^n\prod_{j=1}^mp(\mu_j)  & = &
(-1)^{nm}R(q,p).
\end{array}
\end{equation}
\end{corol}
\proof If $\deg q=m(\leq n)$, then the first nonzero
coefficient of $q$ is $b_{n-m}$, according
to~\eqref{polynomial.2}. Therefore,
\begin{equation*}\label{Resultant.formula.2.working.1}
\nabla_{2n}(p,q)=(-1)^{\tfrac{n(n-1)}2}\begin{vmatrix}
    a_0 &a_1 &\dots &a_{n-m-1} & a_{n-m}  &\dots&a_{n}&\dots &a_{2n-m-2} &\dots &a_{2n-1}\\
     0  &a_0 &\dots &a_{n-m-2} & a_{n-m-1}&\dots&a_{n-1}&\dots  & a_{2n-m-3} &\dots &a_{2n-2}\\
    \vdots&\vdots&\ddots&\vdots&\vdots&\ddots&\vdots&\ddots&\vdots&\ddots&\vdots\\
     0&0&\dots&0&a_0&\dots&a_m&\dots&a_{n-2}&\dots&a_{n+m-1}\\
    \vdots&\vdots&\ddots&\vdots&\vdots&\ddots&\vdots&\ddots&\vdots&\ddots&\vdots\\
     0  &  0 &\dots &0 & 0&\dots&a_{0}&\dots&a_{n-m} &\dots &a_{n}\\
    0 &0 &\dots &0 & b_{n-m} &\dots&b_n &\dots &b_{2n-m-1}&\dots&b_{2n-1}\\
     0  &0 &\dots &0 & 0&\dots&b_{n-1} &\dots &b_{2n-m-2}&\dots&b_{2n-2}\\
    \vdots&\vdots&\ddots&\vdots&\vdots&\ddots&\vdots&\ddots&\vdots&\ddots&\vdots\\
     0  &  0 &\dots &0 & 0&\dots&0&\dots & b_{n-m} &\dots&b_{n}\\
\end{vmatrix}.
\end{equation*}
Thus, we have
\begin{equation*}\label{Resultant.formula.2.working.2}
\nabla_{2n}(p,q)=(-1)^{\tfrac{n(n-1)}2}a_0^{n-m}\mathbf{R}(p,q).
\end{equation*}
This relation and the formul\ae~\eqref{Hurwitz.General.Relations},
\eqref{Resultant.formula.working.8}
yield~\eqref{Resultant.formula.2} when $p$ has only simple roots.
But~\eqref{Resultant.formula.2} is also valid when $p$ has
multiple zeros, which can be proved by an approximation argument
as above. \eop

The formul\ae~\eqref{Resultant.formula}
and~\eqref{Resultant.formula.2} now imply the well-known property
of the resultant:

\begin{corol}
 $\mathbf{R}(p,q)=0$ \textit{if and only if\/}
the polynomials $p$ and $q$ have common roots.
\end{corol}

Next, we  consider a function that allows us to test whether a single
polynomial has multiple roots.

\begin{definition}\label{definition.discriminant}
Given a  polynomial~\eqref{polynomial.1} with roots  $\lambda_i$
$(i=1,\ldots,n)$, its  \textit{discriminant\/}  is defined as
\begin{equation}\label{discriminant}
\mathbf{D}(p) \eqbd
a_0^{2n-2}\prod_{j<i}^n(\lambda_i-\lambda_j)^2.
\end{equation}
\end{definition}
\noindent It is clear from~\eqref{discriminant} that the
discriminant of a polynomial is equal to zero if and only if
the~polynomial has multiple zeros. But multiple zeros of a
polynomial are the zeros that it shares with  its derivative. The
following connection between the discriminant of $p$ and  the
resultant of $p$ and $p'$ should not come as a surprise.
\begin{theorem}\label{Th.discriminant.formula}
Given a polynomial~\eqref{polynomial.1} of degree~$n$, we have
\begin{equation}\label{discriminant.formula}
\mathbf{R}(p,p')=(-1)^{\tfrac{n(n-1)}2}a_0\mathbf{D}(p).
\end{equation}
\end{theorem}
\proof Indeed, the resultant~$\mathbf{R}(p,p')$
satisfies~\eqref{Resultant.formula.2}. Together
with~\eqref{Resultant.formula.working.7} and~\eqref{discriminant},
it gives
\begin{equation*}\label{discriminant.formula.working.1}
\mathbf{R}(p,p')=a_0^{n-1}\prod_{i=1}^np'(\lambda_i)=a_0^{2n-1}(-1)^{\tfrac{n(n-1)}2}\prod_{j<i}(\lambda_i-\lambda_j)^2=(-1)^{\tfrac{n(n-1)}2}a_0\mathbf{D}(p).
\end{equation*}
\eop

From~\eqref{Resultant.formula.2} we obtain
\begin{equation*}\label{discriminant.formula.working.2}
\nabla_{2n}(p,p')=(-1)^{\tfrac{n(n-1)}2}a_0\mathbf{R}(p,p')=a_0^2\mathbf{D}(p).
\end{equation*}
Thus, the discriminant of the polynomial~$p$ is the following
$(2n-1)\times(2n-1)$ determinant:
\begin{equation}\label{discriminant.determinant.formula}
\mathbf{D}(p)=\dfrac1{a_0}
\begin{vmatrix}
na_0&(n-1)a_1&(n-2)a_2&\dots& a_{n-1}&     0  &\dots&0\\
 a_0&     a_1&     a_2&\dots& a_{n-1}&   a_n  &\dots&0\\
0   &    na_0&(n-1)a_1&\dots&2a_{n-2}&a_{n-1} &\dots&0\\
0   &     a_0&     a_1&\dots& a_{n-2}&a_{n-1} &\dots&0\\
0   &     0  &    na_0&\dots&3a_{n-3}&2a_{n-2}&\dots&0\\
0   &     0  &     a_0&\dots& a_{n-3}& a_{n-2}&\dots&0\\
\vdots&\vdots&\vdots&\ddots&\vdots&\vdots&\ddots&\vdots\\
0   &0       &0       &\dots&a_1&a_2&\dots&a_{n}\\
0   &0       &0       &\dots&na_0&(n-1)a_1&\dots&a_{n-1}
\end{vmatrix}
\end{equation}

We now give one more formula for the resultant, which is very
close to the well-known Orlando formula~\cite{Orlando} (see
also~\cite{Gantmakher}). More precisely, the application of this
formula to the resultant of the odd and the even parts of a
polynomial yields exactly the Orlando formula.
\begin{theorem}\label{Th.Orlando.formula.general.1}
Let the polynomials $p$ and $q$ be given
by~\eqref{polynomial.1}--\eqref{polynomial.2}, with $\deg p =n$
and $\deg q=m\leq n-1$. Then the resultant of these
polynomials can be computed as follows:
\begin{equation}\label{Orlando.formula.general.simple}
\mathbf{R}(p,q)=(-1)^{\tfrac{n(n-1)}2} a_0^{m+n}\prod_{1\leq
i<k\leq 2n}(z_i+z_k),
\end{equation}
where $z_i$ $(i=1,\ldots,2n)$ are the zeros of the following
polynomial of degree $2n$
\begin{equation}\label{basic.poly.for.orlando}
h(z)=p(z^2)+zq(z^2).
\end{equation}
\end{theorem}
\proof%
Let $\lambda_i$ $(i=1,\ldots,n)$ be the zeros of the
polynomial $p$, and let $\mu_j$ $(j=1,\ldots,m)$ be the zeros of
the polynomial $q$. From~\eqref{basic.poly.for.orlando} it follows
that
\begin{equation}\label{proof.orlando.1}
p(z^2)=\dfrac{h(z)+h(-z)}2,\qquad q(z^2)=\dfrac{h(z)-h(-z)}{2z}.
\end{equation}
and
\begin{equation}\label{proof.orlando.2}
p(\mu_j)=h(\pm\sqrt{\mu_j}),\qquad j=1,\ldots,m,\qquad\quad
q(\lambda_i)=\pm\dfrac{h(\pm\sqrt{\lambda_i})}{\sqrt{\lambda_i}},\qquad
i=1,\ldots,n.
\end{equation}
From~\eqref{proof.orlando.1} we obtain
\begin{equation}\label{proof.orlando.3}  \qquad \qquad
p(z_k^2)=\dfrac{h(-z_k)}2,\qquad\quad
q(z_k^2)=-\dfrac{h(-z_k)}{2z_k},\qquad k=1,\ldots,2n,
\end{equation}
where $z_k$ are the zeros of the polynomial $h$. Since $\deg q=m$
by assumption, we conclude that $b_{n-m}\neq 0$ but
$b_0=\ldots=b_{n-m-1}=0$. Thus, \eqref{Resultant.formula.2}
implies
\begin{equation}\label{proof.orlando.4}
R(p,q)=b_{n-m}^{n}\prod_{j=1}^{m}p(\mu_j).
\end{equation}
Using~\eqref{proof.orlando.4},~\eqref{proof.orlando.2}
and~\eqref{proof.orlando.3}, we have
\begin{eqnarray*}%\label{proof.orlando.5}
 \left(\mathbf{R}(p,q)\right)^2&=&b_{n-m}^{2n}\prod_{j=1}^{m}h(\sqrt{\mu_j})h(-\sqrt{\mu_j})=b_{n-m}^{2n}\prod_{j=1}^{m}a_0^{2}\prod_{k=1}^{2n}(\mu_j-z_k^2)
=a_0^{2m}\prod_{k=1}^{2n}\left[b_{n-m}\prod_{j=1}^{m}(z_k^2-\mu_j)\right] \\
&=&a_0^{2m}\prod_{k=1}^{2n}q(z_k^2)=a_0^{2m}\prod_{k=1}^{2n}\dfrac{h(-z_k)}{2z_k}
= a_0^{2m}\prod_{k=1}^{2n}\dfrac{a_0}{2z_k}\prod_{i=1}^{2n}(z_i+z_k)  \\
&=& a_0^{2m+2n}\prod_{k=1}^{2n}\prod_{\substack{i=1\\
i\neq k}}^{2n}(z_i+z_k)=a_0^{2m+2n}\prod_{1\leq i<k\leq
2n}(z_i+z_k)^2.
\end{eqnarray*}
Thus, we obtain
\begin{equation*}\label{proof.orlando.6}
\mathbf{R}(p,q)=\pm a_0^{m+n}\prod_{1\leq i<k\leq
2n}(z_i+z_k).
\end{equation*}
To determine the sign, we consider the special case $h(z)=(z-1)^{2n}$.
Then $$
p(z^2) = {(z-1)^{2n} + (z+1)^{2n} \over 2}, \qquad
q(z^2) = {(z-1)^{2n} - (z+1)^{2n} \over 2z}, $$
so that $\deg p=n$, $\deg q=n-1$.
The roots of the odd part $q$ can be determined from
the equation $(z-1)^{2n}=(z+1)^{2n}$ except that the
zero root should be discarded. This shows that the roots
of $q$ are $\{ w^2_k :k=1, \ldots, n-1 \}$ where $w_k$ is
defined from the equation  $2/(w_k +1)  = 1-e^{\pi i k\over n}$,
$k=1, \ldots, n-1$.
Consequently,
\begin{eqnarray*}
 \prod_{k=1}^{n-1} p(w^2_k) &= & \prod_{k=1}^{n-1} \left(
2 \over 1- e^{\tfrac{\pi i k}{n}} \right)^{2n}\; = \;
\prod_{k=1}^{n-1} \left( i e^{\tfrac{-\pi i k}{2n}} \over \sin
\left(\pi k\over n \right)  \right)^{2n}\\  & = & i^{2n(n-1)}
\prod_{k=1}^{n-1} e^{-\pi i k } {1\over \sin^{2n} \left(\pi k\over
n \right) } \; = \; (-1)^{\sum\limits_{i=1}^{n-1} k }
\prod_{k=1}^{n-1} {1\over \sin^{2n} \left(\pi k\over n \right) }.
\end{eqnarray*}
Thus, according to the last formula in~(\ref{Resultant.formula.2}),
 the sign of the resultant $R(p,q)$ in our special case is
$$(-1)^{n(n-1)} \sgn b_1^n (-1)^{\tfrac{n(n-1)}2}=(-1)^{\tfrac{n(n+1)}2}=(-1)^{n+\tfrac{n(n-1)}2}.$$
As we already established, the resultant of $p$ and $q$ is a
polynomial in the coefficients of $p$ and $q$, hence in the
coefficients of $h$. The expression $(-1)^{\tfrac{n(n-1)}2}
a_0^{2n-1} \prod_{i<k} (z_i+z_k)$ is a symmetric function of the
roots of $h$ multiplied by its leading coefficient to the power
$2n-1$, and hence also a polynomial in the coefficients of $h$.
Since the two polynomials must be identically equal, we conclude
that the sign $(-1)^{\tfrac{n(n-1)}2}$ occurs at all times
whenever $\deg p=n$, $\deg q=n-1$.

We now show how to produce a formula for the case $m<n-1$ from the already established
formula for $m=n-1$. Given a polynomial $q$ of degree $n-1$, set $b_1$ through
$b_{n-m-1}$ to zero, thus obtaining a polynomial of degree $m$.  Observe what happens to the determinantal
expression~\eqref{Resultant}. Since the lower left block of size $n\times(n-m-1)$
is now filled with zeros, the upper-triangular submatrix above produces the
factorization $$ \widetilde{\mathbf{R}}(p,q)=a_0^{n-m-1}\mathbf{R}(p,q), $$
where $\widetilde{\mathbf{R}}(p,q)$ denotes the ``old'' resultant  of $p$ and $q$
constructed  as if $\deg q$ were  equal to $(n-1)$.
Thus the ``new'' resultant $\mathbf{R}(p,q)$ satisfies the equation
$$ a_0^{n-m-1}\mathbf{R}(p,q)=(-1)^{n(n-1)\over 2}a_0^{2n-1} \prod_{1\leq i<k \leq 2n} (z_i+z_j),$$
hence $\mathbf{R}(p,q)=(-1)^{n(n-1)\over 2} a_0^{n+m}  \prod_{1\leq i<k \leq 2n} (z_i+z_j)$.
\eop
Our next statement can be proved analogously.
\begin{theorem}\label{Th.Orlando.formula.general.2}
Let the polynomials $p$ and $q$ be given
by~\eqref{polynomial.1}--\eqref{polynomial.2}, and let $\deg
q=m\leq n=\deg p$. Then the resultant of these polynomials
can be computed by the formula
\begin{equation}\label{Orlando.formula.general.simple.2}
\mathbf{R}(p,q)=(-1)^{\tfrac{n(n-1)}2}a_0^{m+n}\prod_{1\leq
i<k\leq 2n+1}(z_i+z_k),
\end{equation}
where $z_i$ $(i=1,\ldots,2n+1)$ are the zeros of the polynomial
\begin{equation*}\label{basic.poly.for.orlando.2}
g(z)=q(z^2)+zp(z^2).
\end{equation*}
\end{theorem}

The famous Orlando formula is a simple consequence of
Theorems~\ref{Th.Orlando.formula.general.1}--\ref{Th.Orlando.formula.general.2}.
Before proving the Orlando formula, we introduce the following
determinants for the polynomial~\eqref{polynomial.1}
\begin{equation}\label{delta}
\Delta_{j}(p)=
\begin{vmatrix}
a_1&a_3&a_5&a_7&\dots&a_{2j-1}\\
a_0&a_2&a_4&a_6&\dots&a_{2j-2}\\
0  &a_1&a_3&a_5&\dots&a_{2j-3}\\
0  &a_0&a_2&a_4&\dots&a_{2j-4}\\
\vdots&\vdots&\vdots&\vdots&\ddots&\vdots\\
0  &0  &0  &0  &\dots&a_{j}
\end{vmatrix},\quad j=1,\ldots,n,
\end{equation}
where we set $a_i\eqbd 0$ for $i>n$.
\begin{definition}\label{def.Hurwitz.dets}
The determinants $\Delta_{j}(p)$ $(j=1,\ldots,n)$ are called
\textit{the Hurwitz determinants} or \textit{the Hurwitz minors}
of the polynomial~$p$.
\end{definition}

It is easy to see that the polynomial $p$ can be always
represented as follows
\begin{equation}\label{app.poly.odd.even}
p(z)=p_0(z^2)+zp_1(z^2),
\end{equation}
where
\begin{equation}\label{poly1.12}
\begin{split}
&p_0(u)=a_1u^l+a_3u^{l-1}+\cdots+a_{n},\\
&p_1(u)=a_0u^l+a_2u^{l-1}+\cdots+a_{n-1},
\end{split}
\end{equation}
if the degree $n$ of the polynomial $p(z)$ is odd: $n=2l+1$, and
\begin{equation}\label{poly1.13}
\begin{split}
&p_0(u)=a_0u^l+a_2u^{l-1}+\cdots+a_{n},\\
&p_1(u)=a_1u^{l-1}+a_3u^{l-2}+\cdots+a_{n-1},
\end{split}
\end{equation}
if $n=2l$.

\begin{theorem}[Orlando, \cite{Orlando,{Gantmakher}}]\label{Th.Orlando.formula}
Let the polynomial $p$ of degree $n$ be given
by~\eqref{polynomial.1}. Then the determinant $\Delta_{n-1}(p)$
defined by~\eqref{delta} can be computed from the formula
\begin{equation}\label{Orlando.formula}
\Delta_{n-1}(p)=(-1)^{\tfrac{n(n-1)}2}a_0^{n-1}\prod_{1\leq
i<j\leq n}(z_i+z_j),
\end{equation}
where $z_i$, $i=1,\ldots,n$, are the zeros of the polynomial $p$.
\end{theorem}
This equality is known as \textit{the Orlando formula}.

\proof At first, let the degree $n$ of $p$ be odd, $n=2l+1$.
Then~\eqref{app.poly.odd.even}--\eqref{poly1.12} show that $\deg
p_0\leq l$, $\deg p_1 =l$, and the leading coefficient of
$p_1$ is $a_0$.
Thus,~\eqref{Hurwitz.General.Determinants},~\eqref{Resultant.formula.2}
and~\eqref{Orlando.formula.general.simple.2} imply

\begin{equation*}\label{Orlando.formula.proof.1}
\Delta_{n-1}(p)=(-1)^l\nabla_{2l}(p_1,p_0)=(-1)^{\tfrac{l(l+1)}2}a_0^{l-\deg
p_0}\mathbf{R}(p_1,p_0)=(-1)^la_0^{2l}\prod_{1\leq
i<j\leq2l+1}(z_i+z_j),
\end{equation*}
which coincides with~\eqref{Orlando.formula} since $n=2l+1$.

If $n=2l$, then~\eqref{app.poly.odd.even} and~\eqref{poly1.13}
imply $\deg p_1\leq l-1$, $\deg p_0=l$, and the leading
coefficient of the~polynomial $p_0$ is $a_0$. As above, the
formul\ae~\eqref{Hurwitz.General.Determinants},~\eqref{Resultant.formula.2}
and~\eqref{Orlando.formula.general.simple} can now be combined to
obtain
\begin{equation*}\label{Orlando.formula.proof.2}
\Delta_{n-1}(p)=a_0^{-1}\nabla_{2l}(p_0,p_1)=(-1)^{\tfrac{l(l-1)}2}a_0^{l-\deg
p_1-1}\mathbf{R}(p_0,p_1)=(-1)^la_0^{2l-1}\prod_{1\leq
i<j\leq2l}(z_i+z_j),
\end{equation*}
which is exactly the formula~\eqref{Orlando.formula}.
\eop

%%%%%%%%%%%%%%%%%%%%%%%%%%%%%%%%%%%%%%%%%%%%%%%%%%%%%%%%%%%%%%%%%%%%%
\subsection{\label{s:Euclidean.algorithm.general.complex.case}Euclidean
algorithm, the greatest common divisor, and continued fractions:
general case}
%%%%%%%%%%%%%%%%%%%%%%%%%%%%%%%%%%%%%%%%%%%%%%%%%%%%%%%%%%%%%%%%%%%%%

In the previous subsection, we considered the resultant of two
polynomials and observed that it is equal to zero  if and only if
these polynomials have a nontrivial common divisor. The standard
way to find their greatest common divisor is via the Euclidean
algorithm.

Let us consider polynomials $p$ and $q$ given
by~\eqref{polynomial.1}--\eqref{polynomial.2} and let us
denote\footnote{Thus, $\deg f_1 <\deg f_0 $.
Obviously, if $\deg q <\deg p$, that is, if $b_0=0$, then
$f_1=q$.}
$$f_0(z) \eqbd p(z), \qquad f_1(z) \eqbd
q(z)-\dfrac{b_0}{a_0}p(z).$$ Construct a sequence of polynomials
${f_0,f_1,\ldots,f_k}$ $(k\leq n)$ by the following formula
\begin{equation}\label{Euclidean.algorithm}
f_{j-1}(z) = q_{j}(z)f_{j}(z)+f_{j+1}(z),\quad j=1,\ldots,k\quad
(f_{k+1}(z)= 0),
\end{equation}
where $q_j$ is the quotient and $f_{j+1}$ is the remainder from
the division of $f_{j-1}$ by $f_j$. The last polynomial in this
sequence, $f_k$, is the greatest common divisor of the polynomials
$f_0$ and $f_1$ (and also all other polynomials $f_j$ in the
sequence). In other words,
\begin{equation}\label{polys.via.gcd}
f_j(z)=h_j(z)f_k(z),\quad j=0,1,\ldots,k,
\end{equation}
where $h_k(z)=1$. Now, denote
\begin{equation}\label{sequence.of.rat.functions}
R_j(z) \eqbd
\dfrac{f_{j}(z)}{f_{j-1}(z)}=\dfrac{h_{j}(z)}{h_{j-1}(z)}.
\end{equation}
Rewriting~\eqref{Euclidean.algorithm}, we obtain
\begin{equation}\label{Euclidean.cont.frac}
R_{j}(z)=\dfrac1{q_{j}(z)+R_{j+1}(z)},\quad j=1,\ldots,k,
\end{equation}
where $R_{k+1}(z)\equiv 0$. Using this equality, we can represent the
function $R_1(z)$ as a continued fraction\footnote{The functions
$R=q/p$ and $R_1$ are related via the formula $R(z)=R_1(z)+b_0/a_0$.}:
\begin{equation}\label{continued.fraction.general}
R_1(z)=\dfrac{f_{1}(z)}{f_0(z)}=\dfrac{h_{1}(z)}{h_0(z)}=\dfrac1{q_1(z)+\cfrac1{q_2(z)+\cfrac1{q_3(z)+\cfrac1{\ddots+\cfrac1{q_k(z)}}}}} ~.
\end{equation}
It is easy to see that, for each  $j=0,\ldots,k{-}1$, the polynomial $h_j$
is the leading principal minor of order $k{-}j$ of the following
$k\times k$ tridiagonal matrix
\begin{equation}\label{Generalized.Jacobi.matrix}
\mathcal{J}(z)=\begin{pmatrix}
  q_k(z) &   -1        &  0          & \dots  & 0       & 0     \\
   1     &  q_{k-1}(z) & -1          & \dots  & 0       & 0      \\
  0      &  1          &  q_{k-2}(z) & \dots  & 0       & 0       \\
  \vdots &   \vdots    & \vdots      & \ddots & \vdots  & \vdots   \\
  0      & 0           & 0           & \cdots &  q_2(z) & -1        \\
  0      & 0           & 0           & \cdots &  1      &  q_1(z)
\end{pmatrix}.
\end{equation}
In particular, $h_0(z)=\det\mathcal{J}(z)$. Thus, the determinant
of the matrix $\mathcal{J}(z)$ is the denominator of $R(z)$. This
is a very  useful observation, as certain properties of the
function $R$ turn out to be connected to the location of the
eigenvalues of the generalized (matrix polynomial) eigenvalue
problem
\begin{equation}\label{Generilized.eigenvalue.problem}
\mathcal{J}(z)u=0.
\end{equation}
Conversely, the behavior of the eigenvalues of the
problem~\eqref{Generilized.eigenvalue.problem} can provide
information about certain properties of the function $R$.
Later in this paper we will give examples of such interrelation.

More generally, in addition to the
fraction~\eqref{continued.fraction.general} we may consider the
fractions
\begin{equation*}\label{continued.fraction.general.parts}
R_{j}(z)=\dfrac{h_{j}(z)}{h_{j-1}(z)}=\dfrac1{q_{j}(z)+\cfrac1{q_{j+1}(z)+\cfrac1{q_{j+2}(z)+\cfrac1{\ddots+\cfrac1{q_k(z)}}}}},\quad
j=1,\ldots,k.
\end{equation*}
The continued fraction
expansion~\eqref{continued.fraction.general} can be found
efficiently in terms of the corresponding Hankel minors, as the
next theorem shows.
\begin{theorem}\label{Th.relation.continued.fraction.Hankel.minors}
Two rational functions $R$ and $G$ vanishing at $\infty$ are
related via the identity
\begin{equation}\label{Th.for.criterion.J-fraction.statement.1}
R(z)=\dfrac1{g(z)+G(z)},
\end{equation}
where $g$ is a polynomial of degree $m(\geq 1)$ with leading
coefficient $\alpha$
\begin{equation}\label{Th.for.criterion.J-fraction.statement.1.5}
g(z)=\alpha z^m+\cdots
\end{equation}
if and only if
\begin{equation}\label{Th.for.criterion.J-fraction.statement.1.8}
D_{1}(R)=D_{2}(R)=\cdots=D_{m-1}(R)=0,\quad(\text{when}\quad m>1)
\end{equation}
and
\begin{equation}\label{Th.for.criterion.J-fraction.statement.2}
D_{m+j}(R)=(-1)^{\tfrac{m(m-1)}2}\cdot\dfrac{(-1)^{j}D_{j}(G)}{\alpha^{m+2j}},\quad
j=0,1,2,\ldots,
\end{equation}
where $D_0(G)\eqbd 1$, and the determinants $D_j(R)$ and $D_j(G)$
are defined by~\eqref{Hankel.determinants.1}.

\end{theorem}
\proof
Since the functions $R$ and $G$ vanish at $\infty$, they can be
expanded into Laurent series
\begin{equation}\label{Th.for.criterion.J-fraction.proof.1}
R(z)=\dfrac{s_{0}}{z}+\dfrac{s_{1}}{z^{2}}+\cdots+\dfrac{s_{m-1}}{z^{m}}+\dfrac{s_{m}}{z^{m+1}}+\cdots
, \qquad
G(z)=\dfrac{t_{0}}{z}+\dfrac{t_{1}}{z^{2}}+\dfrac{t_{2}}{z^{3}}+\cdots
.
\end{equation}
Moreover, the
conditions~\eqref{Th.for.criterion.J-fraction.statement.1}--\eqref{Th.for.criterion.J-fraction.statement.1.5}
hold if and only if
\begin{equation}\label{Th.for.criterion.J-fraction.proof.1.5}
s_0=s_1=\cdots=s_{m-2}=0\quad\text{and}\quad s_{m-1}\neq0.
\end{equation}
In fact, if $R$
satisfies~\eqref{Th.for.criterion.J-fraction.statement.1}--\eqref{Th.for.criterion.J-fraction.statement.1.5},
then
\begin{equation}\label{Th.for.criterion.J-fraction.proof.2}
s_i=\displaystyle\lim_{z\to\infty}z^{i+1}R(z)=\displaystyle\lim_{z\to\infty}\dfrac1{\dfrac{g(z)}{z^{i+1}}+\dfrac{G(z)}{z^{i+1}}}=0,\quad
i=0,1,2,\ldots,m-2.
\end{equation}
and
\begin{equation}\label{Th.for.criterion.J-fraction.proof.2.5}
s_{m-1}=\displaystyle\lim_{z\to\infty}z^{m}R(z)=\displaystyle\lim_{z\to\infty}\dfrac1{\alpha+
\dfrac{\gamma_1z^{m-1}+\cdots+\gamma_m}{z^{m}}+\dfrac{G(z)}{z^{m}}}=\dfrac1{\alpha}\neq0.
\end{equation}

Now assume that the
condition~\eqref{Th.for.criterion.J-fraction.proof.1.5} holds and
that $R(z)=\dfrac1{f(z)+G(z)}$ for some polynomial $f$. If $\deg f
=j$ for some $1\leq j<m$,  that is\footnote{The degree $j$ of
$f$ cannot be $0$, since  the limit
$\displaystyle\lim_{z\to\infty}R(z)=\dfrac1{\gamma}$ is nonzero
for a constant nonzero function $f(z)=\gamma$.}, if $f(z)=\gamma
z^j+\cdots$, where $\gamma\neq 0$, then
$s_{j-1}=\displaystyle\lim_{z\to\infty}z^jR(z)=\dfrac1{\gamma}\neq0$,
contrary to~\eqref{Th.for.criterion.J-fraction.proof.1.5}. On the
other hand, if the degree of $f$ is greater than $m$, then
$s_{m-1}=\displaystyle\lim_{z\to\infty}z^mR(z)=0$. Thus, $f$ must
be of exact degree~$m$ since otherwise a contradiction arises with
one of the
conditions~\eqref{Th.for.criterion.J-fraction.proof.1.5}.

Also note that the
equalities~\eqref{Th.for.criterion.J-fraction.proof.1.5} are
equivalent to~\eqref{Th.for.criterion.J-fraction.statement.1.8}.
Moreover, from~\eqref{Th.for.criterion.J-fraction.proof.1.5} we
have
\begin{equation}\label{Th.for.criterion.J-fraction.proof.2.7}
D_m(R)=
\begin{vmatrix}
    0 &0 &0 &\dots &0 &s_{m-1}\\
    0 &0 &0 &\dots&s_{m-1} &s_m\\
    \vdots&\vdots&\vdots&\ddots&\vdots&\vdots\\
    0 &s_{m-1} &s_{m} &\dots&s_{2m-4} &s_{2m-3}\\
    s_{m-1} &s_m &s_{m+1} &\dots&s_{2m-3} &s_{2m-2}
\end{vmatrix}=(-1)^{\tfrac{m(m-1)}2}s_{m-1}^m\neq0.
\end{equation}
For convenience, denote the coefficients of the polynomial $g$ as
follows
\begin{equation*}\label{Th.for.criterion.J-fraction.proof.2.8}
g(z)\bdeq \alpha
z^m+t_{-m}z^{m-1}+t_{-m+1}z^{m-2}+\cdots+t_{-2}z+t_{-1}.
\end{equation*}
If the functions $R$ and $G$
satisfy~\eqref{Th.for.criterion.J-fraction.statement.1}--\eqref{Th.for.criterion.J-fraction.statement.1.5},
then, according
to~\eqref{Th.for.criterion.J-fraction.proof.1}--\eqref{Th.for.criterion.J-fraction.proof.1.5},
we have
\begin{equation*}\label{Th.for.criterion.J-fraction.proof.3}
\left[\dfrac{s_{m-1}}{z^{m}}+\dfrac{s_{m}}{z^{m+1}}+\dfrac{s_{m+1}}{z^{m+2}}+\cdots\right]\left[\alpha
z^m+t_{-m}z^{m-1}+\cdots+t_{-1}+\dfrac{t_{0}}{z}+\dfrac{t_{1}}{z^{2}}+\dfrac{t_{2}}{z^{3}}+\cdots\right]\equiv 1.
\end{equation*}
This identity implies the following relations:
\begin{equation}\label{Th.for.criterion.J-fraction.proof.4}
\begin{array}{lll}
\displaystyle s_{m-1} & = & \dfrac1{\alpha},\\
% \\
\displaystyle s_{m+j} & = &
-\sum\limits_{i=0}^{j}\dfrac{t_{-m+i}}{\alpha}\cdot
s_{m+j-i-1},\qquad j=0,1,2,\ldots
\end{array}
\end{equation}
Now, the equality~\eqref{Th.for.criterion.J-fraction.statement.2}
for $D_{m}(R)$ follows
from~\eqref{Th.for.criterion.J-fraction.proof.2.7}--\eqref{Th.for.criterion.J-fraction.proof.4}.

For a fixed number $j\geq 1$, consider the determinant
$D_{m+j}(R)$:
\begin{equation}\label{Th.for.criterion.J-fraction.proof.5}
D_{m+j}(R)=
\begin{vmatrix}
    0 &  0 &\dots& 0 &s_{m-1}&s_{m}&\dots & %s_{m+j-2}&
    s_{m+j-1}\\
    0 &  0 &\dots&s_{m-1}&s_{m}&s_{m+1}&\dots & %s_{m+j-1}&
    s_{m+j}\\
    \vdots&\vdots&\ddots&\vdots&\vdots&\vdots&\ddots&%\vdots&
    \vdots\\
    0 &s_{m-1}&\dots&s_{2m-4}&s_{2m-3}&s_{2m-2}&\dots&%s_{2m+j-4}&
    s_{2m+j-3}\\
    s_{m-1} &s_{m}&\dots&s_{2m-3}&s_{2m-2}&s_{2m-1}&\dots&%s_{2m+j-3}&
    s_{2m+j-2}\\
    s_{m} &s_{m+1}&\dots&s_{2m-2}&s_{2m-1}&s_{2m}&\dots&%s_{2m+j-2}&
    s_{2m+j-1}\\
    \vdots&\vdots&\ddots&\vdots&\vdots&\vdots&\ddots&%\vdots&
    \vdots\\
%    s_{m+j-2} &s_{m+j-1}&\dots&s_{2m+j-4}&s_{2m+j-3}&s_{2m+j-2}
%    &\dots&s_{2m+2j-4}&s_{2m+2j-3}\\
    s_{m+j-1} &s_{m+j}&\dots &s_{2m+j-3}&s_{2m+j-2}&s_{2m+j-1}&\dots&%s_{2m+2j-3}&
    s_{2m+2j-2}
\end{vmatrix}.
\end{equation}
Add to each $i^{\mathrm{th}}$ column $(i=m+j,m+j-1,\ldots,3,2)$
columns $i-1,i-2,\ldots,1$ multiplied by
$\dfrac{t_{-m}}{\alpha},\dfrac{t_{-m+1}}{\alpha},\ldots$,
$\dfrac{t_{j-3}}{\alpha},\dfrac{t_{j-2}}{\alpha}$, respectively.
This eliminates the entries in the upper right corner of the
determinant~\eqref{Th.for.criterion.J-fraction.proof.5}
using~\eqref{Th.for.criterion.J-fraction.proof.4}. So, we  can
rewrite the original determinant as a product of the following two
determinants of order $m$ and $j$, respectively:
\begin{equation*}\label{Th.for.criterion.J-fraction.proof.6}
D_{m+j}(R)=
\begin{vmatrix}
    0&0&\dots&0&s_{m-1}\\
    0&0&\dots&s_{m-1}&0\\
    \vdots&\vdots&\ddots&\vdots&\vdots\\
    0&s_{m-1}&\dots &0&0\\
    s_{m-1}&0&\dots&0&0
\end{vmatrix}\cdot
\begin{vmatrix}
    d_{11} &d_{12} &\dots &d_{1,j-1} &d_{1,j}\\
    d_{21} &d_{22} &\dots &d_{2,j-1} &d_{2,j}\\
    d_{31} &d_{32} &\dots &d_{3,j-1} &d_{3,j}\\
    \vdots&\vdots&\ddots&\vdots&\vdots\\
%    d_{j-1,1} &d_{j-1,2} &\dots &d_{j-1,j-1} &d_{j-1,j}\\
    d_{j1} &d_{j2} &\dots &d_{j,j-1} &d_{jj}\\
\end{vmatrix},
\end{equation*}
where
\begin{equation}\label{Th.for.criterion.J-fraction.proof.7}
d_{i_1,i_2}\eqbd
-s_{m+i_1-2}\cdot\dfrac{t_{i_2-1}}{\alpha}-s_{m+i_1-3}\cdot\dfrac{t_{i_2}}{\alpha}-\cdots-s_{m}\cdot\dfrac{t_{i_1+i_2-3}}{\alpha}-s_{m-1}\cdot\dfrac{t_{i_1+i_2-2}}{\alpha}.
\end{equation}
From~\eqref{Th.for.criterion.J-fraction.proof.4}
and~\eqref{Th.for.criterion.J-fraction.proof.7} we obtain
\begin{equation*}\label{Th.for.criterion.J-fraction.proof.8}
D_{m+j}(R)=(-1)^{\tfrac{m(m-1)}2}\cdot\dfrac{(-1)^{j}}{\alpha^{m+j}}
%\begin{vmatrix}
%    s_{m-1}&s_{m}&\dots&s_{2m-3}&s_{2m-2}\\
%    0&s_{m-1}&\dots&s_{2m-2}&s_{2m-1}\\
%    \vdots&\vdots&\ddots&\vdots&\vdots\\
%    0&0&\dots&s_{m-1}&s_{m}\\
%    0&0&\dots&0&s_{m-1}
%\end{vmatrix}
\begin{vmatrix}
    s_{m-1}&0&\dots&0&0\\
    s_{m}&s_{m-1}&\dots&0&0\\
    \vdots&\vdots&\ddots&\vdots&\vdots\\
    s_{m+j-3}&s_{m+j-4}&\dots&s_{m-1}&0\\
    s_{m+j-2}&s_{m+j-3}&\dots&s_{m}&s_{m-1}
\end{vmatrix}\cdot
\begin{vmatrix}
    t_0 &t_{1} &\dots &t_{j-2} &t_{j-1}\\
    t_1 &t_{2} &\dots &t_{j-1} &t_{j}\\
    \vdots&\vdots&\ddots&\vdots&\vdots\\
    t_{j-2} &t_{j-1} &\dots &t_{2j-4} &t_{2j-3}\\
    t_{j-1} &t_{j} &\dots &t_{2j-3} &t_{2j-2}\\
\end{vmatrix}.
\end{equation*}
Conversely, if
the~equalities~\eqref{Th.for.criterion.J-fraction.statement.1.8}--\eqref{Th.for.criterion.J-fraction.statement.2}
hold, then  we obtain $s_0=s_1=\cdots=s_{m-2}=0$
from~\eqref{Th.for.criterion.J-fraction.statement.1.8} by
induction, which was already proved to be equivalent to the fact
that $R$
satisfies~\eqref{Th.for.criterion.J-fraction.statement.1}, with a
polynomial $g$ of degree at least $m$ and
$D_{m}(R)=(-1)^{\tfrac{m(m-1)}2}s_{m-1}^m$. If $\deg g >m$, then
\begin{equation*}\label{Th.for.criterion.J-fraction.proof.9}
s_{m-1}=\displaystyle\lim_{z\to\infty}z^{m}R(z)=\displaystyle\lim_{z\to\infty}\dfrac1{\dfrac{g(z)}{z^{m}}+\dfrac{G(z)}{z^{m}}}=0,
\end{equation*}
therefore,
$D_{m}(R)=(-1)^{\tfrac{m(m-1)}2}s_{m-1}^m=(-1)^{\tfrac{m(m-1)}2}\dfrac1{\alpha^m}=0$
and $\dfrac1{\alpha}=0$, according
to~\eqref{Th.for.criterion.J-fraction.statement.2}. Thus,
from~\eqref{Th.for.criterion.J-fraction.statement.1.8}--\eqref{Th.for.criterion.J-fraction.statement.2}
we get $D_i(R)=0$ for all $i\in\mathbb{N}$, which is impossible,
since $R$ is a rational function, hence at least one of the
minors $D_i(R)$ must be nonzero due to
Theorem~\ref{Th.Hankel.matrix.rank.2}. Consequently, the
polynomial $g$
satisfies~\eqref{Th.for.criterion.J-fraction.statement.1.5} with
$\alpha\neq 0$.
\eop

If a rational function $R$ with exactly $r$ poles is expanded into
a continued fraction~\eqref{continued.fraction.general}, then the
rational functions $R_j$ defined
in~\eqref{sequence.of.rat.functions} satisfy the
relations~\eqref{Euclidean.cont.frac}, where
\begin{equation}\label{Euclidean.cont.frac.quotients}
q_j(z)=\alpha_jz^{n_j}+\cdots, \qquad \alpha_j\neq 0,\quad
j=1,2,\ldots,k.
\end{equation}
Here $n_1+n_2+\cdots+n_k=r$ and $n_1\geq 1$, since $\deg f_1
< \deg f_0$, as is remarked above. The other degrees~$n_i$ are
greater or equal to $1$ due to the structure of the Euclidean
algorithm~\eqref{Euclidean.algorithm}.

From~\eqref{Th.for.criterion.J-fraction.statement.1.8}--\eqref{Th.for.criterion.J-fraction.statement.2}
we obtain, for a fixed integer $j$ $(j=1,2,\ldots,k)$,
the following product formul\ae:
\begin{equation*}%\label{Th.J-fraction.criterion.proof.3}
\begin{array}{lcl}
D_{n_1+n_2+\cdots+n_j}(R_1) & =
&(-1)^{\tfrac{n_1(n_1-1)}2}(-1)^{\sum\limits_{i=2}^{j}n_i}
\cdot\dfrac1{\alpha_1^{n_1+2\sum\limits_{i=2}^{j}n_i}}\cdot D_{n_2+n_3+\cdots+n_j}(R_2)=\\
 \\
& = &
(-1)^{\tfrac{n_1(n_1-1)}2}(-1)^{\sum\limits_{i=2}^{j}n_i}\cdot\dfrac1{\alpha_1^{n_1+2\sum\limits_{i=2}^{j}n_i}}
\cdot(-1)^{\tfrac{n_2(n_2-1)}2}(-1)^{\sum\limits_{i=3}^{j}n_i}\times\\
 \\
&& \times\dfrac1{\alpha_2^{n_2+2\sum\limits_{i=3}^{j}n_i}}\cdot
D_{n_3+n_4+\cdots+n_j}(R_3)=\cdots
\end{array}
\end{equation*}
This chain of equalities results in the formula
\begin{equation}\label{continued.fraction.determinants.formula}
\displaystyle
D_{n_1+n_2+\cdots+n_j}(R)=\prod_{i=1}^{j}(-1)^{\tfrac{n_i(n_i-1)}2}
\cdot(-1)^{\sum\limits_{i=0}^{j-1}in_{i+1}}\cdot\prod_{i=1}^{j}
\dfrac1{\alpha_i^{n_i+2\sum\limits_{\rho=i+1}^{j}n_{\rho}}},\quad
j=1,2,\ldots,k.
\end{equation}

\begin{remark}\label{remark.1.2}
Our discussion above can be summarized as follows:
Suppose that a rational function $R$ with $r$ poles has a continued fraction
expansion~\eqref{continued.fraction.general} with
polynomials~$q_j$ satisfying~\eqref{Euclidean.cont.frac.quotients}.
Then
Theorems~\ref{Th.Hankel.matrix.rank.2},~\ref{Th.Hankel.matrix.rank.1}
and~\ref{Th.relation.continued.fraction.Hankel.minors} and
formul\ae~\eqref{continued.fraction.determinants.formula} imply
that all Hankel minors $D_j(R)$ are equal to zero, except for the
minors\footnote{Recall that $n_1+n_2+\cdots+n_k=r$.}
$D_{n_1}(R),D_{n_1+n_2}(R)$, $\ldots,D_{n_1+n_2+\cdots+n_k}(R)$,
which are not zero and which can be calculated from the
formul\ae~\eqref{continued.fraction.determinants.formula}.
\end{remark}

Using Theorem~\ref{Th.relation.continued.fraction.Hankel.minors}
and the formul\ae~\eqref{continued.fraction.determinants.formula},
we now describe equivalence classes of rational functions whose
sequences of Hankel minors are the same.
\begin{theorem}\label{Th.class.of.Hankel.equivalence}
Two rational functions
\begin{equation*}\label{Th.class.of.Hankel.equivalence.condition.1}
R(z)=\dfrac{s_{0}}{z}+\dfrac{s_{1}}{z^{2}}+\dfrac{s_{2}}{z^{3}}+\cdots\quad\text{and}\quad
G(z)=\dfrac{t_{0}}{z}+\dfrac{t_{1}}{z^{2}}+\dfrac{t_{2}}{z^{3}}+\cdots~,
\end{equation*}
both vanishing at infinity,  have equal Hankel minors
\begin{equation}\label{Th.class.of.Hankel.equivalence.statement.1}
D_j(R)=D_j(G),\qquad j=1,2,\ldots~,
\end{equation}
if and only if their continued fraction expansions
\begin{equation*}\label{Th.class.of.Hankel.equivalence.condition.2}
R(z)=\dfrac1{q_1(z)+\cfrac1{q_2(z)+\cfrac1{q_3(z)+\cfrac1{\ddots+\cfrac1{q_{k_1}(z)}}}}}\quad\text{and}\quad G(z)=
\dfrac1{\widetilde{q}_1(z)+\cfrac1{\widetilde{q}_2(z)+\cfrac1{\widetilde{q}_3(z)+\cfrac1{\ddots+\cfrac1{\widetilde{q}_{k_2}(z)}}}}}
\end{equation*}
satisfy
\begin{equation}\label{Th.class.of.Hankel.equivalence.statement.2}
k_1=k_2 \bdeq k,
\end{equation}
and the polynomials $q_j$ and $\widetilde{q}_j$, for each $j$
$(j=1,2,\ldots,k)$, have equal degrees and equal leading
coefficients:
\begin{equation}\label{Th.class.of.Hankel.equivalence.statement.3}
\begin{array}{c}
q_j(z)=\alpha_jz^{n_j}+\cdots,\\[2mm]
\widetilde{q}_j(z)=\alpha_jz^{n_j}+\cdots,
\end{array}\qquad j=1,2,\ldots,k.
\end{equation}
\end{theorem}
\proof
If the continued fraction expansions of the functions $R$ and $G$
satisfy the
conditions~\eqref{Th.class.of.Hankel.equivalence.statement.2}--\eqref{Th.class.of.Hankel.equivalence.statement.3},
then the
equalities~\eqref{Th.class.of.Hankel.equivalence.statement.1}
follow from Remark~\ref{remark.1.2} and the
formul\ae~\eqref{continued.fraction.determinants.formula}.

Conversely, let the Hankel minors associated with the functions
$R$ and $G$
satisfy~\eqref{Th.class.of.Hankel.equivalence.statement.1}.
Therefore, by Theorems~\ref{Th.Hankel.matrix.rank.2}
and~\ref{Th.Hankel.matrix.rank.1}, the functions $R$ and $G$ have
equal number of poles. Moreover, they can be then represented as
follows:
\begin{equation*}\label{Th.class.of.Hankel.equivalence.proof.1}
R(z)=R_1(z)=\dfrac1{q_1(z)+R_2(z)}\quad\text{and}\quad
G(z)=G_1(z)=\dfrac1{\widetilde{q}_1(z)+G_2(z)}.
\end{equation*}
It follows from
Theorem~\ref{Th.relation.continued.fraction.Hankel.minors}
 that the degrees and the leading coefficients of the polynomials
$q_1$ and $\widetilde{q}_1$ coincide. According
to~\eqref{Th.for.criterion.J-fraction.statement.1.8}, the
functions $R_2$ and $G_2$ have equal Hankel minors. So we can
apply the same argument to them. Thus,
Theorem~\ref{Th.relation.continued.fraction.Hankel.minors} shows
that the degrees and the leading coefficients of each pair of the
polynomials $q_j,\widetilde{q}_j$ are equal to each other and,
consequently, the number of those polynomials must be the same as
well. Since the functions $R$ and $G$ have equal number of poles
as was proved above, the
equality~\eqref{Th.class.of.Hankel.equivalence.statement.2} thus
follows. \eop

\begin{remark} Note that the equalities $D_j(R)=D_j(G)$ do not imply
the equality of the functions $R\equiv G$. Counterexamples are quite
easy to construct. For instance, we can take
\begin{eqnarray*}
 R(z) &= & {1 \over z-1 - \cfrac{1}{z-2} }={1\over z} +{1\over z^2} +{2\over z^3}+{5\over z^4}+\cdots, \\
 G(z)&=&{1 \over z-1 - \cfrac{1}{z-3} } = {1\over z} +{1\over z^2} +{2\over z^3}+{6\over z^4}+\cdots~.
\end{eqnarray*}
Then $R\not\equiv G$ but $D_1(R)=D_2(R)=D_1(G)=D_2(G)=1$.
\end{remark}
\medskip

Finally, let the function $R$ have a continued fraction
expansion~\eqref{continued.fraction.general} and also a Laurent
series expansion
\begin{equation}\label{power.series}
R(z)=\dfrac{s_{n_1-1}}{z^{n_1}}+\dfrac{s_{n_1}}{z^{n_1+1}}+\dfrac{s_{n_1+1}}{z^{n_1+2}}+\cdots,
\end{equation}
where $n_1=\deg q_1 \geq 1$.

Consider the functions
\begin{equation}\label{continued.fraction.general.parts.2}
F_j(z) \eqbd
\dfrac{Q_j(z)}{P_j(z)}=\dfrac1{q_1(z)+\cfrac1{q_2(z)+\cfrac1{q_3(z)+\cfrac1{\ddots+\cfrac1{q_j(z)}}}}},\qquad
j=1,\ldots,k,
\end{equation}
constructed using the
polynomials~\eqref{Euclidean.cont.frac.quotients} by the Euclidean
algorithm~\eqref{Euclidean.algorithm}.

\begin{definition}
The polynomials $Q_j$ are called {\em partial numerators,\/}
the polynomials $P_j$ {\em partial denominators,\/} and
the fractions $F_j$ {\em partial quotients\/} of $R$.
\end{definition}

The denominator $P_j(z)$ of the fraction $F_j(z)$ is the $j$th
leading principal minor of the matrix
\begin{equation*}\label{Generalized.Jacobi.matrix.2}
\begin{pmatrix}
  q_1(z) &   -1        &  0          & \dots  & 0       & 0     \\
   1     &  q_{2}(z) & -1          & \dots  & 0       & 0      \\
  0      &  1          &  q_{3}(z) & \dots  & 0       & 0       \\
  \vdots &   \vdots    & \vdots      & \ddots & \vdots  & \vdots   \\
  0      & 0           & 0           & \cdots &  q_{k-1}(z) & -1        \\
  0      & 0           & 0           & \cdots &  1      &  q_k(z)
\end{pmatrix},
\end{equation*}
and $P_k(z)=h_0(z)$ (see~\eqref{continued.fraction.general}). Let
$m_j$ denote the sum of the degrees $n_1$ through $n_j$
(see~\eqref{Euclidean.cont.frac.quotients}):
\begin{equation}\label{degrees}
m_j=n_1+n_2+\cdots+n_j,\quad j=1,2,\ldots,k.
\end{equation}
Then $\deg P_j = m_j$.

Notice that, for a fixed number $j$ $(1\leq j\leq k)$,
the initial terms of the Laurent series of the function $F_j$
coincide with those in the Laurent series~\eqref{power.series} of
$R$ up to, and including, the term $\dfrac{s_{2m_j-1}}{z^{2m_j}}$:
\begin{equation}\label{series.cutted.continued.fraction}
F_j(z)=\dfrac{Q_j(z)}{P_j(z)}=\dfrac{s_{n_1-1}}{z^{n_1}}+\dfrac{s_{n_1}}{z^{n_1+1}}+\cdots+\dfrac{s_{2m_j-1}}{z^{2m_j}}+\dfrac{s^{(j)}_{2m_j}}{z^{2m_j+1}}+\cdots
\end{equation}
In fact, each coefficient $s_i$ of the series~\eqref{power.series}
can be found from the recurrence relation
\begin{equation}\label{recurrent.formula.for.power.series}
\displaystyle
s_i=\lim_{z\to\infty}\left[z^{i+1}R(z)-s_{n_1-1}z^{i-n_1+1}-s_{n_1}z^{i+n_1}-\cdots-s_{i-1}z\right],\qquad
i=n_1-1,n_1,n_1+1,\ldots
\end{equation}
Using the expansions~\eqref{continued.fraction.general}
and~\eqref{continued.fraction.general.parts.2} of the functions
$R$ and $F_j$, respectively, together with the
formula~\eqref{recurrent.formula.for.power.series}, we
obtain~\eqref{series.cutted.continued.fraction}.

To find an explicit formula for the polynomials $P_j$ that depends
only on the coefficients $s_i$, we introduce the following
notation:
\begin{equation*}\label{quasi-orthogonal.quotients}
P_j(z)\bdeq
P_{0,j}z^{m_j}+P_{1,j}z^{m_j-1}+\cdots+P_{m_j-1,j}z+P_{m_j,j},\qquad
j=1,2,\ldots,k.
\end{equation*}
For a fixed number $j$ between $1$ and $k$, the formula
\eqref{series.cutted.continued.fraction} implies
\begin{equation}\label{series.cutted.continued.fraction.2}
Q_j(z)=P_j(z)\left[\dfrac{s_{n_1-1}}{z^{n_1}}+\dfrac{s_{n_1}}{z^{n_1+1}}+\cdots+\dfrac{s_{2m_j-1}}{z^{2m_j}}+\dfrac{s^{(j)}_{2m_j}}{z^{2m_j+1}}+\cdots\right].
\end{equation}
We must require that the  coefficients of
$z^{-1},z^{-2},\ldots,z^{-m_j}$ be zero, which leads to the system
\begin{equation}\label{system.for.coeffs.quasi-orthogonal.quotients}
\begin{array}{llll}
&s_{0}P_{m_j,j}+s_{1}P_{m_j-1,j}+\cdots+s_{m_j-1}P_{1,j}+s_{m_j}P_{0,j}&=&0,\\
&s_{1}P_{m_j,j}+s_{2}P_{m_j-1,j}+\cdots+s_{m_j}P_{1,j}+s_{m_j+1}P_{0,j}&=&0,\\
&\qquad \qquad  \qquad \qquad \qquad \vdots\qquad &\; \vdots&\vdots \\
&s_{m_j-1}P_{m_j,j}+s_{m_j}P_{m_j-1,j}+\cdots+s_{2m_j-2}P_{1,j}+s_{2m_j-1}P_{0,j}&=&0,\\
\end{array}
\end{equation}
where we set $s_0 \eqbd s_1 \eqbd \cdots \eqbd s_{n_1-2} \eqbd 0$.
By Cramer's rule, the solution to the
system~\eqref{system.for.coeffs.quasi-orthogonal.quotients}
satisfies
\begin{equation*}\label{coeffs.quasi-orthogonal.quotients}
P_{m_j-i,j}=\dfrac{(-1)^{m_j-i}P_{0,j}}{D_{m_j}(R)}\cdot
\begin{vmatrix}
    s_0 &s_1&\dots &s_{i-1} &s_{i+1} &\dots &s_{m_j}\\
    s_1 &s_2&\dots &s_{i} &s_{i+2} &\dots &s_{m_j+1}\\
    \vdots&\vdots&\ddots&\vdots&\vdots&\ddots&\vdots\\
    s_{m_j-1} &s_{m_j}&\dots &s_{m_j+i-2} &s_{m_j+i} &\dots &s_{2m_j-1}
\end{vmatrix},\quad
\begin{array}{l}
i=0,1,2,\dots,m_j-1,\\
j=1,2,\ldots,k.
\end{array}
\end{equation*}
This formula implies the following representation for the polynomials $P_j$:
\begin{equation}\label{quasi-orthogonal.quotients.via.determinants}
P_j(z)=\dfrac{P_{0,j}}{D_{m_j}(R)}
\begin{vmatrix}
    s_0 &s_1 &s_{2} &\dots &s_{m_j}\\
    s_1 &s_2 &s_{3} &\dots &s_{m_j+1}\\
    \vdots&\vdots&\vdots&\ddots&\vdots\\
    s_{m_j-1} &s_{m_j}&s_{m_j+1} &\dots &s_{2m_j-1}\\
    1 &z &z^2 &\dots &z^{m_j}\\
\end{vmatrix},\quad j=1,2,\ldots,k.
\end{equation}
Here, according to the notation~\eqref{Euclidean.cont.frac.quotients},
\begin{equation*}\label{leading.coeffs.of.quasi-orthogonal.quotients}
P_{0,j}=\prod_{i=1}^{j}\alpha_i,\quad j=1,2,\ldots,k.
\end{equation*}
This formula,  combined
with~\eqref{continued.fraction.determinants.formula} and the
notation~\eqref{degrees}, yields
\begin{eqnarray*}%\label{leading.coeffs.of.quasi-orthogonal.quotients.even}
 P_{0,2j} & = & (-1)^{j}\prod_{i=1}^{j}
\left({\dfrac{(-1)^{\tfrac{n_{2i-1}(n_{2i-1}-1)}2}D_{m_{2i-1}}(R)}{D_{m_{2i-2}}(R)}}
\right)^{\tfrac1{n_{2i-1}}}
% \cdot
\left({\dfrac{(-1)^{\tfrac{n_{2i}(n_{2i}-1)}2}D_{m_{2i-1}}(R)}{D_{m_{2i}}(R)}}\right)^{\tfrac1{n_{2i}}}  \\
&& {\rm for} \quad j=1,2,\ldots,\left\lfloor\dfrac{k}2\right\rfloor.  \\
%\left[\dfrac{(-1)^{\frac{n_{2i-1}(n_{2i-1}-1)}2}D_{m_{2i-1}}(R)}{D_{m_{2i-2}}(R)}\right]^{1\over {n_{2i-1}}}
%\left[\dfrac{(-1)^{\frac{n_{2i}(n_{2i}-1)}2}D_{m_{2i-1}}(R)}{D_{m_{2i}}(R)}\right]^{1 \over {n_{2i}}},
%
%\label{leading.coeffs.of.quasi-orthogonal.quotients.odd}
 P_{0,2j+1} &  = & (-1)^{j}\prod_{i=1}^{j}
\left({\dfrac{(-1)^{\tfrac{n_{2i-1}(n_{2i-1}-1)}2}D_{m_{2i-2}}(R)}{D_{m_{2i-1}}(R)}}\right)^{\tfrac1{n_{2i-1}}}
%\cdot
\left({\dfrac{(-1)^{\tfrac{n_{2i}(n_{2i}-1)}2}D_{m_{2i}}(R)}{D_{m_{2i-1}}(R)}}\right)^{\tfrac1{n_{2i}}}\times\\
&&   \times
\left({\dfrac{(-1)^{\tfrac{n_{2j+1}(n_{2j+1}-1)}2}D_{m_{2j}}(R)}{D_{m_{2j+1}}(R)}}\right)^{\tfrac1{n_{2i+1}}}\qquad
\qquad {\rm for} \quad
j=0,1,2,\ldots,\left\lceil\dfrac{k}2\right\rceil-1.
\end{eqnarray*}
In case the Euclidean algorithm is {\it regular,\/} the notion
that will be introduced in the sequel, these formul\ae\, take a
simpler form.

The coefficients of $Q_j$ can be obtained from  the coefficients
of $P_j$ via the formula
\eqref{series.cutted.continued.fraction.2}: denoting
\begin{equation*}\label{quasi-orthogonal.numerator}
Q_j(z) =
Q_{0,j}z^{m_j-n_1}+Q_{1,j}z^{m_j-n_1-1}+\cdots+Q_{m_j-n_1-1,j}z+Q_{m_j-n_1,j},
\end{equation*}
we find the coefficients of $Q_j$ by the following
\begin{equation*}
Q_{i,j}=\dfrac{(-1)^{m_j}P_{0,j}}{D_{m_j}(R)}\cdot
\begin{vmatrix}
    0 &0&\dots &s_{0} &s_{1} &\dots &s_{n_1+i-1}\\
    s_0 &s_1&\dots &s_{m_j-n_1-i+1} &s_{m_j-n_1-i+2} &\dots &s_{m_j}\\
    s_1 &s_2&\dots &s_{m_j-n_1-i+2} &s_{m_j-n_1-i+3} &\dots &s_{m_j+1}\\
    \vdots&\vdots&\ddots&\vdots&\vdots&\ddots&\vdots\\
    s_{m_j-1} &s_{m_j}&\dots &s_{2m_j-n_1-i} &s_{2m_j-n_1-i+1} &\dots &s_{2m_j-1}
\end{vmatrix}
\end{equation*}
for $i=0,1,\ldots,m_j-n_1$.

Note that the rational functions $F_j(z)=\dfrac{Q_j(z)}{P_j(z)}$
are exactly the diagonal elements of the Pad\'e table of the
function~$R$, see~\cite{Baker_Graves,Henrici}.
%%
%
%
%
%from the recurrence relations
%
%\begin{eqnarray*}\label{system.for.coeffs.quasi-orthogonal.numerators}
%Q_{0,j} &= &s_{n_1-1}P_{0,j},\\
%Q_{1,j} &= &s_{n_1}P_{0,j}+s_{n_1-1}P_{1,j},\\
%\;\vdots & \vdots &  \qquad \vdots \\
%%&Q_{m_j-n_1-1,j}=s_{m_j-2}P_{0,j}+s_{m_j-3}P_{1,j}+\cdots+s_{n_1}P_{m_j-n_1-2,j}+s_{n_1-1}P_{m_j-n_1-1,j},\\
%Q_{m_j-n_1,j}&=&s_{m_j-1}P_{0,j}+s_{m_j-2}P_{1,j}+\cdots+s_{n_1}P_{m_j-n_1-1,j}+s_{n_1-1}P_{m_j-n_1,j},\\
%\end{eqnarray*}
%

%%%%%%%%%%%%%%%%%%%%%%%%%%%%%%%%%%%%%%%%%%%%%%%%%%%%%%%%%%%%%%%%%%%%%%%
\subsection{\label{s:Euclidean.regular.case}Euclidean algorithm: regular case.
Finite continued fractions of Jacobi type}
%%%%%%%%%%%%%%%%%%%%%%%%%%%%%%%%%%%%%%%%%%%%%%%%%%%%%%%%%%%%%%%%%%%%%%%

In this section we discuss the best known case of the Euclidean algorithm,
which leads to orthogonal polynomials, $3$-term recurrence relations
and other related phenomena.  We give conditions for regularity and discuss
the form of partial quotients and the generalized eigenvalue
problem corresponding to the regular case.

Suppose that all the polynomials $q_j$ resulting from an
application the Euclidean algorithm~\eqref{Euclidean.algorithm}
are linear:
\begin{equation}\label{reg.Euclidean.algorithm}
q_j(z)=\alpha_jz+\beta_j,\quad \alpha_j,\beta_j\in\mathbb{C},\qquad
\alpha_j\neq0,\quad j=1,\ldots,r.
\end{equation}
We call this situation the {\em regular case} of the Euclidean algorithm.
In the regular case,
\begin{equation}
\deg f_j(z)=n-j,\qquad j=0,\ldots,r,
\end{equation}
and $r=n-l$, where $l$ is the degree of the greatest common
divisor~$f_r$ of the polynomials $f_0$ and $f_1$ (and of all
other polynomials $f_j$ in the sequence).

Thus, the polynomials $f_j$ satisfy the following
{\em three-terms recurrence relation}:
\begin{equation}\label{three.term.recurrence.relation.1}
f_{j-1}(z)=(\alpha_jz+\beta_j)f_j(z)+f_{j+1}(z),\quad
j=1,\ldots,r.
\end{equation}
Consequently, the function $R$ expands into the continued
fraction
\begin{equation}\label{continued.fraction.Jacobi}
R(z)=\dfrac{f_{1}(z)}{f_0(z)}=\dfrac{h_{1}(z)}{h_0(z)}=\dfrac1{\alpha_1z+\beta_1+\cfrac1{\alpha_2z+\beta_2+\cfrac1{\alpha_3z+\beta_3+\cfrac1{\ddots+\cfrac1{\alpha_rz+\beta_r}}}}},
\end{equation}
where the polynomials $h_0$ and $h_1$ are defined by~\eqref{polys.via.gcd}.

\begin{definition}\label{def.J-fraction}
Continued fractions of the type~\eqref{continued.fraction.Jacobi}
are usually called \textit{$J$-fractions} or continued fractions
of Jacobi type.
\end{definition}

\begin{remark}\label{remark.1.3}
If a rational function $G$ satisfies the condition
$\displaystyle\lim_{z\to\infty}G(z)=c$, $0<|c|<\infty$, then we
will say that $G(z)$ has a $J$-fraction expansion if the function
$R(z)=G(z)-\displaystyle\lim_{z\to\infty}G(z)$ can be represented
as in~\eqref{continued.fraction.Jacobi}.
\end{remark}

Theorem~\ref{Th.relation.continued.fraction.Hankel.minors} with
some simple modifications implies the following corollary, which will
be useful later.
\begin{corol}\label{corol.for.criterion.J-fraction}
Two rational functions $R$ and $G$, both vanishing at~$\infty$, satisfy
the following condition
\begin{equation*}%\label{corol.for.criterion.J-fraction.statement.1}
R(z)=\dfrac1{\alpha z+\beta+G(z)},\quad\alpha\neq0,
\end{equation*}
if and only if
\begin{equation}\label{corol.for.criterion.J-fraction.statement.2}
D_{j}(R)=\dfrac{(-1)^{j-1}}{\alpha^{2j-1}}D_{j-1}(G),\quad
j=1,2,\ldots,
\end{equation}
where $D_0(G)\eqbd 1$.

\end{corol}

Using Corollary~\ref{corol.for.criterion.J-fraction}, we can now prove
a criterion when a rational function expands into a $J$-fraction.
\begin{theorem}[\cite{PerronII}]\label{Th.J-fraction.criterion}
A rational function
\begin{equation}\label{J-fractions.criterion.function}
R(z) = s_{-1}+\dfrac{s_{0}}{z}+\dfrac{s_{1}}{z^{2}}+\dfrac{s_{2}}{z^{3}}+\cdots
\end{equation}
with exactly $r$ poles has a $J$-fraction expansion if and only if
\begin{equation}\label{J-fractions.determinant.inequalities}
D_j(R)\neq0,\quad j=1,\ldots,r,
\end{equation}
where $D_j(R)$ are the Hankel determinants defined
in~\eqref{Hankel.determinants.1}.
\end{theorem}
\proof
In view of Remark~\ref{remark.1.3} it is sufficient to consider
the case $s_{-1}=0$.
Suppose that the function $R$ has a $J$-fraction
expansion~\eqref{continued.fraction.Jacobi}, i.e., a
continued fraction expansion~\eqref{continued.fraction.general}
with polynomials~$q_j$ satisfying~\eqref{reg.Euclidean.algorithm}
with $k=r$. Then the formul\ae~\eqref{continued.fraction.determinants.formula}
yield
\begin{equation}\label{J-fraction.determinants.formula}
D_{j}(R)=(-1)^{\tfrac{j(j-1)}2}\prod_{i=1}^{j}\dfrac1{\alpha_i^{2j-2i+1}}\neq0,\quad
j=1,2,\ldots,r.
\end{equation}

Now suppose that the
inequalities~\eqref{J-fractions.determinant.inequalities} hold.
Then $D_1(R_1)=s_0^{(1)}\neq 0$,  where we denote\footnote{Recall
that we assumed $s_{-1}=0$.}
\begin{equation*}\label{Th.J-fraction.criterion.proof.3}
R_1(z)\eqbd
R(z)=\dfrac{s_{0}^{(1)}}{z}+\dfrac{s_{1}^{(1)}}{z^{2}}+\dfrac{s_{2}^{(1)}}{z^{3}}+\cdots
\end{equation*}
Therefore, $R_1(z)$ can be represented as follows
\begin{equation*}\label{Th.J-fraction.criterion.proof.4}
R_1(z)=\dfrac1{\alpha_1
z+\beta_1+R_2(z)},\quad\alpha_1=\dfrac1{s_0^{(1)}}\neq0,
\end{equation*}
Now, Corollary~\ref{corol.for.criterion.J-fraction} implies that
$D_j(R_1)=\dfrac{(-1)^{j-1}}{\alpha_1^{2j-1}}D_{j-1}(R_2)$,
$j=1,2,\ldots$. Therefore, the function~$R_2$ is of the same type
as $R_1$, that is, $R_1$ satisfies the conditions\footnote{For
$j\geq r$, we have $D_{j+1}(R_1)=D_{j}(R_2)=0,$ and therefore
$R_2$ has exactly $r-1$ poles.}
\begin{equation*}\label{Th.J-fraction.criterion.proof.44444}
D_{j}(R_2)\neq0,\quad j=1,2,\ldots,r-1,
\end{equation*}
which are analogous
to~\eqref{J-fractions.determinant.inequalities}. In particular,
$s_0^{(2)}\neq 0$, where
\begin{equation*}\label{Th.J-fraction.criterion.proof.5}
R_2(z)=\dfrac{s_{0}^{(2)}}{z}+\dfrac{s_{1}^{(2)}}{z^{2}}+\dfrac{s_{2}^{(2)}}{z^{3}}+\cdots
\end{equation*}
If we continue this process, then for each function
\begin{equation*}\label{Th.J-fraction.criterion.proof.6}
R_j(z)=\dfrac{s_{0}^{(j)}}{z}+\dfrac{s_{1}^{(j)}}{z^{2}}+\dfrac{s_{2}^{(j)}}{z^{3}}+\cdots
\end{equation*}
we obtain that $s_0^{(j)}\neq 0$, $j=1,2,\ldots,r$. Consequently,
\begin{equation*}\label{Th.J-fraction.criterion.proof.7}
R_j(z)=\dfrac1{\alpha_jz+\beta_j+R_{j+1}(z)},\quad j=1,2,\ldots,r,
\end{equation*}
where $R_{r+1}\equiv 0$. This means that the function $R$
can be represented as a $J$-fraction
(see~\eqref{Euclidean.cont.frac}--\eqref{continued.fraction.general}).
\eop

This theorem and Theorem~\ref{Th.Hurwitz.relations} imply the
following statement, which was proved in~\cite{Wall} (Theorem 41.1)
with inessential differences.
\begin{corol}\label{corol.J-fraction.criterion}
Let $R(z)=\dfrac{q(z)}{p(z)}$, where $p$ and $q$ are defined
in~\eqref{polynomial.1}--\eqref{polynomial.2}, and let
the function $R$  have exactly $r$ poles. The function $R$ can be
expanded into a $J$-fraction~\eqref{continued.fraction.Jacobi}
if and only if
\begin{equation}\label{J-fractions.nabla.inequalities}
\nabla_{2j}(p,q)\neq0,\qquad j=1,\ldots,r,
\end{equation}
where $\nabla_{2j}(p,q)$ are defined
in~\eqref{Hurwitz.General.Determinants}.
\end{corol}
\begin{corol}\label{regular.Euclidean.agorithm.criterion}
Given a rational function $R(z)=\dfrac{f_1(z)}{f_0(z)}$, where
$\deg f_1<\deg f_0$, the Euclidean
algorithm~\eqref{Euclidean.algorithm} applied to the polynomials
$f_0$ and $f_1$ is regular if and only if the
inequalities~\eqref{J-fractions.determinant.inequalities} hold.
\end{corol}
If we make an equivalence transformation of the
fraction~\eqref{continued.fraction.Jacobi}, as in~\cite[p.~166]{Wall},
and set
\begin{equation}\label{J-fractions.second.form.change.var.1}
d_0\eqbd \dfrac1{\alpha_1},\quad
d_j\eqbd -\dfrac1{\alpha_j\alpha_{j+1}},\qquad j=1,2,\ldots,r-1.
\end{equation}
\begin{equation}\label{J-fractions.second.form.change.var.2}
\qquad \quad e_j\eqbd -\dfrac{\beta_j}{\alpha_j},\qquad \qquad
\quad j=1,2,\ldots,r,
\end{equation}
then the $J$-fraction~\eqref{continued.fraction.Jacobi} takes the
form
\begin{equation}\label{J-fractions.second.form}
R(z)=\dfrac{d_0}{z-e_1-\cfrac{d_1}{z-e_2-\cfrac{d_2}{\ddots-\cfrac{d_{r-1}}{z-e_r}}}}.
\end{equation}
Moreover, from~\eqref{J-fraction.determinants.formula}
and~\eqref{J-fractions.second.form.change.var.1} we have
(see~\cite[p.~167]{Wall})
\begin{equation}\label{J-fraction.determinants.formula.2}
D_{j}(R)=D_{j-1}(R)\cdot\prod_{i=1}^{j}d_{i-1}.
\end{equation}
In fact,
\begin{eqnarray*}\label{J-fraction.determinants.formula.2.proof}
D_{j}(R) & = &
(-1)^{\tfrac{j(j-1)}2}\prod_{i=1}^{j}\dfrac1{\alpha_i^{2j-2i+1}}=(-1)^{\tfrac{j(j-1)}2}\dfrac1{\alpha_j}\cdot\prod_{i=1}^{j-1}\dfrac1{\alpha_i^2}
\cdot\prod_{i=1}^{j-1}\dfrac1{\alpha_i^{2j-2i-1}}=\\
& = &
(-1)^{\tfrac{j(j-1)}2}(-1)^{\tfrac{(j-2)(j-1)}2}(-1)^{j-1}\dfrac1{\alpha_1}\cdot\prod_{i=1}^{j-1}\dfrac{-1}{\alpha_{i}\alpha_{i+1}}
\cdot D_{j-1}(R)=D_{j-1}(R)\cdot\prod_{i=1}^{j}d_{i-1}.
\end{eqnarray*}
The expression~\eqref{J-fraction.determinants.formula.2} implies
an interesting formula:
\begin{equation}\label{J-fraction.determinants.formula.3}
d_j=\dfrac{D_{j-1}(R)D_{j+1}(R)}{D_{j}^2(R)},\qquad
j=0,1,\ldots,r-1,
\end{equation}
where we set $D_{-1}(R)\eqbd D_0(R)\eqbd 1$.

\begin{remark}\label{remark.1.4} From the
formul\ae~\eqref{J-fractions.second.form.change.var.1},~\eqref{J-fractions.second.form.change.var.2}
and~\eqref{J-fraction.determinants.formula.3} it follows that both forms
 of $J$-fraction expansions~\eqref{continued.fraction.Jacobi}
and~\eqref{J-fractions.second.form} of the function $R(z)$ are
\textit{unique}, that is, all coefficients
in~\eqref{continued.fraction.Jacobi}
and~\eqref{J-fractions.second.form} are defined uniquely for
the~function $R$.
\end{remark}
\begin{remark}\label{remark.1.5}(cf.~\cite[p.170]{Wall})
Suppose that the function~$R$ has a $J$-fraction
expansion~\eqref{continued.fraction.Jacobi}. Then
\begin{equation}\label{continued.fraction.Jacobi.for.Stieltjes}
R(-z)=-\dfrac{1}{\alpha_1z-\beta_1+\cfrac1{\alpha_2z-\beta_2+\cfrac1{\alpha_3z-\beta_3+\cfrac1{\ddots+\cfrac1{\alpha_rz-\beta_r}}}}}.
\end{equation}
According to Remark~\ref{remark.1.4}, (a properly normalized) $J$-fraction
expansion is unique. Therefore, if the function $R$ is odd, i.e.,
$R(z)\equiv -R(-z)$, then all $\beta_j$
in~\eqref{continued.fraction.Jacobi.for.Stieltjes} must be equal
to zero. Conversely, if $\beta_j$ are all equal to zero,
then~$R(z)$ is obviously an odd function of $z$.
\end{remark}

\vspace{2mm}

Thus, the formula~\eqref{J-fraction.determinants.formula.2} implies
\begin{equation}\label{J-fraction.determinants.formula.2213}
D_{j}(R)=\prod_{i=0}^{j}d^{j-i}_{i}
\end{equation}
In other words, if the function $R$ has a $J$-fraction
expansion~\eqref{J-fractions.second.form}, the minors $D_j(R)$ do
not depend on the coefficients $e_j$, $j=1,\ldots,r$, whereas
the minors $\widehat{D}_j(R)$ obviously do. In order to establish
the dependence of these minors on the coefficients of the
fraction~\eqref{J-fractions.second.form}, we first prove the
following simple fact.

\begin{lemma}\label{lemma.rat.func.det}
Let the complex rational function
\begin{equation}\label{J-fractions.criterion.function.2}
R(z) =
\dfrac{s_{0}}{z}+\dfrac{s_{1}}{z^{2}}+\dfrac{s_{2}}{z^{3}}+\cdots
\end{equation}
with exactly $r$ poles, has a
$J$-fraction expansion~\eqref{J-fractions.second.form}. Then
\begin{equation}\label{lemma.det.formula}
\dfrac{\widehat{D}_r(R)}{D_r(R)}=\prod\limits_{k=1}^{r}\mu_k,
\end{equation}
where $\mu_k$ are the poles of the function $R$, and the minors
$D_j(R)$ and $\widehat{D}_j(R)$, $j=1,2,\ldots$, are defined
by~\eqref{Hankel.determinants.1}
and~\eqref{Hankel.determinants.2}.
\end{lemma}
\begin{proof}
If $R$ has a pole at zero, then the lemma holds true, since
in this case $D_r(R)\neq0$ by
Theorem~\ref{Th.J-fraction.criterion}, while $\widehat{D}_r(R)=0$
by Corollary~\ref{corol.zero.pole}. Now assume that $R$ has no a
pole at zero.

At first, suppose that the function $R$ has only simple poles
$\mu_k\neq\mu_j$ whenever $k\neq j$, so it can be represented as a
sum of partial fractions
\begin{equation*}\label{lemma.det.proof.1}
R(z)=\sum_{k=1}^r\dfrac{\nu_k}{z-\mu_k}
\end{equation*}
This formula together
with~\eqref{J-fractions.criterion.function.2} gives the following
well-known formul\ae
\begin{equation}\label{lemma.det.proof.2}
s_j=\sum_{k=1}^r\nu_k\mu_k^{j},\qquad j=0,1,2,\ldots
\end{equation}
On the other hand, \eqref{Resultant.formula.working.3} implies
\begin{equation}\label{lemma.det.proof.3}
D_r(G)=\prod\limits_{k=1}^{r}\nu_k\cdot
\begin{vmatrix}
    1 &1 &\dots &1\\
    \mu_1 &\mu_2 &\dots &\mu_r\\
    \vdots&\vdots&\ddots&\vdots\\
    \mu_1^{r-1} &\mu_2^{r-1} &\dots &\mu_r^{r-1}\\
\end{vmatrix}^2~.
\end{equation}
Further, recall that $\widehat{D}_r(R)=D_r(\Phi)$, where
\begin{equation*}\label{lemma.det.proof.4}
\Phi(z)=zR(z)=\sum_{k=1}^r\nu_k+\sum_{k=1}^r\dfrac{\nu_k\mu_k}{z-\mu_k}=s_0+\dfrac{s_1}{z}+\dfrac{s_2}{z^2}+\dfrac{s_3}{z^3}+\ldots
\end{equation*}
So, analogously to~\eqref{lemma.det.proof.3}, we obtain
\begin{equation}\label{lemma.det.proof.5}
\widehat{D}_r(R)=D_r(\Phi)=\prod\limits_{k=1}^{r}(\nu_k\mu_k)\cdot
\begin{vmatrix}
    1 &1 &\dots &1\\
    \mu_1 &\mu_2 &\dots &\mu_r\\
    \vdots&\vdots&\ddots&\vdots\\
    \mu_1^{r-1} &\mu_2^{r-1} &\dots &\mu_r^{r-1}\\
\end{vmatrix}^2~.
\end{equation}
Now~\eqref{lemma.det.formula} follows
from~\eqref{lemma.det.proof.3} and~\eqref{lemma.det.proof.5}.

\vspace{2mm}

Let $R=\dfrac{q}{p}$ and suppose that the function $R$ has
multiple poles, or, equivalently, that the polynomial~$p$ has
multiple zeros. In this case, we can consider an approximating
polynomial $p_\varepsilon$ with simple zeros such that
\begin{equation*}\label{lemma.det.proof.6}
\displaystyle\lim_{\varepsilon\to 0} p_\varepsilon(z) =p(z) \quad
\hbox{\rm for all } \;  z.
\end{equation*}
Then
\begin{equation*}\label{lemma.det.proof.7}
\displaystyle D_{r}(R_\varepsilon)\xrightarrow[\varepsilon\to0]{}
D_{r}(R),\qquad\displaystyle\widehat{D}_{r}(R_\varepsilon)\xrightarrow[\varepsilon\to0]{}\widehat{D}_{r}(R),
\end{equation*}
where $R_\varepsilon(z) \eqbd \dfrac{q(z)}{p_\varepsilon(z)}$. The
formula~\eqref{lemma.det.formula} is valid for the polynomials $q$
and $p_\varepsilon$, so it is also valid at the limit, i.e., for
the polynomials $q$ and $p$, since the product of all zeros of
the polynomial $p_{\varepsilon}(z)$ tends to the product of all
zeros of the polynomial $p$ whenever $\varepsilon\to0$.
\end{proof}

Let us consider the following tridiagonal matrix
\begin{equation}\label{Jacobi.matrix}
\mathcal{J}_r=
\begin{pmatrix}
    e_1 & \sqrt{d_1} &  0 &\dots&   0   & 0 \\
    \sqrt{d_1} & e_2 &\sqrt{d_2} &\dots&   0   & 0 \\
     0  &\sqrt{d_2} & e_3 &\dots&   0   & 0 \\
    \vdots&\vdots&\vdots&\ddots&\vdots&\vdots\\
     0  &  0  &  0  &\dots&e_{r-1}&\sqrt{d_{r-1}}\\
     0  &  0  &  0  &\dots&\sqrt{d_{r-1}}&e_r\\
\end{pmatrix}
\end{equation}
constructed using the coefficients of the
$J$-fraction~\eqref{J-fractions.second.form}.

Suppose that the coefficient $d_0$
in~\eqref{J-fractions.second.form} equals $1$ and consider the
partial quotients of the
$J$-fraction~\eqref{J-fractions.second.form}
\begin{equation}\label{J-fractions.second.form.partial.quot}
F_j(z)=\dfrac{Q_j(z)}{P_j(z)}=\dfrac{1}{z-e_1-\cfrac{d_1}{z-e_2-\cfrac{d_2}{\ddots-\cfrac{d_{j-1}}{z-e_j}}}},\qquad
j=1,\ldots,r.
\end{equation}
Then $F_r=R$ and the polynomial $P_j$ is the
characteristic polynomial of the leading principal submatrix of
$\mathcal{J}_r$ of order $j$, for any $j=1,\ldots,r$.

Using the formula~\eqref{lemma.det.formula}, it is now easy to
prove the following theorem.
\begin{theorem}\label{theorem.tridiag.minors}
Let the matrix $J_n$ be defined by~\eqref{Jacobi.matrix} and let
the function $R$  defined
by~\eqref{J-fractions.criterion.function.2} have a $J$-fraction
expansion~\eqref{J-fractions.second.form} (with $d_0=1$). Then
the~leading principal minors
$\displaystyle|\mathcal{J}_r|_{1}^{m}$, $m=1,\ldots,r$, of the
matrix $\mathcal{J}_r$ can be found by the following formul\ae:
\begin{equation}\label{theorem.tridiag.minors.formula}
\displaystyle|\mathcal{J}_r|_{1}^{m}=\dfrac{\widehat{D}_m(R)}{D_m(R)},\qquad
m=1,\ldots,r.
\end{equation}
\end{theorem}
\begin{proof}
At first, we note that the
formula~\eqref{theorem.tridiag.minors.formula} holds for
$\det(\mathcal{J}_r)=\displaystyle|\mathcal{J}_r|_{1}^{r}$. Indeed,
on the one hand, the determinant of the matrix
$\mathcal{J}_r$ is the product of its eigenvalues.
On the other hand, the eigenvalues of the matrix $\mathcal{J}_r$
are the poles of the function $R$, so
by~\eqref{lemma.det.formula}, their product equals
$\dfrac{\widehat{D}_r(R)}{D_r(R)}$.

Let us now turn to the functions $F_j$ introduced
in~\eqref{J-fractions.second.form.partial.quot}. According
to~\eqref{series.cutted.continued.fraction}, we have
\begin{equation}\label{series.cutted.continued.fraction.new}
F_m(z)=\dfrac{Q_m(z)}{P_m(z)}=\dfrac{s_{0}}{z}+\dfrac{s_{1}}{z^{2}}+\cdots+\dfrac{s_{2m-1}}{z^{2m}}+\dfrac{s^{(m)}_{2m}}{z^{2m+1}}+\cdots,\qquad
m=1,\ldots,r,
\end{equation}
where the coefficients $s_i$, $i=0,1,\ldots,2m-1$ coincide with
those of the function $R$ defined
in~\eqref{J-fractions.criterion.function.2}. Consequently,
\begin{equation*}\label{Minors.equality}
D_j(F_m)=D_j(R),\qquad \widehat{D}_j(F_m)=\widehat{D}_j(R),\qquad
j=1,\ldots,m.
\end{equation*}
Thus, applying Lemma~\ref{lemma.rat.func.det} to the function
$F_m$ and to the leading principal submatrix of
$\mathcal{J}_r$ of order $m$, we obtain
\begin{equation*}\label{Minors.equality.2}
\displaystyle|\mathcal{J}_r|_{1}^{m}=\dfrac{\widehat{D}_j(F_m)}{D_j(F_m)}=\dfrac{\widehat{D}_m(R)}{D_m(R)},\qquad
m=1,\ldots,n,
\end{equation*}
as required.
\end{proof}

Now suppose that the function $R$ has a $J$-fraction
expansion~\eqref{J-fractions.second.form}, where $d_0$ may
differ from  $1$. Then we can consider the function
$G(z):=\dfrac{R(z)}{d_0}$. It is clear that
\begin{equation*}\label{theorem.tridiag.minors.formula.334q5}
D_m(R)=d_0^mD_m(G),\qquad\widehat{D}_j(R)=d_0^m\widehat{D}_j(G)\qquad
m=1,\ldots,r.
\end{equation*}
Consequently, according to~\eqref{theorem.tridiag.minors.formula}, we get
\begin{equation}\label{theorem.tridiag.minors.formulanew.222}
\displaystyle|\mathcal{J}_r|_{1}^{m}=\dfrac{\widehat{D}_m(G)}{D_m(G)}=\dfrac{\widehat{D}_m(R)}{D_m(R)},\qquad
m=1,\ldots,r.
\end{equation}

Thus, if the function $R$ has a $J$-fraction
expansion~\eqref{J-fractions.second.form}, then it follows
from~\eqref{theorem.tridiag.minors.formulanew.222}
and~\eqref{theorem.tridiag.minors.formula} that the
minors $\widehat{D}_j(R)$ can be found as follows
\begin{equation*}\label{theorem.tridiag.minors.formula.3345}
\widehat{D}_m(R)=\displaystyle|\mathcal{J}_r|_{1}^{m}\cdot\prod_{i=0}^{m-1}d^{m-i}_{i},\qquad
m=1,\ldots,r,
\end{equation*}
where $|\mathcal{J}_r|_{1}^{m}$ is the leading principal minor of
order $m$ of the matrix $\mathcal{J}_r$, $m=1,\ldots,r$.

Finally, note that the matrix $\mathcal{J}_j$ can be replaced throughtout
the discussion above by the matrix
$$ %\mathcal{J}_r=
\begin{pmatrix}
    e_1 & 1 &  0 &\dots&   0   & 0 \\
    d_1 & e_2 & 1 &\dots&   0   & 0 \\
     0  & d_2  & e_3 &\dots&   0   & 0 \\
    \vdots&\vdots&\vdots&\ddots&\vdots&\vdots\\
     0  &  0  &  0  &\dots& e_{r-1}& 1 \\
     0  &  0  &  0  &\dots& d_{r-1} &e_r\\
\end{pmatrix}
$$
which is diagonally similar to the matrix $\mathcal{J}_r$ via the
transformation $$ {\rm diag} (1, \sqrt{d_1}, \sqrt{d_1 d_2}, \ldots,
\sqrt{d_1, \ldots, d_{r-1}}), $$
which also preserves all principal minors.

\vspace{3mm}

Let us now return to the
equation~\eqref{Generilized.eigenvalue.problem}. If all the
polynomials $q_i$ in the
fraction~\eqref{continued.fraction.general} are linear, as
in~\eqref{reg.Euclidean.algorithm}, then the
equation~\eqref{Generilized.eigenvalue.problem} becomes
a~generalized eigenvalue problem
\begin{equation}\label{Generilized.eigenvalue.Jacobi.problem}
(\mathrm{A} z+\mathrm{B})u=0,
\end{equation}
for the matrix pair $(\mathrm{A},\mathrm{B})$ of order $r$  where $\mathrm{A}$
is diagonal and $\mathrm{B}$ is tridiagonal (for a similar setup, see  \cite{Gesz_Simon}):
\begin{equation}\label{D.E.matrices}
\mathrm{A}=\begin{pmatrix}
   \alpha_r &    0   &    0   & \dots &    0   &   0    \\
     0 &\alpha_{r-1} &    0   & \dots &    0   &   0    \\
     0 &    0   &\alpha_{r-2} & \dots &    0   &   0    \\
\vdots & \vdots & \vdots & \ddots& \vdots & \vdots \\
     0 &    0   &    0   & \dots &   \alpha_2  &   0    \\
     0 &    0   &    0   & \dots &    0   &  \alpha_1
\end{pmatrix},\quad
\mathrm{B}=\begin{pmatrix}
   \beta_r &   -1   &    0   & \dots &    0   &   0    \\
     1 & \beta_{r-1}&   -1   & \dots &    0   &   0    \\
     0 &    1   & \beta_{r-2}& \dots &    0   &   0    \\
\vdots & \vdots & \vdots & \ddots& \vdots & \vdots \\
     0 &    0   &    0   & \dots &   \beta_2  &  -1    \\
     0 &    0   &    0   & \dots &    1   &  \beta_1
\end{pmatrix},
\end{equation}
and the function $R$ expands into a
$J$-fraction~\eqref{continued.fraction.Jacobi}.

The polynomials $h_j$ $(j=0,\ldots,r-1)$
in~\eqref{continued.fraction.Jacobi} are known to be the leading
principal minors of order $r-j$ of the matrix
$\mathcal{J}(z)=z\mathrm{A}+\mathrm{B}$ as we mentioned above (see
also~\cite{Barkovsky.2,{GantKrein}}). In particular,
$h_0(z)=\det(z\mathrm{A}+\mathrm{B})$.

As we already mentioned in
Section~\ref{s:Euclidean.algorithm.general.complex.case}, we can
localize the eigenvalues of the
problem~\eqref{Generilized.eigenvalue.Jacobi.problem} using
properties of the function~$R$. For example, if
$\alpha_j>0,\beta_j>0$ for all $j=1,\ldots,r$
in~\eqref{reg.Euclidean.algorithm} (or, equivalently,
in~\eqref{continued.fraction.Jacobi}),   then every polynomial
$q_j$ is a function mapping the closed right half-plane into the
open right half-plane. In fact, $\Re q_j(z)=\alpha_j\Re
z+\beta_j>0$, whenever $\Re z\geq 0$. Now note: if functions
$F_1$ and~$F_2$ map the closed right-half plane to the open
right-half plane, then so do the functions $F_1+F_2$ and
$\dfrac1{F_j}$ $(j=1,2)$. Consequently, in this case the function
$R$ represented by~\eqref{continued.fraction.Jacobi}  also maps
the closed right half-plane to the open right half-plane, being a
composition of such maps. But since the function $R$ has a
positive real part in the closed right half-plane and is finite
there, then all its zeros and poles and, subsequently, all zeros
of the polynomials $h_0$ and $h_1$ lie in the open left
half-plane. In summary, \textit{if $\alpha_j,\beta_j>0$,
$j=1,2\ldots,r$, then the eigenvalue
problem~\eqref{Generilized.eigenvalue.Jacobi.problem} is stable},
that is, \textit{all its eigenvalues lie in the open left
half-plane}. This example constitutes the subject of  Problem~7.1
in~\cite{Barkovsky.2}. A similar  result was  obtained
in~\cite{Genin}.

Finally, at the end of this subsection,  let us discuss the form taken by the
partial quotients $F_j$ defined by~\eqref{continued.fraction.general.parts.2}
in the regular case of the Euclidean algorithm. Since all polynomials
$q_j$ are of degree one (see~\eqref{reg.Euclidean.algorithm}), we have
 $k=r$, the number of poles of $R$. Moreover, for a fixed integer $j$
($1\leq j\leq r$) we have $\deg P_j =m_j=j$, where $P_j$
is denominator of the fraction $F_j$ of the form
(cf.~\eqref{quasi-orthogonal.quotients.via.determinants})
\begin{equation*}\label{reg.quasi-orthogonal.quotients.via.determinants}
P_j(z)=\dfrac{P_{0,j}}{D_{j}(R)}
\begin{vmatrix}
    s_0 &s_1 &s_{2} &\dots &s_{j}\\
    s_1 &s_2 &s_{3} &\dots &s_{j+1}\\
    \vdots&\vdots&\vdots&\ddots&\vdots\\
    s_{j-1} &s_{j}&s_{j+1} &\dots &s_{2j-1}\\
    1 &z &z^2 &\dots &z^{j}\\
\end{vmatrix},\qquad j=1,2,\ldots,r.
\end{equation*}
The leading coefficients $P_{0,j}$ can be determined by the formul\ae\,
(see, e.g.,~\cite{Markov})
\begin{eqnarray*}\label{reg.leading.coeffs.of.quasi-orthogonal.quotients.even}
P_{0,2j} &= &(-1)^{j}\dfrac{D_{1}^2(R)D_{3}^2(R)\cdots
D_{2j-1}^2(R)}{D_{2}^2(R)D_{4}^2(R)\cdots
D_{2j-2}^2(R)D_{2j}(R)},\qquad
j=1,2,\ldots,\left\lfloor\dfrac{r}2\right\rfloor. \\
\label{reg.leading.coeffs.of.quasi-orthogonal.quotients.odd}
P_{0,2j+1} &= & (-1)^{j}\dfrac{D_{2}^2(R)D_{4}^2(R)\cdots
D_{2j}^2(R)}{D_{1}^2(R)D_{3}^2(R)\cdots
D_{2j-1}^2(R)D_{2j+1}(R)},\quad
j=0,1,2,\ldots,\left\lceil\dfrac{r}2\right\rceil-1.
\end{eqnarray*}
%

%%%%%%%%%%%%%%%%%%%%%%%%%%%%%%%%%%%%%%%%%%%%%%%%%%%%%%%%%%%%%%%%%%%%%%
\subsection{\label{s:Euclidean.doubly.regular}Euclidean algorithm: doubly regular case.
Finite continued fractions of Stieltjes type}
%%%%%%%%%%%%%%%%%%%%%%%%%%%%%%%%%%%%%%%%%%%%%%%%%%%%%%%%%%%%%%%%%%%%%%

Assume, as above, that the rational function $R$ has
a series expansion~\eqref{J-fractions.criterion.function}, where
$s_{-1}=0$, and consider the function
\begin{equation}\label{function.for.Stieltjes.fraction}
F(z)\eqbd zR(z^2)=\dfrac{s_{0}}{z}+\dfrac{s_{1}}{z^{3}}+\dfrac{s_{2}}{z^{5}}+\cdots
\end{equation}
The function $F$ can be also represented as a series
\begin{equation}\label{function.for.Stieltjes.fraction.series.1}
F(z)=\dfrac{t_{0}}{z}+\dfrac{t_{1}}{z^{2}}+\dfrac{t_{2}}{z^{3}}+\cdots,
\end{equation}
where
\begin{equation}\label{function.for.Stieltjes.fraction.series.2}
\begin{array}{lcl}
t_{2j} & = & s_j,\\[2mm]
t_{2j+1}& = &0,
\end{array}\qquad j=0,1,\ldots
\end{equation}

\begin{remark}\label{remark.1.6}
Let $r$ denote the number of poles of the function $R$ (counting
multiplicities).  Note that the function $F$ has $2r$ poles if
and only if $R$ has  no pole at zero. Otherwise, $F$ has only
$2r{-}1$ poles.
\end{remark}

\begin{lemma}\label{lem.minors.relations.for.Stieltjes.formula}
The following relations hold between the minors
$D_j(R)$, $\widehat{D}_j(R)$ and $D_j(F)$ defined
in~\eqref{Hankel.determinants.1}
and~\eqref{Hankel.determinants.2}:
\begin{equation}\label{minors.relations.for.Stieltjes.formula}
\begin{array}{lcl}
D_{2j}(F) & = & D_j(R)\cdot\widehat{D}_j(R),\\[2mm]
D_{2j-1}(F) & = & D_j(R)\cdot\widehat{D}_{j-1}(R),
\end{array}
\qquad j=1,2,\ldots,
\end{equation}
where we set $\widehat{D}_0(R)\eqbd 1$.
\end{lemma}
\proof
First interchange the rows and columns of the determinant $D_{2j}(F)$
\begin{equation*}\label{lem.minors.relations.for.Stieltjes.formula.proof.1}
D_{2j}(F)=
\begin{vmatrix}
    t_0 &t_1 &\dots &t_{2j-2} &t_{2j-1}\\
    t_1 &t_2 &\dots &t_{2j-1} &t_{2j}\\
    t_2 &t_3 &\dots &t_{2j} &t_{2j+1}\\
    \vdots&\vdots&\ddots&\vdots&\vdots\\
    t_{2j-3} &t_{2j-2} &\dots &t_{4j-4}&t_{4j-3}\\
    t_{2j-2} &t_{2j-1} &\dots &t_{4j-3}&t_{4j-2}\\
    t_{2j-1} &t_{2j}   &\dots &t_{4j-2}&t_{4j-2}
\end{vmatrix}=
\begin{vmatrix}
    s_0 &0   &s_1 &\dots &s_{j-1}&0\\
    0   &s_1 &0   &\dots &0      &s_{j}\\
    s_1 &0   &s_2 &\dots &s_{j}  &0\\
    \vdots&\vdots&\vdots&\ddots&\vdots&\vdots\\
    0 &s_{j-1} &0 &\dots &0&s_{2j-2}\\
    s_{j-1}&0 &s_{j} &\dots &s_{2j-2}&0\\
    0 &s_{j} &0 &\dots &0&s_{2j-1}
\end{vmatrix}
\end{equation*}
so that $(2i{-}1)$st row and $(2i{-}1)$st column move to the $i$th
position, for each $i=2,3,\ldots,j$; this produces
\begin{equation*}\label{lem.minors.relations.for.Stieltjes.formula.proof.2}
D_{2j}(F)=
\begin{vmatrix}
    s_0     &s_1   &\dots &s_{j-1} &0   &0   &\dots &0\\
    s_1     &s_2   &\dots &s_{j}   &0   &0   &\dots &0\\
    \vdots&\vdots&\ddots&\vdots&\vdots &\vdots&\ddots&\vdots\\
    s_{j-1} &s_{j} &\dots &s_{2j-2}&0   &0   &\dots &0\\
    0       &0     &\dots &0       &s_1 &s_2 &\dots &s_{j}\\
    0       &0     &\dots &0       &s_2 &s_3 &\dots &s_{j+1}\\
    \vdots&\vdots&\ddots&\vdots&\vdots&\vdots&\ddots&\vdots\\
    0       &0     &\dots &0       &s_j &s_{j+1} &\dots &s_{2j-1}\\
\end{vmatrix}=D_j(R)\cdot\widehat{D}_j(R).
\end{equation*}
The formula for the determinants $D_{2j-1}(F)$ can be proved in the same way.
\eop

This lemma, Theorem~\ref{Th.J-fraction.criterion} and
Corollary~\ref{corol.zero.pole} can now be combined to derive the following
corollary, which will be important later.
\begin{corol}\label{corol.J-fraction.of.function.Stieltjes.fraction}
The function $F$ defined
in~\eqref{function.for.Stieltjes.fraction} has a $J$-fraction
expansion if and only if
\begin{eqnarray}\label{minors.inequalities.for.Stieltjes.fraction.1}
& D_{j}(R)\neq 0, & j=1, 2, \ldots, r,  \\
\label{minors.inequalities.for.Stieltjes.fraction.2}
& \widehat{D}_{j}(R)\neq 0, &  j=1, 2, \ldots, r{-}1, \\
\label{minors.inequalities.for.Stieltjes.fraction.3}
& D_{j}(R)=\widehat{D}_{j}(R)=0, & j=r{+}1, \, r{+}2, \ldots
\end{eqnarray}
where $r$ is the number of the poles of the function $R$,
which generates $F$.  In addition, $\widehat{D}_{r}(R)=0$ if and only if the
function $R$ (hence also the function $F$) has a pole at $0$.
\end{corol}

Since $F$ is evidently an odd function, Remark~\ref{remark.1.5} shows that
its $J$-fraction expansion (if any) has to be of the following form:
\begin{equation}\label{continued.fraction.Jacobi-Stieltjes}
F(z)=\dfrac1{c_1z+\cfrac1{c_2z+\cfrac1{c_{3}z+\cfrac1{\ddots+\cfrac1{c_{k}z}}}}},\quad
c_j\neq0,\quad j=1,2,\ldots,k.
\end{equation}
Here $k=2r$ if the function $R$ has no pole at $0$ and $k=2r-1$ if it does,
as follows from Remark~\ref{remark.1.6} and
Theorem~\ref{Th.J-fraction.criterion}.
Making the equivalence transformation~\eqref{J-fractions.second.form.change.var.1}--\eqref{J-fractions.second.form.change.var.2}
in~\eqref{continued.fraction.Jacobi-Stieltjes} with $c_i$ replacing
 $\alpha_i$ and with $e_i=\beta_i=0$, we obtain by induction
\begin{eqnarray}\label{even.coeff.Stieltjes.fraction.formula}
 c_{2j} &=&-\dfrac{\displaystyle\prod_{i=0}^{j-1}d_{2i}}{\displaystyle\prod_{i=0}^{j}d_{2i-1}},\qquad \qquad   \qquad \qquad  \qquad
j=1,2,\ldots,r, \\
\label{odd.coeff.Stieltjes.fraction.formula}
c_1&=&\dfrac1{d_0}, \qquad
c_{2j-1}\; =\; -\dfrac{\displaystyle
\prod_{i=0}^{j-1}d_{2i-1}}{\displaystyle\prod_{i=0}^{j-1}d_{2i}},\qquad
j=2,3,\ldots,r,
\end{eqnarray}
since $c_{j+1}=-\dfrac1{c_jd_j}$
from~\eqref{J-fractions.second.form.change.var.1}. The
formul\ae~\eqref{J-fraction.determinants.formula.3}
and~\eqref{minors.relations.for.Stieltjes.formula} imply
\begin{eqnarray}\label{even.coeff.Stieltjes.fraction.formula.2}
&& d_{2i}=\dfrac{D_{2i-1}(F)\cdot
D_{2i+1}(F)}{D_{2i}^2(F)}=\dfrac{D_{i+1}(R)\cdot\widehat{D}_{i-1}(R)}{D_{i}(R)\cdot\widehat{D}_{i}(R)},    \qquad
i=1,2,\ldots,r,  \\
\label{odd.coeff.Stieltjes.fraction.formula.2}
&& d_{2i-1}=\dfrac{D_{2i-2}(F)\cdot
D_{2i}(F)}{D_{2i-1}^2(F)}=\dfrac{D_{i-1}(R)\cdot\widehat{D}_{i}(R)}{D_{i}(R)\cdot\widehat{D}_{i-1}(R)},   \qquad \quad
i=1,2,\ldots,r,
\end{eqnarray}
Combining~\eqref{even.coeff.Stieltjes.fraction.formula}--\eqref{odd.coeff.Stieltjes.fraction.formula}
and~\eqref{even.coeff.Stieltjes.fraction.formula.2}--\eqref{odd.coeff.Stieltjes.fraction.formula.2},
we get very interesting and ultimately very useful relations
between the coefficients $c_j$ and the minors $D_j(R)$ and
$\widehat{D}_j(R)$ (see~\cite{Nudelman}):
\begin{eqnarray}\label{even.coeff.Stieltjes.fraction.main.formula}
c_{2j}=-\dfrac{D_{j}^2(R)}{\widehat{D}_{j-1}(R)\cdot\widehat{D}_{j}(R)},\qquad
j=1,2,\ldots,r,  \\
\label{odd.coeff.Stieltjes.fraction.main.formula}
c_{2j-1}=\dfrac{\widehat{D}_{j-1}^2(R)}{D_{j-1}(R)\cdot
D_{j}(R)},\qquad j=1,2,\ldots,r.
\end{eqnarray}
\begin{remark}\label{remark.1.7}
According to Corollary~\ref{corol.zero.pole}, the function $R$ has
a pole at zero if and only if $\widehat{D}_{r-1}(R)\neq0$ and
$\widehat{D}_r(R)=0$.
From~\eqref{even.coeff.Stieltjes.fraction.main.formula}
and~\eqref{minors.inequalities.for.Stieltjes.fraction.2} we
conclude that $c_{2r}=\infty$ in this case.
\end{remark}

Upon making an equivalence transformation
in~\eqref{continued.fraction.Jacobi-Stieltjes}, replacing
$z^2$ by $z$, and removing the factor $z$  (as in~\cite[p.~170]{Wall}), we obtain
\begin{equation}\label{Stieltjes.fraction.1}
R(z)=\dfrac1{c_1z+\cfrac1{c_2+\cfrac1{c_{3}z+\cfrac1{\ddots+\cfrac1{T}}}}},\quad
c_j\neq0,\quad\text{where}\quad
T=\begin{cases}
         c_{2r} & \text{if}\ |R(0)|<\infty,\\
         c_{2r-1}z &  \text{if}\ R(0)=\infty.
       \end{cases}
\end{equation}
\begin{definition}\label{def.Stieltjes.fraction}
Continued fractions of type~\eqref{Stieltjes.fraction.1} are called
\textit{continued fractions of Stieltjes type} or \textit{Stieltjes
continued fraction}. Accordingly, if~\eqref{Stieltjes.fraction.1}
holds for a function $R$ with $r$ poles,  then we say that $R$ has a
Stieltjes continued fraction expansion.
\end{definition}
\begin{remark}\label{remark.1.8}
If $R(\infty)=c_0$, where $0<|c_0|<\infty$, then we say that
$R$ has a Stieltjes continued fraction expansion whenever the
function $G(z)\eqbd R(z)-c_0$ has one.
\end{remark}

Summarizing all the previous results, we obtain the following
criterion for a rational function to have a Stieltjes continued
fraction expansion.
\begin{theorem}[\cite{Stieltjes1,{Stieltjes2},{Stieltjes3},{Stieltjes4},{Wall},{Nudelman}}]\label{Th.Stieltjes.fraction.criterion}
Suppose that a rational function $R$ is finite at infinity,  has
exactly $r$ poles, and can be represented
as a series~\eqref{J-fractions.criterion.function}. The function
$R$ has a Stieltjes continued fraction expansion~\eqref{Stieltjes.fraction.1}
if and only if it satisfies the
conditions~\eqref{minors.inequalities.for.Stieltjes.fraction.1}--\eqref{minors.inequalities.for.Stieltjes.fraction.3}.
In that case, the~coefficients of the Stieltjes continued fraction can be
found from the
formul\ae~\eqref{even.coeff.Stieltjes.fraction.main.formula}--\eqref{odd.coeff.Stieltjes.fraction.main.formula}.
\end{theorem}

Assume that
\begin{equation}\label{new.correction.1}
F(z)=zR(z^2)=\dfrac{g_1(z)}{g_0(z)},
\end{equation}
and $F(\infty)=0$. The function $F$ has a $J$-fraction
expansion~\eqref{continued.fraction.Jacobi-Stieltjes} if and only
if the~Euclidean algorithm applied to the polynomials $g_0$ and
$g_1$ is regular. Since $F$ is an odd function,
\eqref{continued.fraction.Jacobi-Stieltjes} implies
\begin{equation*}
g_{j-1}(z)=c_jzg_j(z)+g_{j+1}(z),\quad c_j\neq 0, \quad
j=1,2,\ldots,k,
\end{equation*}
where $k$ is equal to $2r{-}1$ or $2r$ depending on whether or not $R$
has a pole at zero.  Performing the transformation
\begin{equation*}
\widetilde{g}_{2i}(z)=g_{2i}(z),\qquad
\widetilde{g}_{2i-1}(z)=\dfrac{g_{2i-1}(z)}{z},
\qquad \quad i=1,2,\ldots,r,
\end{equation*}
we obtain
\begin{equation*}
\widetilde{g}_{j-1}(z)=\widetilde{q}_j(z)\widetilde{g}_j(z)+\widetilde{g}_{j+1}(z),\quad
q_j(z)\not\equiv0,\quad j=1,2,\ldots,k,
\end{equation*}
where
\begin{equation}\label{auxiliary.quotients.0}
\begin{array}{lcll}
\widetilde{q}_{2i}(z) & = & c_{2i},  & \qquad i=1,2,\ldots,\left\lfloor\dfrac{k}2\right\rfloor,\\[3mm]
\widetilde{q}_{2i-1}(z) & = & c_{2i-1}z^2,   & \qquad i=1,2,\ldots,r.\\
\end{array}
\end{equation}
If $R$ and, subsequently, $F$ have a pole at zero, then
$\widetilde{q}_{2r}(z)$ does not exist, according to
Remark~\ref{remark.1.7}.

Since $F(z)=\dfrac{g_1(z)}{g_0(z)}$ is an odd function
(see~\eqref{new.correction.1}), the function
$\dfrac{\widetilde{g}_1(z)}{\widetilde{g}_0(z)}=\dfrac{F(z)}{z}$
is even. Therefore, both polynomials $\widetilde{g}_0$ and
$\widetilde{g}_1$ are even, hence so are all subsequent
 polynomials $\widetilde{g}_{j}(z)$,  $j=2,\ldots,k$. Equivalently,
\begin{equation*}
\widetilde{g}_{j}(z)=f_{j}(z^2).
\end{equation*}
From~\eqref{auxiliary.quotients.0} we see that the polynomials
$\widetilde{q}_j(z)$ are even, so are functions of $z^2$. Denoting
$q_j(z^2)=\widetilde{q}_j(z)$, we obtain
\begin{equation*}
f_{j-1}(z^2)=q_j(z^2)f_j(z^2)+f_{j+1}(z^2),\qquad j=1,2,\ldots,k.
\end{equation*}
Since
$\dfrac{f_1(z^2)}{f_0(z^2)}=\dfrac{\widetilde{g}_1(z)}{\widetilde{g}_0(z)}=\dfrac{F(z)}{z}=R(z^2)$,
replacing $z^2$ by $z$, we get
\begin{equation*}
R(z)=\dfrac{f_1(z)}{f_0(z)}
\end{equation*}
and
\begin{equation}\label{doubly.regular.Euclidean.algorithm}
f_{j-1}(z)=q_j(z)f_j(z)+f_{j+1}(z),\quad q_j(z)\not\equiv0,\quad
j=1,2,\ldots,k,
\end{equation}
and
\begin{equation}\label{auxiliary.quotients}
\begin{array}{lcll}
q_{2i}(z) & = & c_{2i}, & \quad i=1,2,\ldots,\left\lfloor\dfrac{k}2\right\rfloor,\\[2mm]
q_{2i-1}(z) &= & c_{2i-1}z, & \quad i=1,2,\ldots,r.\\
\end{array}
\end{equation}
Recalling that we consider the case $R(\infty)=0$, we see that
the polynomials $f_j$ have fixed degrees
%
%\begin{itemize}
%
%\item[]
%
\begin{equation}\label{degrees.of.doble.regular.algorithm.polys}
\begin{array}{lcll}
f_{2i}(z) & = & h_{2i}z^{n-i}+\cdots, & \qquad i=0,1,2,\ldots,\left\lfloor\dfrac{k}2\right\rfloor,\\[2mm]
f_{2i-1}(z) & = & h_{2i-1}z^{n-i}+\cdots, & \qquad i=1,2,\ldots,r.\\
\end{array}
\end{equation}
%
%\item[] if $\deg f_0(z)=\deg f_1(z)$
%
%\begin{equation}\label{degrees.of.doble.regular.algorithm.polys.2}
%\begin{split}
%&f_{2i}(z)=h_{2i}z^{n-i}+\cdots,\quad i=0,1,2,\ldots,\left\lfloor\dfrac{k}2\right\rfloor,\\
%
%&f_{2i-1}(z)=h_{2i-1}z^{n+1-i}+\cdots,\quad i=1,2,\ldots,r,\\
%\end{split}
%\end{equation}
%
%\end{itemize}
where $n=\deg f_0 \geq r$ and $h_j \neq 0$ for
$j=0,1,\ldots,k$.

So, the function  $R(z)=\dfrac{f_1(z)}{f_0(z)}$ has a Stieltjes
continued fraction expansion~\eqref{Stieltjes.fraction.1} if and
only if the application of the Euclidean algorithm to the
polynomials $f_0$ and $f_1$ has the
form~\eqref{doubly.regular.Euclidean.algorithm}--\eqref{auxiliary.quotients}
produces their the greatest common divisor  $f_k(z)$, where
$k=2r$, if $|R(0)|<\infty$, and $k=2r-1$ otherwise. We already
know from the
condition~\eqref{minors.inequalities.for.Stieltjes.fraction.1}
that the Euclidean
algorithm~\eqref{doubly.regular.Euclidean.algorithm}--\eqref{auxiliary.quotients}
must be regular for the function $R$ to have the
expansion~\eqref{Stieltjes.fraction.1}. But in this case we obtain
one more set of inequalities,
namely,~\eqref{minors.inequalities.for.Stieltjes.fraction.2}. This
justified calling such an instance of the algorithm \textit{doubly
regular}.

Now consider again the rational function
\begin{equation}\label{working.function}
R(z)=\dfrac{q(z)}{p(z)}=s_{-1}+\dfrac{s_{0}}{z}+\dfrac{s_{1}}{z^{2}}+\dfrac{s_{2}}{z^{3}}+\cdots,
\end{equation}
where the polynomials $p$ and $q$ are defined
in~\eqref{polynomial.1}--\eqref{polynomial.2}.

We now introduce another (infinite) matrix associated with the
function~\eqref{working.function}. This object differs
significantly from the Hankel matrix constructed
in~\eqref{Hakel.matrix.Rat.func.correspondence} from the
coefficients $s_{j}$. In particular, this new matrix is made of
the coefficients of the polynomials $p$ and $q$.

\begin{definition}\label{def.Hurwitz.matrix.infinite} Given polynomials
 $p$ and $q$ from~\eqref{polynomial.1}--\eqref{polynomial.2}, define
the infinite matrix $H(p,q)$ as follows:
if $\deg q < \deg p $, that is, if $b_0=0$,
then\footnote{Generally speaking, $b_1$ may be allowed to be zero.
However, in this section we consider functions expanding into Stieltjes
continued fractions, and $D_1(R)=s_0\neq 0$ is one of the
necessary conditions for such an expansion to exist by
Corollary~\ref{corol.J-fraction.of.function.Stieltjes.fraction}.
At the same time, $s_0 \neq 0$ implies $b_1=a_0s_0 \neq 0$.}
\begin{equation}\label{Hurwitz.matrix.infinite.case.1}
H(p,q) \eqbd
\begin{pmatrix}
a_0&a_1&a_2&a_3&a_4&a_5&\dots\\
0  &b_1&b_2&b_3&b_4&b_5&\dots\\
0  &a_0&a_1&a_2&a_3&a_4&\dots\\
0  &0  &b_1&b_2&b_3&b_4&\dots\\
\vdots &\vdots  &\vdots  &\vdots &\vdots  &\vdots &\ddots
\end{pmatrix};
\end{equation}
if $\deg q =\deg p$, that is, $b_0\neq0$, then
\begin{equation}\label{Hurwitz.matrix.infinite.case.2}
H(p,q) \eqbd
\begin{pmatrix}
b_0&b_1&b_2&b_3&b_4&b_5&\dots\\
0  &a_0&a_1&a_2&a_3&a_4&\dots\\
0  &b_0&b_1&b_2&b_3&b_4&\dots\\
0  &0  &a_0&a_1&a_2&a_3&\dots\\
\vdots &\vdots  &\vdots  &\vdots &\vdots  &\vdots &\ddots
\end{pmatrix}.
\end{equation}
The matrix $H(p,q)$ is called an \textit{infinite matrix of Hurwitz type}.
\end{definition}

\begin{remark}\label{remark.1.9}
The matrix $H(p,q)$ is of infinite rank since its submatrix
obtained by deleting the even or odd rows of the original matrix
is a triangular infinite matrix with $a_0 \neq 0$ on the main
diagonal.
\end{remark}

Together with the infinite matrix $H(p,q)$, we consider its
specific finite submatrices:
\begin{definition}\label{def.Hurwitz.matrix.finite} Let the polynomials
$p$ and $q$ be given by~\eqref{polynomial.1}--\eqref{polynomial.2}.
If $\deg q < \deg p = n$, let $\mathcal{H}_{2n}(p,q)$ denote
the following ${2n} \times {2n}$-matrix:
\begin{equation}\label{Hurwitz.matrix.finite.case.1}
\mathcal{H}_{2n}(p,q)=
\begin{pmatrix}
    b_1 &b_2 &b_3 &\dots &b_{n}   &   0    &0      &\dots &0&0\\
    a_0 &a_1 &a_2 &\dots &a_{n-1} & a_{n}  &0      &\dots &0&0\\
     0  &b_1 &b_2 &\dots &b_{n-1} & b_{n}  &0      &\dots &0&0\\
     0  &a_0 &a_1 &\dots &a_{n-2} & a_{n-1}&a_{n}  &\dots &0&0\\
    \vdots&\vdots&\vdots&\ddots&\vdots&\vdots&\vdots&\ddots&\vdots&\vdots\\
     0  &  0 &  0 &\dots &a_{0} & a_{1}&a_{2}&\dots &a_{n}&0\\
     0  &  0 &  0 &\dots &   0  & b_{1}&b_{2}&\dots &b_{n}&0\\
     0  &  0 &  0 &\dots &0     & a_{0}&a_{1}&\dots &a_{n-1}&a_{n}\\
\end{pmatrix}.
\end{equation}
If $\deg q = \deg p =n$, let $\mathcal{H}_{2n+1}(p,q)$ denote the following
 $(2n{+}1)\times (2n{+}1)$-matrix
\begin{equation}\label{Hurwitz.matrix.finite.case.2}
\mathcal{H}_{2n+1}(p,q)=
\begin{pmatrix}
    a_0 &a_1 &a_2 &\dots &a_{n-1} & a_{n}  &0&\dots &0&0\\
    b_0 &b_1 &b_2 &\dots &b_{n-1} & b_{n}  &0&\dots &0&0\\
     0  &a_0 &a_1 &\dots &a_{n-2} & a_{n-1}&a_{n}&\dots &0&0\\
     0  &b_0 &b_1 &\dots &b_{n-2} & b_{n-1}&b_{n}&\dots &0&0\\
    \vdots&\vdots&\vdots&\ddots&\vdots&\vdots&\vdots&\ddots&\vdots&\vdots\\
     0  &  0 &  0 &\dots &a_{0} & a_{1}&a_{2}&\dots &a_{n}&0\\
     0  &  0 &  0 &\dots &b_{0} & b_{1}&b_{2}&\dots &b_{n}&0\\
     0  &  0 &  0 &\dots &0     & a_{0}&a_{1}&\dots &a_{n-1}&a_{n}\\
\end{pmatrix}   .
\end{equation}
Both kinds of matrices $\mathcal{H}_{2n}(p,q)$ and $\mathcal{H}_{2n+1}(p,q)$  are
called \textit{finite matrices of Hurwitz type}. The leading principal minors of
these matrices will be denoted by\footnote{That is, $\Delta_j(p,q)$ is the leading
principal minor  of the matrix $\mathcal{H}_{2n}(p,q)$ of order
$j$ if $\deg q < \deg p $. Otherwise (when $\deg q =\deg p$),
 $\Delta_j(p,q)$ denotes the  leading principal minor of the
matrix $\mathcal{H}_{2n+1}(p,q)$ of order $j$.} $\Delta_j(p,q)$.
\end{definition}

\vspace{0.2cm}

The infinite Hurwitz matrix $H(p,g)$ has an interesting factorization property:
% related to the greatest common divisor of $p$ and $q$.
%
\begin{theorem}\label{Th.Hurwitz.matrix.with.gcd}
If $g(z)=g_0z^l+g_1z^{l-1}+\cdots+g_l$, then
\begin{equation}\label{Hurwitz.matrix.with.gcd}
H(p\cdot g,q\cdot g)=H(p,q)\mathcal{T}(g),
\end{equation}
where $\mathcal{T}(g)$ is the infinite upper triangular Toeplitz
matrix made of the coefficients of the polynomial $g$:
\begin{equation}\label{Toeplitz.infinite.matrix}
\mathcal{T}(g)=
\begin{pmatrix}
g_0&g_1&g_2&g_3&g_4&\dots\\
0  &g_0&g_1&g_2&g_3&\dots\\
0  &0  &g_0&g_1&g_2&\dots\\
0  &0  &0  &g_0&g_1&\dots\\
0  &0  &0  &0  &g_0&\dots\\
\vdots &\vdots  &\vdots  &\vdots  &\vdots &\ddots
\end{pmatrix}.
\end{equation}
Here we set $g_i \eqbd 0$ for all $i>l$.
\end{theorem}
\proof
Straightforward multiplication of matrices $H(p,q)$ and
$\mathcal{T}(g)$.
\eop

Denote by $\eta_{j}(p,q)$ the leading principal minor of the
matrix $H(p,q)$ of order $j$ $(j=1,2,\ldots)$. We now derive
some key connections between these minors and the minors $D_j$,
$\widehat{D}_j$ and $\nabla_{2j}$ we encountered before.
\begin{lemma}\label{lemma.relations.eta.nabla.D}
Let the polynomials $p$ and $q$ be defined
by~\eqref{polynomial.1}--\eqref{polynomial.2} and let
\begin{equation*}
R(z)=\dfrac{q(z)}{p(z)}=s_{-1}+\dfrac{s_0}{z}+\dfrac{s_1}{z^2}+\cdots
\end{equation*}
The following relations hold between the determinants
$\eta_j(p,q)$ and $D_j(R),\widehat{D}_j(R),\nabla_{2j}(p,q)$
defined by~\eqref{Hankel.determinants.1},
\eqref{Hankel.determinants.2}
and~\eqref{Hurwitz.General.Determinants}, respectively.
\begin{itemize}
\item[] If  $\deg q < \deg p $, then
%
%\begin{align}\label{Hurwitz.determinants.relations.infinite.2.even}
%\eta_{2j}(p,q) & \; = \; \nabla_{2j}(p,q)\quad \;\;\; = \; a_0^{2j}D_j(R), &  j=1,2,\ldots, \quad \; \\
%
%\label{Hurwitz.determinants.relations.infinite.2.odd}
%\eta_{2j+1}(p,q)& \; = \; a_0\nabla_{2j}(zq,p) \; = \;
%(-1)^{j}a_0^{2j+1} \widehat{D}_j(R),& j=0,1,2,\ldots  ~.
%\end{align}
\begin{eqnarray}\label{Hurwitz.determinants.relations.infinite.2.even}
& \eta_{2j}(p,q)=\nabla_{2j}(p,q)=a_0^{2j}D_j(R), &
j=1,2,\ldots, \\
\label{Hurwitz.determinants.relations.infinite.2.odd}
& \eta_{2j+1}(p,q)=a_0\nabla_{2j}(zq,p)=(-1)^{j}a_0^{2j+1}
\widehat{D}_j(R), & j=0,1,2,\ldots.
\end{eqnarray}
\item[] If   $\deg q = \deg p $, then
%
%\begin{align}\label{Hurwitz.determinants.relations.infinite.1.odd}
%\eta_{2j+1}(p,q) & \; = \; b_0\nabla_{2j}(p,q)\qquad \; = \; b_0a_0^{2j}D_j(R), & \quad
%j=0,1,2,\ldots, \\
%
%\label{Hurwitz.determinants.relations.infinite.1.even}
%\eta_{2j}(p,q) & \; = \; a_0b_0\nabla_{2j-2}(h,p)\; = \; (-1)^{j-1}b_0a_0^{2j-1}\widehat{D}_{j-1}(R), &
%j=1,2,\ldots,\quad
%\end{align}
%
\begin{eqnarray}\label{Hurwitz.determinants.relations.infinite.1.odd}
& \eta_{2j+1}(p,q)=b_0\nabla_{2j}(p,q)= b_0a_0^{2j}D_j(R), &
j=0,1,2,\ldots, \\
\label{Hurwitz.determinants.relations.infinite.1.even}
& \eta_{2j}(p,q)=a_0b_0\nabla_{2j-2}(h,p)=
(-1)^{j-1}b_0a_0^{2j-1}\widehat{D}_{j-1}(R), & j=1,2,\ldots,
\end{eqnarray}
\ \ \ where $h(z) \eqbd zq(z)-\dfrac{b_0}{a_0}zp(z)$ and
$D_0(R)\eqbd \widehat{D}_0(R) \eqbd 1$.
\end{itemize}
\end{lemma}
\proof First, we prove the more complicated
equalities~\eqref{Hurwitz.determinants.relations.infinite.1.odd}--\eqref{Hurwitz.determinants.relations.infinite.1.even}.
The formula~\eqref{Hurwitz.determinants.relations.infinite.1.odd}
follows
from~\eqref{Hurwitz.matrix.infinite.case.2},~\eqref{Hurwitz.General.Determinants}
and~\eqref{Hurwitz.General.Relations}. To
prove~\eqref{Hurwitz.determinants.relations.infinite.1.even}, we
consider the function
\begin{equation}\label{additional.function}
G(z) \eqbd zR(z)-s_{-1}z=\dfrac{h(z)}{p(z)}=s_0+\dfrac{s_1}{z}+\dfrac{s_2}{z^2}+\cdots
\end{equation}
Here we used the fact that $s_{-1}=\dfrac{b_0}{a_0}$.
From~\eqref{Hankel.determinants.1}
and~\eqref{Hankel.determinants.2} it follows that
\begin{equation}\label{additional.function.minors}
D_j(G)=\widehat{D}_j(R),\qquad j=1,2,\ldots,
\end{equation}
Next, for a fixed index $j=1,2,\ldots$, we have
\begin{equation*}
\eta_{2j}(p,q)=
\begin{vmatrix}
    b_0 &b_1 &b_2 &\dots &b_{j-1} & b_{j}  &\dots &b_{2j-1}\\
     0  &a_0 &a_1 &\dots &a_{j-2} & a_{j-1}&\dots &a_{2j-2}\\
     0  &b_0 &b_1 &\dots &b_{j-2} & b_{j-1}&\dots &b_{2j-2}\\
     0  &0   &a_0 &\dots &a_{j-3} & a_{j-2}&\dots &a_{2j-3}\\
    \vdots&\vdots&\vdots&\ddots&\vdots&\vdots&\ddots&\vdots\\
     0  &  0 &  0 &\dots &a_{0} & a_{1}&\dots &a_{j}\\
     0  &  0 &  0 &\dots &b_{0} & b_{1}&\dots &b_{j}\\
     0  &  0 &  0 &\dots &0     & a_{0}&\dots &a_{j-1}\\
\end{vmatrix}.
\end{equation*}
For each $i=1,\ldots,j-1$, we now subtract the $(2i)$th row
multiplied by $\dfrac{b_0}{a_0}$ from the $(2i{+}1)$st row to
obtain
\begin{equation}\label{lemma.relations.eta.nabla.D.proof}
\eta_{2j}(p,q)=a_0b_0\nabla_{2j-2}(h,p)=a_0b_0(-1)^{j-1}\nabla_{2j-2}(p,h)~.
\end{equation}
This proves the first part
of~\eqref{Hurwitz.determinants.relations.infinite.1.even}. Now
from~\eqref{Hurwitz.General.Relations},~\eqref{additional.function}
and~\eqref{additional.function.minors} we have
\begin{equation*}
\nabla_{2j}(p,h)=a_0^{2j}D_{j}(G)=a_0^{2j}\widehat{D}_j(R).
\end{equation*}
Combined with~\eqref{lemma.relations.eta.nabla.D.proof}, this
implies~\eqref{Hurwitz.determinants.relations.infinite.1.even}.  The
formul\ae~\eqref{Hurwitz.determinants.relations.infinite.2.even}--\eqref{Hurwitz.determinants.relations.infinite.2.odd}
can be proved analogously
to~\eqref{Hurwitz.determinants.relations.infinite.1.odd}--\eqref{Hurwitz.determinants.relations.infinite.1.even},
applying Theorem~\ref{Th.Hurwitz.relations} to the functions
$R(z)$ and $zR(z)$ and using Definition~\ref{def.Hurwitz.matrix.infinite}.
\eop

We next turn to finite matrices of Hurwitz type and consider some
of their properties.

\begin{theorem}\label{Th.relations.Delta.and.D}
Let the polynomials $p$ and $q$ be defined as
in~\eqref{polynomial.1}--\eqref{polynomial.2} and let
\begin{equation*}
R(z)=\dfrac{q(z)}{p(z)}=s_{-1}+\dfrac{s_0}{z}+\dfrac{s_1}{z^2}+\cdots
\end{equation*}
\begin{itemize}
\item[] If $\deg q < \deg p$, then
\begin{eqnarray}\label{Hurwitz.determinants.relations.finite.2.odd}
\Delta_{2j-1}(p,q) &= &a_0^{2j-1}D_j(R),\qquad \quad \;\; j=1,2,\ldots,n, \\
\label{Hurwitz.determinants.relations.finite.2.even}
\Delta_{2j}(p,q) &= &(-1)^{j}a_0^{2j}\widehat{D}_j(R),\qquad
j=1,2,\ldots,n.
\end{eqnarray}
\item[] If $\deg q = \deg p $, then
\begin{eqnarray}\label{Hurwitz.determinants.relations.finite.1.even}
\Delta_{2j}(p,q)&=&a_0^{2j}D_j(R),\qquad \qquad \quad \;\, j=1,2,\ldots,n, \\
\label{Hurwitz.determinants.relations.finite.1.odd}
\Delta_{2j+1}(p,q)&=&(-1)^{j}a_0^{2j+1}\widehat{D}_{j}(R),\qquad
j=0,1,2,\ldots,n,
\end{eqnarray}

\end{itemize}
where $\widehat{D}_0(R)\eqbd 1$ and the determinants
$D_j(R),\widehat{D}_j(R)$ are defined
by~\eqref{Hankel.determinants.1},~\eqref{Hankel.determinants.2},
respectively.
\end{theorem}
\proof
Apply Theorem~\ref{Th.Hurwitz.relations} to the
functions $R(z)$ and $zR(z)$ if $\deg q < \deg p$ or to the
functions $R(z)$ and $zR(z)-z\dfrac{b_0}{a_0}$ if $\deg q =\deg
p$, as in the proof of Lemma~\ref{lemma.relations.eta.nabla.D}.
\eop

We now summarize all previous results and add one more fact: %a few another facts:
\begin{theorem}\label{Th.Stieltjes fraction.global.criterion}
Given polynomials
\begin{eqnarray*}
 p(z) & =  & a_0z^n+a_1z^{n-1}+\cdots+a_n,  \qquad   a_1,\dots,a_n\in\mathbb
C, \quad a_0\neq 0,  \\
q(z)&= & b_0z^{n}+b_1z^{n-1}+\cdots+b_{n}, \qquad \;  b_0,\dots,b_{n}\in\mathbb C,
\end{eqnarray*}
let the function
\begin{equation*}
R(z)=\dfrac{q(z)}{p(z)}=
s_{-1}+\dfrac{s_0}{z}+\dfrac{s_1}{z^2}+\cdots
\end{equation*}
have exactly $r$ poles $(r\leq n)$, counting multiplicities.
The following conditions are equivalent:
\begin{itemize}
\item [$1)$] the Hankel determinants $D_j(R)$ and $\widehat{D}_j(R)$ defined
in~\eqref{Hankel.determinants.1}--\eqref{Hankel.determinants.2}
satisfy
\begin{eqnarray*}\label{minors.inequalities.for.Stieltjes.fraction.1.repeat}
& D_{j}(R)\neq 0, & \quad j=1,2,\ldots,r, \\
\label{minors.inequalities.for.Stieltjes.fraction.2.repeat}
& \widehat{D}_{j}(R) \neq 0, & \quad j=1, 2, \ldots, r-1, \\
\label{minors.inequalities.for.Stieltjes.fraction.3.repeat}
& D_{j}(R)=\widehat{D}_{j}(R)=0, & \quad j=r+1, r+2, \ldots;
\end{eqnarray*}
moreover, $\widehat{D}_r(R)=0$ if and only if $R$ has a pole at $0$;

\item [$2)$] the function $R$ has a Stieltjes continued fraction
expansion~\eqref{Stieltjes.fraction.1} whose coefficients $c_j$
can be found by the
formul\ae~\eqref{even.coeff.Stieltjes.fraction.main.formula}--\eqref{odd.coeff.Stieltjes.fraction.main.formula};
\item [$3)$] the Euclidean algorithm applied to the polynomials
$g_0(z) \eqbd p(z^2)$ and\footnote{If $\deg q <\deg p $, then
$b_0=0$, according to our convention.} $g_1(z) \eqbd
zq(z^2)-\dfrac{b_0}{a_0}zp(z^2)$ is regular and $\deg q \geq
n-1$;
\item [$4)$] the infinite matrix $H(p,q)$ factors as follows:
\begin{itemize}
\item [] if $\deg q <\deg p $, then
\begin{equation}\label{infinite.Hurwitz.matrix.factorization.1}
H(p,q) = J(c_{1})J(c_{2})\cdots J(c_k)H(0,1)\mathcal{T}(g),
\end{equation}
\item [] if $\deg q = \deg p $, then
\begin{equation}\label{infinite.Hurwitz.matrix.factorization.2}
H(p,q)=J(c_{0})J(c_{1})\cdots J(c_k)H(0,1)\mathcal{T}(g),
\end{equation}
\end{itemize}
where $k=2r-1$, if $R(0)=\infty$, and $k=2r$, otherwise. The
polynomial~$g$ is the general common divisor of $p$ and
$q$, the matrix $\mathcal{T}(g)$ is defined
in~\eqref{Toeplitz.infinite.matrix} and
\begin{equation}\label{def.matrix.J(c)}
J(c) \eqbd
\begin{pmatrix}
c & 1 & 0 & 0 & 0 &\dots\\
0 & 0 & 1 & 0 & 0 &\dots\\
0 & 0 & c & 1 & 0 &\dots\\
0 & 0 & 0 & 0 & 1 &\dots\\
0 & 0 & 0 & 0 & c &\dots\\
\vdots &\vdots  &\vdots  &\vdots &\vdots &\ddots
\end{pmatrix},\quad
H(0,1) \eqbd
\begin{pmatrix}
1 & 0 & 0 & 0 & 0 &\dots\\
0 & 0 & 0 & 0 & 0 &\dots\\
0 & 0 & 1 & 0 & 0 &\dots\\
0 & 0 & 0 & 0 & 0 &\dots\\
0 & 0 & 0 & 0 & 1 &\dots\\
\vdots &\vdots  &\vdots  &\vdots &\vdots &\ddots
\end{pmatrix}.
\end{equation}
The coefficients $c_j$ are defined by the
formul\ae~\eqref{even.coeff.Stieltjes.fraction.main.formula}--\eqref{odd.coeff.Stieltjes.fraction.main.formula}
or, equivalently, as follows:
\begin{itemize}
\item [] if $\deg q <\deg p$, then
\begin{eqnarray}\label{odd.coeff.Stieltjes.fraction.main.formula.2.1}
c_{2j-1}&=&\dfrac{\eta_{2j-1}^2(p,q)}{\eta_{2j-2}(p,q)\cdot
\eta_{2j}(p,q)},\qquad j=1,2,\ldots,r,  \\
\label{even.coeff.Stieltjes.fraction.main.formula.2.1}
c_{2j}&=&\dfrac{\eta_{2j}^2(p,q)}{\eta_{2j-1}(p,q)\cdot\eta_{2j}(p,q)},\qquad
j=1,2,\ldots,\left\lfloor\dfrac{k}2\right\rfloor;
\end{eqnarray}
\item [] if $\deg q =\deg p $, then
\begin{eqnarray}\label{odd.coeff.Stieltjes.fraction.main.formula.2.2}
c_{2j-1}&=&\dfrac{\eta_{2j}^2(p,q)}{\eta_{2j-1}(p,q)\cdot
\eta_{2j+1}(p,q)},\qquad j=1,2,\ldots,r, \\
\label{even.coeff.Stieltjes.fraction.main.formula.2.2}
c_{2j}&=&\dfrac{\eta_{2j+1}^2(p,q)}{\eta_{2j}(p,q)\cdot\eta_{2j+2}(p,q)},\qquad
\quad j=0,1,2,\ldots,\left\lfloor\dfrac{k}2\right\rfloor.
\end{eqnarray}
Here we set $\eta_0(p,q)\eqbd 1$ and $k \eqbd 2r-1$ if $R(0)=\infty$,
whereas $k \eqbd 2r$ if $|R(0)|<\infty$.
\end{itemize}
%
%\item[5)]
%
\end{itemize}
\end{theorem}
\proof
Theorem~\ref{Th.Stieltjes.fraction.criterion} establishes
the~equivalence of conditions $1)$ and $2)$.

By Theorem~\ref{Th.J-fraction.criterion}, condition $3)$ is
equivalent to the fact that the function
$z\mapsto F(z)=\dfrac{g_1(z)}{g_0(z)}$ has a $J$-fraction expansion.
At the same time, this is equivalent to condition $1)$,
according to Theorem~\ref{Th.J-fraction.criterion} and
Lemma~\ref{lem.minors.relations.for.Stieltjes.formula}. Thus,
we have proved the equivalence of conditions $1)$ and $3)$.

\vspace{2mm}

Now we prove that condition $4)$ follows from condition~$3)$. As
was shown above, whenever the Euclidean algorithm applied to the
polynomials $g_0$ and $g_1$ is regular, the polynomials
$f_0=p$ and $f_1=q$
satisfy~\eqref{doubly.regular.Euclidean.algorithm}--\eqref{auxiliary.quotients}
with the coefficients $c_j$ determined by the
formul\ae~\eqref{even.coeff.Stieltjes.fraction.main.formula}--\eqref{odd.coeff.Stieltjes.fraction.main.formula}.
From Lemma~\ref{lemma.relations.eta.nabla.D} we see that the
coefficients $c_j$ can be found by the
formul\ae~\eqref{odd.coeff.Stieltjes.fraction.main.formula.2.1}--\eqref{even.coeff.Stieltjes.fraction.main.formula.2.1}
or~\eqref{odd.coeff.Stieltjes.fraction.main.formula.2.2}--\eqref{even.coeff.Stieltjes.fraction.main.formula.2.2}.

\vspace{1mm}

At first, let $\deg q  < \deg p$ and let $g=\gcd(p,q)$. The
algorithm~\eqref{doubly.regular.Euclidean.algorithm}--\eqref{auxiliary.quotients}
produces a sequence of polynomials $f_0$, $f_1$, $\ldots$, $f_k$,
where $k=2r$ if $p(0)\neq 0$ or $k=2r{-}1$ otherwise. Note that
$f_0(z)=\dfrac{p(z)}{g(z)}$ and $f_1(z)=\dfrac{q(z)}{g(z)}$, that
is, $\gcd(f_0,f_1)= 1$. Consider any four consecutive polynomials
$f_{j-1}$, $f_{j}$, $f_{j+1}$, $f_{j+2}$ $(j=1,\ldots,2r-3)$ in
this sequence such that $\deg f_{j-1} >\deg f_{j}$, i.e., $j$ is
odd. Then  the matrix $H(f_{j-1},f_j)$ (see
Definition~\ref{def.Hurwitz.matrix.infinite}) satisfies
\begin{equation}\label{Th.Stieltjes fraction.global.criterion.proof.1}
H(f_{j-1},f_j)=J(c_j)H(f_{j+1},f_j)=J(c_j)J(c_{j+1})H(f_{j+1},f_{j+2}).
\end{equation}
This formula can be easily obtained by straightforward calculation
using~\eqref{doubly.regular.Euclidean.algorithm}--\eqref{degrees.of.doble.regular.algorithm.polys}.
For the polynomials $f_{2r-2}$ and $f_{2r-1}$, the
formula~\eqref{Th.Stieltjes fraction.global.criterion.proof.1}
has a different form:
\begin{itemize}
\item[] if $k=2r-1$, then
\begin{equation}\label{Th.Stieltjes fraction.global.criterion.proof.2}
H(f_{2r-2},f_{2r-1})=J(c_{2r-1})H(f_{2r},f_{2r-1})=J(c_{2r-1})H(0,1),
\end{equation}
since $f_{2r-1}=f_{k}=\gcd(f_0,f_1) = 1$, and therefore
$f_{2r}(z)=f_{k+1}(z)\equiv 0$,
\item[] if $k=2r$, then
\begin{equation}\label{Th.Stieltjes fraction.global.criterion.proof.3}
\begin{array}{lcl}
H(f_{2r-2},f_{2r-1})&=&J(c_{2r-1})H(f_{2r},f_{2r-1})
 =  J(c_{2r-1})J(c_{2r})H(f_{2r+1},f_{2r})\\[1.3mm] &= &J(c_{2r-1})J(c_{2r})H(0,1),
\end{array}
\end{equation}
since $f_{2r}=f_{k}=\gcd(f_0,f_1)= 1$, and
$f_{2r+1}(z)=f_{k+1}(z) \equiv 0$.
\end{itemize}
Thus, from the formul\ae~\eqref{Th.Stieltjes
fraction.global.criterion.proof.1}--\eqref{Th.Stieltjes
fraction.global.criterion.proof.3} we obtain
\begin{equation}\label{Th.Stieltjes fraction.global.criterion.proof.4}
H(f_{0},f_{1})=J(c_1)J(c_2)\cdots J(c_{k-1})J(c_{k})H(0,1).
\end{equation}
At the same time, Theorem~\ref{Th.Hurwitz.matrix.with.gcd} implies
\begin{equation}\label{Th.Stieltjes fraction.global.criterion.proof.5}
H(p,q)=H(f_{0},f_{1})\mathcal{T}(g).
\end{equation}
The formul\ae~\eqref{Th.Stieltjes
fraction.global.criterion.proof.4}--\eqref{Th.Stieltjes
fraction.global.criterion.proof.5}
yield~\eqref{infinite.Hurwitz.matrix.factorization.1}.

As for the case $\deg p=\deg q$, we denote
$f_{0}(z)\eqbd \dfrac{p(z)}{g(z)}$,
$f_{1}(z)\eqbd \dfrac{q(z)-c_0p(z)}{g(z)}$, where
$c_0 \eqbd \dfrac{b_0}{a_0}$, and find by straightforward calculation
that
\begin{equation}\label{Th.Stieltjes fraction.global.criterion.proof.6}
H(p,q)=J(c_0)H(f_{0},f_{1})\mathcal{T}(g).
\end{equation}
Since $\deg f_0 > \deg f_1$, the matrix $H(f_0,f_1)$
satisfies~\eqref{Th.Stieltjes fraction.global.criterion.proof.4}.
Thus, from~\eqref{Th.Stieltjes fraction.global.criterion.proof.6}
and~\eqref{Th.Stieltjes fraction.global.criterion.proof.4} we
obtain the
factorization~\eqref{infinite.Hurwitz.matrix.factorization.2}.

Conversely, if condition $4)$ holds, then we can reconstruct
the
algorithm~\eqref{doubly.regular.Euclidean.algorithm}--\eqref{auxiliary.quotients}
using the
factorizations~\eqref{infinite.Hurwitz.matrix.factorization.1}
or~\eqref{infinite.Hurwitz.matrix.factorization.2} as follows: the
coefficients of the polynomials $f_{j-1}$ and $f_j$ (if
$\deg f_{j-1} > \deg f_{j}$) are the entries in the first
and the second rows, respectively, of the matrix
\begin{equation}\label{Th.Stieltjes fraction.global.criterion.proof.7}
H(f_{j-1},f_{j})=J(c_j)J(c_{j+1})\cdots
J(c_{k-1})J(c_{k})H(0,1),\quad j=1,\ldots,k.
\end{equation}
Here we have $f_0(z)=\dfrac{p(z)}{g(z)}$ and
$f_1(z)=\dfrac{q(z)-c_0p(z)}{g(z)}$, where\footnote{As was
mentioned above, $b_0=0$ if $\deg p > \deg q$.}
$c_0=\dfrac{b_0}{a_0}$. Note that $g_0(z)=f_0(z^2)$ and
$g_1(z)=zf_1(z^2)$ in our notation. As was shown above, the
algorithm~\eqref{doubly.regular.Euclidean.algorithm}--\eqref{auxiliary.quotients}
for the polynomials $f_0$ and $f_1$ is equivalent to a regular
 Euclidean algorithm for the polynomials $g_0(z)$ and $g_1(z)$.
Therefore,  condition $4)$ implies  condition $3)$.
\eop

\begin{remark}\label{remark.1.11}
Regarding equivalence classes of rational functions, we should
note the following: Suppose that  two rational functions $R$ and
$G$ with exactly $r$ poles each satisfy the
inequalities~\eqref{minors.inequalities.for.Stieltjes.fraction.1}--\eqref{minors.inequalities.for.Stieltjes.fraction.2}.  Then
 $R(z)\equiv G(z)$ if and only if
\begin{equation*}\label{class.of.equivalence}
\begin{array}{c}
D_j(R)= D_j(G),\\[2mm]
\widehat{D}_j(R)=\widehat{D}_j(G),
\end{array}\qquad j=1,2,\ldots,r,
\end{equation*}
since these equalities guarantee that the corresponding Stieltjes coefficients of $R$ and $G$ 
coincide by~\eqref{even.coeff.Stieltjes.fraction.main.formula}--\eqref{odd.coeff.Stieltjes.fraction.main.formula}. 
We would like to remind the reader that, according to Theorem~\ref{Th.class.of.Hankel.equivalence}, the
equality of the minors $D_j(R)$ and $D_j(G)$, for each $j$,  \textit{per se\/}
does not guarantee that the functions $R$ and $G$ are equal.
\end{remark}

\setcounter{equation}{0}
%
%\setcounter{notezz}{0}
%

%%%%%%%%%%%%%%%%%%%%%%%%%%%%%%%%%%%%%%%%%%%%%%%%%%%%%%%%%%%%%%%%%%%%
\section{Real rational functions and related topics\label{s:rational}}
%%%%%%%%%%%%%%%%%%%%%%%%%%%%%%%%%%%%%%%%%%%%%%%%%%%%%%%%%%%%%%%%%%%%

In this section we develop connections among several notions:
the Euclidean algorithm and its variant, the Sturm algorithm
(Section~\ref{s:sturm}), Cauchy indices (Section~\ref{s:cauchy}),
various representations of rational functions, and their associated
Hankel minors. Those diverse topics turn out to be connected to the
same basic question of counting roots or poles of rational functions.
Throughout this section, we assume that all our rational functions are
\emph{real}.

Thus, consider a \textit{real} rational function
\begin{equation}\label{basic.real.rational.function}
R(z)=\dfrac{q(z)}{p(z)}=s_{-1}+\dfrac{s_0}{z}+\dfrac{s_1}{z^2}+\dfrac{s_2}{z^3}+\cdots,\quad
s_i\in\mathbb{R}, \qquad i=-1,0,1,2,\ldots.
\end{equation}
where $p$ and $q$ are \textit{real} polynomials
\begin{eqnarray}\label{real.polynomial.1}
& p(z) \; \eqbd \; a_0z^n+a_1z^{n-1}+\cdots+a_n,  &
a_0,a_1,\dots,a_n\in\mathbb{R},\ a_0\neq0, \\
\label{real.polynomial.2}
& q(z) \; \eqbd \; b_0z^{n}+b_1z^{n-1}+\cdots+b_{n},\,  &
b_0,\dots,b_{n}\in\mathbb{R},
\end{eqnarray}
In what follows we assume that the function $R$ has exactly
$r=n-l$ poles (counting multiplicities) where $l$ is the degree of
the greatest common divisor of the polynomials $p$ and $q$
($0\leq l\leq n$).

%%%%%%%%%%%%%%%%%%%%%%%%%%%%%%%%%%%%%%%%%%%%%%%%%%%%%%%%%%%%%%%%%%%
\subsection{Sturm algorithm and Frobenius rule of signs for
Hankel minors\label{s:sturm}}
%%%%%%%%%%%%%%%%%%%%%%%%%%%%%%%%%%%%%%%%%%%%%%%%%%%%%%%%%%%%%%%%%%%

For real polynomials $f_0\eqbd p$ and $f_1\eqbd q$, it is convenient to use a modification
of the Euclidean algorithm, namely, its variant known as the Sturm algorithm.
Suppose that we run the Euclidean algorithm~\eqref{Euclidean.algorithm}
starting with the polynomials $f_0$ and $f_1$.  If we denote
\begin{eqnarray*}\label{Sturm.sequence.construction.1}
& \widetilde{f}_j(z) \; \eqbd  \; (-1)^{\frac{j(j-1)}2}f_j(z), & \;\;
 j=0,1,\ldots,k,  \\
\label{Sturm.sequence.construction.2}
& \widetilde{q}_j(z) \; \eqbd \; (-1)^jq_j(z), \quad \;\;\; & \;\; j=1,\ldots,k,
\end{eqnarray*}
then the polynomials $\widetilde{f}_j$ and $\widetilde{q}_k$
satisfy the following relations:
\begin{equation}\label{Sturm.algorithm}
\widetilde{f}_{j-1}(z)\; = \; \widetilde{q}_j(z)\widetilde{f}_j(z)-\widetilde{f}_{j+1}(z),
\qquad j=0,1,\ldots,k,
\end{equation}
where $\widetilde{f}_{k+1}(z)\equiv 0$.

\begin{definition}\label{def.Sturm.algorithm}
The relations~\eqref{Sturm.algorithm} represent the so-called
\textit{Sturm algorithm}. If all the polynomials $\widetilde{q}_j$ are
linear, the algorithm is called \textit{regular}.
\end{definition}

The polynomial $f_k$ is the greatest common divisor of the polynomials $f_0$ and $f_1$.
Thus, the Sturm algorithm also produces the greatest common divisor of two initial
polynomials, but the Sturm form turns out to have advantages over the Euclidean form
in the real case, as we will clarify later. Very roughly, in the real case signs of some
quantities are more easily traced using the Sturm algorithm than the Euclidean algorithm.
In the complex case, the  issue of signs has no comparable significance.

%%%%%%%%%%%%%%%%%%%%%%%%%%%%%%%%%%%%%%%%%%%%%%%%%%%%%%%%%%%%%%%%%%%%%
%\subsection{Sturm method}
%%%%%%%%%%%%%%%%%%%%%%%%%%%%%%%%%%%%%%%%%%%%%%%%%%%%%%%%%%%%%%%%%%%%%

In connection with the Sturm method, we  mention briefly the so-called
Sturm sequences, which, however, will not be used much in the sequel:

\begin{definition}\label{Sturm.sequence}
A sequence of polynomials $g_0$, $g_1$, $\ldots$, $g_n$ is called
a \textit{Sturm sequence} on the interval $(a,b)$ if
\begin{itemize}
\item [1)] $g_0(z_*)=0$ for some $z_*\in(a,b)\quad\Rightarrow\quad g_1(z_*)\neq 0$;
\item [2)] $g_j(z_*)=0$ for some $z_*\in(a,b)\quad \Rightarrow \quad g_{j-1}(z_*)g_{j+1}(z_*)<0$
\; for $j=1,2\ldots,n-1$;
\item [3)] $g_n(z)\neq 0 \quad \forall z\in(a,b)$.
\end{itemize}
\end{definition}

The sequence $\widetilde{f}_0$, $\widetilde{f}_1$, $\ldots$, $\widetilde{f}_r$
from the Sturm algorithm~\eqref{Sturm.algorithm} is easily seen to be a Sturm
sequence on any interval where $\widetilde{f}_r(z)$ does not vanish. Moreover,
if we denote
\begin{equation*}\label{polys.via.gcd.Sturm}
\widetilde{h}_j(z) \eqbd \dfrac{\widetilde{f}_j(z)}{\widetilde{f}_r(z)},
\qquad j=0,1,\ldots,r,
\end{equation*}
then $\widetilde{h}_r(z)\equiv 1$ and the sequence
$\widetilde{h}_0$, $\widetilde{h}_1$, $\ldots$, $\widetilde{h}_r$
is a Sturm sequence on the real axis.

%In this and other sections, we consider only the Sturm algorithm.
%Some formul\ae\, in the case of this algorithm have another form
%than in the case of Euclidean algorithm. We will show changed
%formul\ae\, for the Sturm algorithm, when we use them.

Next, we must introduce several notions of sign changes for sequences of real numbers.
\begin{definition}
Given a sequence $\mathbf{t}\eqbd (t_0,t_1,\ldots,t_n)$ without zeros, we say
that its \textit{number of sign changes} % $\SC(\mathbf{t})=\SC(t_0, \ldots, t_n)$
is the number of indices $j$ between $1$ and $n$ satisfying $t_{j-1}t_j<0$.
The \textit{number of sign retentions}  % $\SR(\mathbf{t})=\SR(t_0, \ldots, t_n)$
is the number of indices $j$ between $1$ and $n$ satisfying $t_{j-1}t_j>0$.

For a sequence with zeros, the maximum number of sign changes obtainable
by an appropriate choice of signs of any zero entry is called the
\textit{number of weak sign changes} and is denoted by
$\SC^+(\mathbf{t})=\SC^+(t_0, \ldots, t_n)$.
The minimum number so obtainable is called the
\textit{number of strong sign changes} and is denoted by
$\SC^-(\mathbf{t})=\SC^-(t_0, \ldots, t_n)$.
The \textit{number of weak $\SR^+(\mathbf{t})$ and strong $\SR^-(\mathbf{t})$
sign retentions} can be defined correspondingly.
\end{definition}

Note that the number of weak sign changes does not increase and the
number of strong sign changes does not decrease under small
perturbations of the elements of a given sequence. Also note that
the number of strong sign changes can be determined simply by
discarding all zero elements and counting the number of ordinary
sign changes in the obtained sequence. Finally, note that
the number of sign changes and the number of sign retentions
(be it ordinary, weak, or strong, respectively) always sum up
to $n$ if $n+1$ is the length of the sequence:
\begin{eqnarray}\label{sign.changes.retentions.relation}
%\SC(t_0,t_1,\ldots,t_n)+\SR(t_0,t_1,\ldots,t_n)=n. \\
%\label{sign.changes.retentions.relation2}
\SC^\pm(t_0,t_1,\ldots,t_n)+\SR^\pm(t_0,t_1,\ldots,t_n)=n.
\end{eqnarray}

In the sequel, we will need only the notion of strong sign changes, so our discussion
of weak sign changes above is included for completeness only.  We will also need
another important method of counting sign changes
(and sign retentions) specifically introduced by Frobenius for sequences of Hankel
minors\footnote{Since our function $R$ is real, all its minors $D_j(R)$ and
$\widehat{D}_j(R)$ and therefore the sequences
$(D_0(R),D_1(R),D_2(R),\ldots,D_r(R))$  and
$(\widehat{D}_0(R),\widehat{D}_1(R),\widehat{D}_2(R),\ldots,\widehat{D}_r(R))$
are real.} $(D_0(R),D_1(R),D_2(R),\ldots,D_r(R))$ and
$(\widehat{D}_0(R),\widehat{D}_1(R),\widehat{D}_2(R),\ldots,\widehat{D}_r(R))$
for a real rational function $R$.
As usual, we set $D_0(R)\eqbd \widehat{D}_0(R)\eqbd 1$.

%In his remarkable work~\cite{Frobenius}, Frobenius proved the
%following fact which enables us to calculate the number $V(D_0(R),D_1(R),\ldots,D_r(R))$
%of sign changes in case when the sequence
%$\{D_0(R),D_1(R),\ldots,D_r(R)\}$ contains a few successive zeros.
%
\begin{rrule}[Frobenius \cite{Frobenius,{Gantmakher.1}}]\label{Th.Frobenius}
If, for some integers $i$ and $j$ $(0\leq i<j)$,
\begin{equation}\label{Th.Frobenius.condition}
D_i(R)\neq 0, \quad D_{i+1}(R)=D_{i+2}(R)=\cdots=D_{i+j}(R)=0,\quad
D_{i+j+1}(R)\neq 0,
\end{equation}
then the \textit{number $\SC^F(D_0(R),D_1(R),D_2(R),\ldots,D_r(R))$ of Frobenius sign changes}
should be calculated by assigning  signs as follows:
\begin{equation}\label{Th.Frobenius.rule}
\sgn D_{i+\nu}(R)=(-1)^{\tfrac{\nu(\nu-1)}2}\sgn D_{i}(R),\quad
\nu=1,2,\ldots,j.
\end{equation}
The \textit{number of Frobenius sign retentions}  $\SR^F(D_0(R),D_1(R),D_2(R),\ldots,D_r(R))$
is defined accordingly.
\end{rrule}
\
This assignment has an interesting property, which will be useful later.
\begin{corol}\label{corol.Frobenius.rule}
If the condition~\eqref{Th.Frobenius.condition} holds and if
integers $\mu$ and $\nu$ $(i<\mu\leq \nu\leq j)$ have
equal parities, then
\begin{equation}\label{Frobenius.rule.1}
\SC^F(D_{\mu}(R),D_{\mu+1}(R),\ldots,D_{\nu}(R))=\SR^F(D_{\mu}(R),D_{\mu+1}(R),\ldots,D_{\nu}(R)).
\end{equation}
%
%!!!!!!!!!!!!!!!!!!!!!!!!!!!!!!!!!
%
%\begin{equation}\label{Frobenius.rule.2}
%P(D_{2i}(R),D_{2i+1}(R),\ldots,D_{2\nu+1}(R))=V(D_{2i}(R),D_{2i+1}(R),\ldots,D_{2\nu+1}(R))+1.
%\end{equation}
%
%\begin{equation}\label{Frobenius.rule.3}
%P(D_{2i-1}(R),D_{2i}(R),\ldots,D_{2\nu}(R))=V(D_{2i-1}(R),D_{2i}(R),\ldots,D_{2\nu}(R))-1.
%\end{equation}
%
\end{corol}
The Frobenius method of counting sign changes is so pervasive in the rest of this paper
that we adopt the notational convention $$\SCF(\mathbf{t}) \eqbd \SC^F(\mathbf{t}),
\qquad \SRF(\mathbf{t}) \eqbd \SR^F(\mathbf{t}).$$

%%%%%%%%%%%%%%%%%%%%%%%%%%%%%%%%%%%%%%%%%%%%%%%%%%%%%%%%%%%%%
\subsection{Cauchy indices and their properties\label{s:cauchy}}
%%%%%%%%%%%%%%%%%%%%%%%%%%%%%%%%%%%%%%%%%%%%%%%%%%%%%%%%%%%%%

In this section, we introduce a special counter known as the Cauchy index.
Consider a real rational function
$F$, which may in principle have a pole at $\infty$, unlike the function $R$.

\begin{definition}\label{Def.Cauchy.index.at.pole}
The quantity
\begin{equation}\label{Cauchy.index.at.odd.pole}
\Ind\nolimits_{\omega}(F)\eqbd
\begin{cases}
   \; +1&\;\text{if}\quad F(\omega-0)<0<F(\omega+0),\\
   \; -1&\;\text{if}\quad F(\omega-0)>0>F(\omega+0),
\end{cases}
\end{equation}
is called the \textit{index\/} of the function $F$ at its
\textit{real\/} pole $\omega$ of \textit{odd\/} order.
\end{definition}

We also set
\begin{equation}\label{Cauchy.index.at.even.pole}
\Ind\nolimits_{\omega}(F)\eqbd 0
\end{equation}
if $\omega$ is a real pole of the function $F$ of \textit{even\/}
order.

Suppose that the function $F$ has $m$ real poles in total, viz.,
$\omega_1<\omega_2<\dots<\omega_m$.

\begin{definition}\label{Def.Cauchy.index.on.interval}
The quantity
\begin{equation}\label{Cauchy.index.on.interval}
\Ind\nolimits_a^b(F)\eqbd \sum\limits_{i\,\colon\,a<\omega_i<b}
\Ind\nolimits_{\omega_i}(F).
\end{equation}
is called the \textit{Cauchy index} of the function $F$ on the
interval~$(a,b)$.
\end{definition}
We are primarily interested in the quantity $\IndC(F)$, the Cauchy
index of $F$ on the real line. However, since the function $F$
may have a pole at the point $\infty$, it is convenient for us to
consider this pole as \textit{real}. From this point of view, let
us introduce the index at $\infty$ following, e.g.,~\cite{Barkovsky.2}.
To do so, we consider the function~$F$ as a map on the \textit{projective
line} %\footnote{For more information on projective spaces, see, for example, ?????.}%
$\PR^1\eqbd \mathbb{R}^1\cup\{\infty\}$ into itself. So, if the
function $F$ has a pole at $\infty$, then we let
\begin{equation}\label{Cauchy.index.at.infty}
\IndI(F) \eqbd
\begin{cases}
\; +1 & \quad\text{if}\quad F(+\infty)<0<F(-\infty),\\
\; -1& \quad\text{if}\quad F(+\infty)>0>F(-\infty),\\
\; \quad\! 0 & \quad\text{if}\quad \sgn F(+\infty)=\sgn F(-\infty).
\end{cases}
\end{equation}
Thus, the generalized Cauchy  index of the function $F$ on the
projective real line is
\begin{equation}\label{Cauchy.index.on.projective.line}
\IndPR(F)\eqbd \IndC(F)+\IndI(F).
\end{equation}

\begin{remark}\label{remark.2.1}
Obviously, if the function $F$ has no pole at $\infty$, then the generalized
Cauchy index  $\IndPR(F)$ coincides with the usual Cauchy index $\IndC(F)$.
\end{remark}

Following~\cite{Barkovsky.2}, we list a few properties of
generalized Cauchy indices, which will be of use later.

First, note that a polynomial $q(z)=cz^{\nu}+\cdots$
$(\nu=\deg q)$ can be viewed as a rational function with
a single pole at $\infty$, hence
\begin{equation}\label{Cauchy.index.of.poly}
\IndI(q)=
\begin{cases}
-\sgn c\; &\text{if}\quad\nu\quad\text{is odd},\\
\;\;\; 0     \; &\text{if}\quad\nu\quad\text{is even}.\\
\end{cases}
\end{equation}

The following theorem collects all properties of Cauchy indices that we
will need.
\begin{theorem}[see~\cite{Barkovsky.2}]\label{Th.Cauchy.index.properties}
Let $F$ be a real rational function.
\begin{itemize}
\item[1)] If $d$ is a real constant, then $\IndPR(d+F)=\IndPR(F)$.
\item[2)] If $q$ is a real polynomial and $|F(\infty)|<\infty$,
then $\IndPR(q+F)=\IndC(F)+\IndI(q)$.
\item[3)] If $G$ and $F$ are real rational functions that have no
real poles in common, then
\begin{equation}\label{Cauchy.index.of.functions.sum}
\IndC(F+G)=\IndC(F)+\IndC(G).
\end{equation}
\item[4)] $\IndPR\left(-\dfrac1F\right)=\IndPR(F)$.
\end{itemize}
\end{theorem}
\proof
Properties $1),2)$ and $3)$ follow immediately from the definition of Cauchy
indices~\eqref{Cauchy.index.on.interval}--\eqref{Cauchy.index.at.infty}
and from~\eqref{Cauchy.index.of.poly}.
The proof of Property $4)$ is reproduced from~\cite{Barkovsky.2}.
The projective line $\PR^1$ is divided by two points, $0$ and
$\infty$, into its positive $(0,\infty)$ and negative
$(-\infty,0)$ rays. Suppose that a variable $z$ traverses $\PR^1$
and returns to its starting point. Clearly, the number of
crossings of $F(z)$ from $(-\infty,0)$ to $(0,\infty)$ must equal
the number of reverse crossings. The crossings through $\infty$
occur at the poles of the function $F$; they are accounted for in
the sum~\eqref{Cauchy.index.on.projective.line} with the
appropriate sign. The crossings through $0$ occur at the zeros of
$F$, i.e., at the poles of the function $z\mapsto \dfrac1{F(z)}$,
and they are accounted for in the analogous formula for
$\IndPR\left(\dfrac1{F}\right)$. As a result,
$\IndPR\left(\dfrac1F\right)+\IndPR(F)=0$. \eop

Let us apply the Sturm algorithm~\eqref{Sturm.algorithm} to the
numerator and denominator of the fraction
$R$~(see~\eqref{basic.real.rational.function}). As a result, we
obtain another kind of continued fraction, which slightly differs
from~\eqref{continued.fraction.general}:
\begin{equation}\label{continued.fraction.general.via.Sturm}
R(z)=s_{-1}+\dfrac1{q_1(z)-\cfrac1{q_2(z)-\cfrac1{q_3(z)-\cfrac1{\ddots-\cfrac1{q_k(z)}}}}}
\end{equation}
Here the polynomials $q_i$ have the form
\begin{equation}\label{Sturm.cont.frac.quotients}
q_j(z)=\alpha_jz^{n_j}+\cdots,\quad\;\;\;  \alpha_j\neq 0,\qquad
j=1,2,\ldots,k,
\end{equation}
where\footnote{Recall that $r$ is the number of poles, counted
with multiplicities, of the function $R$.} $n_1+n_2+\cdots+n_k=r$
and $n_i\geq 1$, $i=1,\ldots,k$. From
Theorem~\ref{Th.relation.continued.fraction.Hankel.minors}  it is
easy\footnote{One should replace $G$ by $-G$ in the
formula~\eqref{Th.for.criterion.J-fraction.statement.1} and apply
the formula~\eqref{Th.for.criterion.J-fraction.statement.2}
inductively to the
function~\eqref{continued.fraction.general.via.Sturm}, taking into
account that $(-1)^jD_j(G)=D_j(-G)$.} to specialize the
formula~\eqref{continued.fraction.determinants.formula} to the
function~\eqref{continued.fraction.general.via.Sturm}:
\begin{equation}\label{continued.fraction.determinants.formula.Sturm}
\displaystyle
D_{n_1+n_2+\cdots+n_j}(R)=\prod_{i=1}^{j}(-1)^{\tfrac{n_i(n_i-1)}2}
\cdot\prod_{i=1}^{j}\dfrac1{\alpha_i^{n_i+2\sum_{\rho=i+1}^{j}n_{\rho}}},
\qquad j=1,2,\ldots,k.
\end{equation}
Applying property $1)$ and then inductively properties $4)$, $2)$
and $3)$ of Theorem~\ref{Th.Cauchy.index.properties} to the
function~\eqref{continued.fraction.general.via.Sturm}, we get the
following result:
\begin{theorem}[\cite{Barkovsky.2}]\label{Th.index.via.quotients}
If a rational function $R$ is represented by a continued
fraction~\eqref{continued.fraction.general.via.Sturm}, then
\begin{equation}\label{index.via.quotients}
\IndPR(R)=-\sum_{j=1}^{k}\IndI(q_j).
\end{equation}
\end{theorem}

This theorem implies an important fact recorded in
Theorem~\ref{Th.index.via.determinants} below. That fact is closely
connected to the theory of quadratic forms. In fact, it was initially
proved  for rational functions with simple poles
by Hermite using quadratic forms~\cite{Hermite} (see
also~\cite[pp.397--414]{Hermite2}) and for arbitrary
rational functions by Hurwitz~\cite{Hurwitz} (see
also~\cite{KreinNaimark,{Gantmakher},{Barkovsky.2}}). Our proof
differs from those proofs as well as from the proofs of Gantmacher
in~\cite{Gantmakher} and Barkovsky in~\cite{Barkovsky.2}
%\footnote{Hermite, Hurwitz and Barkovsky approached the result
%of Theorem~\ref{Th.index.via.determinants} only in terms
%of quadratic forms.}%
in that it does not use the theory of quadratic forms, only  Frobenius
Rule~\ref{Th.Frobenius}, some properties of continued
fractions, and Theorem~\ref{Th.index.via.quotients} about Cauchy indices.

\begin{theorem}[\cite{Hermite,Hermite2}]\label{Th.index.via.determinants}
If a rational function $R$ with exactly $r$ poles is represented by a
series~\eqref{basic.real.rational.function}, then
\begin{equation}\label{index.via.determinants.1}
\IndC(R)=r-2\SCF(D_0(R),D_1(R),D_2(R),\ldots,D_r(R)).
\end{equation}
where the determinants $D_j(R)$ are defined
in~\eqref{Hankel.determinants.1} and where $D_0(R)\eqbd 1$.
\end{theorem}
\proof
According to~\eqref{sign.changes.retentions.relation}, the
assertion of the theorem is equivalent to the following formula
\begin{equation}\label{index.via.determinants.2}
\IndC(R)=\SRF(D_0(R),D_1(R),D_2(R),\ldots,D_r(R))-\SCF(D_0(R),D_1(R),D_2(R),\ldots,D_r(R)).
\end{equation}
Moreover, according to Remark~\ref{remark.2.1}, we have
$\IndPR(R)=\IndC(R)$.

First, let us assume that the function $R$ has a $J$-fraction
expansion obtained via a \emph{regular} Sturm algorithm (see
Definition~\ref{def.Sturm.algorithm}), that is, where all the
polynomials~\eqref{Sturm.cont.frac.quotients} are linear and
$k=r$. Thus, $n_i=1$ $(i=1,2,\ldots,r)$, and the
formula~\eqref{continued.fraction.determinants.formula.Sturm}
takes a simpler form (cf.~\eqref{J-fraction.determinants.formula})
\begin{equation}\label{continued.fraction.determinants.formula.Sturm.regular}
\displaystyle D_{j}(R)=\prod_{i=1}^{j}\dfrac1{\alpha_i^{2j-2i+1}},
\qquad j=1,2,\ldots,r.
\end{equation}
This implies the following formula:
\begin{equation}\label{leading.coeffs.quotients.determinants.relations.Sturm.regular}
\displaystyle
\alpha_j=\dfrac{D_{j-1}(R)}{D_{j}(R)}\prod_{i=1}^{j-1}\dfrac1{\alpha_i^2},
\qquad j=1,2,\ldots,r.
\end{equation}
In our (regular) case,  from~\eqref{Cauchy.index.of.poly}
and~\eqref{index.via.quotients} we obtain
\begin{equation}\label{index.via.quotients.regular}
\IndC(R)=\sum_{j=1}^{r}\sgn\alpha_j.
\end{equation}
But~\eqref{leading.coeffs.quotients.determinants.relations.Sturm.regular}
implies
\begin{equation}\label{signum.of.leading.coeffs.quotients.determinants.relations.Sturm.regular}
\displaystyle
\sgn\alpha_j=\sgn\dfrac{D_{j-1}(R)}{D_{j}(R)}=\SRF(D_{j-1}(R),D_{j}(R))-\SCF(D_{j-1}(R),D_{j}(R)),\qquad
j=1,2,\ldots,r.
\end{equation}
Combining~\eqref{index.via.quotients.regular}
and~\eqref{signum.of.leading.coeffs.quotients.determinants.relations.Sturm.regular},
we obtain~\eqref{index.via.determinants.2}.

Now let all  polynomials $q_j$ in~\eqref{continued.fraction.general.via.Sturm}
be of  odd degrees, that is, let all the numbers $n_j$ be~\textit{odd}.
Then~\eqref{Cauchy.index.of.poly} and~\eqref{index.via.quotients} imply
\begin{equation}\label{index.via.quotients.odd}
\IndC(R)=\sum_{j=1}^{k}\sgn\alpha_j.
\end{equation}
Using notation~\eqref{degrees} from~\eqref{continued.fraction.determinants.formula.Sturm},
we see that
\begin{equation}\label{leading.coeffs.quotients.determinants.relations.Sturm.odd}
\displaystyle \alpha_j^{n_j}=(-1)^{\tfrac{n_j(n_j-1)}2}
\cdot\dfrac{D_{m_{j-1}}(R)}{D_{m_j}(R)}\cdot
\prod_{i=1}^{j-1}\dfrac1{\alpha_i^{2n_j}}, \qquad j=1,2,\ldots,k.
\end{equation}
Take into account that  the numbers  $m_{j-1}$ and $m_j-1$ have
equal parities for every $j$, $1\leq j\leq k$, since
their difference $n_j-1$ is  even. Hence, according to
Corollary~\ref{corol.Frobenius.rule}
(see~\eqref{Frobenius.rule.1}), we have
\begin{equation}\label{Th.index.via.determinants.proof.1}
\SRF(D_{m_{j-1}}(R),D_{m_{j-1}+1}(R),\ldots,D_{m_{j}-1}(R))-\SCF(D_{m_{j-1}}(R),D_{m_{j-1}+1}(R),\ldots,D_{m_{j}-1}(R))=0.
\end{equation}
The Frobenius Rule~\eqref{Th.Frobenius.rule} gives
\begin{equation*}\label{Th.index.via.determinants.proof.2}
\sgn D_{m_j-1}(R)=\sgn
D_{m_{j-1}+n_j-1}(R)=(-1)^{\tfrac{(n_j-1)(n_j-2)}2}\sgn
D_{m_j}(R), \qquad j=1,2,\ldots,k.
\end{equation*}
And from~\eqref{leading.coeffs.quotients.determinants.relations.Sturm.odd} we obtain
\begin{equation}\label{Th.index.via.determinants.proof.3}
\begin{array}{rcl}
\displaystyle
\sgn\alpha_j & = & \sgn\alpha_j^{n_j} \; = \; (-1)^{\tfrac{n_j(nj-1)}2}\sgn
\dfrac{D_{m_{j-1}}(R)}{D_{m_j}(R)}\;= \\
 \\
& =& (-1)^{\tfrac{n_j(n_j-1)}2}\cdot(-1)^{\tfrac{(n_j-1)(n_j-2)}2}\cdot\sgn
\dfrac{D_{m_{j}-1}(R)}{D_{m_j}(R)} \; = \; \sgn
\dfrac{D_{m_{j}-1}(R)}{D_{m_j}(R)}\;= \\
 \\
& = & \SRF(D_{m_j-1}(R),D_{m_j}(R))-\SCF(D_{m_j-1}(R),D_{m_j}(R)), \qquad
j=1,2,\ldots,k.
\end{array}
\end{equation}
Here we used the condition that all $n_j$ $(j=1,2,\ldots,k)$  are \textit{odd}.
From~\eqref{Th.index.via.determinants.proof.1}
and~\eqref{Th.index.via.determinants.proof.3} we obtain
\begin{equation}\label{Th.index.via.determinants.proof.4}
\begin{array}{rcl}
\displaystyle \sgn\alpha_j & = &
\SRF(D_{m_{j-1}}(R),D_{m_{j-1}+1}(R),\ldots,D_{m_{j}-1}(R),D_{m_{j}}(R))\;-\\
 \\
&& -\; \SCF(D_{m_{j-1}}(R),D_{m_{j-1}+1}(R),\ldots,D_{m_{j}-1}(R),D_{m_{j}}(R)),
\qquad j=1,2,\ldots,k.
\end{array}
\end{equation}
Substituting this formula into~\eqref{index.via.quotients.odd}
yields~\eqref{index.via.determinants.2}.

Let us now consider the general case and let $1\leq
j_1<j_2<\cdots<j_{\eta}\leq k$ be the indices of those
polynomials $q_{i_1}$, $q_{i_2}$, $\ldots$, $q_{i_{\eta}}$
in~\eqref{continued.fraction.general.via.Sturm} that have
\textit{odd} degrees.
Then~\eqref{Cauchy.index.of.poly},~\eqref{Sturm.cont.frac.quotients}
and~\eqref{index.via.quotients} imply
\begin{equation}\label{index.via.quotients.general}
\IndC(R)=\sum_{i=1}^{\eta}\sgn\alpha_{j_i}.
\end{equation}
Let $i$ and $j$ $(1\leq i<j\leq k)$ be integers such
that $n_i$ and $n_j$ are \textit{odd} and
$n_{i+1},n_{i+2},\ldots,n_{j-1}$ are all \textit{even}. Then the
numbers $m_i,m_{i+1},\ldots,m_{j-1},m_{j}-1$ (see~\eqref{degrees})
have equal parities and Corollary~\ref{corol.Frobenius.rule}
yields
\begin{equation}\label{Th.index.via.determinants.proof.5}
\SRF(D_{m_{i}}(R),D_{m_{i}+1}(R),\ldots,D_{m_{j}-1}(R))-\SCF(D_{m_{i}}(R),D_{m_{i}+1}(R),\ldots,D_{m_{j}-1}(R))=0.
\end{equation}
On the other hand, the formula~\eqref{Th.index.via.determinants.proof.3}
is valid in the present case, since $n_j$ is \textit{odd} by
assumption. Therefore,
\begin{equation}\label{Th.index.via.determinants.proof.6}
\begin{array}{rcl}
\displaystyle \sgn\alpha_j & = &
\SRF(D_{m_{i}}(R),D_{m_{i}+1}(R),\ldots,D_{m_{j}-1}(R),D_{m_{j}}(R))\;-\\
 \\
&& - \; \SCF(D_{m_{i}}(R),D_{m_{i}+1}(R),\ldots,D_{m_{j}-1}(R),D_{m_{j}}(R)).
\end{array}
\end{equation}
Applied to  the indices $i_1$, $i_2$, $\ldots$, $i_{\eta}$~, the formula
~\eqref{Th.index.via.determinants.proof.6} yields
\begin{equation}\label{Th.index.via.determinants.proof.7}
\begin{array}{rcl}
\displaystyle \sgn\alpha_{j_i} & = &
\SRF(D_{m_{j_{i-1}}}(R),D_{m_{j_{i-1}}+1}(R),\ldots,D_{m_{j_i}}(R))\\
 \\
&& - \; \SCF(D_{m_{j_{i-1}}}(R),D_{m_{j_{i-1}}+1}(R),\ldots,D_{m_{j_i}}(R)),\qquad
i=1,2,\ldots,\eta.
\end{array}
\end{equation}
If $j_{\eta}<k$, then the indices
$n_{j_{\eta}+1},n_{j_{\eta}+2},\ldots,n_k$ are all \textit{even}, so
all indices
$m_{j_{\eta}},m_{j_{\eta}+1},m_{j_{\eta}+2},\ldots,m_k$ have equal
parities, hence Corollary~\ref{corol.Frobenius.rule} implies
\begin{equation}\label{Th.index.via.determinants.proof.8}
\SRF(D_{m_{j_{\eta}}}(R),D_{m_{j_{\eta}}+1}(R),\ldots,D_{m_{k}}(R))-\SCF(D_{m_{j_{\eta}}}(R),D_{m_{j_{\eta}}+1}(R),\ldots,D_{m_{k}}(R))=0.
\end{equation}
The formul\ae~\eqref{Th.index.via.determinants.proof.7}--\eqref{Th.index.via.determinants.proof.8}
with~\eqref{index.via.quotients.general}
yield~\eqref{index.via.determinants.2}, as desired.
\eop

Now we show how Theorem~\ref{Th.index.via.determinants} can help in
calculating the Cauchy index of a rational function if that function
has a $J$-fraction or Stieltjes continued fraction expansion.
\begin{theorem}\label{Th.index.via.J-fraction}
If a real rational function~\eqref{basic.real.rational.function}
with exactly $r$ poles has a $J$-fraction
expansion
\begin{equation*}\label{J-fraction.via.Sturm}
R(z)=s_{-1}+
\dfrac1{\alpha_1z+\beta_1-\cfrac1{\alpha_2z+\beta_2-\cfrac1{\alpha_3z+\beta_3-\cfrac1{\ddots-\cfrac1{\alpha_rz+\beta_r}}}}}
\end{equation*}
and $m$ $(0\leq m\leq r)$ is the number of negative
coefficients $\alpha_j$ $(j=1,2,\ldots,r)$, then
\begin{equation}\label{index.via.J-fraction}
\IndC(R)=r-2m.
\end{equation}
\end{theorem}
\proof
From~\eqref{signum.of.leading.coeffs.quotients.determinants.relations.Sturm.regular}
we get that $\alpha_j$ is negative if and only if there is a sign
change in the sequence $\{D_{j-1}(R),D_{j}(R)\}$. Consequently,
$m=\SCF(D_0(R),D_1(R),D_2(R),\ldots,D_r(R))$,
and~\eqref{index.via.J-fraction} follows
from~\eqref{index.via.determinants.1}.
Also~\eqref{index.via.J-fraction} follows
from~\eqref{index.via.quotients.regular}.
\eop
\begin{theorem}\label{Th.index.via.Stieltjes.fraction}
If a real rational function~\eqref{basic.real.rational.function}
with exactly $r$ poles has a Stieltjes continued fraction
expansion~\eqref{Stieltjes.fraction.1} and if $m$ $(0\leq
m\leq r)$ is a number of negative coefficients $c_{2j-1}$
$(j=1,2,\ldots,r)$, then the Cauchy index $\IndC(R)$ can be found
by the formula~\eqref{index.via.J-fraction}.
\end{theorem}
\proof
The formula~\eqref{index.via.J-fraction} follows
from~\eqref{odd.coeff.Stieltjes.fraction.main.formula}
and~\eqref{index.via.determinants.1}.
\eop

Using Theorem~\ref{Th.index.via.determinants}, we can also produce
formul\ae\ for calculating the Cauchy index of a rational function
on the intervals $(-\infty,0)$ and $(0,+\infty)$.
\begin{theorem}\label{Th.number.of.negative.positive.indices.of.real.rational.function.1}
Let a real rational function $R$ with exactly $r$ poles have a series
expansion~\eqref{basic.real.rational.function}. Then the indices
of the function $R$ on the intervals~$(0,+\infty)$
and~$(-\infty,0)$ can be found from the following formul\ae:
\hskip 2mm
%
%\begin{itemize}
%
%\item[]
$\bullet\;$ if $|R(0)|<\infty$, then
\begin{eqnarray}\label{index.on.positive.line.1}
\Ind\nolimits_{0}^{+\infty}(R)\; = \; r-\left[\SCF(1,D_1(R),D_2(R),\ldots,D_r(R))+\SCF(1,\widehat{D}_1(R),\widehat{D}_2(R),\ldots,\widehat{D}_r(R))\right],  \\
\label{index.on.negative.line.1}
\Ind\nolimits_{-\infty}^{0}(R)\; = \; \SCF(1,\widehat{D}_1(R),\widehat{D}_2(R),\ldots,\widehat{D}_r(R))-\SCF(1,D_1(R),D_2(R),\ldots,D_r(R)),  \qquad\quad
\end{eqnarray}
%
%\item[]
$\bullet\;$ if $z=0$ is a pole of $R$ of order $\nu$, then
\begin{eqnarray}\label{index.on.positive.line.2}
\Ind\nolimits_{0}^{+\infty}(R)\; = \; r-\dfrac{1+\sigma_1}2-\left[\SCF(1,D_1(R),D_2(R),\ldots,D_r(R))+\SCF(1,\widehat{D}_1(R),\widehat{D}_2(R),\ldots,\widehat{D}_{r-1}(R))\right], \quad \\
\label{index.on.negative.line.2}
\Ind\nolimits_{-\infty}^{0}(R)\; = \; \dfrac{1-\sigma_2}2+\SCF(1,\widehat{D}_1(R),\widehat{D}_2(R),\ldots,\widehat{D}_{r-1}(R))-\SCF(1,D_1(R),D_2(R),\ldots,D_r(R)),
\qquad \qquad \end{eqnarray}
where $\sigma_1=\sgn\left(\lim\limits_{z\to0}z^{\nu}R(z)\right)$,
and
$\sigma_2=\begin{cases}\;\; \sigma_1 & \text{if}\;\;\; \Ind_{0}R(z)\neq 0,\\
-\sigma_1 & \text{if}\;\;\; \Ind_{0}R(z)=0.
\end{cases}$
%\end{itemize}
\end{theorem}
\proof
At first, let the function $R$ have no pole at $0$, that is, suppose that
$|R(0)|<\infty$. Represent $R(z)$ in the form
\begin{equation}\label{Th.number.of.negative.positive.poles.of.real.rational.function.1.proof.0}
R(z)=R^{(1)}(z)+R^{(2)}(z)+G(z),
\end{equation}
where the function $G$ has no real poles of  odd order, that
is, where $\IndC(G)=0$, and
\begin{equation}\label{Th.number.of.negative.positive.poles.of.real.rational.function.1.proof.1}
\begin{array}{lcl}
R^{(1)}(z)&=& \; \; \displaystyle \sum^{m_{-}}_{j=1}\left(\dfrac
{A_{l_j}^{(j)}}{(z+\omega_j)^{l_j}}+\dfrac
{A^{(j)}_{l_j-1}}{(z+\omega_j)^{l_j-1}}+\cdots+\dfrac{A^{(j)}_1}{z+\omega_j}\right),\\ \\
R^{(2)}(z)&=& \displaystyle \sum^{m_{-}+m_{+}}_{j=m_{-}+1}\left(\dfrac
{A_{l_j}^{(j)}}{(z-\omega_j)^{l_j}}+\dfrac
{A^{(j)}_{l_j-1}}{(z-\omega_j)^{l_j-1}}+\cdots+\dfrac{A^{(j)}_1}{z-\omega_j}\right),
\end{array}
\end{equation}
where all $l_j$ are odd  and $\omega_j>0$. Here $m_{-}$
( $m_{+}$) is the numbers of negative (positive) poles of
odd order of the function $R$. It follows from
Definitions~\ref{Def.Cauchy.index.at.pole}
and~\ref{Def.Cauchy.index.on.interval} and from
the~formul\ae~\eqref{Th.number.of.negative.positive.poles.of.real.rational.function.1.proof.0}--\eqref{Th.number.of.negative.positive.poles.of.real.rational.function.1.proof.1}  that
\begin{equation}\label{Th.number.of.negative.positive.poles.of.real.rational.function.1.proof.1.1}
\begin{array}{rclclcl}
\displaystyle\Ind\nolimits_{-\infty}^{0}(R) & = & \Ind\nolimits_{-\infty}^{0}(R^{(1)})&=&\sum\limits^{m_{-}}_{j=1}\sgn\left(A_{l_j}^{(j)}\right),\\
 \\
\displaystyle\Ind\nolimits_{0}^{+\infty}(R) &=
&\Ind\nolimits_{0}^{+\infty}(R^{(2)})&=&\sum\limits^{m_{-}+m_{+}}_{j=m_{-}+1}\sgn\left(A_{l_j}^{(j)}\right).
\end{array}
\end{equation}

Now consider the function $z\mapsto F(z)=zR(z)$.
From~\eqref{basic.real.rational.function}
and~\eqref{Th.number.of.negative.positive.poles.of.real.rational.function.1.proof.0}
we have
\begin{equation}\label{Th.number.of.negative.positive.poles.of.real.rational.function.1.proof.2}
F(z)=zR^{(1)}(z)+zR^{(2)}(z)+zG(z)=s_{-1}z
+s_0+\frac{s_1}z+\frac{s_2}{z^2}+\frac{s_3}{z^3}+\cdots,
\end{equation}
where
\begin{equation}\label{Th.number.of.negative.positive.poles.of.real.rational.function.1.proof.2.5}
\begin{array}{lcl}
\displaystyle zR^{(1)}(z)&=&-\sum\limits^{m_{-}}_{j=1}\left(\frac
{\omega_jA_{l_j}^{(j)}}{(z+\omega_j)^{l_j}}+\frac
{\omega_jA^{(j)}_{l_j-1}-A_{l_j}^{(j)}}{(z+\omega_j)^{l_j-1}}+\cdots+\dfrac{\omega_jA^{(j)}_1-A_{2}^{(j)}}{z+\omega_j}\right)+\sum\limits^{m_{-}}_{j=1}A^{(j)}_1,\\
 \\
\displaystyle
zR^{(2)}(z)&=&\sum\limits^{m_{-}+m_{+}}_{j=m_{-}+1}\left(\frac
{\omega_jA_{l_j}^{(j)}}{(z-\omega_j)^{l_j}}+\frac
{\omega_jA^{(j)}_{l_j-1}-A_{l_j}^{(j)}}{(z-\omega_j)^{l_j-1}}+\cdots+\dfrac{\omega_jA^{(j)}_1-A_{2}^{(j)}}{z-\omega_j}\right)+\sum\limits^{m_{-}+m_{+}}_{j=m_{-}+1}A^{(j)}_1,
\end{array}
\end{equation}
From Definitions~\ref{Def.Cauchy.index.at.pole}
and~\ref{Def.Cauchy.index.on.interval} and from
the~formul\ae~\eqref{Th.number.of.negative.positive.poles.of.real.rational.function.1.proof.1.1}--\eqref{Th.number.of.negative.positive.poles.of.real.rational.function.1.proof.2.5}
we obtain
\begin{eqnarray}\label{Th.number.of.negative.positive.poles.of.real.rational.function.1.proof.3}
\Ind\nolimits_{-\infty}^{0}(F)\;=\;\Ind\nolimits_{-\infty}^{0}(zR^{(1)})\;=\;
-\sum^{m_{-}}_{j=1}\sgn\left(\omega_jA_{l_j}^{(j)}\right)\;=\;
-\sum^{m_{-}}_{j=1}\sgn\left(A_{l_j}^{(j)}\right)\;=\;-\Ind\nolimits_{-\infty}^{0}(R), \qquad \; \\
\label{Th.number.of.negative.positive.poles.of.real.rational.function.1.proof.4}
\Ind\nolimits_{0}^{+\infty}(F)\;=\;\Ind\nolimits_{0}^{+\infty}(zR^{(2)})\;=\;
\sum^{m_{-}+m_{+}}_{j=m_{-}+1}\sgn\left(\omega_jA_{l_j}^{(j)}\right)\;=\;
\sum^{m_{-}+m_{+}}_{j=m_{-}+1}\sgn\left(A_{l_j}^{(j)}\right)\;=\;\Ind\nolimits_{0}^{+\infty}(R),\quad \;\;
\end{eqnarray}
since all $\omega_j$ are positive.

Theorems~\ref{Th.Cauchy.index.properties} and~\ref{Th.index.via.determinants}
and the formula\ae~\eqref{Th.number.of.negative.positive.poles.of.real.rational.function.1.proof.3}--\eqref{Th.number.of.negative.positive.poles.of.real.rational.function.1.proof.4} yield
\begin{equation*}\label{Th.number.of.negative.positive.poles.of.real.rational.function.1.proof.5}
\begin{array}{lcl}
\IndC(R)&=&\Ind\nolimits_{-\infty}^{0}(R)+\Ind\nolimits_{0}^{+\infty}(R)=r-2\SCF(1,D_1(R),D_2(R),\ldots,D_r(R)),\\
 \\
\IndC(F)&=&\Ind\nolimits_{-\infty}^{0}(F)+\Ind\nolimits_{0}^{+\infty}(F)=
-\Ind\nolimits_{-\infty}^{0}(R)+\Ind\nolimits_{0}^{+\infty}(R)\;=\\
 \\
& =& r-2\SCF(1,D_1(F),D_2(F),\ldots,D_r(F))=
r-2\SCF(1,\widehat{D}_1(R),\widehat{D}_2(R),\ldots,\widehat{D}_r(R)),
\end{array}
\end{equation*}
These formul\ae\
imply~\eqref{index.on.positive.line.1}--\eqref{index.on.negative.line.1}.

If the function $R$ has a pole of order $\nu\,(\geq 1)$ at
$0$, that is, if $R(0)=\infty$, then instead
of~\eqref{Th.number.of.negative.positive.poles.of.real.rational.function.1.proof.0}
and~\eqref{Th.number.of.negative.positive.poles.of.real.rational.function.1.proof.2}
we obtain
\begin{equation*}\label{Th.number.of.negative.positive.poles.of.real.rational.function.1.proof.7}
\begin{array}{lcl}
R(z)&=&R^{(1)}(z)+R^{(2)}(z)+\dfrac{C}{z^{\nu}}+G(z),\\
 \\
F(z)&=&zR^{(1)}(z)+zR^{(2)}(z)+\dfrac{C}{z^{\nu-1}}+zG(z),
\end{array}
\end{equation*}
where $C=\lim\limits_{z\to0}z^{\nu}R(z)$, the functions $R^{(1)}$
and $R^{(2)}$ are the same as
in~\eqref{Th.number.of.negative.positive.poles.of.real.rational.function.1.proof.1},
and $\IndC(G)=0$. The
formul\ae~\eqref{Th.number.of.negative.positive.poles.of.real.rational.function.1.proof.3}--\eqref{Th.number.of.negative.positive.poles.of.real.rational.function.1.proof.4}
remain the same.

If $\nu$ is  odd, then $\Ind_{0}(R)=\sigma_1\neq 0$ but
$\Ind_{0}(F)=0$, $D_r(F)=\widehat{D}_r(R)=0$ (see
Corollary~\ref{corol.zero.pole}), and
\begin{equation}\label{Th.number.of.negative.positive.poles.of.real.rational.function.1.proof.8}
\begin{array}{lcl}
\IndC(R)&=&\Ind\nolimits_{-\infty}^{0}(R)+\Ind\nolimits_{0}^{+\infty}(R)+\sigma_1=r-2\SCF(1,D_1(R),D_2(R),\ldots,D_r(R)),\\[1.7mm]
\IndC(F)&=&\Ind\nolimits_{-\infty}^{0}(F)+\Ind\nolimits_{0}^{+\infty}(F)=
-\Ind\nolimits_{-\infty}^{0}(R)+\Ind\nolimits_{0}^{+\infty}(R)\;=\\[1.7mm]
&=& r-1-2\SCF(1,D_1(F),D_2(F),\ldots,D_{r-1}(F))\;=\\[1mm]
&=& r-1-2\SCF(1,\widehat{D}_1(R),\widehat{D}_2(R),\ldots,\widehat{D}_{r-1}(R)).
\end{array}
\end{equation}
If $\nu$ is even, then $\Ind_{0}(R)=0$ but
$\Ind_{0}(F)=\sigma_1\neq0$, $D_r(F)=\widehat{D}_r(R)=0$, and
\begin{equation}\label{Th.number.of.negative.positive.poles.of.real.rational.function.1.proof.9}
\begin{array}{lcl}
\IndC(R)&=&\Ind\nolimits_{-\infty}^{0}(R)+\Ind\nolimits_{0}^{+\infty}(R)=r-2\SCF(1,D_1(R),D_2(R),\ldots,D_r(R)),\\[1.7mm]
\IndC(F)&=&\Ind\nolimits_{-\infty}^{0}(F)+\Ind\nolimits_{0}^{+\infty}(F)+\sigma_1=
-\Ind\nolimits_{-\infty}^{0}(R)+\Ind\nolimits_{0}^{+\infty}(R)\;=\\[1.7mm]
&=& r-1-2\SCF(1,D_1(F),D_2(F),\ldots,D_{r-1}(F))\;=\\[1mm]
&=& r-1-2\SCF(1,\widehat{D}_1(R),\widehat{D}_2(R),\ldots,\widehat{D}_{r-1}(R)).
\end{array}
\end{equation}
The
formul\ae~\eqref{Th.number.of.negative.positive.poles.of.real.rational.function.1.proof.8}--\eqref{Th.number.of.negative.positive.poles.of.real.rational.function.1.proof.9}
give~\eqref{index.on.positive.line.2}--\eqref{index.on.negative.line.2}.
\eop

If a rational function $R$ has a Stieltjes fraction expansion, then using
Theorems~\ref{Th.number.of.negative.positive.indices.of.real.rational.function.1}
and formul\ae~\eqref{even.coeff.Stieltjes.fraction.main.formula},
we can establish relations between $\Ind_{-\infty}^{0}(R)$,
$\Ind_{0}^{+\infty}(R)$ and the signs of the coefficients in the
Stieltjes fraction  of $R$.
\begin{theorem}\label{Th.number.of.negative.positive.indices.of.real.rational.function.via.S-fraction}
Suppose that a real rational function $R$ with exactly $r$ poles has a Stieltjes
continued fraction expansion~\eqref{Stieltjes.fraction.1}.  Then the indices of
the function $R$ on the intervals~$(0,+\infty)$ and~$(-\infty,0)$
can be found by the following formul\ae:
\begin{itemize}
\item[] if $|R(0)|<\infty$, then
\begin{eqnarray}\label{index.on.positive.line.1.via.S-fraction}
\Ind\nolimits_{0}^{+\infty}(R) & = & n_{\text{e}}-n_{\text{o}}, \\
\label{index.on.negative.line.1.via.S-fraction}
\Ind\nolimits_{-\infty}^{0}(R) & = & r-\left[n_{\text{o}}+n_{\text{e}}\right],
\end{eqnarray}
where $n_{\text{o}}$ is the number of negative coefficients
$c_{2i-1}$, $i=1,2,\ldots,r$, and  $n_{\text{e}}$ is the number
of negative coefficients $c_{2i}$, $i=1,2,\ldots,r$.
\item[] if $z=0$ is a pole of $R$ of order $\nu$, then
\begin{eqnarray}\label{index.on.positive.line.2.via.S-fraction}
\Ind\nolimits_{0}^{+\infty}(R) & = & \dfrac{1-\delta_1}2+n_{\text{e}}-n_{\text{o}},  \\
\label{index.on.negative.line.2.via.S-fraction}
\Ind\nolimits_{-\infty}^{0}(R) & = & r-\dfrac{1+\delta_2}2-\left[n_{\text{o}}+n_{\text{e}}\right],
\end{eqnarray}
where $n_{\text{o}}$ is the number of negative coefficients
$c_{2i-1}$, $i=1,2,\ldots,r$, $n_{\text{e}}$ is the number of
negative coefficients $c_{2i}$, $i=1,2,\ldots,r-1$,
$\delta_1=\sgn\left(\lim\limits_{z\to0}z^{\nu}R(z)\right)$, and
$\delta_2=\begin{cases} \;\; \delta_1 & \text{if}\;\;\; \Ind_{0}(R)\neq 0,\\
-\delta_1 & \text{if}\;\;\; \Ind_{0}(R)=0.
\end{cases}$
\end{itemize}
\end{theorem}
\proof
From the
formul\ae~\eqref{even.coeff.Stieltjes.fraction.main.formula}--\eqref{odd.coeff.Stieltjes.fraction.main.formula}
we obtain
\begin{equation*}\label{Th.number.of.negative.positive.indices.of.real.rational.function.via.S-fraction.proof.1}
n_{\text{o}}=\SCF(1,D_1(R),D_2(R),\ldots,D_r(R)),
\end{equation*}
\begin{equation*}\label{Th.number.of.negative.positive.indices.of.real.rational.function.via.S-fraction.proof.2}
n_{\text{e}}=\SCF(1,\widehat{D}_1(R),\widehat{D}_2(R),\ldots,\widehat{D}_{k}(R)),
\end{equation*}
where $k=r$ if $|R(0)|<\infty$, and $k=r-1$ if $R(0)=\infty$.
The assertion of the theorem follows from these formul\ae\ and from
Theorem~\ref{Th.number.of.negative.positive.indices.of.real.rational.function.1}.
\eop

%
%\newpage
%
\setcounter{equation}{0}

%\setcounter{notezz}{0}

%%%%%%%%%%%%%%%%%%%%%%%%%%%%%%%%%%%%%%%%%%%%%%%%%%%%%%%%%%%%%%%%%%%
\section{\label{s:R-functions}Rational functions mapping the upper
half-plane to the lower half-plane}
%%%%%%%%%%%%%%%%%%%%%%%%%%%%%%%%%%%%%%%%%%%%%%%%%%%%%%%%%%%%%%%%%%%

%%%%%%%%%%%%%%%%%%%%%%%%%%%%%%%%%%%%%%%%%%%%%%%%%%%%%%%%%%%%%%%%%%%
\subsection{\label{s:R-functions.general.theory}General theory}
%%%%%%%%%%%%%%%%%%%%%%%%%%%%%%%%%%%%%%%%%%%%%%%%%%%%%%%%%%%%%%%%%%%

Now we specialize general properties of complex and real rational
functions stated in the previous sections to the following very
important class of functions:

\begin{definition}\label{def.S-function}
A rational function $R$ is called an \textit{$R$-function of
negative type} (respectively, \textit{positive type}) if
it maps the upper half-plane of the complex plane to the lower
half-plane (respectively, to itself):
\begin{equation*}\label{def.S-function.negative.type}
\Im z>0\Longrightarrow\Im
R(z)<0\quad\text{---}\quad\text{\emph{negative type}};
\end{equation*}
\begin{equation*}\label{def.S-function.positive.type}
\Im z>0\Longrightarrow\Im
R(z)>0\quad\text{---}\quad\text{\emph{positive type}}.
\end{equation*}
\end{definition}
The name \textit{R-function} appears first in the works of
M.G.\,Krein and his progeny in connection with the theory of Stieltjes string and Stieltjes
continued fractions (see, for
example,~\cite{GantKrein,{KreinKatz}}). Below we discuss
several well-known and some new relationships between
$R$-functions and continued fractions of Stieltjes type
(see Definition~\ref{def.Stieltjes.fraction}). By now, these functions,
as well as their meromorphic analogues, have been considered by many authors
and have acquired various names. For instance, these functions are called
\emph{strongly real functions} in the monograph~\cite{Sheil-Small} due
to their property to take real values \textit{only} for real values of the
argument (more general and detailed consideration can be found
in~\cite{CM}, see also
Theorem~\ref{Th.R-function.general.properties} below).

\begin{remark}\label{remark.3.1}
In the sequel, we will deal with \textit{R}-functions of negative
type only. But if an \textit{R}-function $F$ is of negative type,
then the function $-F$ is evidently an \textit{R}-function of
positive type. Hence all results obtained for \textit{R}-functions
of negative type can be easy reformulated for \textit{R}-functions
of  positive type.
\end{remark}

At first, let us prove one necessary condition for a rational
function to be an \textit{R}-function.
\begin{lemma}\label{lem.R.func.necessary.condition}
If a real rational function
\begin{equation*}\label{rat.func.R.func.necessary.condition}
z \mapsto R(z)=\dfrac{q(z)}{p(z)}
\end{equation*}
is an \textit{R}-function of negative type, where $p$ and $q$ are
real polynomials, then
\begin{equation}\label{rat.func.R.func.necessary.condition.1}
|\deg p-\deg q|\leq1,
\end{equation}
and
\begin{equation}\label{rat.func.R.func.necessary.condition.2}
R(z)=-\alpha z+\beta+\widetilde{R}(z),\qquad
\alpha\geq0,\,\,\,\beta\in\mathbb{R},\,\,\,\widetilde{R}(\infty)=0.
\end{equation}
If $\alpha=0$, then $\widetilde{R}$ is an \textit{R}-function of
negative type.
\end{lemma}
\proof
Indeed,
\begin{equation*}\label{lem.R.func.necessary.condition.proof.1}
R(z)=h(z)+\widetilde{R}(z),\qquad\widetilde{R}(\infty)=0,
\end{equation*}
where $h(z)=cz^j+\cdots$ is a real polynomial of some degree $j$.
Obviously, for $z\in\mathbb{C}$ such that $\Im z>0$ and sufficiently
large, we have
\begin{equation*}\label{lem.R.func.necessary.condition.proof.2}
\sgn\left(\Im R(z)\right)=\sgn\left(\Im
h(z)\right)=\sgn(c\sin(j\varphi)),
\end{equation*}
where $\varphi=\arg z$. If $j\geq2$, then $\Im R(z)$  takes
both positive and negative values for some $z$ from the upper
complex half-plane, say, for $\arg z=\dfrac{\pi}{2j}$ and $\arg
z=\dfrac{3\pi}{2j}$. Therefore, the degree of $h$ is necessarily
at most $1$: $h(z)=-\alpha z+\beta$, where
$\alpha,\beta\in\mathbb{R}$. But $\Im h(z)=-\alpha\Im z$, thus,
$\alpha$ must be nonnegative. Moreover, if $\alpha=0$, then $\Im
R(z)=\Im \widetilde{R}(z)<0$ for $z$ such that $\Im z>0$ , and
$\widetilde{R}(z)$ is an \textit{R}-function of negative type.

Thus, if $\deg q>\deg p$ and the function $R=\dfrac{q}{p}$ is an
\textit{R}-function of negative type, then
\begin{equation}\label{lem.R.func.necessary.condition.proof.3}
\deg q\leq\deg p+1.
\end{equation}
Let $\deg q<\deg p$ and let the function~$R=\dfrac{q}{p}$ be an
\textit{R}-function of negative type. Then the
function~$-\dfrac{p}{q}$ is an \textit{R}-function of negative
type too and, therefore,
\begin{equation}\label{lem.R.func.necessary.condition.proof.4}
\deg q\leq\deg p+1.
\end{equation}
From~\eqref{lem.R.func.necessary.condition.proof.3}--\eqref{lem.R.func.necessary.condition.proof.4}
we obtain~\eqref{rat.func.R.func.necessary.condition.1}.
\eop

The following theorem lists the most important properties of
\textit{R}-functions. Parts of this theorem can be found
in~\cite{Pick,{KreinNaimark},{CM},{Gantmakher},{Atkinson},Atkinson_rus,
{Barkovsky.2},{Sheil-Small}}.

\begin{theorem}\label{Th.R-function.general.properties}
Let $p$ and $q$ be real and coprime\footnote{This condition is
introduced for simplicity and just means that the number of poles
of the function $R$ equals the number of zeros of the polynomial
$p$.} polynomials
satisfying~\eqref{rat.func.R.func.necessary.condition.1}. For the
real rational function
\begin{equation*}\label{Rational.function.for.R-functions}
z \mapsto R(z)=\dfrac{q(z)}{p(z)}
\end{equation*}
with exactly $n=\deg p$ poles, the following conditions are
equivalent:
\begin{itemize}
\item[$1)$] $R$ is an \textit{R}-function of negative type:
\begin{equation}\label{R-function.negative.type.condition}
\Im z>0\Rightarrow\Im R(z)<0;
\end{equation}

\item[$2)$] The function $R$ can
be represented in the form
\begin{equation}\label{Mittag.Leffler.1}
\displaystyle R(z)=-\alpha z+\beta+\sum^{n}_{j=1}\frac
{\gamma_j}{z-\omega_j},\qquad\alpha\geq0,\,\beta\in\mathbb{R},\,\omega_j\in\mathbb{R},\quad
j=1,2,\ldots,n,
\end{equation}
where
\begin{equation}\label{Mittag.Leffler.1.poles.formula}
\gamma_j=\dfrac{q(\omega_j)}{p'(\omega_j)}>0,\quad j=1,\ldots,n;
\end{equation}
\item[$3)$] The index of the function $R$ is maximal:
\begin{equation}\label{Index.Cauchy.R-functions}
\IndPR(R)=\max\left(\deg p,\deg q\right);
\end{equation}
\item[$4)$]
The function $R$ has a $J$-fraction expansion:
\begin{equation}\label{R-function.J-fraction}
R(z)=-\alpha z+\beta+
\dfrac1{\alpha_1z+\beta_1-\cfrac1{\alpha_2z+\beta_2-\cfrac1{\alpha_3z+\beta_3-\cfrac1{\ddots-\cfrac1{\alpha_nz+\beta_n}}}}},
\end{equation}
where $\alpha_j>0,\beta_j\in\mathbb{R}$, $\alpha\geq0$ and
$\beta\in\mathbb{R}$;
\item[$5)$] The polynomials $p$ and $q$ have only real roots and satisfy the
inequality
\begin{equation}\label{R-functions.decreasing.condition}
p(\omega)q'(\omega)-p'(\omega)q(\omega)<0\quad\text{for all}\quad
\omega\in\mathbb{R};
\end{equation}
\item[$6)$] The roots of the polynomials $p$ and
$q$ are real, simple and interlacing, that is, between any two
consecutive roots of one of the polynomials there is exactly one root
of the other polynomial, and
\begin{equation}\label{R-functions.normalization}
\exists\; \omega\in\mathbb{R}:\quad
p(\omega)q'(\omega)-p'(\omega)q(\omega)<0;
\end{equation}
\item[$7)$] The polynomial
\begin{equation}\label{R-functions.linear.combination}
z \mapsto g(z)=\lambda p(z)+\mu q(z),
\end{equation}
has only real zeros for any real $\lambda$ and $\mu$,
$\lambda^2+\mu^2\neq0$, and the
condition~\eqref{R-functions.normalization} is satisfied;
\item[$8)$] The function $R(z)$ has real values only for real~$z$:
\begin{equation}\label{R-functions.pure.reality.condition}
R(z)\in\mathbb{R}\quad\Longrightarrow\quad z\in\mathbb{R},
\end{equation}
and
\begin{equation}\label{R-functions.normalization.2}
\exists\;\omega\in\mathbb{R}:\quad-\infty<R'(\omega)<0;
\end{equation}
\item[$9)$]
Let the function $R$ be represented by the series
\begin{equation}\label{R-function.series}
R(z)=-\alpha
z+\beta+\frac{s_0}z+\frac{s_1}{z^2}+\frac{s_2}{z^3}+\cdots
\end{equation}
with $\alpha\geq0$ and $\beta\in\mathbb{R}$. The following
inequalities hold
\begin{equation}\label{Grommer.condition.1}
D_j(R)>0,\quad j=1,2,\dots,n,
\end{equation}
where the determinants $D_j(R)$ are defined
in~\eqref{Hankel.determinants.1}.
\end{itemize}
\end{theorem}

The equivalence of conditions $1)$  and $2)$ is usually called the
Chebotarev theorem~\cite{Cheb,{CM}}. The equivalence of conditions
$1)$ and $9)$ in case of meromorphic functions is the famous
Grommer theorem~\cite{Grommer,{AkhiezerKrein},{CM}}. The
equivalence between $2)$ and $4)$ was also proved by
Grommer~\cite{Grommer} (see also~\cite{Wall,{Lange}}). Finally,
the equivalence of $1)$ and~$6)$ is a modification of the famous
Hermite-Biehler theorem (for
example,~\cite{KreinNaimark,{Gantmakher}}).

\proof The scheme of our proof is as follows:
\begin{equation*}\label{Th.R-function.general.properties.scheme.of.proof}
\begin{matrix}
1)      &\Longrightarrow &2)        &\Longrightarrow &3)        &\Longrightarrow &4)        &\Longleftrightarrow              &9) \\
\Uparrow &{}                  &{}                  &{}                  &{}                  &{}                  &\Downarrow &{}              &{} \\
8)      &\Longleftarrow      &7) &\Longleftarrow &6)
 &\Longleftarrow &5)&{}              &{}
\end{matrix}
\end{equation*}

\noindent $1)\Longrightarrow2)$\ \  Let the function $R$
satisfy~\eqref{R-function.negative.type.condition}.
Lemma~\ref{lem.R.func.necessary.condition} guarantees that $R$
can be represented in the
form~\eqref{rat.func.R.func.necessary.condition.2}. Thus, we have
to prove that the function $\widetilde{R}$ in the
representation~\eqref{rat.func.R.func.necessary.condition.2} has
the following form
\begin{equation}\label{Th.R-function.general.properties.proof.0.1}
\widetilde{R}(z)=\dfrac{\tilde{q}(z)}{p(z)}=\sum^{n}_{j=1}\frac
{\gamma_j}{z-\omega_j},\qquad\omega_j\in\mathbb{R},\quad
j=1,2,\ldots,n,
\end{equation}
where\footnote{Here $\alpha$ and $\beta$ are the numbers from the
representation~\eqref{rat.func.R.func.necessary.condition.2}.}
$\tilde{q}(z)=q(z)-(\alpha z+\beta)p(z)$ and
\begin{equation}\label{Th.R-function.general.properties.proof.0.2}
\gamma_j=\dfrac{\tilde{q}(\omega_j)}{p'(\omega_j)}=\dfrac{q(\omega_j)}{p'(\omega_j)}>0,\qquad
j=1,\ldots,n.
\end{equation}

At first, assume that the function $R$ and, therefore, the
function $\widetilde{R}$ have a nonreal pole $\lambda$ of some
multiplicity $j(\geq1)$. Then for $z=\lambda+\varepsilon$,
$|\varepsilon|\to0$, we have
\begin{equation*}\label{Th.R-function.general.properties.proof.0.3}
\sgn\left(\Im R(z)\right)=\sgn\left(\Im
\widetilde{R}(z)\right)=\sgn\left(\Im\dfrac{c}{\varepsilon^j}\right)=\sgn\left(\sin(\arg
c-j\cdot\arg\varepsilon)\right).
\end{equation*}
From these equalities one can see that we can obviously choose
such complex numbers $\varepsilon_1$ and $\varepsilon_1$ that
$z_1=\lambda+\varepsilon_1$ and $z_2=\lambda+\varepsilon_2$ are
from the upper half plane, but $\sgn\left(\Im
R(z_1)\right)=-\sgn\left(\Im R(z_2)\right)$. It contradicts
with~\eqref{R-function.negative.type.condition}. Thus, the
functions $R$ and $\widetilde{R}$ have only real poles.

Now let us assume that $R$ and $\widetilde{R}$ have a real pole
$\mu$ of some multiplicity $j\geq2$. And let
$z=\mu+\varepsilon$. Since $\mu$ is a real number, we have $\Im
z=\Im\varepsilon$. If we take $z$ from the upper half-plane of
complex plane and sufficiently close to $\mu$, that is,
$0<\arg\varepsilon<\pi$ and $|\varepsilon|\to0$, then
\begin{equation*}\label{Th.R-function.general.properties.proof.0.4}
\sgn\left(\Im R(z)\right)=\sgn\left(\Im
\widetilde{R}(z)\right)=\sgn\left(\Im\dfrac{A}{\varepsilon^j}\right)=-\sgn\left(A\sin(j\cdot\arg\varepsilon)\right).
\end{equation*}
These equalities show that if we choose $\varepsilon_1$ and
$\varepsilon_2$ such that $\arg\varepsilon_1=\dfrac{\pi}{2j}$ and
$\arg\varepsilon_2=\dfrac{3\pi}{2j}$, then $z_1=\mu+\varepsilon_1$
and $z_2=\mu+\varepsilon_2$ both belong to the upper half-plane,
if $j\geq0$, but $\sgn\left(\Im R(z_1)\right)=-\sgn\left(\Im
R(z_2)\right)$. Thus, the function $R$ (and the function
$\widetilde{R}$) has only simple real poles, and $\widetilde{R}$
satisfies~\eqref{Th.R-function.general.properties.proof.0.1},
where the residues $\gamma_j$ are determined by the formula
\begin{equation*}\label{Th.R-function.general.properties.proof.0.5}
\gamma_j=\lim_{z\to\omega_j}\left(R(z)(z-a)\right)=\lim_{z\to\omega_j}\dfrac{\tilde{q}(z)(z-a)}{p(z)}=\dfrac{\tilde{q}(\omega_j)}{p'(\omega_j)}=\dfrac{q(\omega_j)}{p'(\omega_j)},\qquad
j=1,\ldots,n.
\end{equation*}

For $z$ sufficiently close to a pole $\omega_j$ $(j=1,2,\ldots,n)$, we have
\begin{equation*}\label{Th.R-function.general.properties.proof.0.6}
\sgn\left(\Im R(z)\right)=\sgn\left(\Im
\widetilde{R}(z)\right)=\sgn\left(\Im\dfrac{\gamma_j}{z-\omega_j}\right)=-\sgn\left(\gamma_j\Im
z\right).
\end{equation*}
These equalities combined
with~\eqref{R-function.negative.type.condition} yield the
positivity of all $\gamma_j$. Consequently, the function
$\widetilde{R}$
satisfies~\eqref{Th.R-function.general.properties.proof.0.1}--\eqref{Th.R-function.general.properties.proof.0.2},
but $R$ has the
representation~\eqref{Mittag.Leffler.1}--\eqref{Mittag.Leffler.1.poles.formula}.

\vspace{2mm}

\noindent $2)\Longrightarrow3)$\ \ Let the function $R$
satisfy~\eqref{Mittag.Leffler.1}. If $\deg p\geq\deg q$, then
$\alpha=0$ and $\max(\deg p,\deg q)=\deg p=n$. In this case, $R$
has no pole at~$\infty$, so from~\eqref{Cauchy.index.at.odd.pole}
we obtain
\begin{equation*}\label{Th.R-function.general.properties.proof.1}
\IndPR(R)=\IndC(R)=n=\max(\deg p,\deg q).
\end{equation*}
If $\deg q=\deg p+1$, that is, $\alpha>0$, then $R$ has a pole
at~$\infty$, so according to~\eqref{Cauchy.index.at.odd.pole},
and~\eqref{Cauchy.index.at.infty}--\eqref{Cauchy.index.of.poly},
\begin{equation*}\label{Th.R-function.general.properties.proof.2}
\IndPR(R)=\sgn\alpha+\IndC(R)=1+n=\deg q=\max(\deg p,\deg q).
\end{equation*}

\vspace{2mm}

\noindent $3)\Longrightarrow4)$\ \ Since the function $R$
satisfies~\eqref{rat.func.R.func.necessary.condition.1}, it can be
represented as follows:
\begin{equation}\label{Th.R-function.general.properties.proof.3}
R(z)=-\alpha z+\beta+\widetilde{R}(z),
\end{equation}
where $\alpha,\beta\in\mathbb{R}$ and $\widetilde{R}(\infty)=0$.
Then
from~\eqref{Cauchy.index.at.infty}--\eqref{Cauchy.index.of.functions.sum}
we have
\begin{equation}\label{Th.R-function.general.properties.proof.4}
\IndPR(R)=\sgn\alpha+\IndC(R)=\sgn\alpha+\IndC(\widetilde{R}),
\end{equation}
therefore, according to~\eqref{Cauchy.index.at.odd.pole}
and~\eqref{Cauchy.index.at.infty},
\begin{itemize}
\item[] if $\deg q=\deg p+1$, then
\begin{equation}\label{Th.R-function.general.properties.proof.5}
-n-1\leq\IndPR(R)\leq n+1,
\end{equation}
\item[] if $\deg q\leq\deg p$, then
\begin{equation}\label{Th.R-function.general.properties.proof.6}
-n\leq\IndPR(R)\leq n,
\end{equation}
\end{itemize}

The condition~\eqref{Index.Cauchy.R-functions} is equivalent to
\begin{equation}\label{Th.R-function.general.properties.proof.7}
\IndPR(R)=
\begin{cases}
n+1 & \text{if}\quad\deg q=\deg p+1;\\
n &  \text{if}\quad\deg q\leq\deg p.
\end{cases}
\end{equation}
From~\eqref{Th.R-function.general.properties.proof.4}--\eqref{Th.R-function.general.properties.proof.7}
we obtain that $\alpha\geq0$
in~\eqref{Th.R-function.general.properties.proof.3} and
$\IndC(R)=\IndC(\widetilde{R})=n$.

Let us expand the function $R$ into a continued
fraction~\eqref{continued.fraction.general.via.Sturm}. Then
Theorem~\ref{Th.index.via.quotients} yields
\begin{equation*}\label{Th.R-function.general.properties.proof.8}
\IndPR(\widetilde{R})=-\sum_{j=1}^{k}\IndI(q_j),
\end{equation*}
where $k\leq n$. Since $\IndPR(\widetilde{R})$ must be equal
to $n$, we see that $k=n$ and that all polynomials $q_j$
$(j=1,2,\ldots,n)$ are linear with positive leading coefficients:
\begin{equation}\label{Th.R-function.general.properties.proof.9}
q_j(z)=\alpha_j z+\beta_j,\qquad
\alpha_j>0,\;\;\beta_j\in\mathbb{R},\quad j=1,2,\ldots,n.
\end{equation}
This follows from the fact that $\displaystyle\sum_{j=1}^{k}\deg
q_j=n$ and from~\eqref{Cauchy.index.of.poly}.

Thus, if the function $R$
satisfies~\eqref{Index.Cauchy.R-functions}, then we obtain
from~\eqref{Th.R-function.general.properties.proof.3}
(where~$\alpha$ must be nonnegative)
and~\eqref{Th.R-function.general.properties.proof.9}
that $R$ has a $J$-fraction
expansion~\eqref{R-function.J-fraction}.

\vspace{2mm}

\noindent $4)\Longrightarrow5)$\ \ If the function~$R(z)$ has a
$J$-fraction expansion~\eqref{R-function.J-fraction}, then
\begin{equation}\label{Th.R-function.general.properties.proof.10}
R(z)=\dfrac{q(z)}{p(z)}=-\alpha
z+\beta+\dfrac{f_1(z)}{f_0(z)},\quad\alpha\geq0,\;\;\;\beta\in\mathbb{R},
\end{equation}
where the function~$\dfrac{f_1}{f_0}$ has a $J$-fraction
expansion~\eqref{R-function.J-fraction} and vanishes at~$\infty$.
Therefore, starting from the polynomials $f_0$, $f_1$, we can
construct a sequence $f_0,f_1,\ldots,f_n$ by the Sturm algorithm:
\begin{equation}\label{Th.R-function.general.properties.proof.11}
f_{j-1}(z)=(\alpha_jz+\beta_j)f_j(z)-f_{j+1}(z),\quad\alpha_j>0,\;\;\;\beta_j\in\mathbb{R},\;\;\;
j=0,1,\ldots,n,
\end{equation}
where $f_{n+1}(z)\equiv0$, and $\gcd(f_0,f_1)=f_n(z)\equiv1$,
since the polynomials $p$ and $q$ (and, therefore, $f_0$ and
$f_1$) are coprime by assumption.
From~\eqref{Th.R-function.general.properties.proof.11} we have
that, for all $\omega\in\mathbb{R}$,
\begin{equation}\label{Th.R-function.general.properties.proof.12}
\begin{array}{l}
f_0(\omega)f'_1(\omega)-f'_0(\omega)f_1(\omega)=-\alpha_1f^2_1(\omega)+f_1(\omega)f'_2(\omega)-f'_1(\omega)f_2(\omega)\;=\\
 \\
=-\alpha_1f^2_1(\omega)-\alpha_2f^2_2(\omega)+f_2(\omega)f'_3(\omega)-f'_2(\omega)f_3(\omega)=\ldots=-\sum^{n}_{j=1}\alpha_jf^2_j(\omega)<0
\end{array}
\end{equation}
since all $\alpha_j>0$ and all $f^2_j(\omega)\geq0$ and
$f^2_n(\omega)>0$ as well.

Now from~\eqref{Th.R-function.general.properties.proof.10} we
obtain that $q(z)=(-\alpha z+\beta)f_0(z)+f_1(z)$ and $p=f_0$.
Thus,~\eqref{Th.R-function.general.properties.proof.12} implies
\begin{equation*}\label{Th.R-function.general.properties.proof.13}
p(\omega)q'(\omega)-p'(\omega)q(\omega)=-\alpha
f_0^2(\omega)+f_0(\omega)f'_1(\omega)-f'_0(\omega)f_1(\omega)<0,
\end{equation*}
for any $\omega\in\mathbb{R}$, as required.

Let $z$ be a complex number such that $\Im z\neq 0$. It is easy to
see that
\begin{equation*}\label{Th.R-function.general.properties.proof.13.3}
\sgn\left(\Im\dfrac{-1}{\alpha z
+\beta-f(z)}\right)=\alpha\sgn(\Im z) -\sgn\left(\Im
f(z)\right),\quad\alpha\geq0,\,\beta\in\mathbb{R},
\end{equation*}
where $f$ is any function of a complex variable. Thus, if the
function~$R$ has a $J$-fraction
expansion~\eqref{R-function.J-fraction}, then we obtain
\begin{equation*}\label{Th.R-function.general.properties.proof.13.5}
\sgn\left(\Im
R(z)\right)=\sgn\left(\Im\dfrac{-1}{R(z)}\right)=-\sgn\left[\Im
z\left(\alpha+\sum_{i=1}^{r}\alpha_i\right)\right]\neq0,
\end{equation*}
whenever $\Im z\neq0$. Thus, $R$ necessarily has only
\textit{real} zeros and poles\footnote{At the same time, this
proves the implication $4)\Longrightarrow1)$.}.

\vspace{2mm}

\noindent $5)\Longrightarrow6)$\ \ If the polynomials $p$ and $q$ are
real-rooted and satisfy the
inequality~\eqref{R-functions.decreasing.condition}, then the
condition~\eqref{R-functions.normalization} holds for them.
From~\eqref{R-functions.decreasing.condition} it also immediately
follows that all zeros of $p$ and $q$ are simple. Otherwise, there
must be a \textit{real} number $\omega$ such that
$p(\omega)q'(\omega)-p'(\omega)q(\omega)=0$, which
contradicts~\eqref{R-functions.decreasing.condition}.

Let real numbers $\omega_1,\omega_2$ $(\omega_1<\omega_2)$ be two
consecutive (simple) zeros of $p$. By Rolle's theorem,
we have
\begin{equation}\label{Th.R-function.general.properties.proof.14}
p'(\omega_1)p'(\omega_2)<0.
\end{equation}
Then from~\eqref{R-functions.decreasing.condition} it follows that
$p'(\omega_1)q(\omega_1)<0$ and $p'(\omega_2)q(\omega_2)<0$.
Together with~\eqref{Th.R-function.general.properties.proof.14},
these inequalities imply $q(\omega_1)q(\omega_2)<0$. Thus, in the
interval $(\omega_1,\omega_2)$, the polynomial $q$ has an odd
number of (simple) zeros. In the same way, one can prove that
between any consecutive zeros of $q$, there is an odd number of
zeros of~$p$. Therefore, all zeros of $p$ and $q$ are simple, and
between any two consecutive zeros of one of the polynomials there
is only one zero of the other polynomial, as required.

\vspace{2mm}

\noindent $6)\Longrightarrow7)$\ \ Now let the polynomials $p$ and
$q$ have simple real interlacing zeros. If numbers $\omega_1$ and
$\omega_2$ $(\omega_1<\omega_2)$ are some two consecutive zeros of
the polynomial $p$, then by interlacing we have
$q(\omega_1)q(\omega_2)<0$ and the polynomial $g$
defined~\eqref{R-functions.linear.combination} satisfies
the~inequality $g(\omega_1)g(\omega_2)<0$. Therefore, all zeros of
the polynomial $g$ are real and simple and interlace the zeros of
$p$ since $|\deg p-\deg g|\leq1$.

\vspace{2mm}

\noindent $7)\Longrightarrow8)$\ \ Let the polynomial $g$
defined by~\eqref{R-functions.linear.combination} have only real zeros
and suppose, without loss of generality, that $\mu\neq0$. Then the
equation
\begin{equation}\label{Th.R-function.general.properties.proof.15}
p(z)\left(\mu\dfrac{q(z)}{p(z)}+\lambda\right)=0
\end{equation}
has only real solutions for any real $\lambda$ and $\mu\neq0$.
Therefore, the function $R=\dfrac{q}{p}$ cannot take real values
for nonreal~$z$. Otherwise, the
equation~\eqref{Th.R-function.general.properties.proof.15} would
have nonreal solutions for some real $\lambda$ and $\mu$. Thus,
the function $R=\dfrac{q}{p}$
satisfies~\eqref{R-functions.pure.reality.condition}.

The condition~\eqref{R-functions.normalization}
implies~\eqref{R-functions.normalization.2}. Indeed, if a real
$\omega$ in~\eqref{R-functions.normalization} is not a zero
of~$p$, then
\begin{equation*}\label{Th.R-function.general.properties.proof.15.1}
R'(\omega)=\dfrac{p(\omega)q'(\omega)-p'(\omega)q(\omega)}{p^2(\omega)}<0.
\end{equation*}
If $\omega$ from~\eqref{R-functions.normalization} is a zero of
the polynomial $p$, then $\omega$ is a pole of the function $R'$.
However, \eqref{R-functions.normalization} shows that, for
sufficiently small $\varepsilon>0$,
\begin{equation*}\label{Th.R-function.general.properties.proof.15.2}
p(\omega+\varepsilon)q'(\omega+\varepsilon)-p'(\omega+\varepsilon)q(\omega+\varepsilon)<0,
\end{equation*}
since the function $pg'-p'q$ is continuous and preserves its sign
in some vicinity of the point $\omega$. Therefore,
$R'(\omega+\varepsilon)<0$, as required.

\vspace{2mm}

\noindent $8)\Longrightarrow1)$\ \ If the function $R$
satisfies~\eqref{R-functions.pure.reality.condition}, then it has
no complex zeros, since $0$ is a real value. Thus, if $\Im z>0$,
then $\Im R(z)\neq 0$, i.e., $R$ is an \textit{R}-function of
positive or negative type. Suppose that $\Im R(z)>0$ whenever
$\Im z>0$, that is, $R$ is an \textit{R}-function of positive
type. Then the function $F=-R=\dfrac{q}{p}$ is an
\textit{R}-function of negative type, i.e., it
satisfies~\eqref{R-function.negative.type.condition}. But we have
already proved that
$1)\Longrightarrow2)\Longrightarrow3)\Longrightarrow4)\Longrightarrow5)$,
therefore, we have
\begin{equation*}\label{Th.R-function.general.properties.proof.16}
p(\omega)q'(\omega)-p'(\omega)q(\omega)<0\quad\text{for
all}\quad\omega\in\mathbb{R}.
\end{equation*}
Consequently,
\begin{equation*}\label{Th.R-function.general.properties.proof.17}
F'(\omega)=\dfrac{p(\omega)q'(\omega)-p'(\omega)q(\omega)}{p^2(\omega)}<0,
\end{equation*}
for any \textit{real} $\omega$ such that $F'(\omega)$ exists. This
means that $R'(\omega)=-F'(\omega)>0$ for any \textit{real}
$\omega$ such that $R'(\omega)$ exists. This
contradicts~\eqref{R-functions.normalization.2}. Therefore, $R$ is
an \textit{R}-function of negative type.

\vspace{2mm}

\noindent $4)\Longleftrightarrow9)$\ \
Theorem~\ref{Th.J-fraction.criterion} guarantees that the function
$R$ has a $J$-fraction expansion~\eqref{R-function.J-fraction}
with all $\alpha_j$ nonzero and \textit{real}, $j=1,2,\ldots,n$,
if and only if the Hankel minors $D_j(R)$, $j=1,2,\ldots,n$, are
\textit{nonzero}. But for
$J$-fraction~\eqref{R-function.J-fraction}, the
formula~\eqref{continued.fraction.determinants.formula.Sturm.regular}
holds (see
also~\eqref{leading.coeffs.quotients.determinants.relations.Sturm.regular}).
This formula implies that the
inequalities~\eqref{Grommer.condition.1} hold if and only if all
$\alpha_j>0$ for $j=1,2,\ldots,n$. \eop

\begin{remark}\label{remark.3.2}
Comparing representations~\eqref{Mittag.Leffler.1}
and~\eqref{R-function.series}, one can obtain the following
formula:
\begin{equation}\label{Moment.formula}
s_i=\sum_{j=1}^{n}\gamma_j\omega_j^i,\qquad i=0,1,2,\ldots
\end{equation}
\end{remark}
\begin{remark}\label{remark.3.3}
From the proof of Theorem~\ref{Th.R-function.general.properties}
one can see that a sum of two \textit{R}-functions of negative
type is an \textit{R}-functions of negative type. Also if $R(z)$
is an \textit{R}-function of negative type, then the functions
$z\mapsto-R(-z)$ and $z\mapsto-\dfrac1{R(z)}$ are also
\textit{R}-functions of negative type.
\end{remark}

Using the equivalence of conditions $1)$ and $7)$ of
Theorem~\ref{Th.R-function.general.properties}, one can obtain the
following simple fact.

\begin{corol}\label{Corol.differentiation.of.R-functions}
Let $p$ and $q$ be real coprime polynomials
satisfying~\eqref{rat.func.R.func.necessary.condition.1}, $\deg
p\geq 2$. If the function $R=q/p$ is an $R$-function of
negative type, then the functions $R_j=q^{(j)}/p^{(j)}$,
$j=1,\ldots,\deg p-1$, are also $R$-functions of negative type.
\end{corol}
\proof%
It suffices to prove that the function $R_1$ is an
\textit{R}-function. First, consider the case $\deg q=\deg p-1$:
\begin{equation*}
p(z)=\sum\limits_{j=0}^na_jz^{n-j}, \qquad
q(z)=\sum\limits_{j=1}^{n}b_jz^{n-j-1},\quad\text{where}\quad
a_0>0,\ \ b_1\neq0,\ \ n\geq 2.
\end{equation*}
Then $R(z)\to 0$ as $z\to+\infty$, and, for sufficiently large
positive $z$, $\sgn R(z)=\sgn q(z)$. Suppose that $R$ is an
\textit{R}-function of negative type. By
Theorem~\ref{Th.R-function.general.properties}, the polynomial $g$
defined by~\eqref{R-functions.linear.combination} has only real
zeros. Therefore, the polynomial $g'$ also has only real zeros, so
the function $R_1=q'/p'$ is an \textit{R}-function of positive or
negative type according to
Theorem~\ref{Th.R-function.general.properties} and
Remark~\ref{remark.3.1}

Since $R$ is an \textit{R}-function of negative type, $R$ is
decreasing between its poles (this follows, for example,
from~\eqref{Mittag.Leffler.1}). Therefore, if a real number $\xi$
is the largest zero of the polynomial $p$, then $R$ must be
positive in the~interval~$(\xi,+\infty)$, so $R(z)\to+0$ as
$z\to+\infty$. Consequently, the polynomial $q(z)$ is positive for
sufficiently large positive $z$, i.e., $b_1>0$. Now it is easy to
see that
\begin{equation*}
p'(z)q''(z)-p''(z)q'(z)\sim-n(n-1)a_0b_1z^{2n-4}<0\quad\text{as}\quad
z\to\pm\infty
\end{equation*}
Hence there exists a real $\omega$ such that
$p'(\omega)q''(\omega)-p''(\omega)q'(\omega)<0$, so $R_1$ is an
\textit{R}-function of negative type by
Theorem~\ref{Th.R-function.general.properties}.

\vspace{2mm}

Let now $\deg q=\deg p$. Then
\begin{equation}\label{Corol.differentiation.of.R-functions.proof.1}
R(z)=\dfrac{q(z)}{p(z)}=%\dfrac{b_0}{a_0}
\beta+\dfrac{h(z)}{p(z)},\qquad
h(z)=q(z)-%\dfrac{b_0}{a_0}p(z),\qquad \deg h<\deg p,
\beta p(z),\qquad \deg h<\deg p,
\end{equation}
where $\beta\in\mathbb{R}$ and the function $h/p$ is an
\textit{R}-function of negative type by
Theorem~\ref{Th.R-function.general.properties}
(see~\eqref{Mittag.Leffler.1}). Therefore, the function $h'/p'$
is also an \textit{R}-function of negative type. But
from~\eqref{Corol.differentiation.of.R-functions.proof.1} it
follows that
\begin{equation*}\label{Corol.differentiation.of.R-functions.proof.2}
R_1(z)=\dfrac{q'(z)}{p'(z)}=%\dfrac{b_0}{a_0}
\beta+\dfrac{h'(z)}{p'(z)}.
\end{equation*}
Thus, the function $R_1$ is an \textit{R}-function of negative
type by Theorem~\ref{Th.R-function.general.properties}.

\vspace{2mm}

Finally, if $\deg q=\deg p+1$, then we may apply the previous
result to the functions $F=-p/q$ and $F_1=-p'/q'$.
\eop

This theorem and Theorem~\ref{Th.R-function.general.properties}
immediately imply the following result due to V.A.\,Markov
(see~\cite[Theorem~9, Chapter~1]{CM}).
\begin{theorem}[V.A.\,Markov]\label{Theorem.V.A.Markov}
If the zeros of two real polynomials $p$ and $q$ are simple, real
and interlacing, then the zeros of their derivatives $p'$ and $q'$
also are real, simple and interlacing.
\end{theorem}

If a rational function $R$ is an \textit{R}-function, then, using
Theorem~\ref{Th.number.of.negative.positive.indices.of.real.rational.function.1},
one can easily find the numbers of negative and positive poles of
this function.
\begin{theorem}\label{Th.number.of.negative.poles.of.R-frunction}
Let a rational function $R$ with exactly $r$ poles be an
\textit{R}-function of negative type and let $R$ have a series
expansion~\eqref{R-function.series}. Then the number $r_{-}$ of
negative poles of $R$ equals\footnote{Recall that the number
$\SCF(1,\widehat{D}_1(R),\widehat{D}_2(R),\ldots,\widehat{D}_k(R))$
of Frobenius sign changes must be calculated according to Frobenius
Rule~\ref{Th.Frobenius}.}
\begin{equation}\label{number.negative.poles.R-functions}
r_{-}=\SCF(1,\widehat{D}_1(R),\widehat{D}_2(R),\ldots,\widehat{D}_k(R)),
\end{equation}
where $k=r-1$, if $R(0)=\infty$, and $k=r$, if $|R(0)|<\infty$.
The~determinants $\widehat{D}_j(R)$ are defined
in~\eqref{Hankel.determinants.2}.
\end{theorem}
\proof In fact, Theorem~\ref{Th.R-function.general.properties}
states that $\SCF(1,D_1(R),D_2(R),\ldots,D_r(R))=0$  if (and only
if) $R$ is an \textit{R}-function
(see~\eqref{Grommer.condition.1}). Moreover, all poles of $R$ are
real and simple, and all residues at those poles are positive
(see~\eqref{Mittag.Leffler.1.poles.formula}), therefore we get
$r_{-}=\Ind_{-\infty}^{0}(R)$. Thus, the
formula~\eqref{number.negative.poles.R-functions} follows
from~\eqref{index.on.negative.line.1}--\eqref{index.on.negative.line.2},
where $\sigma_2=\sigma_1=1$. \eop

It is convenient for us to consider separately the extreme cases
of Theorem~\ref{Th.number.of.negative.poles.of.R-frunction}.
\begin{corol}\label{corol.R-function.negative.poles}
Let a rational function $R$ with exactly~$r$ poles be an
\textit{R}-function of negative type. All poles of $R$ are
negative if and only if
\begin{equation}\label{R-function.negative.poles.condition}
\widehat{D}_{j-1}(R)\widehat{D}_j(R)<0,\quad j=1,2,\ldots,r,
\end{equation}
where $\widehat{D}_{0}(R)\eqbd1$, and the determinants
$\widehat{D}_j(R)$ are defined by~\eqref{Hankel.determinants.2}.
\end{corol}
\proof Indeed,
Theorem~\ref{Th.number.of.negative.poles.of.R-frunction} implies
\begin{equation}\label{R-function.negative.poles.condition.1}
\SCF(1,\widehat{D}_1(R),\widehat{D}_2(R),\ldots,\widehat{D}_r(R))=r.
\end{equation}
According to the Frobenius rule~\eqref{Th.Frobenius.rule}, if
$\widehat{D}_j(R)=0$ for some $j\leq r-1$ but
$\widehat{D}_{j-1}(R)\neq0$, then $\sgn \widehat{D}_{j}(R)=\sgn
\widehat{D}_{j-1}(R)$, that is,
$\SCF(\widehat{D}_{j-1}(R),\widehat{D}_{j}(R))=0$. Therefore, if
there are zero determinants $\widehat{D}_{j}(R)$ in the sequence
$(1,\widehat{D}_1(R),\widehat{D}_2(R),\ldots,\widehat{D}_r(R))$,
then the equality~\eqref{R-function.negative.poles.condition.1}
cannot hold. Consequently, all minors $\widehat{D}_{j}(R)$ for
$j=1,2,\ldots,r$ are not equal to zero, and
from~\eqref{R-function.negative.poles.condition.1} we
obtain~\eqref{R-function.negative.poles.condition}.
\eop
\begin{remark}\label{remark.3.4}
The inequalities~\eqref{R-function.negative.poles.condition} are
equivalent to the following inequalities
\begin{equation}\label{R-function.negative.poles.condition.2}
(-1)^{j}\widehat{D}_j(R)>0,\quad j=1,2,\ldots,r,
\end{equation}
\end{remark}
\begin{corol}\label{corol.R-function.positive.poles}
Let a rational function $R$ with exactly~$r$ poles (counting
multiplicities) be an \textit{R}-function of negative type. All
poles of $R$ are positive if and only if
\begin{equation}\label{R-function.positive.poles.condition}
\widehat{D}_j(R)>0,\quad j=1,2,\ldots,r,
\end{equation}
where the determinants $\widehat{D}_j(R)$ are defined
in~\eqref{Hankel.determinants.2}.
\end{corol}
\proof If the function $R$ is an \textit{R}-function of negative
type with positive poles and has a series
expansion~\eqref{R-function.series}, then the function
\begin{equation*}\label{corol.R-function.positive.poles.proof.1}
F(z)=-zR(-z)=-\alpha
z-\beta+\frac{s_0}z-\frac{s_1}{z^2}+\frac{s_2}{z^3}-\frac{s_3}{z^4}+\cdots,
\end{equation*}
where $\alpha\geq0$ and $\beta\in\mathbb{R}$, is an
\textit{R}-function of negative type with negative poles. In fact,
since the function $R$ has the form~\eqref{Mittag.Leffler.1} with
all positive $\omega_j$ and $\gamma_j$, then $F$ can be
represented as follows
\begin{equation*}\label{corol.R-function.positive.poles.proof.2}
\displaystyle R(z)=-\alpha z-\beta+\sum^{r}_{j=1}\frac
{\gamma_j}{z+\omega_j},\qquad\alpha\geq0,\,\beta\in\mathbb{R},\,\omega_j,\gamma_j>0,\quad
j=1,2,\ldots,r,
\end{equation*}
Therefore, according to
Theorem~\ref{Th.R-function.general.properties}, the function $F$
is an \textit{R}-function of negative type with only negative
poles. On the other hand, for a fixed integer $j$ $(j \geq1)$
we have
\begin{equation*}\label{corol.R-function.positive.poles.proof.3}
\begin{array}{c}
\widehat{D}_j(F)=
\begin{vmatrix}
    -s_1 &s_2 &-s_3 &\ldots &(-1)^{j}s_{j}\\
    s_2 &-s_3 &s_4 &\ldots &(-1)^{j+1}s_{j+1}\\
    -s_3 &s_4 &-s_5 &\ldots &(-1)^{j+2}s_{j+2}\\
    \vdots&\vdots&\vdots&\ddots&\vdots\\
    (-1)^{j}s_{j} &(-1)^{j+1}s_{j+1} &(-1)^{j+2}s_{j+2} &\ldots &-s_{2j-1}
\end{vmatrix}=
\begin{vmatrix}
    -1 &0 &0 &\ldots &0\\
    0&1 &0 &\ldots &0\\
    0&0 &-1 &\ldots &0\\
    \vdots&\vdots&\vdots&\ddots&\vdots\\
    0 &0 &0 &\ldots &(-1)^{j}
\end{vmatrix}
\\
 \\
\times
\begin{vmatrix}
    s_1 &s_2 &s_3 &\ldots &s_{j}\\
    s_2 &s_3 &s_4 &\ldots &s_{j+1}\\
    s_3 &s_4 &s_5 &\ldots &s_{j+2}\\
    \vdots&\vdots&\vdots&\ddots&\vdots\\
    s_{j} &s_{j+1} &s_{j+2} &\ldots &s_{2j-1}
\end{vmatrix}
\cdot
\begin{vmatrix}
    1 &0 &0 &\ldots &0\\
    0&-1 &0 &\ldots &0\\
    0&0 &1 &\ldots &0\\
    \vdots&\vdots&\vdots&\ddots&\vdots\\
    0 &0 &0 &\ldots &(-1)^{j-1}
\end{vmatrix}=
(-1)^{\tfrac{j(j+1)}2}\cdot\widehat{D}_j(R)\cdot(-1)^{\tfrac{j(j-1)}2}\\
 \\
=(-1)^{j}\widehat{D}_j(R),
\end{array}
\end{equation*}
Consequently, from these formul\ae\ and the
inequalities~\eqref{R-function.negative.poles.condition.2} we
obtain~\eqref{R-function.positive.poles.condition}.
\eop
%

%%%%%%%%%%%%%%%%%%%%%%%%%%%%%%%%%%%%%%%%%%%%%%%%%%%%%%%%%%%%%%%%%%%%%%%%%
\subsection{Some classes of infinite Hankel matrices and the
finite moment problem on the real axis\label{s:moments}}
%%%%%%%%%%%%%%%%%%%%%%%%%%%%%%%%%%%%%%%%%%%%%%%%%%%%%%%%%%%%%%%%%%%%%%%%%

Here we connect the characterization of $R$-functions from
Section~\ref{s:R-functions.general.theory} to a particular moment problem, viz., a discrete moment
problem on the real line with a measure supported at finitely many points.
This problem is quite well known (see, e.g., \cite{Akhiezer,Shohat_Tamarkin}):

\begin{problem}[Finite moment problem on $\mathbb{R}$]\label{p:moments}
Given an infinite sequence of real numbers
$$(s_0,s_1,s_2,\ldots),$$ it is required to determine
numbers\footnote{For $k=0$, it is understood that
$\omega_1\geq 0$. Analogously, if $k=n$, we have
$\omega_n<0$.}
\begin{equation*}\label{Moment.problem}
\gamma_1>0,\gamma_2>0,\ldots,\gamma_n>0, \qquad
\omega_1<\omega_2<\ldots<\omega_k<0\leq\omega_{k+1}<\omega_{k+2}<\ldots<\omega_n
\end{equation*}
so that the equations~\eqref{Moment.formula} hold:
$$ \qquad \qquad \qquad \qquad \qquad \qquad \quad
s_i=\sum_{j=1}^{n}\gamma_j\omega_j^i,\qquad i=0,1,2,\ldots
\qquad \qquad \qquad \qquad \qquad \qquad\qquad \quad \;\; \eqref{Moment.formula}
$$
%\
\end{problem}

From Remark~\ref{remark.3.2} it follows that the
equalities~\eqref{Moment.formula} are equivalent to
the following series representation
\begin{equation}\label{Moment.problem.main.equality.2}
F(z):=\sum_{j=1}^{n}\dfrac{\gamma_j}{z-\omega_j}=\frac{s_0}z+\frac{s_1}{z^2}+\frac{s_2}{z^3}+\frac{s_3}{z^4}+\cdots
\end{equation}
In this case, the infinite Hankel matrix
$S=\|s_{i+j}\|_0^{\infty}$ is of finite rank $n$. Thus, our moment
problem has a solution if and only if the function $F$ determined
by the series~\eqref{Moment.problem.main.equality.2} is an
\textit{R}-function of negative type with $r$ poles and  with
exactly $k\,(\leq n)$ negative poles. The solution of the
moment problem is unique, since the positive numbers $\gamma_j$
and $\omega_j$ $(j=1,2,\ldots,n)$ are uniquely determined from the
expansion~\eqref{Moment.problem.main.equality.2}. We will see that
Theorems~\ref{Th.R-function.general.properties}
and~\ref{Th.number.of.negative.poles.of.R-frunction} provide a
solution to this problem in
Theorem~\ref{Th.moment.problem.solution} below, with two important
special cases provided
in~Theorem~\ref{Th.R-functions.via.strict.total.positivity} and
Corollary~\ref{Cor.R-functions.via.sign.regularity}. However,
before proceeding, we must introduce and discuss some relevant
matrix notions of positivity/nonnegativity and sign regularity.

\begin{definition}\label{def.stricly.pos.def.infinite.matrix}
An infinite symmetric matrix $A$ of finite rank is called
$r$-\textit{positive definite} if all its principal minors up to
order $r(\geq1)$ (inclusive) are positive. If the rank of the
matrix $A$ equals $r$ and $A$ is $r$-positive definite,
 then $A$ is called \textit{positive definite}.
\end{definition}

\begin{definition}\label{def.stricly.tot.pos.infinite.matrix}
An infinite matrix $A$ of finite rank is called
$r$-\textit{strictly totally positive} if all its minors up to
order~$r$ (inclusive) are positive. If the rank of the matrix $A$
equals $r$ and $A$ is $r$-strictly totally positive, then $A$ is called \textit{strictly totally positive}.
\end{definition}

\begin{definition}\label{def.tot.noneg.infinite.matrix}
An infinite matrix is called \textit{totally nonnegative} if all
its minors are nonnegative.
\end{definition}

If $A$ is a matrix (finite or infinite), then its minor of order
$j(\geq1)$ whose rows are indexed by $i_1,i_2,\ldots,i_j$ and
whose columns are indexed by $l_1,l_2,\ldots,l_j$ is denoted as
\begin{equation*}\label{minor.denotation}
A\begin{pmatrix}
    i_1 &i_2 &\dots &i_j\\
    l_1 &l_2 &\dots &l_j\\
\end{pmatrix}.
\end{equation*}

\begin{definition}[\cite{GantKrein}]\label{def.stricly.sign.reg.infinite.matrix}
An infinite matrix $A$ of finite rank is called $r$-\textit{sign
regular} if all its minors up to order~$r$ (inclusive) satisfy the
following inequalities:
\begin{equation}\label{Sign.regular.matrix.minors.ineq}
(-1)^{\sum\limits_{k=1}^{j}i_k+\sum\limits_{k=1}^{j}l_k}
A\begin{pmatrix}
    i_1 &i_2 &\dots &i_j\\
    l_1 &l_2 &\dots &l_j
\end{pmatrix}>0.
\end{equation}
If, in addition, the rank of the matrix $A$ equals $r$, then $A$ is called
\textit{sign regular}.
\end{definition}

With these definitions in place, we are now ready to state and to prove the following theorem.
\begin{theorem}\label{Th.R-functions.via.strict.total.positivity}
A function
\begin{equation}\label{Th.R-functions.via.strict.total.positivity.condition.1}
R(z)=\frac{s_0}z+\frac{s_1}{z^2}+\frac{s_2}{z^3}+\frac{s_3}{z^4}+\cdots
\end{equation}
is an \textit{R}-function of negative type and has exactly $r$
poles all of which are positive if and only if the matrix
$S=\|s_{i+j}\|_{i,j=0}^{\infty}$ is strictly totally positive of
rank $r$.
\end{theorem}
\proof This proof follows the presentation from~\cite[p.~238]{Gantmakher}.

At first, let us assume that the function $R$ is an
\textit{R}-function of negative type with $r$ poles and all poles
are positive. Then by
Theorem~\ref{Th.R-function.general.properties} we have
\begin{equation}\label{Th.R-functions.via.strict.total.positivity.proof.1}
\displaystyle R(z)=\sum^{r}_{j=1}\frac
{\gamma_j}{z-\omega_j},\qquad \gamma_j>0,\quad j=1,2,\ldots,r,
\end{equation}
where
\begin{equation}\label{Th.R-functions.via.strict.total.positivity.proof.2}
0<\omega_1<\omega_2<\ldots<\omega_r.
\end{equation}
According to formul\ae~\eqref{Moment.formula} (with $n=r$), an
arbitrary submatrix of the matrix $S$ of order $k\,(\leq r)$
can be represented as follows
\begin{equation}\label{Th.R-functions.via.strict.total.positivity.proof.3}
\begin{array}{l}
\begin{pmatrix}
    s_{i_1+j_1} &\ldots &\ldots&s_{i_1+j_k}\\
    \vdots&\vdots&\vdots&\vdots\\
    s_{i_1+j_k} &\ldots &\ldots &s_{i_k+j_k}
\end{pmatrix}=\\
 \\
=\begin{pmatrix}
    \gamma_1\omega_1^{i_1} &\gamma_2\omega_2^{i_1} &\dots &\gamma_r\omega_r^{i_1}\\
    \gamma_1\omega_1^{i_2} &\gamma_2\omega_2^{i_2} &\dots &\gamma_r\omega_r^{i_2}\\
    \vdots&\vdots&\ddots&\vdots\\
    \gamma_1\omega_1^{i_k} &\gamma_2\omega_2^{i_k} &\dots &\gamma_r\omega_r^{i_k}\\
\end{pmatrix}\cdot
\begin{pmatrix}
    \omega_1^{j_1} &\omega_2^{j_2} &\dots &\omega_r^{j_k}\\
    \omega_1^{j_1} &\omega_2^{j_2} &\dots &\omega_r^{j_k}\\
    \vdots&\vdots&\ddots&\vdots\\
    \omega_1^{j_1} &\omega_2^{j_2} &\dots &\omega_r^{j_k}\\
\end{pmatrix}.
\end{array}
\end{equation}
Therefore,
\begin{equation}\label{Th.R-functions.via.strict.total.positivity.proof.4}
\begin{array}{l}
S\begin{pmatrix}
    i_1 &i_2 &\cdots &i_k\\
    j_1 &j_2 &\cdots &j_k
\end{pmatrix}=
\begin{vmatrix}
    s_{i_1+j_1} &\ldots &\ldots&s_{i_1+j_k}\\
    \vdots&\vdots&\vdots&\vdots\\
    s_{i_1+j_k} &\ldots &\ldots &s_{i_k+j_k}
\end{vmatrix}=\\
 \\
=\displaystyle\sum_{1\leq\sigma_1<\sigma_2<\ldots<\sigma_k\leq
r}\gamma_{\sigma_1}\gamma_{\sigma_2}\cdots\gamma_{\sigma_k}
\begin{vmatrix}
    \omega_{\sigma_1}^{i_1} &\omega_{\sigma_2}^{i_1} &\dots &\omega_{\sigma_k}^{i_1}\\
    \omega_{\sigma_1}^{i_2} &\omega_{\sigma_2}^{i_2} &\dots &\omega_{\sigma_k}^{i_2}\\
    \vdots&\vdots&\ddots&\vdots\\
    \omega_{\sigma_1}^{i_k} &\omega_{\sigma_2}^{i_k} &\dots &\omega_{\sigma_k}^{i_k}\\
\end{vmatrix}\cdot
\begin{vmatrix}
    \omega_{\sigma_1}^{j_1} &\omega_{\sigma_2}^{j_1} &\dots &\omega_{\sigma_k}^{j_1}\\
    \omega_{\sigma_1}^{j_2} &\omega_{\sigma_2}^{j_2} &\dots &\omega_{\sigma_k}^{j_2}\\
    \vdots&\vdots&\ddots&\vdots\\
    \omega_{\sigma_1}^{j_k} &\omega_{\sigma_2}^{j_k} &\dots &\omega_{\sigma_k}^{j_k}\\
\end{vmatrix}
\end{array}
\end{equation}
Both determinants here are generalized Vandermonde determinants.
Their positivity follows
from~\eqref{Th.R-functions.via.strict.total.positivity.proof.2},
as  was proved in~\cite{GantKrein} (see also~\cite[p.99 and
p.239]{Gantmakher} or~\cite[Part~V, Chapter~1,
Problem~48]{Polya&Szego}). Consequently, any minor of $S$ of
order~$k(\leq r)$ is positive, since all $\gamma_i$s are also
positive.

Thus, in our case, the matrix $S$ is strictly totally positive and
has rank $r$ (the number of poles of the function $R$), according
to Theorem~\ref{Th.Hankel.matrix.rank.1}.

Conversely, if the matrix $S$ is strictly totally positive of rank
$r$, then its leading principal minors $D_j(R)$ are positive up to
order $r$. Since the matrix
$S^{(1)}=\|s_{i+j+1}\|_{i,j=0}^{\infty}$ is a submatrix of the
matrix $S$, it is also strictly totally positive and, in
particular, its leading principal minors $\widehat{D}_j(R)$ are
positive up to order $r$. From
Theorem~\ref{Th.R-function.general.properties} and
Corollary~\ref{corol.R-function.positive.poles} we obtain that the
function $R$ is an \textit{R}-function of negative type with
exactly $r$ poles, which are all positive. \eop
%
%However, if the function $R(z)$ has a pole at zero, then situation
%changes.
%
%\begin{theorem}\label{Th.R-functions.via.r.total.positivity}
%A function $R(z)$ given by
%series~\eqref{Th.R-functions.via.strict.total.positivity.condition.1}
%is an \textit{R}-function of negative type and has exactly~$r$
%poles all of which are nonnegative if and only if the matrix
%$S=\|s_{i+j}\|_{i,j=0}^{\infty}$ is $(r-1)$-totally positive of
%rank~$r$ with positive .
%\end{theorem}
%%
%\proof
%If the function $R$ is an \textit{R}-function of negative type and
%has exactly $r$ poles which are nonnegative, then the matrix $S$
%has rank $r$ according to~Theorem~\ref{Th.Hankel.matrix.rank.1}
%and
%from~\eqref{Th.R-functions.via.strict.total.positivity.proof.4}
%it follows that any minor of the matrix $S$ of order $k(<r)$ is
%positive but minors of order $r$ are nonnegative. Thus, the
%matrix~$S$ is totally nonnegative of rank~$r$ and $(r-1)$-totally
%positive.
%\eop

From this theorem it is easy to obtain the following corollary.
\begin{corol}\label{Cor.R-functions.via.sign.regularity}
A function
\begin{equation*}\label{Cor.R-functions.via.sign.regularity.condition.1}
R(z)=\frac{s_0}z+\frac{s_1}{z^2}+\frac{s_2}{z^3}+\frac{s_3}{z^4}+\cdots
\end{equation*}
is an \textit{R}-function of negative type and has exactly $r$
poles, all of which are negative, if and only if the matrix
$S=\|s_{i+j}\|_{i,j=0}^{\infty}$ is sign regular of rank $r$.
\end{corol}
\proof
The function $R$ is an \textit{R}-function with only negative
poles if and only if the function
\begin{equation}\label{Cor.R-functions.via.sign.regularity.condition.2}
F(z)=-R(-z)=\frac{s_0}z-\frac{s_1}{z^2}+\frac{s_2}{z^3}-\frac{s_3}{z^4}+\cdots=\frac{t_0}z+\frac{t_1}{z^2}+\frac{t_2}{z^3}+\frac{t_3}{z^4}+\cdots
\end{equation}
is an \textit{R}-function of negative type with only positive
poles (with the same number of poles as the function~$R$).

Let $T$ be the infinite matrix defined by
$T\eqbd \|t_{i+j}\|_{i,j=0}^{\infty}$. Since $t_j=(-1)^js_j$
($j=1,2,\ldots$), we have
\begin{equation}\label{Hankel.Matrix.factorization}
T=ESE,
\end{equation}
where the infinite matrix $E$ has the form
\begin{equation}\label{Matrix.Unique.2}
E\eqbd \left( \begin{array}{rrrrr}
    1 &0  &0  &0&\dots\\
    0 &-1 &0  &0&\dots\\
    0 &0  &1  &0&\dots\\
    0 &0  &0  &-1&\dots\\
    \vdots&\vdots&\vdots&\vdots&\ddots
\end{array} \right).
\end{equation}
All minors of this matrix except for its principal minors are zero
and
\begin{equation}\label{Matrix.Unique.2.minors}
E
\begin{pmatrix}
    i_1 &i_2 &\dots &i_k\\
    i_1 &i_2 &\dots &i_k
\end{pmatrix}=(-1)^{\sum\limits_{l=1}^{k}i_l-k},
\end{equation}
since $E(j,j)=(-1)^{j-1}$. From the Binet-Cauchy
formula (see, e.g.,~\cite{Gantmakher,{GantKrein}}) and
from~\eqref{Hankel.Matrix.factorization} we obtain
\begin{equation*}
\begin{split}
&T\begin{pmatrix}
    i_1 &i_2 &\dots &i_k\\
    j_1 &j_2 &\dots &j_k
\end{pmatrix}\\
%\end{equation*}
%\begin{equation*}
%\begin{split}
&=\displaystyle\sum_{\tau_1<\tau_2<\ldots<\tau_m}
\sum_{\varepsilon_1<\varepsilon_2<\ldots<\varepsilon_m}
E\begin{pmatrix}
    i_1 &i_2 &\dots &i_k\\
    \tau_1 &\tau_2 &\dots &\tau_k
\end{pmatrix}
S\begin{pmatrix}
    \tau_1 &\tau_2 &\dots &\tau_k\\
    \varepsilon_1 &\varepsilon_2 &\dots &\varepsilon_k
\end{pmatrix}
E\begin{pmatrix}
    \varepsilon_1 &\varepsilon_2 &\dots &\varepsilon_k\\
    j_1 &j_2 &\dots &j_k
\end{pmatrix}\\
&=E\begin{pmatrix}
    i_1 &i_2 &\dots &i_k\\
    i_1 &i_2 &\dots &i_k\\
\end{pmatrix}
S\begin{pmatrix}
    i_1 &i_2 &\dots &i_k\\
    j_1 &j_2 &\dots &j_k
\end{pmatrix}
E\begin{pmatrix}
    j_1 &j_2 &\dots &j_k\\
    j_1 &j_2 &\dots &j_k
\end{pmatrix}\\
&=(-1)^{\sum\limits_{l=1}^{k}i_l+\sum\limits_{l=1}^{k}j_l-2k}
S\begin{pmatrix}
    i_1 &i_2 &\dots &i_k\\
    j_1 &j_2 &\dots &j_k
\end{pmatrix}=
(-1)^{\sum\limits_{l=1}^{k}i_l+\sum\limits_{l=1}^{k}j_l}
S\begin{pmatrix}
    i_1 &i_2 &\dots &i_k\\
    j_1 &j_2 &\dots &j_k
\end{pmatrix}.
\end{split}
\end{equation*}

Thus,
\begin{equation}\label{SI_Stab.connection.0}
T\begin{pmatrix}
    i_1 &i_2 &\dots &i_k\\
    j_1 &j_2 &\dots &j_k\\
\end{pmatrix}=
(-1)^{\sum\limits_{l=1}^{k}i_l+\sum\limits_{l=1}^{k}j_l}S
\begin{pmatrix}
    i_1 &i_2 &\dots &i_k\\
    j_1 &j_2 &\dots &j_k\\
\end{pmatrix},\qquad k=1,2,\ldots
\end{equation}
This formula and
Theorem~\ref{Th.R-functions.via.strict.total.positivity} imply
the assertion of the corollary.
\eop

Now we are in a position to formulate the solution to Moment Problem~\ref{p:moments}
posed initially. Combined with the results in this section,
Theorems~\ref{Th.R-function.general.properties}
and~\ref{Th.number.of.negative.poles.of.R-frunction} provide
the following solution to this problem:

\begin{theorem}\label{Th.moment.problem.solution}
The infinite moment problem
\begin{equation*}\label{Th.moment.problem.solution.cond.1}
s_i=\sum_{j=1}^{n}\gamma_j\omega_j^i,\qquad i=0,1,2,\ldots
\end{equation*}
\begin{equation*}\label{Th.moment.problem.solution.cond.2}
\gamma_1>0,\gamma_2>0,\ldots,\gamma_n>0;\qquad
\omega_1<\omega_2<\ldots<\omega_k<0\leq\omega_{k+1}<\omega_{k+2}<\ldots<\omega_n,
\end{equation*}
where $s_i$, $i=0,1,2,\ldots$, are given real numbers and
$\gamma_j$ and $\omega_j$, $j=1,2,\ldots,n$, are unknown real
numbers, has a solution if and only if the infinite Hankel matrix
$S=\|s_{i+j}\|_{i,j=0}^{\infty}$ has rank $n$, the determinants
$D_j(S)$, $j=1,2,\ldots,n$, defined
in~\eqref{Hankel.determinants.1}, are positive, and\footnote{Recall
that $\SCF$ is the number of Frobenius sign changes. The determinant
$D_n(S)$ may be equal to zero, but then $D_{n-1}(S)\neq 0$ in that case
by Corollary~\eqref{corol.zero.pole}.}
\begin{equation*}\label{Th.moment.problem.solution.cond.3}
k=\SCF(1,\widehat{D}_1(S),\widehat{D}_2(S),\ldots,\widehat{D}_n(S)),
\end{equation*}
where the determinants $\widehat{D}_j(S)$ $(j=1,2,\ldots,n)$ are
defined by~\eqref{Hankel.determinants.2}. In that case,
the solution is unique.
\end{theorem}

%%%%%%%%%%%%%%%%%%%%%%%%%%%%%%%%%%%%%%%%%%%%%%%%%%%%%%%%%%%%%%%%%%%%%
\subsection{\textit{R}-functions as ratios of polynomials\label{s:R-functions.ratios}}
%%%%%%%%%%%%%%%%%%%%%%%%%%%%%%%%%%%%%%%%%%%%%%%%%%%%%%%%%%%%%%%%%%%%
In this section, we develop determinantal criteria for $R$-functions involving
the Hankel and Hurwitz minors formed from the coefficients of their numerator and
denominator.

Suppose that  a rational function $R$ is  an  \textit{R}-function of negative
type, and write
\begin{equation}\label{Rational.function.for.R-functions.and.Stieltjes.fracrtion}
R(z)=\dfrac{q(z)}{p(z)}=s_{-1}+\dfrac{s_{0}}{z}+\dfrac{s_{1}}{z^{2}}+\dfrac{s_{2}}{z^{3}}+\cdots,
\end{equation}
where $p$ and $q$ are real polynomials
\begin{eqnarray}\label{polynomial.11}
& p(z)\; = \; a_0z^n+a_1z^{n-1}+\cdots+a_n, &
a_0,a_1,\dots,a_n\in\mathbb{R},\ a_0>0, \\
\label{polynomial.12} & q(z)\; =\;
b_0z^{n}+b_1z^{n-1}+\cdots+b_{n}, &
b_0,b_1,\dots,b_{n}\in\mathbb{R}.
\end{eqnarray}
Since $R$ is an \textit{R}-function, we know that $\deg
q\geq\deg p-1$, that is, $b_0^2+b_1^2\neq0$.

\begin{remark}
Generally speaking, the polynomials $p$ and $q$ may have a
non-constant greatest common divisor $g=\gcd(p,q)$ of degree $(n-r)$
for some natural $r(< n)$. In this case, one should
consider the~function~$R$ as a ratio of polynomials
$\widetilde{p}=p/g$ and $\widetilde{q}=q/g$. Then the number of
poles of $R$ equals $r$.
\end{remark}

At first, we describe \textit{R}-functions in terms of the
coefficients of the polynomials $p$ and $q$. More precisely, we
will use the \textit{infinite matrix of Hurwitz type} $H(p,q)$
(see Definition~\ref{def.Hurwitz.matrix.infinite}) and the
\textit{finite matrices of Hurwitz type}
(Definition~\ref{def.Hurwitz.matrix.finite}) constructed using the
coefficients of polynomials $p$ and~$q$.

Our first two theorems cover the case when $\deg q<\deg p$.
\begin{theorem}\label{R-function.criterion.via.infinite.Hurwitz.matrix.case.1}
The
function~\eqref{Rational.function.for.R-functions.and.Stieltjes.fracrtion},
where $\deg q<\deg p$, is an \textit{R}-function of negative type
with exactly $r\,(\leq n)$ poles if and only if
\begin{eqnarray}\label{R-function.criterion.via.infinite.Hurwitz.matrix.case.1.condition.1}
& \eta_{2j}(p,q)\; > \; 0, & \;\; j=1,2,\ldots,r, \\
\label{R-function.criterion.via.infinite.Hurwitz.matrix.case.1.condition.2}
& \;\;  \eta_{i}(p,q)\; = \; 0, &  \;\; i>2r+1,
\end{eqnarray}
where $\eta_i(p,q)$ are the leading principal minors of the matrix
$H(p,q)$, defined by~\eqref{Hurwitz.matrix.infinite.case.1}.
\end{theorem}
\begin{theorem}\label{R-function.criterion.via.finite.Hurwitz.matrix.case.1}
The
function~\eqref{Rational.function.for.R-functions.and.Stieltjes.fracrtion},
where $\deg q<\deg p$, is an \textit{R}-function of negative type
with exactly $r\,(\leq n)$ poles if and only if
\begin{eqnarray}\label{R-function.criterion.via.finite.Hurwitz.matrix.case.1.condition.1}
& \Delta_{2j-1}(p,q)\; > \; 0, & \quad j=1,2,\ldots,r, \\
\label{R-function.criterion.via.finite.Hurwitz.matrix.case.1.condition.2}
& \;\;\;\;\; \Delta_{i}(p,q)\; = \; 0, & \quad i=2r+1,2r+2,\ldots,2n,
\end{eqnarray}
where $\Delta_i(p,q)$ are the leading principal minors of the
Hurwitz matrix $\mathcal{H}_{2n}(p,q)$ defined
in~\eqref{Hurwitz.matrix.finite.case.1}.
\end{theorem}
The next two results cover the case of equal degrees $\deg q =\deg p$:
\begin{theorem}\label{R-function.criterion.via.infinite.Hurwitz.matrix.case.2}
The
function~\eqref{Rational.function.for.R-functions.and.Stieltjes.fracrtion},
where $\deg q=\deg p$, is an \textit{R}-function of negative type
with exactly $r\,(\leq n)$ poles if and only if
\begin{eqnarray}\label{R-function.criterion.via.infinite.Hurwitz.matrix.case.2.condition.1}
& \eta_{2j+1}(p,q)>0, &  \quad j=0,1,2,\ldots,r, \\
\label{R-function.criterion.via.infinite.Hurwitz.matrix.case.2.condition.2}
& \;\;\;\;\; \eta_{i}(p,q)=0, & \quad i>2r+2,
\end{eqnarray}
where $\eta_i(p,q)$ are the leading principal minors of the infinite Hurwitz
matrix $H(p,q)$ defined by~\eqref{Hurwitz.matrix.infinite.case.2}.
\end{theorem}
\begin{theorem}\label{R-function.criterion.via.finite.Hurwitz.matrix.case.2}
The
function~\eqref{Rational.function.for.R-functions.and.Stieltjes.fracrtion},
where $\deg q=\deg p$, is an \textit{R}-function of negative type
with exactly $r\,(\leq n)$ poles if and only if
\begin{eqnarray}\label{R-function.criterion.via.finite.Hurwitz.matrix.case.2.condition.1}
&\Delta_{2j}(p,q) \; > \; 0, & \quad j=1,2,\ldots,r, \\
\label{R-function.criterion.via.finite.Hurwitz.matrix.case.2.condition.2}
&\;\; \Delta_{i}(p,q) \; = \; 0, &  \quad i=2r+2,2r+3,\ldots,2n+1,
\end{eqnarray}
where $\Delta_i(p,q)$ are the leading principal minors of the
Hurwitz matrix $\mathcal{H}_{2n+1}(p,q)$ defined
in~\eqref{Hurwitz.matrix.finite.case.2}.
\end{theorem}

All these results can be easily obtained from
Theorems~\ref{Th.Hankel.matrix.rank.2}
and~\ref{Th.R-function.general.properties} and from the
formul\ae~\eqref{Hurwitz.determinants.relations.infinite.2.even},~\eqref{Hurwitz.determinants.relations.infinite.1.odd},
\eqref{Hurwitz.determinants.relations.finite.2.odd},~\eqref{Hurwitz.determinants.relations.finite.1.even}.
Then Theorem~\ref{Th.number.of.negative.poles.of.R-frunction} and
the
formul\ae~\eqref{Hurwitz.determinants.relations.infinite.2.odd},~\eqref{Hurwitz.determinants.relations.infinite.1.even},
\eqref{Hurwitz.determinants.relations.finite.2.even},~\eqref{Hurwitz.determinants.relations.finite.1.odd}
imply
\begin{theorem}\label{Th.number.of.posotive.poles.of.R-frunction.case.1}
If the
function~\eqref{Rational.function.for.R-functions.and.Stieltjes.fracrtion},
where $\deg q<\deg p$, is an \textit{R}-function of negative type
with exactly $r\,(\leq n)$ poles, then the number of its
positive poles, $r_{+}(\leq r)$, is
\begin{equation*}\label{R-function.positive.poles.condition.case.1}
r_{+}=\SCF(\eta_1(p,q),\eta_3(p,q),\eta_5(p,q),\ldots,\eta_{2k+1}(p,q))=\SCF(1,\Delta_2(p,q),\Delta_4(p,q),\ldots,\Delta_{2k}(p,q)),
\end{equation*}
where $k=r$, if $|R(0)|<\infty$, and $k=r-1$, if $R(0)=\infty$.
\end{theorem}
\begin{theorem}\label{Th.number.of.posotive.poles.of.R-frunction.case.2}
If the
function~\eqref{Rational.function.for.R-functions.and.Stieltjes.fracrtion},
where $\deg q=\deg p$, is an \textit{R}-function of negative type
with exactly $r\,(\leq n)$ poles, then the number of its
positive poles, $r_{+}(\leq r)$, is
\begin{equation*}\label{R-function.positive.poles.condition.case.2}
r_{+}=\SCF(\eta_2(p,q),\eta_4(p,q),\eta_6(p,q),\ldots,\eta_{2k+2}(p,q))=\SCF(\Delta_1(p,q),\Delta_3(p,q),\Delta_5(p,q),\ldots,\Delta_{2k+1}(p,q)),
\end{equation*}
where $k=r$, if $|R(0)|<\infty$, and $k=r-1$, if $R(0)=\infty$.
\end{theorem}
In these theorems, the numbers of sign changes in the sequences of
Hurwitz minors must be calculated by the Frobenius
Rule~\eqref{Th.Frobenius.rule} because of the
equalities~\eqref{Hurwitz.determinants.relations.infinite.2.even}--\eqref{Hurwitz.determinants.relations.infinite.1.even}
and~\eqref{Hurwitz.determinants.relations.finite.2.odd}--\eqref{Hurwitz.determinants.relations.finite.1.odd}
between Hurwitz and Hankel minors.

We next address separately two extreme cases of
Theorems~\ref{Th.number.of.posotive.poles.of.R-frunction.case.1}--\ref{Th.number.of.posotive.poles.of.R-frunction.case.2}, when all poles of the function $R$
 are either negative or positive:
\begin{corol}\label{corol.R-functions.with.positive.poles.via.infinite.Hurwitz.matrix}
The
function~\eqref{Rational.function.for.R-functions.and.Stieltjes.fracrtion}
is an \textit{R}-function of negative type with exactly
$r\,(\leq n)$ poles, all of which are negative, if and only
if
\begin{eqnarray*}\label{R-function.positive poles.criterion.via.infinie.Hurwitz.matrix.condition.1}
& \eta_{i}(p,q)>0, &  \quad j=1,2,\ldots,k,  \\
\label{R-function.positive poles.criterion.via.infinite.Hurwitz.matrix.condition.2}
& \eta_{i}(p,q)=0, &  \quad i>k,
\end{eqnarray*}
where $k=2r+1$ if $\deg q<\deg p$, and $k=2r+2$ if $\deg q=\deg
p$, but $\eta_i(p,q)$ are the leading principal minors of the
infinite Hurwitz matrix $H(p,q)$ defined
in~\eqref{Hurwitz.matrix.infinite.case.1}.
\end{corol}
\begin{corol}\label{corol.R-functions.with.positive.poles.via.finite.Hurwitz.matrix}
The
function~\eqref{Rational.function.for.R-functions.and.Stieltjes.fracrtion}
is an \textit{R}-function of negative type with exactly
$r\,(\leq n)$ poles, all of which are negative, if and only
if
\begin{eqnarray*}\label{R-function.positive poles.criterion.via.finite.Hurwitz.matrix.condition.1}
& \Delta_i(p,q)>0, & \quad   i=1,2,\ldots,k, \\
\label{R-function.positive poles.criterion.via.finite.Hurwitz.matrix.condition.2}
& \Delta_{i}(p,q)=0, & \quad i=k+1,k+2,\ldots,
\end{eqnarray*}
where $k=2r$, if $\deg q<\deg p$, and $k=2r+1$, if $\deg q=\deg
p$, but $\Delta_i(p,q)$ are the leading principal minors of the Hurwitz
matrix $\mathcal{H}_{2n}(p,q)$ or the matrix
$\mathcal{H}_{2n+1}(p,q)$ defined
in~\eqref{Hurwitz.matrix.finite.case.1}--\eqref{Hurwitz.matrix.finite.case.2}.
\end{corol}
\begin{corol}\label{corol.R-functions.with.negative.poles.via.infinite.Hurwitz.matrix}
The
function~\eqref{Rational.function.for.R-functions.and.Stieltjes.fracrtion}
is an \textit{R}-function of negative type with exactly
$r\,(\leq n)$ poles, all of which are positive, if and only
if,
\begin{itemize}
\item for $\deg q<\deg p$, the
inequalities~\eqref{R-function.criterion.via.infinite.Hurwitz.matrix.case.1.condition.1}--\eqref{R-function.criterion.via.infinite.Hurwitz.matrix.case.1.condition.2}
hold and
\begin{equation*}\label{R-function.negative poles.criterion.via.infinite.Hurwitz.matrix.condition.case.1}
\eta_{2j-1}(p,q)\eta_{2j+1}(p,q)<0,\quad j=1,2,\ldots,r,
\end{equation*}
where $\eta_i(p,q)$ are the leading principal minors of the matrix
$H(p,q)$ defined by~\eqref{Hurwitz.matrix.infinite.case.1}.
\item for $\deg q=\deg p$, the
inequalities~\eqref{R-function.criterion.via.infinite.Hurwitz.matrix.case.2.condition.1}--\eqref{R-function.criterion.via.infinite.Hurwitz.matrix.case.2.condition.2}
hold and
\begin{equation*}\label{R-function.negative poles.criterion.via.infinite.Hurwitz.matrix.condition.case.2}
\eta_{2j}(p,q)\eta_{2j+2}(p,q)<0,\quad j=1,2,\ldots,r,
\end{equation*}
where $\eta_i(p,q)$ are the leading principal minors of the matrix
$H(p,q)$ defined by~\eqref{Hurwitz.matrix.infinite.case.2}.
\end{itemize}
\end{corol}
\begin{corol}\label{corol.R-functions.with.negative.poles.via.finite.Hurwitz.matrix}
The
function~\eqref{Rational.function.for.R-functions.and.Stieltjes.fracrtion}
is an \textit{R}-function of negative type with exactly
$r\,(\leq n)$ poles, all of which are positive, if and only
if,
\begin{itemize}
\item for $\deg q<\deg p$, the
inequalities~\eqref{R-function.criterion.via.finite.Hurwitz.matrix.case.1.condition.1}--\eqref{R-function.criterion.via.finite.Hurwitz.matrix.case.1.condition.2}
hold and
\begin{equation*}\label{R-function.negative poles.criterion.via.finite.Hurwitz.matrix.condition.case.1}
\Delta_{2j-2}(p,q)\Delta_{2j}(p,q)<0,\quad
j=1,2,\ldots,r,\quad(\Delta_0(p,q)\eqbd 1)
\end{equation*}
where $\Delta_i(p,q)$ are the leading principal minors of the
matrix $\mathcal{H}_{2n}(p,q)$ defined
in~\eqref{Hurwitz.matrix.finite.case.1}.
\item for $\deg q=\deg p$, the
inequalities~\eqref{R-function.criterion.via.finite.Hurwitz.matrix.case.2.condition.1}--\eqref{R-function.criterion.via.finite.Hurwitz.matrix.case.2.condition.2}
hold and
\begin{equation*}\label{R-function.negative poles.criterion.via.finite.Hurwitz.matrix.condition.case.2}
\Delta_{2j-1}(p,q)\Delta_{2j+1}(p,q)<0,\quad j=1,2,\ldots,r,
\end{equation*}
where $\Delta_i(p,q)$ are the leading principal minors of the
matrix $\mathcal{H}_{2n+1}(p,q)$, defined
in~\eqref{Hurwitz.matrix.finite.case.2}.
\end{itemize}
\end{corol}

At last, there is one more way to find the number of negative (or
positive) poles of an $R$-function. This method turns into
the famous Lienard-Chipart theorem when applied to the theory of
Hurwitz stable polynomials (see~\cite[Chapter~XV, Section 13]{Gantmakher}
Theorem~11). But first, let us introduce (without proof) the
remarkable but little-known consequence of the famous Descartes
Rule of Signs (see, for example,~\cite{Polya&Szego}).

Recall that we denote by  $\SC^-(a_0,a_1,\ldots,a_n)$ the number of weak sign
changes in a real sequence  $(a_0,a_1,\ldots,a_n)$, i.e., the number of sign
changes with zero elements of the sequence removed.
\begin{theorem}[\cite{Uspensky,{Polya&Szego}}]\label{Th.Descartes.rule.realzero.polys}
If a real polynomial $a_0z^n+a_1z^{n-1}+\cdots+a_n$ has only real
roots, then the number of its positive zeros, counting
multiplicities, is equal to $\SC^-(a_0,a_1,\ldots,a_n)$.
\end{theorem}

This fact, together with previous results, implies the following
theorem.

\begin{theorem}\label{Th.generalized.Lienard.Chipart}
If the real rational
function~\eqref{Rational.function.for.R-functions.and.Stieltjes.fracrtion}
is an \textit{R}-function of negative type with exactly $n$
poles\footnote{The latter means that $\gcd(p,q)\equiv 1$.}, then the number
of its positive poles is equal to $SC^-(a_0,a_1,\ldots,a_n)$. In
particular, $R$ has only negative poles if and only
if~\footnote{In fact, the coefficients must be simply of the same sign, but we
already assumed that $a_0>0$ (see~\eqref{polynomial.11}).} $a_j>0$ for
$j=1,2,\ldots,n$, and $R$ has only positive poles if and only if
$a_{j-1}a_j<0$ for $j=1,2,\ldots,n$.
\end{theorem}
\proof Indeed, by Theorem~\ref{Th.R-function.general.properties},
if the function $R$ defined
in~\eqref{Rational.function.for.R-functions.and.Stieltjes.fracrtion}
is an $R$-function, then the polynomial~$p$ has only real
roots which are poles of the function $R$. The number of positive
zeros of $p$ (poles of $R$) equals $\SC^-(a_0,a_1,\ldots,a_n)$,
according to Theorem~\ref{Th.Descartes.rule.realzero.polys}. \eop

For two extreme classes of \textit{R}-functions, i.e., those
 with only positive or only negative poles, we can obtain  a few more criteria.
\begin{theorem}\label{Th.general.Lienard.Chipart.negative.case.1}
The
function~\eqref{Rational.function.for.R-functions.and.Stieltjes.fracrtion},
where $\deg q<\deg p=n$, is an \textit{R}-function of negative type
and has exactly $n$ negative poles if and only if one of the
following conditions holds
\begin{itemize}
\item[$1)$]
$a_n>0, \; a_{n-1}>0, \; \ldots, \; a_0>0, \qquad  \qquad \;\;\;\;\; \Delta_1(p,q)>0, \; \Delta_3(p,q)>0,\; \ldots,
\; \Delta_{2n-1}(p,q)>0$;
\item[$2)$]
$a_n>0,\; b_n>0,\; b_{n-1}>0,\; \ldots,\; b_1>0,\qquad \Delta_1(p,q)>0,\; \Delta_3(p,q)>0,
\; \ldots,\; \Delta_{2n-1}(p,q)>0$;
\item[$3)$]
$a_n>0,\; a_{n-1}>0,\; \ldots,\; a_0>0,\; \qquad \qquad \quad \Delta_2(p,q)>0,\; \Delta_4(p,q)>0,\;
\ldots,\; \Delta_{2n}(p,q)>0$;
\item[$4)$]
$a_n>0,\; b_n>0,\; b_{n-1}>0,\; \ldots,\; b_1>0, \qquad \Delta_2(p,q)>0,\;
\Delta_4(p,q)>0,\; \ldots,\; \Delta_{2n}(p,q)>0$.
\end{itemize}
\end{theorem}
\proof
Condition $1)$ contains the
inequalities~\eqref{R-function.criterion.via.finite.Hurwitz.matrix.case.1.condition.1}
(with $r= n$). By
Theorem~\ref{R-function.criterion.via.finite.Hurwitz.matrix.case.1},
this is equivalent to the function $R$ being an
\textit{R}-function of negative type with exactly $n$ poles. Now
Theorem~\ref{Th.generalized.Lienard.Chipart} implies that the
positivity of the coefficients of $p$ is equivalent to the
negativity of poles of $R$.

\vspace{2mm}

Condition $2)$ also contains the
inequalities~\eqref{R-function.criterion.via.finite.Hurwitz.matrix.case.1.condition.1}
(with $r=n$), which are equivalent to the function $R$ being an
\textit{R}-function of negative type with exactly $n$ poles. By
Theorem~\ref{Th.R-function.general.properties}, all zeros and
poles of $R$ are real and simple.

If the function $R$ has exactly $n$ negative poles, then from
Theorem~\ref{Th.R-function.general.properties} it follows that it
also has negative zeros (since the zeros and the poles of $R$ are
interlacing). Now Theorem~\ref{Th.Descartes.rule.realzero.polys}
yields the positivity of all coefficients of the polynomials $p$ and $q$.

Conversely, if the inequalities $a_n>0$, $b_n>0$, $b_{n-1}>0$,
$\ldots$, $b_1>0$ hold, then $R$ has only negative zeros,
according to Theorem~\ref{Th.Descartes.rule.realzero.polys}. The
interlacing of zeros and poles of $R$ implies that $R$ may have at
most one nonnegative simple pole, that is, the polynomial $p$ may
have at most one nonnegative simple zero. But $a_0>0$ by
assumption, therefore, $p(z)\to+\infty$ as $z\to+\infty$. Since
all zeros of $p$ are simple and real, we have $a_n=p(0)\leq0$
whenever $p$ has one nonnegative zero. This contradicts the
inequality $a_n>0$. Consequently, $p$ has only negative roots, and
$R$ has only negative poles.

\vspace{2mm}

If the function $R$ is an \textit{R}-function of negative type
with exactly $n$ negative poles, then
Corollary~\ref{corol.R-functions.with.positive.poles.via.finite.Hurwitz.matrix}
and Theorem~\ref{Th.Descartes.rule.realzero.polys} imply
condition~$3)$.

Now suppose that condition $3)$ holds. It contains the
inequalities $\Delta_{2j}(p,q)>0$, $j=1,2,\ldots,n$, which are
equivalent to the inequalities $(-1)^j\widehat{D}_j(R)>0$,
$j=1,2,\ldots,n$, according
to~\eqref{Hurwitz.determinants.relations.finite.2.even}. But we
know that $(-1)^j\widehat{D}_j(R)=D_j(F)$, where $F(z)= -zR(z)$,
and the positivity of these determinants is equivalent to the
function $F$ being an \textit{R}-function of negative type by
Theorem~\ref{Th.R-function.general.properties}. From the same
theorem it also follows that all zeros of the polynomials $zq(z)$
and $p(z)$ are real simple and interlacing. From
Theorem~\ref{Th.Descartes.rule.realzero.polys} and from the
inequalities $a_j>0$, $j=0,1,\ldots,n$, we obtain that $p$ has
only negative roots. This fact implies that all zeros of $q$ are
negative too and that they interlace the zeros of the polynomial
$p$ since otherwise the zeros of $zq(z)$ would not interlace the
zeros of $p(z)$. Since the zeros of $p$ and $q$ interlace, the
function $R$ is an \textit{R}-function of negative type by
Theorem~\ref{Th.R-function.general.properties}. It has only
negative poles (the zeros of $p$).

\vspace{2mm}

As in the case of condition $3)$, condition $4)$ holds if the
function $R$ is an \textit{R}-function of negative type with
exactly $n$ negative poles.

If condition $4)$ holds, then, as before, the function $F(z)=
-zR(z)$ is an \textit{R}-function of negative type and, therefore,
the roots of the polynomials $zq(z)$ and $p(z)$ are real and
simple and interlace each other. The inequalities $b_j>0$,
$j=1,2,\ldots,n$, imply that the polynomial $q$ has only negative
zeros. Consequently, $p(z)$ has at most one simple nonnegative
zero because of the interlacing with the zeros of $zq(z)$. As
above, $p$ cannot have a nonnegative zero since otherwise the
inequality $a_n>0$ cannot hold. Since the polynomial $zq(z)$ has
nonpositive zeros but $p(z)$ has only negative zeros and since
their zeros interlace each other, the zeros of $q$ and $p$
interlace too. Now Theorem~\ref{Th.R-function.general.properties}
implies that $R$ is an $R$-function of negative type with exactly
$n$ negative poles. \eop

In the case $\deg q=\deg p$ we obtain similar criteria. The next theorem
addresses the situation when the degrees $\deg p$ and $\deg q$ are equal
and all poles of $R$ are negative.
\begin{theorem}\label{Th.general.Lienard.Chipart.negative.case.2}
The
function~\eqref{Rational.function.for.R-functions.and.Stieltjes.fracrtion},
where $\deg q=\deg p=n$, is an \textit{R}-function of negative type
with exactly $n$ negative poles if and only if one of the
following conditions holds
\begin{itemize}
\item[$1)$]
$a_n>0,\; a_{n-1}>0,\; \ldots,\; a_0>0,\qquad\qquad\quad\;\; \Delta_2(p,q)>0,\; \Delta_4(p,q)>0,
\; \ldots,\;  \Delta_{2n}(p,q)>0$;
\item[$2)$]
$a_n>0,\; b_n>0,\; b_{n-1}>0,\; \ldots,\; b_0>0,\; \qquad\Delta_2(p,q)>0,\; \Delta_4(p,q)>0,
\; \ldots,\; \Delta_{2n}(p,q)>0$;
\item[$3)$]
$a_n>0,\; a_{n-1}>0,\; \ldots,a_0>0,\; \qquad\qquad\quad\;\; \Delta_1(p,q)>0,\; \Delta_3(p,q)>0,
\;\ldots,\;\Delta_{2n+1}(p,q)>0$;
\item[$4)$]
$a_n>0,\; b_n>0,\; b_{n-1}>0,\; \ldots,\; b_0>0,\; \qquad \Delta_1(p,q)>0,\;
\Delta_3(p,q)>0,\; \ldots,\; \Delta_{2n+1}(p,q)>0$.
\end{itemize}
\end{theorem}
\proof This theorem can be proved in the same way as
Theorem~\ref{Th.general.Lienard.Chipart.negative.case.1} using
Theorem~\ref{R-function.criterion.via.finite.Hurwitz.matrix.case.2}
instead of
Theorem~\ref{R-function.criterion.via.finite.Hurwitz.matrix.case.1}
and the
equalities~\eqref{Hurwitz.determinants.relations.finite.1.odd}
instead of the
equalities~\eqref{Hurwitz.determinants.relations.finite.2.even}.
\eop

To address the situation when all poles of $R$ are positive, it
is enough to switch to the function $z\mapsto -R(-z)$ and
apply the same methods as in the proof of
Theorems~\ref{Th.general.Lienard.Chipart.negative.case.1}
and~\ref{Th.general.Lienard.Chipart.negative.case.2}. This
yields the following two results, one for the case $\deg q <\deg p$
and the other for the case $\deg q = \deg p$.

\begin{theorem}\label{Th.general.Lienard.Chipart.positive.case.1}
The
function~\eqref{Rational.function.for.R-functions.and.Stieltjes.fracrtion},
where $\deg q<\deg p=n$, is an \textit{R}-function of negative type
and has exactly $n$ positive poles if and only if one of the
following conditions holds
\begin{itemize}
\item[$1)$] $\qquad\qquad\qquad\qquad(-1)^{n}a_n>0, \; (-1)^{n-1}a_{n-1}>0,\; \ldots,\;a_0>0,$

\vspace{1mm}
$\qquad\qquad\qquad\qquad\Delta_1(p,q)>0,\; \Delta_3(p,q)>0,\; \ldots,\; \Delta_{2n-1}(p,q)>0;$
%\begin{equation*}\label{Th.general.Lienard.Chipart.positive.case.1.cond.1}
%\begin{array}{c}
%(-1)^{n}a_n>0,(-1)^{n-1}a_{n-1}>0,\ldots,a_0>0,\\
% \\
%\Delta_1(p,q)>0,\Delta_3(p,q)>0,\ldots,\Delta_{2n-1}(p,q)>0;
%\end{array}
%\end{equation*}

\vspace{2mm}

\item[$2)$]
$\qquad\qquad\qquad\qquad(-1)^na_n>0,\; (-1)^{n-1}b_n>0,\; (-1)^{n-2}b_{n-1}>0,\; \ldots, \; b_1>0,$

\vspace{1mm}
$\qquad\qquad\qquad\qquad\Delta_1(p,q)>0,\; \Delta_3(p,q)>0,\; \ldots,\; \Delta_{2n-1}(p,q)>0;$
%\begin{equation*}\label{Th.general.Lienard.Chipart.positive.case.1.cond.2}
%\begin{array}{c}
%(-1)^na_n>0,(-1)^{n-1}b_n>0,(-1)^{n-2}b_{n-1}>0,\ldots,b_1>0,\\
% \\
%\Delta_1(p,q)>0,\Delta_3(p,q)>0,\ldots,\Delta_{2n-1}(p,q)>0;
%end{array}
%\end{equation*}

\vspace{2mm}
\item[$3)$] $\qquad\qquad\qquad\qquad (-1)^{n}a_n>0,(-1)^{n-1}a_{n-1}>0,\ldots,a_0>0,$

\vspace{1mm}
$\qquad\qquad\qquad\qquad -\Delta_2(p,q)>0,\; \Delta_4(p,q)>0,\; \ldots,\; (-1)^n\Delta_{2n}(p,q)>0;$
%\begin{equation*}\label{Th.general.Lienard.Chipart.positive.case.1.cond.3}
%\begin{array}{c}
%(-1)^{n}a_n>0,(-1)^{n-1}a_{n-1}>0,\ldots,a_0>0,\\
% \\
%-\Delta_2(p,q)>0,\Delta_4(p,q)>0,\ldots,(-1)^n\Delta_{2n}(p,q)>0;
%\end{array}
%\end{equation*}
%

\vspace{2mm}
\item[$4)$] $\qquad\qquad\qquad\qquad (-1)^na_n>0,\; (-1)^{n-1}b_n>0,\; (-1)^{n-2}b_{n-1}>0,\; \ldots,\; b_1>0,$

\vspace{1mm}
$\qquad\qquad\qquad\qquad -\Delta_2(p,q)>0,\; \Delta_4(p,q)>0,\; \ldots,\; (-1)^n\Delta_{2n}(p,q)>0.$
%\begin{equation*}\label{Th.general.Lienard.Chipart.positive.case.1.cond.4}
%\begin{array}{c}
%(-1)^na_n>0,(-1)^{n-1}b_n>0,(-1)^{n-2}b_{n-1}>0,\ldots,b_1>0,\\
% \\
%-\Delta_2(p,q)>0,\Delta_4(p,q)>0,\ldots,(-1)^n\Delta_{2n}(p,q)>0.
%\end{array}
%\end{equation*}
%
\end{itemize}
\end{theorem}
\begin{theorem}\label{Th.general.Lienard.Chipart.positive.case.2}
The function~\eqref{Rational.function.for.R-functions.and.Stieltjes.fracrtion},
where $\deg q=\deg p$, is an \textit{R}-function of negative type
with exactly $n$ positive poles if and only if one of the
following conditions holds
\begin{itemize}
\item[$1)$] $\qquad\qquad\qquad\qquad (-1)^{n}a_n>0,\; (-1)^{n-1}a_{n-1}>0,\; \ldots, \; a_0>0,$

\vspace{1mm}
$\qquad\qquad\qquad\qquad \Delta_2(p,q)>0,\; \Delta_4(p,q)>0,\; \ldots,\; \Delta_{2n}(p,q)>0;$
%\begin{equation*}
%\begin{array}{c}
%(-1)^{n}a_n>0,(-1)^{n-1}a_{n-1}>0,\ldots,a_0>0,\\
% \\
%\Delta_2(p,q)>0,\Delta_4(p,q)>0,\ldots,\Delta_{2n}(p,q)>0;
%\end{array}
%\end{equation*}
%
\vspace{2mm}
\item[$2)$] $\qquad\qquad\qquad\qquad (-1)^na_n>0,\; (-1)^{n-1}b_n>0,\; (-1)^{n-2}b_{n-1}>0,\;
\ldots,\; b_1>0,$

\vspace{1mm}
$\qquad\qquad\qquad \qquad \Delta_2(p,q)>0,\; \Delta_4(p,q)>0,\; \ldots,\; \Delta_{2n}(p,q)>0.$
%
%\begin{equation*}
%\begin{array}{c}
%(-1)^na_n>0,(-1)^{n-1}b_n>0,(-1)^{n-2}b_{n-1}>0,\ldots,b_1>0,\\
% \\
%\Delta_2(p,q)>0,\Delta_4(p,q)>0,\ldots,\Delta_{2n}(p,q)>0.
%\end{array}
%\end{equation*}
%
\vspace{2mm}
\item[$3)$] $\qquad\qquad\qquad\qquad (-1)^{n}a_n>0,\; (-1)^{n-1}a_{n-1}>0,\; \ldots, \; a_0>0,$

\vspace{1mm}
$\qquad\qquad\qquad\qquad \Delta_1(p,q)>0,\; -\Delta_3(p,q)>0,\; \ldots,\; (-1)^n\Delta_{2n+1}(p,q)>0;$
%
%\begin{equation*}
%\begin{array}{c}
%(-1)^{n}a_n>0,(-1)^{n-1}a_{n-1}>0,\ldots,a_0>0,\\
% \\
%\Delta_1(p,q)>0,-\Delta_3(p,q)>0,\ldots,(-1)^n\Delta_{2n+1}(p,q)>0;
%\end{array}
%\end{equation*}
%
\vspace{2mm}
\item[$4)$] $\qquad\qquad\qquad\qquad (-1)^na_n>0,\; (-1)^{n-1}b_n>0,\; (-1)^{n-2}b_{n-1}>0,\;
\ldots, \; b_1>0,$

\vspace{1mm}
$\qquad\qquad\qquad\qquad \Delta_1(p,q)>0,\; -\Delta_3(p,q)>0,\; \ldots,\; (-1)^n\Delta_{2n+1}(p,q)>0;$
%
%\begin{equation*}
%\begin{array}{c}
%(-1)^na_n>0,(-1)^{n-1}b_n>0,(-1)^{n-2}b_{n-1}>0,\ldots,b_1>0,\\
% \\
%\Delta_1(p,q)>0,-\Delta_3(p,q)>0,\ldots,(-1)^n\Delta_{2n+1}(p,q)>0;
%\end{array}
%\end{equation*}
%
\end{itemize}
\end{theorem}

%%%%%%%%%%%%%%%%%%%%%%%%%%%%%%%%%%%%%%%%%%%%%%%%%%%%%%%%%%%%%%%%%%%%%
\subsection{\label{s:R-functions.Stieltjes}\textit{R}-functions with a Stieltjes continued fraction
expansion}
%%%%%%%%%%%%%%%%%%%%%%%%%%%%%%%%%%%%%%%%%%%%%%%%%%%%%%%%%%%%%%%%%%%%%

We now provide criteria for $R$ functions based on the
coefficients of their Stieltjes continued fractions.

Assume that our
function~\eqref{Rational.function.for.R-functions.and.Stieltjes.fracrtion}
with exactly $r\,(\leq n)$ poles has a Stieltjes continued
fraction expansion
\begin{equation}\label{Stieltjes.fraction.for.real.functions.1}
R(z)=c_0+\dfrac1{c_1z+\cfrac1{c_2+\cfrac1{c_{3}z+\cfrac1{\ddots+\cfrac1{T}}}}},\quad
c_j\in\mathbb{R},\;\;\;c_j\neq0,
\end{equation}
where
\begin{equation}\label{Stieltjes.fraction.for.real.functions.2}
\begin{array}{l}
T=\begin{cases}
         \; c_{2r} & \; {\rm if} \;\; |R(0)|<\infty,\\
         \;  c_{2r-1}z & \; {\rm if} \;\; R(0)=\infty,
       \end{cases}\ \ \ \ \qquad\qquad\qquad\qquad\qquad
\end{array}
\end{equation}
and $c_0=s_{-1}$, $c_0\neq 0$ if and only if $\deg p=\deg q(=n)$. From
Theorem~\ref{Th.Stieltjes fraction.global.criterion} it follows
that the
inequalities~\eqref{minors.inequalities.for.Stieltjes.fraction.1}--\eqref{minors.inequalities.for.Stieltjes.fraction.2}
hold for the function $R$. At the same time the coefficients~$c_i$
in~\eqref{Stieltjes.fraction.for.real.functions.1} can be found by
the
formul\ae~\eqref{even.coeff.Stieltjes.fraction.main.formula}--\eqref{odd.coeff.Stieltjes.fraction.main.formula}
(see~\eqref{Stieltjes.fraction.1}).

\vspace{1mm}

Theorems~\ref{Th.R-function.general.properties}
and~\ref{Th.number.of.negative.positive.indices.of.real.rational.function.via.S-fraction}
yield

\begin{theorem}\label{Th.R-functions.Stiljes.fractions}
Let the
function~\eqref{Rational.function.for.R-functions.and.Stieltjes.fracrtion}
with exactly $r\,(\leq n=\deg p)$ poles have a Stieltjes
continued fraction
expansion~\eqref{Stieltjes.fraction.for.real.functions.1}--\eqref{Stieltjes.fraction.for.real.functions.2}.
The function $R$ is an $R$-function of negative type if and only
if
\begin{equation}\label{R-function.Stieltjes.fraction.condition}
c_{2j-1}>0,\qquad j=1,2,\ldots,r.
\end{equation}
Then the number of negative poles is equal to the number of
positive coefficients $c_{2j}$, $j=1,2,\ldots,k$, where $k=r$, if
$|R(0)|<\infty$, and $k=r-1$, if $R(0)=\infty$.
\end{theorem}

Note that $R$-functions with only positive or only
negative poles have Stieltjes continued fraction expansions,
according to Corollaries~\ref{corol.R-function.negative.poles}
and~\ref{corol.R-function.positive.poles} and
Theorem~\ref{Th.Stieltjes fraction.global.criterion}.
Theorem~\ref{Th.R-functions.Stiljes.fractions} has the following
corollaries.
\begin{corol}\label{corol.R-function.Stieltjes.fractions.positive.poles}
A rational function $R$ with exactly $r$ poles is
an~$R$-function of negative type with all positive poles if
and only if $R$ has a Stieltjes continued fraction
expansion~\eqref{Stieltjes.fraction.for.real.functions.1}--\eqref{Stieltjes.fraction.for.real.functions.2},
where inequalities~\eqref{R-function.Stieltjes.fraction.condition}
hold and where
\begin{equation*}\label{R-function.Stieltjes.fraction.positive.poles.condition}
c_{2j}<0,\qquad j=1,2,\ldots,r.
\end{equation*}
\end{corol}
\begin{corol}\label{corol.R-function.Stieltjes.fractions.nonnegative.poles}
A rational function $R$ with exactly $r$ poles is an~$R$-function
of negative type with all positive poles, except for one at $0$,
if and only if $R$ has a Stieltjes continued fraction
expansion~\eqref{Stieltjes.fraction.for.real.functions.1}--\eqref{Stieltjes.fraction.for.real.functions.2},
where inequalities~\eqref{R-function.Stieltjes.fraction.condition}
hold and where
\begin{equation*}\label{R-function.Stieltjes.fraction.nonnegative.poles.condition}
c_{2j}<0,\qquad j=1,2,\ldots,r-1.
\end{equation*}
\end{corol}
\begin{corol}[Markov, Stieltjes \cite{Markov,{Stieltjes1},{Stieltjes2},{Stieltjes3},{Stieltjes4}}]\label{corol.R-function.Stieltjes.fractions.negative.poles}
A rational function $R$ with exactly $r$ poles is
an~$R$-function of negative type with all negative poles if
and only if $R$ has a Stieltjes continued fraction
expansion~\eqref{Stieltjes.fraction.for.real.functions.1}--\eqref{Stieltjes.fraction.for.real.functions.2},
where
\begin{equation}\label{R-function.Stieltjes.fraction.negative.poles.condition}
c_{i}>0,\qquad\; i=1,2,\ldots,2r.
\end{equation}
\end{corol}
\begin{corol}\label{corol.R-function.Stieltjes.fractions.nonpositive.poles}
A rational function $R$ with exactly $r$ poles is
an~$R$-function of negative type with all negative poles,
except for one at $0$, if and only if $R$ has a
Stieltjes continued fraction
expansion~\eqref{Stieltjes.fraction.for.real.functions.1}--\eqref{Stieltjes.fraction.for.real.functions.2},
where
\begin{equation}\label{R-function.Stieltjes.fraction.nonpositive.poles.condition}
c_{i}>0,\qquad i=1,2,\ldots,2r-1.
\end{equation}
\end{corol}
In the sequel we use the following well-known result.

\begin{theorem}[Aisen, Edrei, Schoenberg,
Whitney,~\cite{AisenEdreiShoenbergWhitney,{ASW},{Edrei_II},{Karlin}}]\label{Th.Shoenberg}
The polynomial $$g(z)=g_0z^l+g_1z^{l-1}+\cdots+g_l$$ has only
nonpositive zeros if and only if its Toeplitz matrix
$\mathcal{T}(g)$ defined by~\eqref{Toeplitz.infinite.matrix} is
totally nonnegative.
\end{theorem}

The functions whose Taylor coefficients generate totally
nonnegative Toeplitz matrices of the form $\mathcal{T}(g)$ was
introduced by Schoenberg~\cite{Schoenberg_ger,{Schoenberg_smooth}}
and also studied by Edrei~\cite{Edrei_1,{Edrei_dub},{Arms_Edrei}}.

\vspace{2mm}
%\begin{definition}\label{def.PF}
%If the polynomial $g$ satisfies the conditions of
%Theorem~\ref{Th.Shoenberg}, then the sequence of its coefficients
%is called totally positive.%

%If, for a given polynomial $g$, all minor of order $\leq m$
%of the matrix $\mathcal{T}(g)$ are nonnegative, then the sequence
%of the coefficients of $g$ is called $m$-positive.
%\end{definition}

%The class of $m$-positive (totally positive) sequences is usually
%denoted by $PF_m$ ($PF_\infty$). The classes of corresponding
%generating functions are also denoted by $PF_m$ ($PF_\infty$). So,

%\vspace{2mm}

We next prove a  criterion of total nonnegativity of infinite
Hurwitz matrices. Previously, only one direction, i.e., the fact
that the infinite Hurwitz matrix of a quasi-stable
polynomial\footnote{The quasi-stable polynomials are polynomials
with zeros in the closed left half-plane of the complex plane.} is
totally nonnegative (see~\cite{Asner,{Kemperman},{Holtz1},Pinkus}),
was known. The necessary and sufficient condition was known only for
finite Hurwitz matrix of Hurwitz stable polynomials.

\begin{theorem}[\textbf{Total Nonnegativity of the Hurwitz Matrix}]\label{Th.Hurwitz.Matrix.Total.Nonnegativity}
The following  are equivalent:
\begin{itemize}
\item[$1)$] The polynomials $p$ and $q$ defined by~\eqref{polynomial.11}--\eqref{polynomial.12} have only nonpositive
zeros\footnote{Here we include the case when $q(z)\equiv0$.}, and
the function $R=q/p$ is either an $R$-function of negative
type or identically zero.
\item[$2)$] The infinite matrix of Hurwitz type $H(p,q)$ defined by~\eqref{Hurwitz.matrix.infinite.case.1}--\eqref{Hurwitz.matrix.infinite.case.2} is totally nonnegative.
\end{itemize}
\end{theorem}
\proof If condition $1)$ holds and $R(z)\not\equiv0$, then,
according to
Corollaries~\ref{corol.R-function.Stieltjes.fractions.negative.poles}--\ref{corol.R-function.Stieltjes.fractions.nonpositive.poles},
the function~$R$ has a Stieltjes fraction
expansion~\eqref{Stieltjes.fraction.for.real.functions.1}--\eqref{Stieltjes.fraction.for.real.functions.2}
and satisfies the
inequalities~\eqref{R-function.Stieltjes.fraction.negative.poles.condition}
or~\eqref{R-function.Stieltjes.fraction.nonpositive.poles.condition},
where $r\,(\leq n)$ is the number of poles of the function
$R={q}/{p}$. According to Theorem~\ref{Th.Stieltjes
fraction.global.criterion}, the matrix $H(p,q)$ has a
factorization~\eqref{infinite.Hurwitz.matrix.factorization.1}
or~\eqref{infinite.Hurwitz.matrix.factorization.2}, where all
matrices $J(c_j)$ are totally nonnegative by inspection since all
$c_j$ are positive. However, by assumption, all zeros of $p$ and
$q$ are nonpositive, therefore all zeros of $g=\gcd(p,q)$ are also
nonpositive. Now from Theorem~\ref{Th.Shoenberg} we obtain that
the matrix $\mathcal{T}(g)$ is totally nonnegative. Since the
matrix~$H(0,1)$ is trivially totally nonnegative, the matrix
$H(p,q)$ is totally nonnegative as a product of totally
nonnegative matrices (see~\cite{GantKrein}).

Let condition $1)$ hold and $R(z)\equiv 0$. Then instead of the
factorizations~\eqref{infinite.Hurwitz.matrix.factorization.1}--\eqref{infinite.Hurwitz.matrix.factorization.2},
we obtain $H(p,q)=H(0,1)\mathcal{T}(p)$. So in this case $H(p,q)$ is
also totally nonnegative.

Conversely, let the matrix $H(p,q)$ be totally nonnegative and
$q(z)\equiv 0$. In this case, $\mathcal{T}(p)$ is also totally
nonnegative as a submatrix of $H(p,q)$. By Theorem~\ref{Th.Shoenberg},
$p$ has only nonpositive zeros, as required.

Let now the matrix $H(p,q)$ be totally nonnegative and let
$q(z)\not\equiv 0$. All submatrices of $H(p,q)$ are also totally
nonnegative. Since matrices $\mathcal{T}(p)$ and $\mathcal{T}(q)$
are submatrices of $H(p,q)$ constructed using all its columns and
all even or all odd rows, they are totally nonnegative. Thus,
according to Theorem~\ref{Th.Shoenberg}, the polynomials~$p$ and
$q$ have only nonpositive zeros. Let $g$ be the greatest common
divisor of $p$ and $q$, so that $p=\widetilde{p}g$ and
$q=\widetilde{q}g$. Then
Theorems~\ref{Th.Hurwitz.matrix.with.gcd} and~\ref{Th.Shoenberg}
imply that
$H(p,q)=H(\widetilde{p},\widetilde{q})\mathcal{T}(g)$, where the
matrix $\mathcal{T}(g)$ is totally nonnegative since $g$ has only
nonpositive zeros.

First, assume that $\deg p<\deg q$, that is, $R(\infty)=0$, and
use notation
\begin{eqnarray*}\label{Th.Hurwitz.Matrix.Total.Nonnegativity.proof.1}
p(z)&\bdeq &f_0(z) \;\,\bdeq \;\, a_0^{(0)}z^n+a_1^{(0)}z^{n-1}+\cdots+a_n^{(0)}, \\
\label{Th.Hurwitz.Matrix.Total.Nonnegativity.proof.2}%
q(z)&\bdeq &f_1(z) \;\, \bdeq \;\, a_0^{(1)}z^{n-1}+a_1^{(1)}z^{n-2}+\cdots+a_{n-1}^{(1)},
\end{eqnarray*}
%
%
%
%
%
%
%\begin{equation*}\label{Th.Hurwitz.Matrix.Total.Nonnegativity.proof.1}
%p(z)=f_0(z)=a_0^{(0)}z^n+a_1^{(0)}z^{n-1}+\cdots+a_n^{(0)},
%\end{equation*}
%
%\begin{equation*}\label{Th.Hurwitz.Matrix.Total.Nonnegativity.proof.2}
%q(z)=f_1(z)=a_0^{(1)}z^{n-1}+a_1^{(1)}z^{n-2}+\cdots+a_{n-1}^{(1)},
%\end{equation*}
%
where $a_0^{(0)}=a_0>0$ by assumption (see~\eqref{polynomial.11}).
From the total nonnegativity of the matrix $H(p,q)$, which is the
same as $H(f_0,f_1)$, we have $a_i^{(0)}\geq 0$,
$a_{i-1}^{(1)}\geq 0$, $i=1,2,\ldots,n$.

Show that $\deg f_1=n-1$. To this end, let us introduce notation
$H_0\eqbd H(f_0,f_1)$ and suppose that
$a_{0}^{(1)}=a_{1}^{(1)}=\cdots=a_{j-1}^{(1)}=0$ and\footnote{This
is possible since $q(z)\not\equiv0$ by assumption.}
$a_{j}^{(1)}>0$ for some integer $1<j\leq n-1$. In this case,
we have
\begin{equation*}\label{Th.Hurwitz.Matrix.Total.Nonnegativity.proof.3}
H_0\begin{pmatrix}
    1 &2 &3\\
    1 &2 &j+2\\
\end{pmatrix}=\begin{vmatrix}
    a_{0}^{(0)} &a_{1}^{(0)} &a_{j+1}^{(0)}\\
         0      &      0     &a_{j}^{(1)}\\
         0      &a_{0}^{(0)} &a_{j}^{(0)}\\
\end{vmatrix}=-\left(a_{0}^{(0)}\right)^2a_{j}^{(1)}<0.
\end{equation*}
This inequality contradicts the total nonnegativity of the matrix
$H_0$. Thus, we obtain that $a_0^{(1)}>0$ and, therefore, the
polynomial $f_1$ is of exact degree $n-1$.

Now we can perform the first step of the
algorithm~\eqref{doubly.regular.Euclidean.algorithm}--\eqref{auxiliary.quotients}
to get the next polynomial $f_2$:
\begin{equation}\label{Th.Hurwitz.Matrix.Total.Nonnegativity.proof.4}
f_2(z)=f_0(z)-c_1zf_1(z)=a_0^{(2)}z^{n-1}+a_1^{(2)}z^{n-2}+\cdots+a_{n-1}^{(2)},
\end{equation}
where $c_1=\dfrac{a_{0}^{(0)}}{a_{0}^{(1)}}>0$. As in the proof of
Theorem~\ref{Th.Stieltjes fraction.global.criterion}, we obtain
the factorization
\begin{equation}\label{Th.Hurwitz.Matrix.Total.Nonnegativity.proof.5}
H_0=J(c_1)H_1,
\end{equation}
where $H_1\eqbd H(f_2,f_1)$ and the matrix $J(c_1)$ is defined as
in~\eqref{def.matrix.J(c)}. If $f_2(z)\equiv0$, then
$f_1=\gcd(f_0,f_1)$
%
%, $H_0=J(c_1)H(0,1)\mathcal{T}(f_1)$, where
%$H(0,1)$ is defined by~\eqref{def.matrix.J(c)},
%
and the function $R(z)=\dfrac1{c_1z}$ is an \textit{R}-function of
negative type. So, in this case, the theorem is proved.
If~$f_2(z)\not\equiv0$, then
from~\eqref{Th.Hurwitz.Matrix.Total.Nonnegativity.proof.4} and
from the total nonnegativity of the matrix $H(f_0,f_1)$ it follows
that
\begin{equation*}\label{Th.Hurwitz.Matrix.Total.Nonnegativity.proof.6}
a_i^{(2)}=\dfrac1{a_0^{(1)}}\cdot\begin{vmatrix}
    a_{0}^{(1)} &a_{i+1}^{(1)}\\
    a_{0}^{(0)} &a_{i+1}^{(0)}\\
\end{vmatrix}=
\dfrac1{a_0^{(0)}a_0^{(1)}}\cdot H_0\begin{pmatrix}
    1 &2 &3\\
    1 &2 &i+3\\
\end{pmatrix}\geq0,\quad i=0,1,\ldots,n-1.
\end{equation*}
Suppose that $\deg f_2<n-1$. Then there exists an integer $j$,
$1\leq j\leq n-1$, such that the coefficients
$a_{0}^{(2)}=a_{1}^{(2)}=\ldots=a_{j-1}^{(2)}=0$ and
$a_{j}^{(2)}>0$. Then
from~\eqref{Th.Hurwitz.Matrix.Total.Nonnegativity.proof.5} we
obtain
\begin{equation*}\label{Th.Hurwitz.Matrix.Total.Nonnegativity.proof.7}
H_0\begin{pmatrix}
    1 &2 &3 &4\\
    1 &2 &3 &j+3\\
\end{pmatrix}=c_1\cdot
H_1\begin{pmatrix}
    1 &3 &4 &5\\
    1 &2 &3 &j+3\\
\end{pmatrix}=
c_1\left(a_0^{(1)}\right)^{2}\cdot
\begin{vmatrix}
          0     &a_{j}^{(2)}\\
    a_{0}^{(1)} &a_{j}^{(1)}\\
\end{vmatrix}<0,
\end{equation*}
which contradicts the total nonnegativity of the matrix $H_0$. So,
the coefficient $a_0^{(2)}$ must be positive and, therefore, $\deg
f_2=n-1$. Thus, we can perform the next step of the
algorithm~\eqref{doubly.regular.Euclidean.algorithm}--\eqref{auxiliary.quotients}
to obtain the next polynomial $f_3$:
\begin{equation}\label{Th.Hurwitz.Matrix.Total.Nonnegativity.proof.8}
f_3(z)=f_1(z)-c_2f_2(z)=a_0^{(3)}z^{n-2}+a_1^{(3)}z^{n-3}+\cdots+a_{n-2}^{(3)},
\end{equation}
where $c_2=\dfrac{a_{0}^{(1)}}{a_{0}^{(2)}}>0$. The
formul\ae~\eqref{Th.Hurwitz.Matrix.Total.Nonnegativity.proof.5}--\eqref{Th.Hurwitz.Matrix.Total.Nonnegativity.proof.8}
imply
\begin{equation*}\label{Th.Hurwitz.Matrix.Total.Nonnegativity.proof.9}
\begin{array}{c}
a_i^{(3)}=\dfrac1{a_0^{(2)}}\cdot
\begin{vmatrix}
    a_{0}^{(2)} &a_{i+1}^{(2)}\\
    a_{0}^{(1)} &a_{i+1}^{(1)}\\
\end{vmatrix}=
\dfrac1{a_0^{(2)}\left(a_0^{(1)}\right)^2}\cdot H_1\begin{pmatrix}
    1 &3 &4 &5\\
    1 &2 &3 &i+4\\
\end{pmatrix} =
\dfrac1{a_0^{(2)}\left(a_0^{(1)}\right)^2}\cdot\dfrac1{c_1}\cdot
H_0\begin{pmatrix}
    1 &2 &3 &4\\
    1 &2 &3 &i+4\\
\end{pmatrix}
\\
 \\
=\dfrac1{a_0^{(0)}a_0^{(1)}a_0^{(2)}}\cdot
H_0\begin{pmatrix}
    1 &2 &3 &4\\
    1 &2 &3 &i+4\\
\end{pmatrix}\geq 0, \qquad  \;\;
i=0,1,\ldots,n-2.\qquad \qquad\qquad  \qquad \quad \qquad\qquad
\end{array}
\end{equation*}
For the matrix $H(f_0,f_1)$, we have the factorization
$$H_0=J(c_1)J(c_2)H_2,$$ where $H_2=H(f_2,f_1)$.

Now let us assume that we have a sequence of polynomials $f_0$,
$f_1$, $\dots$, $f_m$ $(1\leq m<2r)$ constructed from the
polynomials $f_0$ and $f_1$ by the
algorithm~\eqref{doubly.regular.Euclidean.algorithm}--\eqref{auxiliary.quotients},
that is,
\begin{equation}\label{Th.Hurwitz.Matrix.Total.Nonnegativity.proof.10}
\begin{array}{lcll}
f_{2i}(z)& = & f_{2i-2}(z)-c_{2i-1}zf_{2i-1}(z), &  \quad i=1,2,\ldots,\left\lfloor\dfrac{m}2\right\rfloor\\[2mm]
f_{2i+1}(z) &= & f_{2i-1}(z)-c_{2i}f_{2i}(z), & \quad
i=1,2,\ldots,\left\lceil\dfrac{m}2\right\rceil-1,
\end{array}
\end{equation}
where $\deg f_{2i-1}=\deg f_{2i}=n-i$. Let us also assume that all
the coefficients $a_i^{(j)}$ of each polynomial~$f_j$ are
nonnegative. Then the coefficients $c_j$
in~\eqref{Th.Hurwitz.Matrix.Total.Nonnegativity.proof.10} are
positive and determined by the formula
\begin{equation}\label{Th.Hurwitz.Matrix.Total.Nonnegativity.proof.11}
c_j=\dfrac{a_0^{(j-1)}}{a_0^{(j)}}>0,\qquad j=1,2,\ldots,m-1.
\end{equation}
Moreover, if we set
\begin{equation*}\label{Th.Hurwitz.Matrix.Total.Nonnegativity.proof.12}
H_{m-1}\eqbd\begin{cases}
\; H(f_{m-1},f_{m})& \; {\rm if}\ m\ \text{is odd},\\
\; H(f_{m},f_{m-1})& \; {\rm if}\ m\ \text{is even},
\end{cases}
\end{equation*}
then the matrix $H_0$ has the following factorization
\begin{equation}\label{Th.Hurwitz.Matrix.Total.Nonnegativity.proof.13}
H_0=J(c_1)J(c_2)\cdots J(c_{m-1})H_{m-1},
\end{equation}
as it follows from the proof of Theorem~\ref{Th.Stieltjes
fraction.global.criterion}.

Let us perform the next step of the
algorithm~\eqref{doubly.regular.Euclidean.algorithm}--\eqref{auxiliary.quotients}
and obtain the next polynomial $f_{m+1}$:
\begin{equation}\label{Th.Hurwitz.Matrix.Total.Nonnegativity.proof.14}
f_{m+1}(z)=
\begin{cases}
\; f_{m-1}(z)-c_mzf_m(z) & \; \text{if}\ m\ \text{is odd},\\
\; f_{m-1}(z)-c_mf_m(z) \; \, & \; \text{if}\ m\ \text{is even},
\end{cases}
\end{equation}
and denote the coefficients of $f_{m+1}$ by $a_i^{(m+1)}$. If
$f_{m+1}(z)\equiv0$, then $m=2r-1$ or $m=2r-2$ ($r$ is the number
of poles of the function $R=\dfrac{f_1}{f_0}$) and
$f_m=\gcd(f_0,f_1)$. Now the formula~\eqref{Th.Hurwitz.Matrix.Total.Nonnegativity.proof.11} and
Corollaries~\ref{corol.R-function.Stieltjes.fractions.negative.poles}
and~\ref{corol.R-function.Stieltjes.fractions.nonpositive.poles}
show that $R$ is an \textit{R}-function of negative type with
nonnegative poles, which completes the proof  in this case. If
$f_{m+1}(z)\not\equiv0$, then it follows
from~\eqref{Th.Hurwitz.Matrix.Total.Nonnegativity.proof.13}--\eqref{Th.Hurwitz.Matrix.Total.Nonnegativity.proof.14}
and from the total nonnegativity of the matrix $H_0$ that
\begin{equation*}\label{Th.Hurwitz.Matrix.Total.Nonnegativity.proof.15}
\begin{array}{c}
a_i^{(m+1)}=\dfrac1{a_0^{(0)}a_0^{(1)}\cdots a_0^{(m)}}\cdot
H_0\begin{pmatrix}
    1 &2 &\dots &m+1 &m+2\\
    1 &2 &\dots &m+1 &i+m+2\\
\end{pmatrix}\geq0,\quad
i=0,1,2,\ldots, n{-}\left\lfloor\dfrac{m}2\right\rfloor{-}1.
\end{array}
\end{equation*}
We will show that $a_0^{(m+1)}>0$. In fact, if we suppose that it
is not true, then there exists a number $j$ $(1\leq
j\leq n{-}\left\lfloor\dfrac{m}2\right\rfloor{-}1)$ such that
$a_{0}^{(m+1)}=a_{1}^{(m+1)}=\cdots=a_{j-1}^{(m+1)}=0$ and
$a_{j}^{(m+1)}>0$.
Then~\eqref{Th.Hurwitz.Matrix.Total.Nonnegativity.proof.13} yields
\begin{equation*}\label{Th.Hurwitz.Matrix.Total.Nonnegativity.proof.16}
%\begin{array}{c}
H_0\begin{pmatrix}
    1 &2 &\dots &m+2 &m+3\\
    1 &2 &\dots &m+2 &j+m+2\\
\end{pmatrix}=c_1c_2^2\cdots c_m^m\left(a_0^{(m)}\right)^{m+1}\cdot
\begin{vmatrix}
          0     &a_{j}^{(m+1)}\\
    a_{0}^{(m)} &a_{j}^{(m)}\\
\end{vmatrix}<0.
%\end{array}
\end{equation*}
This inequality contradicts the total nonnegativity of the matrix
$H_0$. Consequently, $a_0^{(m+1)}>0$ and we can run the next step
of the
algorithm~\eqref{doubly.regular.Euclidean.algorithm}--\eqref{auxiliary.quotients}
to obtain the next polynomial $f_{m+2}$.

Thus, step by step we construct a sequence of positive numbers
$c_1,c_2,\ldots,c_k$, where $k=2r$ if $|R(0)|<\infty$, and
$k=2r-1$ otherwise. These numbers are exactly the coefficients of
the Stieltjes continued fraction expansion of the function $R$. In
this case, $R$ is an \textit{R}-function of negative type with
nonpositive poles according to
Corollaries~\ref{corol.R-function.Stieltjes.fractions.negative.poles}
and~\ref{corol.R-function.Stieltjes.fractions.nonpositive.poles},
as required.

If $\deg p=\deg q$, then $R(\infty)=c_0=\dfrac{b_0}{a_0}>0$, since
$H(p,q)$ is totally nonnegative. If we denote $f_{-1}=p$ and
$f_0=q$ and run the
algorithm~\eqref{doubly.regular.Euclidean.algorithm}--\eqref{auxiliary.quotients}
as before, then we obtain that the~function $R=p/q$ is an
\textit{R}-function of negative type with nonpositive poles. \eop

\begin{remark}
Note that the easier direction $1)\Longrightarrow2)$ of
Theorem~\ref{Th.Hurwitz.Matrix.Total.Nonnegativity} was proved
in~\cite[Proposition~3.22]{Fisk} as a generalization of results
due to Asner, Kemperman and
Holtz~\cite{Asner,{Kemperman},{Holtz1}}, but the more complicated
implication $2)\Longrightarrow1)$ appears to be new.
\end{remark}

For finite matrices of Hurwitz type, there is no criterion analogous to
Theorem~\ref{Th.Hurwitz.Matrix.Total.Nonnegativity}. However, the
following two theorems can be derived as straightforward consequences of
Theorem~\ref{Th.Hurwitz.Matrix.Total.Nonnegativity}.

\begin{theorem}\label{corol.Hurwitz.Matrix.Total.Nonnegativity}
If the polynomials $p$ and $q$ defined
in~\eqref{polynomial.11}--\eqref{polynomial.12}
%can be represented as follows
%$p=\widetilde{p}g$ and $q=\widetilde{q}g$ such that
%$\widetilde{p}$ and $\widetilde{q}$
have only nonpositive zeros,
%and the matrix $\mathcal{H}_{2l}(g,0)$, where $l=\deg g$, is
%totally nonnegative. The
and the function~$R=q/p$ is either an $R$-function of
negative type or identically zero, then the finite matrix of Hurwitz
type $\mathcal{H}_{k}(p,q)$ is totally nonnegative. Here $k=2n$ if
$\deg q<\deg p$, and $k=2n+1$ if $\deg q=\deg p$.
\end{theorem}
\begin{theorem}[\textbf{Total Nonnegativity of the Finite Hurwitz Matrix}]\label{corol.Hurwitz.Matrix.Total.Nonnegativity.2}
Given the polynomials $p$~and~$q$ defined
in~\eqref{polynomial.11}--\eqref{polynomial.12}, the
function~$R=q/p$ is an \textit{R}-function with exactly $n$
negative poles if and only if the finite matrix of Hurwitz type
$\mathcal{H}_{k}(p,q)$ is nonsingular and totally nonnegative.
Here $k=2n$ if $\deg q<\deg p$, and $k=2n+1$ if $\deg q=\deg p$.
\end{theorem}
\proof%
Indeed, if the function $R$ is an \textit{R}-function with
exactly $n$ \textit{negative poles} poles, then the function
$zR(z)$ also has exactly $n$ poles, so the corresponding Hankel
minor $\widehat{D}_n(R)$ is nonzero according to
Theorem~\ref{Th.Hankel.matrix.rank.1}. By
Theorem~\ref{Th.relations.Delta.and.D}, this means that
$\det\mathcal{H}_{k}(p,q)\neq0$, so the matrix
$\mathcal{H}_{k}(p,q)$ is nonsingular. But $\mathcal{H}_{k}(p,q)$
is totally nonnegative as a submatrix of the totally nonnegative
matrix $H(p,q)$ (see
Theorem~\ref{Th.Hurwitz.Matrix.Total.Nonnegativity}).

Conversely, let the matrix $\mathcal{H}_{k}(p,q)$ be nonsingular
and totally nonnegative. The nonsingularity of the matrix
$\mathcal{H}_{k}(p,q)$ implies that the function $R=q/p$ has
exactly $n=\deg p$ poles, that is, the polynomials $p$ and $q$ are
coprime. Now by the same methods as those used in the proof of
Theorem~\ref{Th.Hurwitz.Matrix.Total.Nonnegativity}, we can show
that the total nonnegativity of the matrix
$\mathcal{H}_{k}(p,q)$ implies that the function $R$ has a
Stieltjes continued fraction
expansion~\eqref{Stieltjes.fraction.for.real.functions.1}--\eqref{Stieltjes.fraction.for.real.functions.2}
with positive coefficients, which is equivalent to the function
$R=q/p$ being an \textit{R}-function with negative poles, according
to
Corollary~\ref{corol.R-function.Stieltjes.fractions.negative.poles}.
\eop

In the particular case when $p$ and $q$ are the even and odd parts of
some polynomial,
Theorem~\ref{corol.Hurwitz.Matrix.Total.Nonnegativity.2} was
first established by Asner~\cite{Asner}.

Let the polynomials $p$ and $q$ be defined
in~\eqref{polynomial.11}--\eqref{polynomial.12} and let $k =2n$ if
$\deg q<\deg p$, and $k =2n+1$ if $\deg q=\deg p$. In the same way
as in Theorem~\ref{Th.Hurwitz.matrix.with.gcd}, one can show the
following: if the polynomials have a common divisor $g$ of degree
$l$ such that $p=\widetilde{p}g$ and $q=\widetilde{q}g$, then
\begin{equation}\label{Hurwitz.matrix.bad.factorization}
\mathcal{H}_{k}(p,q)=\mathcal{H}_{k}(\widetilde{p},\widetilde{q})\mathcal{T}_{k}(g),
\end{equation}
where the matrix $\mathcal{H}_{k}(\widetilde{p},\widetilde{q})$ is
the $k\times k$ principal submatrix of the infinite matrix
$H(\widetilde{p},\widetilde{q})$ indexed by rows (and columns) $2$
through $k+1$
%defined
%in~\eqref{Hurwitz.matrix.infinite.case.1}
, and the matrix $\mathcal{T}_{k}(g)$ is the $k\times k$ leading
principal submatrix of the matrix $\mathcal{T}(g)$ defined
in~\eqref{Toeplitz.infinite.matrix}.

If the matrix $\mathcal{H}_{k}(p,q)$ is singular and totally
nonnegative, then $\mathcal{H}_{k}(p,q)$ can be represented as
in~\eqref{Hurwitz.matrix.bad.factorization}, where the polynomials
$\widetilde{p}$ and $\widetilde{q}$\; have only nonpositive zeros
and $\widetilde{R}=\widetilde{q}/\widetilde{p}$\; is either an
$R$-function or $\widetilde{R}(z)\equiv 0$, but the
polynomial~$g=\gcd(p,q)$ has no nonpositive zeros or $g(z)\equiv
{\rm const}\neq0$. This factorization of the totally nonnegative
matrix $\mathcal{H}_{k}(p,q)$ is possible, for example, if all
minors of order $\leq k$ of the infinite matrix
$\mathcal{T}(g)$ are nonnegative.

\begin{remark}
If all minors of order $\leq k$ of the infinite matrix
$\mathcal{T}(g)$ are nonnegative, then the sequence of the
coefficients of the polynomial $g$ is called $k$-\textit{times
positive} or $k$-\textit{positive}. If $\mathcal{T}(g)$ is totally
nonnegative, then the sequence of the coefficients of the
polynomial $g$ is called \textit{totally
positive}~\cite{Schoenberg.m.pos.1,{Schoenberg.m.pos.2}}. The
functions generating $k$-positive (totally positive) sequences are
usually denoted by $PF_k$ ($PF_\infty$).
\end{remark}

Based on the results above, we make the following conjecture.

\begin{conjecture}
Given two polynomials $p$ and $q$ defined
in~\eqref{polynomial.11}--\eqref{polynomial.12}, the finite
matrix\footnote{Here $k=2n$ if $\deg q<\deg p$, and $k=2n+1$ if
$\deg q=\deg p$.} $\mathcal{H}_k(p,q)$ is totally nonnegative if
and only if\; $\widetilde{p}$ and $\widetilde{q}$\; have only
nonpositive zeros and
$\widetilde{R}=\widetilde{q}/\widetilde{p}$\; is either an
\textit{R}-function or $\widetilde{R}(z)\equiv 0$, and the
polynomial $g=\gcd(p,q)$ has no real zeros  and belongs to
the class~$PF_{k-\deg g}$.
\end{conjecture}

Theorems~\ref{Th.Hurwitz.Matrix.Total.Nonnegativity}
and~\ref{corol.Hurwitz.Matrix.Total.Nonnegativity} imply the
following corollaries.

\begin{corol}\label{Th.Hurwitz.Matrix.Total.Nonnegativity.corol}
The following conditions are equivalent:
\begin{itemize}
\item[$1)$] The polynomials $p$ and $q$ defined by~\eqref{polynomial.11}--\eqref{polynomial.12} are coprime and have only
negative zeros, and the function~$R=q/p$ is an \textit{R}-function
of negative type.
\item[$2)$] The infinite matrix of Hurwitz type $H(p,q)$ defined by~\eqref{Hurwitz.matrix.infinite.case.1}--\eqref{Hurwitz.matrix.infinite.case.2}
is totally nonnegative and $\eta_{k}(p,q)>0$, where $k=2n$ if
$\deg q<\deg p$, and $k=2n+1$ if $\deg q=\deg p$.
\end{itemize}
\end{corol}
\begin{corol}\label{Th.Hurwitz.Matrix.Total.Nonnegativity.corol.2}
The following conditions are equivalent:
\begin{itemize}
\item[$1)$] The polynomials $p$ and $q$ defined by~\eqref{polynomial.11}--\eqref{polynomial.12} are coprime and have only
nonpositive roots, and the function~$R=q/p$ is an
\textit{R}-function of negative type.
\item[$2)$] The finite matrix of Hurwitz type $H_{k}(p,q)$
is totally nonnegative of rank $k-1$, where $k=2n$ if $\deg
q<\deg p$, and $k=2n+1$ if $\deg q=\deg p$.
\end{itemize}
\end{corol}

Sometimes it is convenient to use the inverse indexing of polynomial
coefficients. We now state and prove a result analogous to
Theorem~\ref{Th.Hurwitz.Matrix.Total.Nonnegativity}, using
this alternative ordering of  coefficients.

\begin{corol}\label{corol.Hurwitz.Matrix.Total.Nonnegativity.infinite.matrix.1}
For the polynomials
\begin{eqnarray}\label{corol.Hurwitz.Matrix.Total.Nonnegativity.infinite.matrix.1.poly.1}
& g(z)\;=\; a_0+a_1z+a_2z^2+\cdots+a_{n-1}z^{n-1}+a_nz^n, &
a_n>0,\ \  a_0\neq0, \\
\label{corol.Hurwitz.Matrix.Total.Nonnegativity.infinite.matrix.1.poly.2}%
& h(z)\; = \; b_1+b_2z+\cdots+b_{n-1}z^{n-2}+b_nz^{n-1}, \qquad \;\; &
b_n>0,
\end{eqnarray}
the following conditions are equivalent:
\begin{itemize}
\item[$1)$] The polynomials $g$ and $h$ have only
negative zeros, and the
function~$R=h/g$ is an
\textit{R}-function of negative type.
\item[$2)$] The infinite matrix
\begin{equation*}\label{corol.Hurwitz.Matrix.Total.Nonnegativity.infinite.matrix.1.matrix}
H_{\infty}(g,h) \eqbd
\begin{pmatrix}
a_0&a_1&a_2&a_3&a_4&a_5&\dots\\
0  &b_1&b_2&b_3&b_4&b_5&\dots\\
0  &a_0&a_1&a_2&a_3&a_4&\dots\\
0  &0  &b_1&b_2&b_3&b_4&\dots\\
\vdots &\vdots  &\vdots  &\vdots &\vdots  &\vdots &\ddots
\end{pmatrix}
\end{equation*}
is totally nonnegative.
\end{itemize}
\end{corol}
\proof%
Let the polynomials $g$ and $h$ have only negative zeros such
that~$R=h/g$ is an $R$-function of negative type with exactly
$m\,(\leq n)$ poles. Therefore, by
Theorem~\ref{Th.R-function.general.properties}, $R$ can be
represented as follows
\begin{equation*}\label{corol.Hurwitz.Matrix.Total.Nonnegativity.infinite.matrix.1.proof.1}
{R}(z)=\dfrac{b_1+b_2z+\cdots+b_{n-1}z^{n-2}+b_nz^{n-1}}{a_0+a_1z+a_2z^2+\cdots+a_{n-1}z^{n-1}+a_nz^n}=\sum_{j=1}^{m}\dfrac{\alpha_j}{z+\lambda_j},\quad
\alpha_j,\lambda_j>0.
\end{equation*}
All zeros of this function are also negative. Therefore,
$b_1\neq0$. Consider the function
\begin{equation}\label{corol.Hurwitz.Matrix.Total.Nonnegativity.infinite.matrix.1.proof.2}
\widetilde{R}(z)=\dfrac1{z} R \left(\dfrac1{z}\right)=\dfrac{b_1z^{n-1}+b_2z^{n-2}+\cdots+b_{n-1}z+b_n}{a_0z^n+a_1z^{n-1}+a_2z^{n-2}+\cdots+a_{n-1}z+a_n}=
\sum_{j=1}^{m}\dfrac{\alpha_j/\lambda_j}{z+1/\lambda_j},\quad
\alpha_j,\lambda_j>0.
\end{equation}
Thus, if $R$ is an \textit{R}-function of negative
type with negative poles, then $\widetilde{R}$ is also an \textit{R}-function
of negative type with negative poles. It is easy to show that the
converse statement is also valid. Now by
Theorem~\ref{Th.Hurwitz.Matrix.Total.Nonnegativity} and
by~\eqref{corol.Hurwitz.Matrix.Total.Nonnegativity.infinite.matrix.1.proof.2},
we obtain the equivalence of the conditions $1)$ and $2)$ of the
theorem.
\eop
In the same way, one can prove the following corollary.
\begin{corol}\label{corol.Hurwitz.Matrix.Total.Nonnegativity.infinite.matrix.2}
For the polynomials
\begin{eqnarray}\label{corol.Hurwitz.Matrix.Total.Nonnegativity.infinite.matrix.2.poly.1}
& g(z)\; =\; a_0+a_1z+a_2z^2+\cdots+a_{n-1}z^{n-1}+a_nz^n, &
a_n>0,\\
\label{corol.Hurwitz.Matrix.Total.Nonnegativity.infinite.matrix.2.poly.2}%
& h(z)\; = \; b_0+b_1z+b_2z^2+\cdots+b_{n-1}z^{n-1}+b_nz^{n}, \; &
b_n>0,\ \  b_0\neq 0,
\end{eqnarray}
the following conditions are equivalent:
\begin{itemize}
\item[$1)$] The polynomials $g$ and $h$ have only
negative zeros, and the
function~$R=g/h$ is an
\textit{R}-function of negative type.
\item[$2)$] The infinite matrix
\begin{equation*}\label{corol.Hurwitz.Matrix.Total.Nonnegativity.infinite.matrix.2.matrix}
{H}_{\infty}(g,h)=
\begin{pmatrix}
b_0&b_1&b_2&b_3&b_4&b_5&\dots\\
0  &a_0&a_1&a_2&a_3&a_4&\dots\\
0  &b_0&b_1&b_2&b_3&b_4&\dots\\
0  &0  &a_0&a_1&a_2&a_3&\dots\\
\vdots &\vdots  &\vdots  &\vdots &\vdots  &\vdots &\ddots
\end{pmatrix}
\end{equation*}
is totally nonnegative.
\end{itemize}
\end{corol}

Corollaries~\ref{corol.Hurwitz.Matrix.Total.Nonnegativity.infinite.matrix.1}--\ref{corol.Hurwitz.Matrix.Total.Nonnegativity.infinite.matrix.2}
imply the following two theorems, which, in fact, define a class
of interlacing preservers (see~\cite{Fisk} and references there)
and imply one theorem originally proved by P\'olya.

\begin{theorem}\label{Th.interlacity.preserving.1}
Let the polynomials $g$ and $h$ defined
in~\eqref{corol.Hurwitz.Matrix.Total.Nonnegativity.infinite.matrix.1.poly.1}--\eqref{corol.Hurwitz.Matrix.Total.Nonnegativity.infinite.matrix.1.poly.2}
have only negative zeros, and let the function~$R=h/g$ be an
\textit{R}-function of negative type. Given any two positive
integers $r$ and $l$ such that $rl\leq n<(l+1)r$, the
polynomials
\begin{eqnarray}\label{Th.interlacity.preserving.1.poly.1}
g_{r,l}(z)&\eqbd &a_0+a_{r}z+a_{2r}z^2+a_{3r}z^3+\cdots+a_{rl}z^l,\\
\label{Th.interlacity.preserving.1.poly.2}%
h_{r,l}(z)&\eqbd &b_{r}+b_{2r}z+b_{3r}z^2+\cdots+b_{rl}z^{l-1}
\end{eqnarray}
have only negative zeros, and the function~$R_{r,l}=
h_{r,l}/g_{r,l}$ is an $R$-function of negative type.
\end{theorem}
\proof%
Indeed, if the polynomials $g$ and $h$ have only negative zeros, and if the
function~${R}=h/g$ is an $R$-function of negative type, then by
Corollary~\ref{corol.Hurwitz.Matrix.Total.Nonnegativity.infinite.matrix.2},
the matrix ${H}_{\infty}(g,h)$
is totally nonnegative. Then all its submatrices are totally
nonnegative. In particular, the following infinite submatrix whose
columns are indexed by $1,r+1,2r+1,3r+1,\ldots$ and rows are
indexed by $1,2r+2,4r+3,6r+4,\ldots$
\begin{equation*}\label{Th.interlacity.preserving.1.proof.1}
{H}_{\infty}(g_{r,l},h_{r,l})=
\begin{pmatrix}
a_0&a_r&a_{2r}&a_{3r}&a_{4r}&a_{5r}&\dots\\
0  &b_r&b_{2r}&b_{3r}&b_{4r}&b_{5r}&\dots\\
0  &a_0&a_r&a_{2r}&a_{3r}&a_{4r}&\dots\\
0  &0  &b_{r}&b_{2r}&b_{3r}&b_{4r}&\dots\\
\vdots &\vdots  &\vdots  &\vdots &\vdots  &\vdots &\ddots
\end{pmatrix}
\end{equation*}
is totally nonnegative. Now by
Corollary~\ref{corol.Hurwitz.Matrix.Total.Nonnegativity.infinite.matrix.2},
the
polynomials~\eqref{Th.interlacity.preserving.1.poly.1}--\eqref{Th.interlacity.preserving.1.poly.2}
have only negative zeros, and~${R}_{r,l}= h_{r,l}/g_{r,l}$ is an
$R$-function of negative type, as required.
\eop

The following theorem can be proved in the same fashion.
\begin{theorem}\label{Th.interlacity.preserving.2}
Let the polynomials $g$ and $h$ defined
in~\eqref{corol.Hurwitz.Matrix.Total.Nonnegativity.infinite.matrix.2.poly.1}--\eqref{corol.Hurwitz.Matrix.Total.Nonnegativity.infinite.matrix.2.poly.2}
have only negative zeros, and let the function ${R}=g/h$ be an
$R$-function of negative type. Given any two positive integers
 $r$ and $l$ such that $rl\leq n<(l+1)r$, the polynomials
\begin{eqnarray*}\label{Th.interlacity.preserving.2.poly.1}
g_{r,l}(z)&\eqbd &a_0+a_{r}z+a_{2r}z^2+a_{3r}z^3+\cdots+a_{rl}z^l,\\
\label{Th.interlacity.preserving.2.poly.2}%
h_{r,l}(z)&\eqbd &b_0+b_{r}z+b_{2r}z^2+b_{3r}z^3+\cdots+b_{rl}z^{l}
\end{eqnarray*}
have only negative zeros, and the function~$R_{r,l}\eqbd g_{r,l}/h_{r,l}$
is an  $R$-function of negative type.
\end{theorem}
Theorems~\ref{Th.interlacity.preserving.1}--\ref{Th.interlacity.preserving.2}
imply the following result of P\'olya~\cite[p.~319]{PolyaII}
(also implied by Theorem~\ref{Th.Shoenberg}).
\begin{corol}\label{corol.roots.negativity.preserving}
If the polynomial
\begin{equation}\label{poly.for.negativity}
g(z)=a_0+a_1z+a_2z^2+\cdots+a_{n-1}z^{n-1}+a_nz^n,\quad
a_n>0,
\end{equation}
has only negative zeros, then for any positive integers $r$ and
$l$ satisfying $rl\leq n<(l+1)r$, the polynomial
\begin{equation*}\label{pres.poly}
g_{r,l}(z)= a_0+a_{r}z+a_{2r}z^2+\cdots+a_{rl}z^l
\end{equation*}
also has only negative zeros.
\end{corol}

An analogous result can be obtained for polynomials with only
positive zeros.
\begin{corol}\label{corol.roots.positivity.preserving}
Let the polynomial~\eqref{poly.for.negativity} have only positive
zeros. Then for any positive integers $r$ and $l$ satisfying
$rl\leq n<(l+1)r$, the polynomial~\eqref{pres.poly}
%
%\begin{equation*}
%g_{r,l}(z)=a_0+a_{r}z+a_{2r}z^2+\cdots+a_{rl}z^l
%\end{equation*}
%
also has only positive (negative) zeros whenever $r$ is odd
(even).
\end{corol}

\setcounter{equation}{0}

%%%%%%%%%%%%%%%%%%%%%%%%%%%%%%%%%%%%%%%%%%%%%%%%%%%%%%%%%%%%%%%%%%%
\section{\label{s:counting.zeros}The number of distinct real zeros of polynomials.
Polynomials with all real zeros}
%%%%%%%%%%%%%%%%%%%%%%%%%%%%%%%%%%%%%%%%%%%%%%%%%%%%%%%%%%%%%%%%%%%

In this section we present a sample application of the theory
developed in the previous sections. Using those methods, we analyze
the distribution of zeros of real polynomials with respect to the
real and the imaginary axes.

We should note that polynomials with all real roots have been
studied in control theory, where this property is referred to 
as \emph{aperiodicity}. Among the many relevant papers we note
the work of Jury, Meerov, Fuller and Datta \cite{FullerJury,
FullerRootLoc,FullerAper,FullerRedund,MeerovJury,DattaHankel,BarnettComments} containing
special cases of our results, albeit derived using mostly 
different methods. 

%%%%%%%%%%%%%%%%%%%%%%%%%%%%%%%%%%%%%%%%%%%%%%%%%%%%%%%%%%%%%%%%%%%%
\subsection{\label{s:real.roots.poly.quadratic.forms}The number of
distinct positive, negative and non-real zeros of polynomials.
Stieltjes continued fractions of the logarithmic derivative}
%%%%%%%%%%%%%%%%%%%%%%%%%%%%%%%%%%%%%%%%%%%%%%%%%%%%%%%%%%%%%%%%%%%%%

Consider a real polynomial
\begin{equation}\label{Main.poly.poly}
p(z)=a_0z^n+a_1z^{n-1}+\cdots+a_n,\qquad a_1,\dots,a_n\in\mathbb
R,\ a_0>0,\ n\geq1.
\end{equation}
We denote its logarithmic derivative by $L(z)$:
\begin{equation*}
L(z)\eqbd \dfrac{d\log(p(z))}{dz}=\dfrac{p'(z)}{p(z)}=\dfrac{na_0z^{n-1}+(n-1)a_1z^{n-2}+\cdots
+a_{n-1}}{a_0z^n+a_1z^{n-1}+\cdots+a_n}.
\end{equation*}
If
\begin{equation*}
p(z)=a_0(z-\lambda_1)^{n_1}(z-\lambda_2)^{n_2}\ldots(z-\lambda_m)^{n_m},\qquad
n_1+n_2+\cdots+n_m=n,
\end{equation*}
where $n_j$ is the multiplicity of the zero $\lambda_j$
($j=1,2,\ldots,m$), then the logarithmic derivative of the
polynomial $p$ has the following form
\begin{equation}\label{log.derivative}
L(z)=\sum\limits_{j=1}^{m}\dfrac{n_j}{z-\lambda_j}.
\end{equation}
Moreover, if we expand the function $L$ into its Laurent series at
$\infty$
\begin{equation}\label{series.log.deriv}
L(z)=\dfrac{s_0}z+\dfrac{s_1}{z^2}+\dfrac{s_2}{z^3}+\dfrac{s_3}{z^4}+\cdots,
\end{equation}
then the coefficients $s_j$ are the Newton sums of the polynomial
$p$ (see, for instance,~\cite{Gantmakher}):
\begin{equation}\label{Newton.sums}
s_k=\sum\limits_{j=1}^{m}n_j\lambda_j^k,\quad k=0,1,2,\ldots .
\end{equation}
It is easy to see from~\eqref{log.derivative} that the number of
poles of the function $L$ equals the number of \textit{distinct}
zeros of the polynomial $p$. But the Cauchy index $\IndC(L)$
equals the number of \textit{distinct real} zeros of $p$, since
all poles of $L$ are simple and since the residue of $L$ at each
pole is positive,~\cite{KreinNaimark,{Gantmakher}}. Also
from~\eqref{log.derivative}--\eqref{series.log.deriv} and from
Theorem~\ref{Th.Hankel.matrix.rank.2} it follows that the rank of
the matrix $S=\|s_{i+j}\|_0^{\infty}$ consisting of the Newton
sums~\eqref{Newton.sums} is finite and is equal to~$m\,(\leq
n)$, the number of distinct zeros of $p$.

As before, we denote by $D_j(L)$ ($j=1,2,\ldots,m$) the leading principal
minors of the matrix $S(L)=\|s_{i+j}\|_{0}^{\infty}$
(see~\eqref{Hankel.determinants.1}). Then from
Theorem~\ref{Th.index.via.determinants} we obtain
\begin{equation}\label{index.log.derivative}
\IndC(L)=m-2\SCF(1,D_1(L),D_2(L),\ldots,D_m(L)).
\end{equation}
The above facts and the formula~\eqref{index.log.derivative} imply
the following theorem.
\begin{theorem}[\cite{Gantmakher,KreinNaimark}]\label{Th.number.complex.roots.log.der}
The number of distinct pairs of non-real zeros of the polynomial
$p$ equals $\SCF(1,D_1(L),D_2(L),\ldots,D_m(L))$.
\end{theorem}
\begin{corol}
\begin{equation*}
\SCF(1,D_1(L),D_2(L),\ldots,D_m(L))\leq\left\lfloor
\dfrac{m}2\right\rfloor.
\end{equation*}
\end{corol}

We also consider the determinants $\widehat{D}_j(L)$
($j=1,2,\ldots,m$) defined by~\eqref{Hankel.determinants.2}
and introduce the following notation for counting zeros:

\begin{definition}
Let $r$ denote the number of \textit{distinct real} zeros of the
polynomial $p$ and let $r^{+}$ and $r^{-}$ be the numbers of
\textit{distinct positive} and \textit{negative} zeros of $p$,
respectively.
\end{definition}

Theorem~\ref{Th.number.of.negative.positive.indices.of.real.rational.function.1}
implies the following simple fact.
\begin{theorem}\label{Th.number.of.roots}
Let the numbers $k$ and $l$ be defined as follows\footnote{Recall that
the numbers $\SCF(1,D_1(L),D_2(L),\ldots,D_m(L))$ and
$\SCF(1,\widehat{D}_1(L),\widehat{D}_2(L),\ldots,\widehat{D}_m(L))$
must be calculated according to  Frobenius Rule~\ref{Th.Frobenius}.}:
\begin{equation*}
k= \SCF(1,D_1(L),D_2(L),\ldots,D_m(L)),\qquad l=
\SCF(1,\widehat{D}_1(L),\widehat{D}_2(L),\ldots,\widehat{D}_m(L)),
\end{equation*}
Then the number of distinct pairs of non-real zeros of the
polynomial $p$ equals $k$ and
\begin{equation}\label{number.of.roots}
\begin{split}
r\ \ &=\ \ m-2k;\\
r^{-}&=\ \ l-k;\\
r^{+}&=\begin{cases}
         \; m-k-l & \; \text{if}\;\;\; p(0)\neq 0,\\
         \; m-k-l-1 & \; \text{if}\;\;\; p(0)=0.
       \end{cases}\\
\end{split}
\end{equation}
\end{theorem}

\proof By Theorem~\ref{Th.number.of.roots}, $k$ is the number of
distinct pairs of non-real zeros of the polynomial $p$. Since
$\IndC(L)=r$ as mentioned above, we conclude
\begin{equation}\label{real.roots}
r=m-2k.
\end{equation}
Now we observe that
\begin{equation}\label{positive.negative.roots.general}
r^{+}=\Ind\nolimits_{0}^{+\infty}(L),\qquad
r^{-}=\Ind\nolimits_{-\infty}^{0}(L).
\end{equation}

\vspace{2mm}

\noindent If $p(0)\neq0$, then $r=r^{+}+r^{-}$ and $|L(0)|>0$. So,
from~\eqref{positive.negative.roots.general}
and~\eqref{index.on.positive.line.1}--\eqref{index.on.negative.line.1}
we obtain
\begin{equation*}\label{positive.negative.roots.1}
r^{-}=l-k,\quad r^{+}=m-k-l.
\end{equation*}

\vspace{2mm}

\noindent Now let $p(0)=0$. Then
\begin{equation}\label{real.roots.2}
r=r^{+}+r^{-}+1=m-2k,
\end{equation}
In this case, the number $\sigma_1$ equals $1$ in the
formul\ae~\eqref{index.on.positive.line.2}--\eqref{index.on.negative.line.2}
since all the residues of the function $L$ are positive
(see~\eqref{log.derivative}), whereas the number $\sigma_2$ equals
$\sigma_1$ since all poles of $L$ are simple. Thus,
from~\eqref{positive.negative.roots.general}
and~\eqref{index.on.positive.line.2}--\eqref{index.on.negative.line.2}
we obtain
$\label{positive.negative.roots.2}
r^{-}=l-k,\quad r^{+}=m-k-l-1,
$
as required.
\eop
\begin{corol}
\begin{equation*}
\SCF(1,D_1(L),D_2(L),\ldots,D_m(L))\leq
\SCF(1,\widehat{D}_1(L),\widehat{D}_2(L),\ldots,\widehat{D}_m(L)).
\end{equation*}
\end{corol}
\begin{remark}\label{remark.new.1}
Since $s_0=n$, we have $D_1(L)=s_0=n>0$. This fact means that the
polynomial $p$ has at least one zero. If $D_j(L)=0$ for
$j\geq2$, then $p$ has exactly one zero of multiplicity $n$.
\end{remark}

Our next statement is a slight modification (we use another
continued fraction) and generalization (we cover the
case $p(0)=0$) of Theorem 3.5 from~\cite{Lange} (see
also~\cite{Rogers}). However,~\cite{Lange} uses different methods.
\begin{theorem}\label{Th.number.of.roots.Stieltjes.fraction}
Let the polynomial $p$ be defined by~\eqref{Main.poly.poly}. Then
its logarithmic derivative $L$ has a Stieltjes continued fraction
expansion
\begin{equation}\label{Stiljes_fraction.log.deriv.1}
L(z)=\dfrac1{c_1z+\cfrac1{c_2+\cfrac1{c_{3}z+\cfrac1{\ddots+\cfrac1{T}}}}},\quad
c_j\in\mathbb{R},\;\;\;c_j\neq0,
\end{equation}
where
\begin{equation}\label{Stiljes_fraction.log.deriv.2}
T=\begin{cases}
         \; c_{2m} & \; \text{if}\;\; p(0)\neq0,\\
         \; c_{2m-1}z & \; \text{if}\;\; p(0)=0
       \end{cases}
\end{equation}
if and only if $L$ satisfies the inequalities
\begin{eqnarray*}\label{Stiljes_fraction.log.deriv.3}
& D_j(L)\neq0, & \quad j=1,2,\ldots,m, \\
\label{Stiljes_fraction.log.deriv.4}
& \widehat{D}_j(L)\neq 0, & \quad j=1,2,\ldots,m-1. \\
\label{Stiljes_fraction.log.deriv.5}
& D_j(L)=\widehat{D}_j(L)=0, & \quad j>m.
\end{eqnarray*}
where $m\leq\deg p$.

In that case, $p$ has exactly $m$ distinct zeros. Moreover, if the
number of negative coefficients $c_{2j-1}$ equals~$k$, and the
number of positive coefficients $c_{2j}$ equals~$l$, then the
number of distinct pairs of nonreal zeros of $p$ and
the number of its distinct real, positive and negative zeros are
given by the formul\ae~\eqref{number.of.roots}.
\end{theorem}
\proof The theorem follows immediately from
Theorems~\ref{Th.Stieltjes
fraction.global.criterion},~\ref{Th.number.of.negative.positive.indices.of.real.rational.function.via.S-fraction}
and~\ref{Th.number.of.roots}. \eop

From~\eqref{even.coeff.Stieltjes.fraction.main.formula}--\eqref{odd.coeff.Stieltjes.fraction.main.formula}
it follows that the coefficients $c_i$
in~\eqref{Stiljes_fraction.log.deriv.1}--\eqref{Stiljes_fraction.log.deriv.2}
can be found as follows:
\begin{equation}\label{odd.coeff.Stieltjes.fraction.main.formula.log.der}
c_{2j-1}=\dfrac{\widehat{D}_{j-1}^2(L)}{D_{j-1}(L)\cdot
D_{j}(L)},\quad j=1,2,\ldots,m.
\end{equation}
\begin{equation}\label{even.coeff.Stieltjes.fraction.main.formula.log.der}
c_{2j}=-\dfrac{D_{j}^2(L)}{\widehat{D}_{j-1}(L)\cdot\widehat{D}_{j}(L)},\quad
j=1,2,\ldots,\widetilde{m},
\end{equation}
where $\widetilde{m}= m$ if $p(0)\neq 0$, $\widetilde{m}= m-1$ if
$p(0)=0$, and $D_0(L)\equiv 1$, $\widehat{D}_0(L)\equiv 1$.

Thus, Theorem~\ref{Th.number.of.roots} expresses the numbers of
positive, negative and non-real zeros in terms of the number of
sign changes in the sequences of the minors $D_j(L)$ and $\widehat{D}_j(L)$,
and Theorem~\ref{Th.number.of.roots.Stieltjes.fraction} does the same
in terms of the Stieltjes continued fraction of $L$, provided that $L$
has such a continued fraction expansion. Now we will obtain formul\ae~for
those numbers in terms of the coefficients $a_j$ of the given
polynomial $p$.

Consider the following $2n\times2n$ matrix.
\begin{equation}\label{Hurvitz_like_Matrix}
\mathcal{D}_{2n}(p) \eqbd
\begin{pmatrix}
na_0&(n-1)a_1&(n-2)a_2&\dots& a_{n-1}&     0  &\dots&0&0\\
 a_0&     a_1&     a_2&\dots& a_{n-1}&   a_n  &\dots&0&0\\
0   &    na_0&(n-1)a_1&\dots&2a_{n-2}&a_{n-1} &\dots&0&0\\
0   &     a_0&     a_1&\dots& a_{n-2}&a_{n-1} &\dots&0&0\\
0   &     0  &    na_0&\dots&3a_{n-3}&2a_{n-2}&\dots&0&0\\
0   &     0  &     a_0&\dots& a_{n-3}& a_{n-2}&\dots&0&0\\
\vdots&\vdots&\vdots&\ddots&\vdots&\vdots&\ddots&\vdots&\vdots\\
0   &0       &0       &\dots&na_0&(n-1)a_1&\dots&a_{n-1} &0\\
0   &0       &0       &\dots&a_0&a_1&\dots&a_{n-1} &a_n
\end{pmatrix},
\end{equation}
which is a Hurwitz-type matrix constructed with the coefficients
of the polynomials $p$ and with the coefficients of its derivative
$p'$. According to Definition~\ref{def.Hurwitz.matrix.finite}, the
matrix $\mathcal{D}_{2n}(p)$ is $\mathcal{H}_{2n}(p,q)$
(see~\eqref{Hurwitz.matrix.finite.case.1}), where $q=p'$. Denote
the leading principal minors if the matrix $\mathcal{D}_{2n}(p)$
by $\delta_j(p)$, $j=1,2,\ldots,2n$. We remind the reader that
$\delta_{2n-1}(p)=a_0\mathbf{D}(p)$, where $\mathbf{D}(p)$ is the
discriminant of $p$
(see~\eqref{discriminant.determinant.formula}).

\begin{theorem}\label{Th.number.of.roots.coeffs.poly}
The polynomial $p$ has exactly $m\leq n$ distinct zeros if
and only if
\begin{equation}\label{Determinants.ineq.1}
\delta_{2m-1}(p)\neq0,\quad\text{and}\quad\delta_{j}(p)=0\quad\text{for}\quad
j>2m.
\end{equation}
At the same time, if%\footnote{The numbers
%$\SCF(1,\delta_1(p),\delta_3(p),\ldots,\delta_{2m-1}(p))$ and
%$\SCF(1,\delta_2(p),\delta_4(p),\ldots,\delta_{2m}(p))$ must be
%calculated according to the Frobenius Rule~\ref{Th.Frobenius}.}
%
\begin{equation*}
k= \SCF(1,\delta_1(p),\delta_3(p),\ldots,\delta_{2m-1}(p)),\qquad
l= \SCF(1,\delta_2(p),\delta_4(p),\ldots,\delta_{2m}(p)),
\end{equation*}
then the number of distinct pairs of non-real roots of the
polynomial $p$ equals $k$ and
\begin{equation*}\label{number.of.roots.deltas}
\begin{split}
& r\;\; = \; m-2k\,;\\
&r^{+}= \; l-k\,;\\
&r^{-}=\begin{cases}
         \;m-k-l & \; \text{if }\;\ p(0)\neq 0,\\
         \;m-k-l-1 &\; \text{if }\;\ p(0)=0,
       \end{cases}\\
\end{split}
\end{equation*}
where $r$ is the number of distinct real roots of $p$, and $r^{-}$
and $r^{+}$ are the numbers of distinct negative and positive
roots of $p$, respectively.
\end{theorem}
\proof
By Definition~\ref{def.Hurwitz.matrix.finite}, we have
\begin{equation*}\label{Th.number.of.roots.coeffs.poly.proof.1}
\delta_j(p)=\Delta_j(p,p'),\quad j=1,2,\ldots
\end{equation*}
Therefore, from Theorem~\ref{Th.relations.Delta.and.D}
(see~\eqref{Hurwitz.determinants.relations.finite.2.odd}--\eqref{Hurwitz.determinants.relations.finite.2.even})
we obtain
\begin{equation}\label{Determinants.relations.5}
\delta_{2j-1}(p)=a_0^{2j-1}D_j(L),\quad j=1,2,\ldots,n;
\end{equation}
\begin{equation}\label{Determinants.relations.6}
\delta_{2j}(p)=(-1)^{j}a_0^{2j}\widehat{D}_j(L),\quad
j=1,2,\ldots,n.
\end{equation}

Then, from Theorem~\ref{Th.Hankel.matrix.rank.2}
and~\eqref{Determinants.relations.5} it follows that the rank of
the matrix $S=\|s_{j+k}\|_0^{\infty}$ equals $m$ if and only if
the condition~\eqref{Determinants.ineq.1} holds. At the same time,
we have $\widehat{D}_{m-1}(L)\neq0$, $\widehat{D}_m(L)=0$ if and
only if $p(0)=0$, and $\widehat{D}_m(L)\neq0$ if and only if
$p(0)\neq0$, and $\widehat{D}_j(L)=0$ for $j>m$.

It suffices to note that~\eqref{Determinants.relations.5} implies
that
\begin{equation}\label{k}
k=\SCF(1,\delta_1(p),\delta_3(p),\ldots,\delta_{2m-1}(p))=\SCF(1,D_1(L),D_2(L),\ldots,D_{m}(L))
\end{equation}
since $a_0>0$. But from~\eqref{Determinants.relations.6} we have
\begin{equation}\label{l}
\begin{split}
l\;=\;\SCF(1,\delta_2(p),\delta_4(p),\ldots,\delta_{2m}(p))=\SCF(1,-\widehat{D}_1(L),\widehat{D}_2(L),\ldots,(-1)^m\widehat{D}_{m}(L))\\
%\end{equation*}
%\begin{equation*}
=m-\SCF(1,\widehat{D}_1(L),\widehat{D}_2(L),\ldots,\widehat{D}_{m}(L)), \quad\,
\end{split}
\end{equation}
which gives
$\SCF(1,\widehat{D}_1(L),\widehat{D}_2(L),\ldots,\widehat{D}_{m}(L))=m-l$
if $\widehat{D}_{m}(L)\neq0$, and if\footnote{In this case,
$\widehat{D}_{m-1}(L)\neq0$, according to
Corollary~\ref{corol.zero.pole}.} $\widehat{D}_{m}(L)=0$, then
$\SCF(1,\widehat{D}_1(L),\widehat{D}_2(L),\ldots,\widehat{D}_{m}(L))=m-1-l$.
Now the assertion of the theorem follows from~\eqref{k},~\eqref{l}
and from Theorem~\ref{Th.number.of.roots}. \eop
\begin{remark}\label{remark.new.2}
For the polynomial $p$ we have $\delta_1(p)=a_0D_1(L)=a_0n>0$.
\end{remark}

Using the
formul\ae~\eqref{Determinants.relations.5}--\eqref{Determinants.relations.6},
we can represent the coefficients  $c_i$ of the Stieltjes
continued fraction
expansion~\eqref{Stiljes_fraction.log.deriv.1}--\eqref{Stiljes_fraction.log.deriv.2}
of the function $L$ in terms of the determinants~$\delta_i(p)$:
\begin{equation}\label{coeff.Stieltjes.fraction.main.formula.log.der}
c_{i}=\dfrac{\delta^2_{i-1}(p)}{\delta_{i-2}(p)\cdot\delta_{i}(p)},\quad
i=1,2,\dots,2m,
\end{equation}
where $\delta_{-1}(p)\equiv\dfrac1{a_0}$,  $\delta_{0}(p)\equiv
1$.

%%%%%%%%%%%%%%%%%%%%%%%%%%%%%%%%%%%%%%%%%%%%%%%%%%%%%%%%%%%%%%%%%%%%%%%
\subsection{\label{s:real.roots.poly.coeffs}Polynomials with real zeros}
%%%%%%%%%%%%%%%%%%%%%%%%%%%%%%%%%%%%%%%%%%%%%%%%%%%%%%%%%%%%%%%%%%%%%%%
We now provide several explicit criteria for a
polynomial to have only real zeros. These criteria, just as those developed
before, use the Hankel and Hurwitz minors made of the coefficients of
a given polynomial.

Let us again consider the polynomial
\begin{equation}\label{polynomial.for.real.roots}
p(z)=a_0z^n+a_1z^{n-1}+\cdots+a_n,\qquad a_1,\dots,a_n\in\mathbb
R,\ a_0>0.
\end{equation}
and let
\begin{equation}\label{log.derivative.series}
L(z)=
\dfrac{p'(z)}{p(z)}=\dfrac{s_0}z+\dfrac{s_1}{z^2}+\dfrac{s_2}{z^3}+\cdots
\end{equation}
be its logarithmic derivative. Theorem~\ref{Th.number.of.roots}
directly implies the following results.
\begin{theorem}\label{Th.real-rooted.criteria}
The polynomial $p$ has exactly $m\,(\leq n)$ distinct zeros,
all of which are real, if and only if
\begin{eqnarray*}\label{Th.number.of.real.roots.condition.1}
& D_j(L)>0, & \quad j=1,2,\ldots,m, \\
\label{Th.number.of.real.roots.condition.2}
& D_j(L)=0, & \quad j>m.
\end{eqnarray*}
\end{theorem}
\begin{theorem}\label{Th.number.of.real.roots.}
Let the polynomial $p$ have exactly $m\,(\leq n)$ distinct
zeros all of which are real. Then the number of distinct negative
zeros of the polynomial $p$ equals %\footnote{The number
%$\SCF(1,\widehat{D}_1(L),\widehat{D}_2(L),\ldots,\widehat{D}_m(L))$
%must be calculated according to Frobenius Rule~\ref{Th.Frobenius}.}
$\SCF(1,\widehat{D}_1(L),\widehat{D}_2(L),\ldots,\widehat{D}_m(L))$.
Moreover,
\begin{equation*}\label{Th.number.of.real.roots.condition.3}
\widehat{D}_j(L)=0,\qquad j>m.
\end{equation*}
and
\begin{equation*}\label{Th.number.of.real.roots.condition.4}
\begin{cases}
\; \widehat{D}_m(L)\neq0 & \; \text{if}\quad p(0)\neq0, \\
\; \widehat{D}_m(L)=0\ \ \text{and}\ \  \widehat{D}_{m-1}(L)\neq0 \
& \; \text{if}\quad p(0)=0.
\end{cases}
\end{equation*}
\end{theorem}
\proof The polynomial $p$ takes value $0$ at $0$ if and only if the function $L$ has
a pole at zero. Therefore, the assertion of the theorem follows
from Corollary~\ref{corol.zero.pole} and
Theorem~\ref{Th.number.of.roots}. \eop
\begin{remark}\label{remark.number.of.pos.roots}
According to Theorem~\ref{Th.Descartes.rule.realzero.polys}, if
the polynomial $p$ has only real zeros, then the number of its
positive zeros equals %\footnote{If there are some zero elements in
%$the sequence, then we should count the number
%$\SC^-(a_0,a_1,a_2,\ldots,a_n)$ excluding zero elements from the sequence.}
$\SC^-(a_0,a_1,a_2,\ldots,a_n)$, i.e.,
 the number of strong sign changes in the
sequence of the coefficients of the polynomial $p$. Obviously, the
number of its negative zeros is equal to
$\SR^-(a_0,a_1,a_2,\ldots,a_n)$, i.e., the number of strong sign
retentions in the sequence of its
coefficients.
\end{remark}

Also from Theorem~\ref{Th.number.of.roots.Stieltjes.fraction} we
obtain the following simple corollary.
\begin{corol}\label{corol.number.of.roots.Stieltjes.fraction.real-rooted}
Let the polynomial $p$ have exactly $m\,(\leq n)$ distinct
zeros, all of which are real, and let its logarithmic
derivative~$L$ have the Stieltjes continued fraction
expansion~\eqref{Stiljes_fraction.log.deriv.1}--\eqref{Stiljes_fraction.log.deriv.2}.
Then
\begin{equation}\label{corol.number.of.roots.Stieltjes.fraction.real-rooted.condition}
c_{2j-1}>0,\qquad j=1,2,\ldots,m,
\end{equation}
and the number of positive coefficients $c_{2j}$ equals the number
of negative distinct zeros of the polynomial~$p$. The coefficients
$c_i$ can be found by the
formul\ae~\eqref{odd.coeff.Stieltjes.fraction.main.formula.log.der}--\eqref{even.coeff.Stieltjes.fraction.main.formula.log.der}
or~\eqref{coeff.Stieltjes.fraction.main.formula.log.der}.
\end{corol}

Let us consider the Hankel matrix
\begin{equation}\label{log.der.Hakel.matrix}
S(L)=\|s_{i+j}\|_0^{\infty},
\end{equation}
made of the coefficients of the series~\eqref{log.derivative.series}.

From
Theorems~\ref{Th.R-function.general.properties},~\ref{Th.R-functions.via.strict.total.positivity},~\ref{Th.real-rooted.criteria},~\ref{Th.number.of.real.roots.},
Corollary~\ref{Cor.R-functions.via.sign.regularity}
%and~\ref{corol.number.of.roots.Stieltjes.fraction.real-rooted}
and Remark~\ref{remark.number.of.pos.roots} we obtain the
following straightforward consequences, the first two addressing
the case when all zeros of $p$ are positive, and the next two when all
zeros are negative.
\begin{corol}\label{corol.pos.real.roots.via.total.positivity}
The polynomial $p$ has only positive zeros and exactly
$m\,(\leq n)$ of them are distinct if and only if the matrix
$S(L)$ defined by~\eqref{log.der.Hakel.matrix} is strictly totally
positive of rank~$m$.
\end{corol}
\begin{corol}\label{corol.pos.real.roots}
The polynomial $p$ has only positive zeros and exactly
$m\,(\leq n)$ of them are distinct if and only if the matrix
$S(L)$ is positive definite of rank $m$ and $a_{j-1}a_{j}<0$
\emph($j=1,2,\ldots,n$\emph).
\end{corol}
\begin{corol}\label{corol.negative.real.roots.via.sign.regularity}
The polynomial $p$ has only negative zeros and exactly
$m\,(\leq n)$ of them are distinct if and only if the matrix
$S(L)$ is strictly sign regular of rank~$m$.
\end{corol}
\begin{corol}\label{corol.negative.real.roots}
Let all the coefficients of the polynomial $p$ be of the same
sign. Then all zeros of the polynomial $p$ are negative and
exactly $m\,(\leq n)$ of them are distinct if and only if the
matrix $S(L)$ is positive definite of rank $m$.
\end{corol}

\noindent If $m=n$ in
Theorems~\ref{Th.real-rooted.criteria}--\ref{Th.number.of.real.roots.}
and
Corollaries~\ref{corol.number.of.roots.Stieltjes.fraction.real-rooted}--\ref{corol.negative.real.roots},
then we obtain criteria of reality (negativity, positivity) and
\textit{simplicity} for all zeros of a given polynomial.

Now we give criteria of reality for all zeros of a given
polynomial in terms of the determinants $\delta_j(p)$, which are
the leading principal minors of the matrix~$\mathcal{D}_{2n}(p)$
defined by~\eqref{Hurvitz_like_Matrix} (see
e.g.~\cite{Katsnelson,Yang}).

\begin{theorem}\label{Th.real.roots.ccc}
The polynomial $p$ has exactly $m\,(\leq n)$ distinct zeros,
all of which are real, if and only if
\begin{equation}\label{delta.pos.1}
\delta_{1}(p)>0,\ \delta_{3}(p)>0,\ldots,\ \delta_{2m-1}(p)>0.
\end{equation}
\begin{equation}\label{corol.pos.real.roots.via.Hurwitz.discr.determinants.condition.2}
\delta_{j}(p)=0,\quad j>2m,
\end{equation}
\end{theorem}
\begin{theorem}\label{Th.number.of.real.roots.ccc}
Let the polynomial $p$ have exactly $m\,(\leq n)$ distinct
zeros, all of which are real. If %\footnote{The number
%$\SCF(1,\delta_2(p),\delta_4(p),\ldots,\delta_{2m}(p))$ must be
%calculated according to Frobenius Rule~\ref{Th.Frobenius}.}
$$l= \SCF(1,\delta_2(p),\delta_4(p),\ldots,\delta_{2m}(p)),$$ then the
number of distinct positive zeros of the polynomial $p$ equals $l$
$($when $p(0)\neq0)$ or $l-1$ \emph(when $p(0)=0$\emph). The
number of all positive zeros of $p$, counting multiplicities, is
equal to $\SC^-(a_0,a_1,a_2,\ldots,a_n)$.
\end{theorem}
\proof The theorem follows from
Theorem~\ref{Th.number.of.roots.coeffs.poly},
formul\ae~\eqref{Determinants.relations.5}--\eqref{Determinants.relations.6}
and Remark~\ref{remark.number.of.pos.roots}. \eop
From Theorems~\ref{Th.number.of.real.roots.ccc}
and~\ref{Th.real.roots.ccc} and our previous results we obtain
the following evident consequences:
\begin{corol}\label{corol.pos.real.roots.via.Hurwitz.discr.determinants}
The polynomial $p$ has exactly $m\,(\leq n)$ distinct zeros,
all of which are positive, if and only
if~\eqref{delta.pos.1}--\eqref{corol.pos.real.roots.via.Hurwitz.discr.determinants.condition.2}
hold and
\begin{equation*}\label{corol.pos.real.roots.via.Hurwitz.discr.determinants.condition.1}
\delta_{2j-2}(p)\delta_{2j}(p)<0,\quad j=1,2,\ldots,m,\quad
(\delta_0(p)\eqbd 1).
\end{equation*}
%
%\begin{equation}\label{corol.pos.real.roots.via.Hurwitz.discr.determinants.condition.2}
%\delta_{j}(p)=0,\quad j>2m.
%\end{equation}
%
\end{corol}
\begin{corol}\label{corol.pos.real.roots.via.Hurwitz.discr.determinants.and.coeffs}
The polynomial $p$ has exactly $m\,(\leq n)$ distinct zeros,
all of which are positive, if and only
if~\eqref{delta.pos.1}--\eqref{corol.pos.real.roots.via.Hurwitz.discr.determinants.condition.2}
hold and
\begin{equation*}\label{corol.pos.real.roots.via.Hurwitz.discr.determinants.and.coeffs.condition.1}
a_{j-1}a_{j}<0,\quad j=1,2,\ldots,n.
\end{equation*}
\end{corol}
\begin{corol}\label{corol.neg.real.roots.via.Hurwitz.discr.determinants}
The polynomial $p$ has exactly $m\,(\leq n)$ distinct zeros,
all of which are negative, if and only if the
equalities~\eqref{corol.pos.real.roots.via.Hurwitz.discr.determinants.condition.2}
hold and
\begin{equation*}\label{corol.neg.real.roots.via.Hurwitz.discr.determinants.condition.1}
\delta_{1}(p)>0,\ \delta_{2}(p)>0,\ldots,\ \delta_{2m-1}(p)>0,\
\delta_{2m}(p)>0.
\end{equation*}
\end{corol}
\begin{corol}\label{corol.negative.real.roots.via.Hurwitz.discr.determinants}
Let all the coefficients of the polynomial $p$ be positive. The
polynomial $p$ has exactly $m\,(\leq n)$ distinct zeros, all
of which are negative, if and only
if~\eqref{delta.pos.1}--\eqref{corol.pos.real.roots.via.Hurwitz.discr.determinants.condition.2}
hold.
\end{corol}

For $m=n$,
Theorems~\ref{Th.real.roots.ccc}--\ref{Th.number.of.real.roots.ccc}
and
Corollaries~\ref{corol.pos.real.roots.via.Hurwitz.discr.determinants}--\ref{corol.negative.real.roots.via.Hurwitz.discr.determinants}
become criteria of reality (negativity, positivity) and
\textit{simplicity} for all zeros of a given polynomial.

Corollary~\ref{corol.neg.real.roots.via.Hurwitz.discr.determinants}
with $m=n$ is \textit{per se} the Hurwitz stability criterion for
polynomials whose zeros are real and simple; it is analogous to the
standard Hurwitz criterion of polynomial
stability~\cite[Chapter~XV, Section~6]{Gantmakher}. In turn,
Corollary~\ref{corol.negative.real.roots.via.Hurwitz.discr.determinants}
with $m=n$ is an analogue of the Li\'enard and Chipart
criterion~\cite[Chapter~XV, Section~13]{Gantmakher}. Just as
the criterion of   Li\'enard and Chipart has a number of
versions~\cite[Chapter~XV, Section~13]{Gantmakher}, we can also obtain other
versions of Corollary~\ref{corol.negative.real.roots.via.Hurwitz.discr.determinants},
e.g., as follows:
\begin{corol}\label{stable.poly.simple.roots.3}
Let all the coefficients of the polynomial $p$ be positive. Then
all roots of the polynomial $p$ are negative and distinct if and
only if the following inequalities hold:
\begin{equation}\label{delta.pos.4}
\delta_{2}(p)>0,\ \delta_{4}(p)>0,\ldots,\ \delta_{2n}(p)>0.
\end{equation}
\end{corol}
\proof
The theorem follows from
Theorem~\ref{Th.general.Lienard.Chipart.negative.case.1} applied
to the pair $p,p'$.
\eop
Note that all criteria of Hurwitz stability require $n$
inequalities on the coefficients of a polynomial of degree
$n$~\cite{Barkovsky.2,{Gantmakher}}, while the criteria of
simplicity and negativity of zeros require $2n$ inequalities on
the coefficients of a polynomial of degree $n$.

The similarity between Hurwitz stable polynomials and polynomials
with simple and negative zeros is also displayed by the following
theorem, which is an analogue of the famous fact that the Hurwitz
matrix of a Hurwitz stable polynomial is totally
nonnegative~\cite{Asner,{Kemperman},{Holtz1},Pinkus}.

\begin{theorem}\label{total.positivity.real.roots}
The polynomial $p$ of degree $n$ has only nonpositive zeros
if and only if its matrix $\mathcal{D}_{2n}(p)$ defined
in~\eqref{Hurvitz_like_Matrix} is totally nonnegative.
\end{theorem}
\proof If $n=0$, then the assertion is evident. Let $n\geq1$.
From
Theorems~\ref{Th.real-rooted.criteria}--\ref{Th.number.of.real.roots.}
and~\ref{Th.R-function.general.properties} it follows that the
polynomial $p$ has nonpositive zeros if and only if its
logarithmic derivative is an \textit{R}-function of negative type
with nonpositive poles and negative roots. Since
$\mathcal{D}_{2n}(p)=\mathcal{H}_{2n}(p,p')$, the necessity
direction of the theorem follows from
Theorem~\ref{corol.Hurwitz.Matrix.Total.Nonnegativity} applied to
the pair $(p, p')$. The sufficiency can be proved as in
Theorem~\ref{Th.Hurwitz.Matrix.Total.Nonnegativity} using the
factorization~\eqref{Hurwitz.matrix.bad.factorization} where
${\mathcal T}_k(g)$ is totally nonnegative whenever $g=\gcd (p,p')$.
%, where the
%polynomial $g$ must have nonpositive zeros because of the
%construction of the polynomial $q\eqbd p'$.
%
\eop

From Theorem~\ref{Th.Hurwitz.Matrix.Total.Nonnegativity} we derive
the following corollary.
\begin{corol}\label{total.positivity.real.roots.inf}
The polynomial $p$ of degree $n$ has only nonpositive zeros if and only
if the infinite matrix $H(p,p')$ defined
in~\eqref{Hurwitz.matrix.infinite.case.1} is totally nonnegative.
\end{corol}

Now, the following corollaries about Stieltjes continued fractions
of logarithmic derivatives can be derived from Theorems~\ref{Th.R-function.general.properties},~\ref{Th.real-rooted.criteria}
and~\ref{Th.number.of.real.roots.} and from
Corollary~\ref{corol.number.of.roots.Stieltjes.fraction.real-rooted}:
\begin{corol}\label{corol.pos.roots.Stieltjes.fraction.real-rooted}
The polynomial $p$ has exactly $m\,(\leq n)$ distinct zeros,
all of which are positive, if and only if its logarithmic
derivative $L$ has a Stieltjes continued fraction
expansion~\eqref{Stiljes_fraction.log.deriv.1}--\eqref{Stiljes_fraction.log.deriv.2}
where the
inequalities~\eqref{corol.number.of.roots.Stieltjes.fraction.real-rooted.condition}
hold and where
\begin{equation*}\label{corol.pos.roots.Stieltjes.fraction.real-rooted.condition}
c_{2j}<0,\quad j=1,2,\ldots,m.
\end{equation*}
\end{corol}
\begin{corol}\label{corol.neg.roots.Stieltjes.fraction.real-rooted}
The polynomial $p$ has exactly $m\,(\leq n)$ distinct zeros,
all of which are negative, if and only if its logarithmic
derivative $L$ has a Stieltjes continued fraction
expansion~\eqref{Stiljes_fraction.log.deriv.1}--\eqref{Stiljes_fraction.log.deriv.2}
where
\begin{equation*}\label{corol.neg.roots.Stieltjes.fraction.real-rooted.condition}
c_{i}>0,\qquad i=1,2,\ldots,2m.
\end{equation*}
\end{corol}
\noindent Note that, in
Corollaries~\ref{corol.pos.roots.Stieltjes.fraction.real-rooted}--\ref{corol.neg.roots.Stieltjes.fraction.real-rooted}, all zeros of
the polynomial $p$ are simple  if and only if $m=n$.

At last, we present the following fact, which is a simple
consequence of
Corollary~\ref{corol.Hurwitz.Matrix.Total.Nonnegativity.infinite.matrix.1}.

\begin{theorem}\label{Theorem.total.positivity.negative.zeros.1}
The polynomial
\begin{equation}\label{Theorem.total.positivity.negative.zeros.1.poly}
g(z)=a_0+a_1z+\cdots+a_nz^n,\quad a_0\neq0,\quad a_n>0
\end{equation}
has all negative zeros if and only if the infinite matrix
\begin{equation}\label{Theorem.total.positivity.negative.zeros.1.infinite.matrix}
\mathcal{D}_{\infty}(g) \eqbd
\begin{pmatrix}
a_0&a_1&a_2 &a_3 &a_4 &a_5 &a_6 &\dots\\
0  &a_1&2a_2&3a_3&4a_4&5a_5&6a_6&\dots\\
0  &a_0&a_1 &a_2 &a_3 &a_4 &a_5 &\dots\\
0  &0  &a_1 &2a_2&3a_3&4a_4&5a_5&\dots\\
0  &0  &a_0 &a_1 &a_2 &a_3 &a_4 &\dots\\
0  &0  &0   &a_1 &2a_2&3a_3&4a_4&\dots\\
\vdots &\vdots  &\vdots  &\vdots &\vdots &\vdots  &\vdots &\ddots
\end{pmatrix}
\end{equation}
is totally nonnegative.
\end{theorem}
\proof%
In fact, consider the logarithmic derivative $L$ of
the polynomial $g$
\begin{equation*}\label{Theorem.total.positivity.negative.zeros.1.proof.1}
{L}(z)=\dfrac{g'(z)}{g(z)}=\dfrac{a_1+2a_2z+\cdots+na_nz^{n-1}}{a_0+a_1z+a_2z^2+\cdots+a_nz^n}=
\dfrac{s_0}z+\dfrac{s_1}{z^2}+\dfrac{s_2}{z^3}+\cdots
\end{equation*}
Suppose that the polynomial $g$ has exactly $m\,(\leq n)$
distinct zeros, all of which are negative.
%:
%
%\begin{equation*}\label{Theorem.total.positivity.negative.zeros.1.proof.2}
%p(z)=a_n\prod_{j=1}^{m}(z+\lambda_j)^{n_j},\quad\lambda_j>0\ \
%\text{for}\ \  j=1,\ldots,m,
%\end{equation*}
%
%where $n_j$ is the multiplicity of the zero $\lambda_j$
%$(j=1,\ldots,m)$. T
Then by
Theorems~\ref{Th.Hankel.matrix.rank.2},~\ref{Th.real-rooted.criteria}
and~\ref{Th.number.of.real.roots.}, we have
\begin{eqnarray}\label{Theorem.total.positivity.negative.zeros.1.proof.3}
& D_j({L})>0, &  \quad j=1,2,\ldots,m, \\
\label{Theorem.total.positivity.negative.zeros.1.proof.4}
& (-1)^j\widehat{D}_j({L})>0, & \quad j=1,2,\ldots,m, \\
\label{Theorem.total.positivity.negative.zeros.1.proof.5}
& D_j({L})=\widehat{D}_j({L})=0, & \quad j>m.
\end{eqnarray}

According to Theorems~\ref{Th.R-function.general.properties}
and~\ref{Th.number.of.negative.poles.of.R-frunction},
from~\eqref{Theorem.total.positivity.negative.zeros.1.proof.3}--\eqref{Theorem.total.positivity.negative.zeros.1.proof.5}
we obtain that the logarithmic derivative~${L}$ of the
polynomial $g$ is an $R$-function of negative
type with all poles and zeros negative. Moreover, all common zeros
of the numerator and the denominator of the function
${L}$ are also negative. The converse statement is
obviously true.

If now we take $h= g'$ in
Corollary~\ref{corol.Hurwitz.Matrix.Total.Nonnegativity.infinite.matrix.1},
then ${R}={L}$ and ${H}_{\infty}(g,g')=\mathcal{D}_{\infty}(g)$.
So, the assertion of the theorem immediately follows from
Corollary~\ref{corol.Hurwitz.Matrix.Total.Nonnegativity.infinite.matrix.1}.
\eop

\begin{remark}
The necessary condition in this theorem is essentially known
(see~\cite[Remark~3.23]{Fisk}) but the~sufficient condition is
most likely new.
\end{remark}

For the
polynomial~\eqref{Theorem.total.positivity.negative.zeros.1.poly}
with negative zeros, the total nonnegativity of the
matrix~\eqref{Theorem.total.positivity.negative.zeros.1.infinite.matrix}
implies two properties of independent interest (cf.~\cite[p.~66]{Fisk}):

$$
a_j^2-a_{j-1}a_{j+1}=
\begin{vmatrix}
a_j&a_{j+1}\\
a_{j-1}&a_{j}
\end{vmatrix}\geq0,\quad j=1,2,\ldots,n-1 \qquad\text{(log-concavity)}
$$
and
$$
a_j^2-\dfrac{j+1}{j}\cdot a_{j-1}a_{j+1}=\dfrac1{j}\cdot
\begin{vmatrix}
ja_j&(j+1)a_{j+1}\\
a_{j-1}&a_{j}
\end{vmatrix}
\geq0,\quad j=1,2,\ldots,n-1 \qquad\text{(weak Newton's
inequalities)}
$$

\begin{remark}
From Theorem~\ref{Theorem.total.positivity.negative.zeros.1} one
can also prove Corollary~\ref{corol.roots.negativity.preserving}.
\end{remark}

\vspace{4mm}

At last, we present a couple of concrete examples.
\begin{example}\label{example.1}
Consider the following real polynomial
\begin{equation*}\label{example.1.polynomial}
f(z)=z^3+az+b,\quad a,b\in\mathbb{R}.
\end{equation*}
By our methods, we obtain a well-known fact~\cite{Kurosh} that the
polynomial $f$ has only real zeros if and only if the coefficients
$a$ and $b$ satisfy the inequality\footnote{We note that $4a^3+27b^2$
is the discriminant of the polynomial $f$.}
\begin{equation}\label{example.1.main.ineq}
4a^3+27b^2\leq0,
\end{equation}
At the same time, we show that the zeros of $f$ cannot all be of
the same sign.

In fact, construct the matrix $\mathcal{A}_{6}(f)$ corresponding
to the polynomial $f$
\begin{equation*}\label{example.1.main.matrix}
\mathcal{D}_6(f)=
\begin{pmatrix}
3&0&a&0&0&0\\
1&0&a&b&0&0\\
0&3&0&a&0&0\\
0&1&0&a&b&0\\
0&0&3&0&a&0\\
0&0&1&0&a&b
\end{pmatrix}
\end{equation*}
and find its leading principal minors $\delta_{j}(f)$
$(j=1,2,\ldots,6)$:
\begin{equation}\label{example.1.main.matrix.minors.1}
\delta_1(f)=3;\;\; \delta_2(f)=
\begin{vmatrix}
3&0\\
1&0
\end{vmatrix}=0;\;\;
\delta_3(f)=
\begin{vmatrix}
3&0&a\\
1&0&a\\
0&3&0
\end{vmatrix}=-6a;
\end{equation}
\begin{equation*}\label{example.1.main.matrix.minors.2}
\delta_4(f)=
\begin{vmatrix}
3&0&a&0\\
1&0&a&b\\
0&3&0&a\\
0&1&0&a
\end{vmatrix}=-4a^2;\;\;
\delta_5(f)=
\begin{vmatrix}
3&0&a&0&0\\
1&0&a&b&0\\
0&3&0&a&0\\
0&1&0&a&b\\
0&0&3&0&a
\end{vmatrix}=-4a^3-27b^2;
\end{equation*}
\begin{equation*}\label{example.1.main.matrix.minors.3}
\delta_6(f)=|\mathcal{D}_6(f)|=b\delta_5(f).
\end{equation*}
According to Theorem~\ref{Th.real.roots.ccc}, the polynomial $f$
has only real zeros if and only if the following inequalities
hold:
\begin{equation}\label{example.ineqs}
\delta_1(f)>0,\;\;\delta_3(f)\geq0,\;\;\delta_5(f)\geq0.
\end{equation}
Moreover, the equality $\delta_3(f)=0$ must imply $\delta_5(f)=0$.
But it is easy to see that the case $\delta_3(f)=\delta_5(f)=0$ is
possible if and only if $a=b=0$.

The first inequality~\eqref{example.ineqs} holds automatically
(see Remark~\ref{remark.new.2}). The second
inequality~\eqref{example.ineqs} gives $a\leq0$. From the
third inequality it follows that the necessary and sufficient
condition for the polynomial $f$ to have only real zeros is the
inequality~\eqref{example.1.main.ineq}. This inequality also holds
for $a=b=0$ and implies the inequality~$a\leq0$.

If the polynomial $f$ has only real zeros, then, according to
Corollaries~\ref{corol.neg.real.roots.via.Hurwitz.discr.determinants}--\ref{corol.negative.real.roots.via.Hurwitz.discr.determinants},
all its zeros are negative or positive if all its coefficients are
nonzero. But the coefficient of $f(z)$ at $z^2$ vanishes,
therefore, $f$ cannot have all zeros of the same sign.
%
%and only if the following inequality hold:
%
%\begin{equation*}\label{example.1.main.ineq.2}
%\delta_2(f)>0,\;\;\delta_4(f)\geq0,\;\;\delta_6(f)\geq0.
%\end{equation*}
%
%But these inequalities do not hold since $\delta_2(f)=0$
%(see~\eqref{example.1.main.matrix.minors.1}). Likewise,
%the polynomial $f$ cannot have only positive roots for any real
%$a$ and $b$, because the determinant $\delta_2(f)$ must be nonzero
%in this case too.

Thus, we proved that  $f$ has only real zeros if and
only if $a$ and $b$ satisfy the
inequality~\eqref{example.1.main.ineq}. We also proved that $f$
cannot have only positive or only negative zeros for any real $a$
and $b$.
\end{example}
\begin{example}\label{example.2}
From Theorem~\ref{Th.real.roots.ccc} it follows that if the
polynomial $p$ defined by~\eqref{polynomial.for.real.roots} is of
degree~$n\geq3$ with $a_1=a_2=0$ and $a_3\neq0$, then $p$
cannot have only real zeros since we have
$$\delta_3(p)=0,$$ and
$$\delta_5(p)=-9na_0^3a_3^2<0.$$ Therefore,
$$\sgn\delta_5(p)=-\sgn a_0=-\sgn\delta_1(p).$$
This contradicts the inequalities~\eqref{delta.pos.1}.
\end{example}

%\cite{Derkach, Akhiezer, Gesz_Simon, H4,Piv_Wor,Piv,Shohat_Tamarkin,Nik_Sor,Prasolov,
%Jones_Thron,Jones_Thron2,Lorentzen_Waadeland,Khruschev,Atkinson_rus}

%\newpage
%\addcontentsline{toc}{section}{Bibliography} %\dotfill}
%  .   .   .   .   .   .   .   .   .   .   .   .   .   .   .   .   .   .   .   .   .   .   .   .   .   .   .   .   .
%  .   .   .  .~~.~~.~~.~~.}
%
%.  .  . .~~~.~~~.~~~.~~~.~~~.~~~.~~~.~~~.~~~.~~~.~~~.~~~.~~~.~~~.~~~.~~~.~~~.~~~.~~~.~~~.~~~.~~~.~~~.~~~.~~~~~}
%\markboth{\thechapter\hspace{1em}Bibliography}{\thechapter\hspace{1em}Bibliography}%

\section*{Acknowledgments}
We are grateful to Victor Katsnelson, Sergei Khrushchev, Allan Pinkus,
Vyacheslav Pivovarchik, Cem Yildirim and the anonymous referees for 
useful comments.

\bibliographystyle{plain}
\bibliography{References}

\end{document}